\documentclass[11pt, a4paper]{article}
\usepackage{amsmath}
\usepackage{amsthm}
\usepackage{mathtools}
\usepackage{mathrsfs}
\usepackage[numbers]{natbib}
\usepackage{lmodern}
\usepackage{mathpazo}
\usepackage{amssymb}
\usepackage[T1]{fontenc}
\numberwithin{equation}{section}
\usepackage[margin=1in]{geometry} 
\usepackage{indentfirst}
\usepackage{enumitem}
\usepackage{array}
\usepackage{placeins}
\usepackage{xcolor}
\usepackage{hyperref}
\hypersetup{
    colorlinks=true,
    linkcolor=blue,
    citecolor=blue,
    urlcolor=blue,
    pdftitle={Classification of finite-dimensional Nichols algebras of rank two in twisted Yetter--Drinfeld categories},
    pdfauthor={Bowen Li},
    pdfsubject={Classification of finite-dimensional Nichols algebras of rank two},
    pdfkeywords={Nichols algebra, Yetter--Drinfeld module, coquasi-Hopf algebra, Cartan graph, Weyl groupoid}
}
\usepackage{cleveref}

\newtheorem{thm}{Theorem}[section]
\newtheorem{definition}[thm]{Definition}
\newtheorem{lemma}[thm]{Lemma}
\newtheorem{cor}[thm]{Corollary}
\newtheorem{prop}[thm]{Proposition}
\newtheorem{rmk}[thm]{Remark}
\newtheorem*{theorema}{Theorem A}
\newtheorem*{theoremb}{Theorem B}
\newcolumntype{L}[1]{>{\raggedright\arraybackslash}p{#1}}
\newcommand{\supp}{\operatorname{supp}}

\newcommand{\GG}{{^{G}_{G}\mathcal{YD}^{\Phi}}}

\newcommand{\brhd}{\blacktriangleright}

\newcommand{\bB}{\mathcal{B}}

\newcommand{\gaplink}[1]{%
  \href{https://github.com/lbwheizi/rank2/blob/835a6ad3388e06c3e67cc71832f20c8c3cc4613a/certificates/#1}%
  {\texttt{GAP source}}}
\makeatletter
\newcommand\blfootnote[1]{%
  \begingroup
  \renewcommand\thefootnote{}%
  \footnotetext[0]{%
    \begin{minipage}[t]{\dimexpr\linewidth-1.8em\relax}
      \setlength{\parindent}{0pt}%
      #1
    \end{minipage}%
  }%
  \endgroup
}
\makeatother
\title{Classification of finite-dimensional Nichols algebras of rank two in twisted Yetter--Drinfeld categories}
\author{Bowen Li}
\date{}

\begin{document}

\maketitle
\begin{center}
\itshape
Dedicated to the memory of my beloved cat Goudan.
\end{center}
\blfootnote{
\textit{Affiliation:}  {School of Mathematical Sciences, Shanghai Jiao Tong University.}

 \textit{Email:}
\href{mailto:lbwwysusu@gmail.com}%
{lbwwysusu@gmail.com}

}

\begin{abstract}
Let \(G\) be a finite non-abelian group, let
\(\Phi\in Z^3(G,\mathbb C^\times)\) be normalized, and let \(V,W\) be
finite-dimensional simple objects of
\({}_G^G\mathcal{YD}^{\Phi}\).  We classify, up to interchange, the
braided-indecomposable pairs \((V,W)\) whose supports generate \(G\)
and for which \(\mathcal B(V\oplus W)\) is finite-dimensional.  The
classification consists of eight cases with five possible support
quandles.  In every case the Cartan graph is standard of type \(A_2\),
\(B_2\), or \(G_2\), and the dimension is determined explicitly.  A
new phenomenon occurs  {in the \(\Gamma _2\) case}: the twisted setting admits a
family of type \(G_2\) absent from the ordinary \(\Gamma _2\)
classification and
we construct an explicit example over a non-abelian group of order
\(16\).
\end{abstract}
\section{Introduction}

Nichols algebras first appeared in the study of bialgebras of type one
\cite{Nichols1978}.  If \(V\) is an object in a braided tensor
category, its Nichols algebra \(\mathcal B(V)\) is the connected graded
braided Hopf algebra generated by \(V\),  {with primitive subobject \(V\)}.  Nichols algebras play a central role in the
lifting method for pointed Hopf algebras and are closely related to
quantum groups and root systems, see
\cite{ASpointed,AS10,rootsys}.  For braided vector spaces of diagonal
type, Weyl groupoids and arithmetic root systems led to the
classification of finite-dimensional Nichols algebras and to explicit
presentations by generators and relations
\cite{Hec06,Hec09,Angiono13,A15,AA17}.

For ordinary Yetter--Drinfeld modules over a non-abelian group, a
finite-dimensional simple object is determined by a finite conjugacy
class and an irreducible representation of the centralizer of a chosen
element.  The conjugation action on the class is part of the braiding
and places strong restrictions on finite-dimensional Nichols algebras,
see \cite{Bcop,AGn03,Gn00,Gn+11}.  Reflections and root systems for
semisimple Yetter--Drinfeld modules were developed in
\cite{AHS10,HS10,HY08,CH15}.  Finite-dimensional Nichols algebras
\(\mathcal B(V\oplus W)\), where \(V\) and \(W\) are simple
Yetter--Drinfeld modules, were studied in
\cite{rank2-1,rank2-2,rank2-3} and completely classified in
\cite{rank2-classification}.  The corresponding
higher-rank cases with finite root systems were classified in
\cite{HV17}.

Quasi-Hopf algebras were introduced by Drinfeld \cite{Drinfeld}, and
twisted quantum doubles of finite groups were constructed in
\cite{DPR}.  In this paper, \(G\) is finite and non-abelian, and we work
in the twisted Yetter--Drinfeld category
\({}_G^G\mathcal{YD}^{\Phi}\).  A simple object is determined by a
conjugacy class and an irreducible projective representation of the
corresponding centralizer. Yetter--Drinfeld modules and bosonization in
quasi-Hopf and coquasi-Hopf settings were studied in
\cite{rightyd,YD-module,quasibon}.  Related classifications of finite
quasi-quantum groups over abelian groups appear in
\cite{rank1,nondiagonal,QQG,FQQR2,huang2024classification}.

The associator changes the projective representations and the braided
adjoint maps.  Thus the ordinary classification cannot be applied
directly.  Reflection theory for Nichols algebras over coquasi-Hopf
algebras was established in
\cite{reflection1,reflection2}.   {The results on Cartan graphs, tensor decompositions, and
finite-dimensionality used below are proved in \cite{reflection3}.}

For a tuple \(M=(V,W)\), we say that \(M\) is
{braided-indecomposable} if
$
c_{W,V}c_{V,W}\ne\operatorname{id}_{V\otimes W}.
$
The main classification is the following.  Theorem~\ref{thm:complete-classification} gives an equivalent
formulation in terms of reflections.

\begin{theorema}
Let \(G\) be a finite non-abelian group, let
\(\Phi\in Z^3(G,\mathbb C^\times)\) be normalized, and let
\(M=(V,W)\) be a tuple of finite-dimensional simple objects in
\({}_G^G\mathcal{YD}^{\Phi}\).  Assume that
\(G=\langle\operatorname{supp}V\cup\operatorname{supp}W\rangle\)
and that \(M\) is braided-indecomposable.  Then
\(\mathcal B(V\oplus W)\) is finite-dimensional if and only if,
after possibly interchanging \(V\) and \(W\), the tuple satisfies one
of the alternatives in statement \textnormal{(3)} of
Theorem~\ref{thm:complete-classification},
equivalently one of the eight cases in
Theorems~\ref{thm:Gamma2-classification},
\ref{thm:Gamma4-classification},
\ref{thm:T-classification}, and
\ref{thm:Gamma3-classification}.
In these cases, 
the Cartan graph is standard of type \(A_2\), \(B_2\), or \(G_2\),
and the dimension is given in
Table~\ref{tab:complete-classification}.
\end{theorema}

The proof proceeds in two stages.  Suppose first that
\(\mathcal B(V\oplus W)\) is finite-dimensional.   {The finite-dimensionality criterion shows that \((V,W)\)
admits all reflections and has a finite Cartan graph.}  After finitely many reflections and, if necessary,
interchanging the entries, one reaches a pair whose Cartan matrix is
of type \(A_2\), \(B_2\), or \(G_2\).   {The corresponding vanishing conditions for the adjoint
objects imply that the support quandle is one of}
\[
Z_T^{4,1},\quad Z_2^{2,2},\quad Z_3^{3,1},\quad
Z_3^{3,2},\quad Z_4^{4,2}.
\]
Their enveloping groups show that \(G\) is an epimorphic image of
\(T,\Gamma _2,\Gamma _3,\Gamma _3\), or \(\Gamma _4\), respectively.
 {Thus necessity reduces to the classification theorems in
Sections~4--7.}   {These calculations yield exactly eight alternatives.
Applying the inverse sequence of reflections shows that, after
interchanging the entries if necessary, the original pair satisfies
one of them.}
Conversely, each alternative is covered by the  {corresponding classification theorem}, which proves finite-dimensionality and identifies the Cartan
graph.  The tensor decomposition gives the dimension formulas
in Table~\ref{tab:complete-classification}. 

Table~\ref{tab:complete-classification} lists the eight cases.  The notation in the dimension formulas is that of the
 {corresponding sections}.  The symbols
\(\mathcal P_i^\Phi\) denote the four \(\Gamma _3\) cases defined in
\eqref{eq:Gamma3-P1}--\eqref{eq:Gamma3-P4}.  The integers
 {%
\(N_\bullet\) in the \(\Gamma _2\)-\(A_2\) row are defined in
 {Section~4}.
}

\begin{table}[htbp]
\centering
\small
\renewcommand{\arraystretch}{1.25}
\begin{tabular}{@{}c c c p{0.38\textwidth}@{}}
\hline
Case & Quandle & Cartan type & \multicolumn{1}{c}{Dimension} \\
\hline
\(\Gamma _2\)-\(A_2\)
& \(Z_2^{2,2}\) & \(A_2\)
& \(64N_VN_{X_1}N_W\)
  (Corollary~\ref{cor:Gamma2-A2-Nichols-dimension}) \\
\(\Gamma _2\)-\(G_2\)
& \(Z_2^{2,2}\) & \(G_2\)
& \(2^{21}\) \\
\(\Gamma _3\)-\(\mathcal P_1^\Phi\)
& \(Z_3^{3,2}\) & \(B_2\)
& \(10368\) \\
\(\Gamma _3\)-\(\mathcal P_2^\Phi\)
& \(Z_3^{3,2}\) & \(B_2\)
& \(2304\) \\
\(\Gamma _3\)-\(\mathcal P_3^\Phi\)
& \(Z_3^{3,1}\) & \(B_2\)
& \(2304\) \\
\(\Gamma _3\)-\(\mathcal P_4^\Phi\)
& \(Z_3^{3,1}\) & \(B_2\)
& \(10368\) \\
\(\Gamma _4\)-\(B_2\)
& \(Z_4^{4,2}\) & \(B_2\)
&  {\(2^{18}\)}
  (Corollary~\ref{cor:Gamma4-Nichols-dimension}) \\
\(T\)-\(G_2\)
& \(Z_T^{4,1}\) & \(G_2\)
& \(6^3\,72^3=80\,621\,568\) \\
\hline
\end{tabular}
\caption{The eight finite-dimensional cases in Theorem~A.}
\label{tab:complete-classification}
\end{table}
\FloatBarrier

The main new result is the \(\Gamma _2\)-\(G_2\) case.
 {For the support quandle \(Z_2^{2,2}\), only Cartan type
\(A_2\) occurs in the ordinary finite-dimensional classification
\cite[Theorem~4.6]{rank2-1}.}  In the twisted setting, a family
of type \(G_2\) also occurs.  Thus the associator changes the
possible Cartan matrix for a fixed support quandle.
The following theorem gives the parameter conditions for this
family, together with its dimension and an explicit realization.

\begin{theoremb}
Let \(G\) be a finite non-abelian quotient of \(\Gamma _2\), let
\(\Phi\in Z^3(G,\mathbb C^\times)\) be normalized, and let
\[
V=M(g,\rho),\qquad W=M(h,\sigma).
\]
The inducing projective representations \(\rho\) and \(\sigma\) are
automatically one-dimensional by
Lemma~\ref{lem:Gamma2-projective-representations-one-dimensional}.
Assume that the tuple \((V,W)\) is braided-indecomposable.  Then
\(\mathcal B(V\oplus W)\) is finite-dimensional and the tuple has a
standard Cartan graph of type \(G_2\) if and only if, after possibly
interchanging \((g,V,\rho)\) and \((h,W,\sigma)\),  {the parameters for this ordering satisfy}
\[
\Delta=1,
\qquad
\lambda'=-1,
\qquad
\lambda^2=\Theta_\Phi(g,h)=-1.
\]
The parameters are defined in
\eqref{eq:Gamma2-parameters}, \eqref{eq:Delta-definition}, and
\eqref{eq:Theta-Phi-definition}.  In this case
\[
\dim\mathcal B(V\oplus W)=2^{21}.
\]
This case is realized by the explicit example over a non-abelian
group of order \(16\) in
Section~\ref{subsec:order16-G2-example}.
\end{theoremb}

In the \(\Gamma _3\), \(\Gamma _4\), and \(T\) cases,
finite-dimensionality forces the pullback of \(\Phi\) to the
corresponding enveloping group to be a coboundary. For
\(\Gamma _3\) and \(\Gamma _4\), this permits a reduction to the
ordinary classification after lifting the support degrees and
twisting.
In the \(T\) case, the adjoint calculations leave one possible
nontrivial cohomology class. We exclude this class by a sequence
of reflections producing infinitely many real roots.  {For the trivial class, the objects and their parameters are
determined by the calculations in \cite{rank2-3}.}

The \(\Gamma _2\)-\(A_2\) case is obtained from the same
adjoint calculations as the \(G_2\) case.  The \(\Gamma _4\)
case  {has} a standard Cartan graph of type \(B_2\), with
finite-dimensionality characterized by explicit conditions
on the cocycle and the inducing projective representations.
 {%
In every case, we determine the dimensions of the Nichols
algebras of the simple objects corresponding to the positive roots.
Their tensor decomposition gives the dimension formulas.
}

The paper is organized as follows.  Section~2 reviews simple objects,
Nichols algebras, braided adjoint actions, reflections, and Cartan
graphs.  Section~3 describes the five quandles and their enveloping
groups.  Sections~4--7 treat the \(\Gamma _2\), \(\Gamma _4\), \(T\),
and \(\Gamma _3\) cases, respectively.  Section~8 proves the complete
classification theorem.  The appendices contain the cocycle identities used in the proofs,
together with detailed coefficient calculations and computations in
the bar complex.

\section{Preliminaries}
Throughout the paper, the base field is
$\Bbbk=\mathbb C$.
\subsection{Twisted Yetter--Drinfeld modules and simple objects}
This subsection fixes the conventions for twisted Yetter--Drinfeld
modules used throughout the paper.  {We recall the action, tensor product, associativity constraint,
and braiding, and then describe simple objects in terms of projective
representations of centralizers.}

Let $G$ be a finite group and let $\Phi$ be a normalized $3$-cocycle
on $G$. We first recall Yetter--Drinfeld modules over the
coquasi-Hopf algebra $(\Bbbk G,\Phi)$. Let $V$ be a left
$\Bbbk G$-comodule. Denote its structure map by
\(\delta_L\colon V\longrightarrow\Bbbk G\otimes V\), and set
\[
V_g:={}^gV:=\{v\in V\mid\delta_L(v)=g\otimes v\},
\qquad
\supp V:=\{g\in G\mid V_g\ne0\}.
\]
Thus \(V=\bigoplus_{g\in G}{}^gV\).
 {An element $v\in{}^gV$ is called homogeneous of degree $g$,
and we write $\deg v=g$.}
The left $\Bbbk G$-comodule $(V,\delta_L)$ is a left-left
Yetter--Drinfeld module over $(\Bbbk G,\Phi)$ if there is a $\mathbb{C}$-linear map
$\rhd:\Bbbk G\otimes V\to V$ such that, for all $e,f\in G$ and
$v\in V_g$,
\begin{align}\label{eq:YD-projective-action}
&e \rhd (f \rhd v) = \frac{\Phi(e, f, g)\Phi(efgf^{-1}e^{-1}, e, f)}{\Phi(e, fgf^{-1}, f)} (ef) \rhd v, \\
&1_G \rhd v = v, \nonumber \\
&e \rhd v \in V_{ege^{-1}}. \nonumber
\end{align}

The category of left-left Yetter--Drinfeld modules over
$(\Bbbk G,\Phi)$ is denoted by $\GG$.
It is a $\mathbb{C}$-linear braided monoidal abelian category. For $M,N\in\GG$, the
coaction and action on $M\otimes N$ are given by
\begin{align}
\delta_L(m_g \otimes n_h) &:= gh \otimes m_g \otimes n_h, \nonumber \\
x \rhd (m_g \otimes n_h) &:= \frac{\Phi(x, g, h)\Phi(xgx^{-1}, xhx^{-1}, x)}{\Phi(xgx^{-1}, x, h)} x \rhd m_g \otimes x \rhd n_h,\label{eq:tensor-product}
\end{align}
for all $x, g, h \in G$ and $m_g \in {}^gM$, $n_h \in {}^hN$.
The associativity constraint $a$ and the braiding $c$ of $\GG$ are
given respectively by
\begin{align*}
a((u_e \otimes v_f) \otimes w_g) &= \Phi(e, f, g)^{-1}u_e \otimes (v_f \otimes w_g), \\
c(u_e \otimes v_f) &:= e \rhd v_f \otimes u_e,
\end{align*}
for all $e, f, g \in G$, $u_e \in {}^eU$, $v_f \in {}^fV$, $w_g \in {}^gW$ and $U, V, W \in \GG$.
These conventions are the specialization of
\cite[Definition~2.5 and equations~(2.9), (2.12)--(2.13)]{reflection1}
to $(\Bbbk G,\Phi)$.
For convenience, set
\begin{align*}
      {\Phi}_x(g,h)&:=\frac{\Phi(g, h, x)\Phi(ghxh^{-1}g^{-1}, g, h)}{\Phi(g, hxh^{-1}, h)},\\
      \Phi^x(g,h)&:=\frac{\Phi(x, g, h)\Phi(xgx^{-1}, xhx^{-1}, x)}{\Phi(xgx^{-1}, x, h)}.
\end{align*}
These functions satisfy the following relation.
\begin{lemma}
\label{lem:local-cocycle-coherence}
For all \(x,y,z,w\in G\), one has
\begin{equation}
\label{eq:local-cocycle-coherence}
\Phi_x(z,w)\Phi_x(y,zw)
=
\Phi_{wxw^{-1}}(y,z)\Phi_x(yz,w).
\end{equation}
\end{lemma}

\begin{proof}
Write \({}^{a}b=aba^{-1}\), and let
\(v={}^{w}x\), \(t={}^{z}v={}^{zw}x\), and
\(u={}^{y}t={}^{yzw}x\).
We use the normalized cocycle identity in the form
\begin{equation}
\label{eq:three-cocycle-identity-local}
\Phi(b,c,d)\Phi(a,bc,d)\Phi(a,b,c)
=\Phi(ab,c,d)\Phi(a,b,cd).
\end{equation}
After substituting the definition of \(\Phi_x\), the quotient
of the left-hand side by the right-hand side in
\eqref{eq:local-cocycle-coherence} is \(AB\), where
\begin{align*}
A&=
\frac{\Phi(z,w,x)\Phi(y,zw,x)\Phi(yz,v,w)}
{\Phi(yz,w,x)\Phi(z,v,w)\Phi(y,z,v)},\\
B&=
\frac{\Phi(t,z,w)\Phi(u,y,zw)\Phi(y,t,z)}
{\Phi(y,t,zw)\Phi(u,y,z)\Phi(u,yz,w)}.
\end{align*}
Applying \eqref{eq:three-cocycle-identity-local} to
\((y,z,w,x)\) and \((y,z,v,w)\), and using \(vw=wx\), gives
\[
A=\frac{\Phi(y,zv,w)}{\Phi(y,z,w)}.
\]
Applying the same identity to \((y,t,z,w)\) and \((u,y,z,w)\),
and using \(uy=yt\) and \(tz=zv\), gives
\[
B=\frac{\Phi(y,z,w)}{\Phi(y,tz,w)}
=\frac{\Phi(y,z,w)}{\Phi(y,zv,w)}.
\]
Thus \(AB=1\), which proves
\eqref{eq:local-cocycle-coherence}.
\end{proof}

\begin{cor}
If \(a,b,c\in C_G(x)\), then
\[
\Phi_x(b,c)\Phi_x(a,bc)
=
\Phi_x(a,b)\Phi_x(ab,c).
\]
Thus the restriction of \(\Phi_x\) to \(C_G(x)\) is a normalized
\(2\)-cocycle.
\end{cor}

\begin{proof}
If \(c\in C_G(x)\), then \(cxc^{-1}=x\), and the assertion follows
from Lemma~\ref{lem:local-cocycle-coherence}.
\end{proof}

Recall that a \(\Phi_x\)-projective representation of \(C_G(x)\) on a vector
space \(U\) is a map
\(\rho:C_G(x)\to\operatorname{GL}(U)\) satisfying
\(\rho(a)\rho(b)=\Phi_x(a,b)\rho(ab)\) for
\(a,b\in C_G(x)\). The simple objects in \(\GG\) are described as follows.
\begin{prop}\label{prop:simple-objects}
Let \(V\) be a simple object of \(\GG\), and let \(x\in\supp V\).
Then \(\supp V=x^G\), where
\(x^G=\{gxg^{-1}\mid g\in G\}\), and \(V_x\) is an irreducible
\(\Phi_x\)-projective representation of \(C_G(x)\). 

Conversely, for every \(x\in G\), each irreducible
\(\Phi_x\)-projective representation \(\rho\) of \(C_G(x)\)
determines a simple object \(M(x,\rho)\) with support \(x^G\).
The \(\Phi_x\)-projective representation of \(C_G(x)\) on
\(M(x,\rho)_x\) is isomorphic to \(\rho\).
After fixing one representative of each conjugacy class, the pairs
\((x,\rho)\), with \(\rho\) taken up to isomorphism, parametrize the
isomorphism classes of simple objects in \(\GG\).
\end{prop}
\begin{proof}
Let \(V\) be a simple object of \(\GG\). Since the sum of the
homogeneous components belonging to a fixed conjugacy class is a
subobject, the support of \(V\) is a single conjugacy class
\(\mathcal O\). Fix \(x\in\mathcal O\). The component \(V_x\) is
stable under \(C_G(x)\), and \eqref{eq:YD-projective-action} gives
\(a\rhd(b\rhd v)=\Phi_x(a,b)(ab)\rhd v\) for
\(a,b\in C_G(x)\) and \(v\in V_x\). Thus \(V_x\) is a
\(\Phi_x\)-projective representation of \(C_G(x)\).

This representation is irreducible. Indeed, suppose that
\(0\neq U\subsetneq V_x\) is \(C_G(x)\)-stable and set
\(N=\sum_{a\in G}a\rhd U\).
The relation \eqref{eq:YD-projective-action} implies that \(N\) is \(G\)-stable.
It is also graded and hence is a subobject of \(V\). Moreover, only
elements of \(C_G(x)\) contribute to the component of degree \(x\).
Therefore \(N_x=U\). Thus \(N\) is a nonzero proper subobject of \(V\),
contrary to the simplicity of \(V\).

Conversely, let \(\rho\) be an irreducible
\(\Phi_x\)-projective representation of \(C_G(x)\), with underlying
space \(W\). Choose elements \(g_y\in G\), \(y\in\mathcal O=x^G\),
such that \(g_yxg_y^{-1}=y\) and \(g_x=1\), and set
\[
M(x,\rho)=\bigoplus_{y\in\mathcal O}M_y,
\qquad
M_y=g_y\otimes W.
\]
For \(w\in W\), the comodule structure on \(M_y\) is
\(\delta_L(g_y\otimes w)=y\otimes(g_y\otimes w)\). The left
\(G\)-action on \(M\) is defined by
\(g\rhd(g_y\otimes w)=\gamma(g,g_y)
(g_{gyg^{-1}}\otimes\rho(c)w)\),
where $gg_y$ decomposes uniquely as $g_{gyg^{-1}}c$ with
$c\in C_G(x)$, and
\begin{equation}\label{eq:induced-action}
\begin{aligned}
\gamma(g,g_y)
=\frac{\Phi_x(g,g_y)}{\Phi_x(g_{gyg^{-1}},c)}
=\frac{
   \Phi(g,g_y,x)\Phi(gyg^{-1},g,g_y)
  }{
   \Phi(g,y,g_y)\Phi(g_{gyg^{-1}},c,x)
  }
  \frac{
   \Phi(g_{gyg^{-1}},x,c)
  }{
   \Phi(gyg^{-1},g_{gyg^{-1}},c)
  }.
\end{aligned}
\end{equation}
These formulas define an object of \({}_G^G\mathcal{YD}^{\Phi}\).
For \(a\in G\) and \(y\in\mathcal O\), let
\(c_{a,y}\in C_G(x)\) be determined by
\(ag_y=g_{aya^{-1}}c_{a,y}\).
Thus the action defined above takes \(M_y\) into \(M_{aya^{-1}}\).
Moreover, \(c_{1,y}=1\), and the normalization of \(\Phi_x\) gives
\(\gamma(1,g_y)=1\). Hence \(1\rhd m=m\) for all \(m\in M\).

 {To verify the first identity in \eqref{eq:YD-projective-action}, observe that}
$
c_{hg,y}
=
c_{h,gyg^{-1}}c_{g,y}.
$
Since \(\rho\) is a \(\Phi_x\)-projective representation of
\(C_G(x)\), we have
\[
\rho(c_{h,gyg^{-1}})\rho(c_{g,y})
=
\Phi_x(c_{h,gyg^{-1}},c_{g,y})\rho(c_{hg,y}).
\]
Let
$
z=gyg^{-1},\
u=hzh^{-1}=hgyg^{-1}h^{-1}.
$
Thus
\[
gg_y=g_zc_{g,y},
\qquad
hg_z=g_uc_{h,z},
\qquad
c_{hg,y}=c_{h,z}c_{g,y}.
\]
Applying \eqref{eq:local-cocycle-coherence}, respectively, to
$
(h,g,g_y),\
(h,g_z,c_{g,y}),\
(g_u,c_{h,z},c_{g,y}),
$
gives
\begin{align}
\Phi_x(g,g_y)\Phi_x(h,gg_y)
&=
\Phi_y(h,g)\Phi_x(hg,g_y),
\label{eq:gamma-reduction-1}
\\
\Phi_x(g_z,c_{g,y})\Phi_x(h,gg_y)
&=
\Phi_x(h,g_z)\Phi_x(hg_z,c_{g,y}),
\label{eq:gamma-reduction-2}
\\
\Phi_x(c_{h,z},c_{g,y})
\Phi_x(g_u,c_{h,z}c_{g,y})
&=
\Phi_x(g_u,c_{h,z})
\Phi_x(hg_z,c_{g,y}).
\label{eq:gamma-reduction-3}
\end{align}
Here \(c_{g,y},c_{h,z}\in C_G(x)\), while
\(g_zc_{g,y}=gg_y\) and \(g_uc_{h,z}=hg_z\).
{
Substituting \eqref{eq:gamma-reduction-2} and
\eqref{eq:gamma-reduction-3} into
\eqref{eq:gamma-reduction-1}, and using the definition of
\(\gamma\), gives
\[
\gamma(g,g_y)\gamma(h,g_z)
\Phi_x(c_{h,z},c_{g,y})
=\Phi_y(h,g)\gamma(hg,g_y).
\]
Since \(\rho\) is a \(\Phi_x\)-projective representation,
this is precisely the first identity in
\eqref{eq:YD-projective-action} for the induced action.
}
By construction, \(\supp M(x,\rho)=x^G\). If \(c\in C_G(x)\), then
\(g_x=1\), and normalization gives \(\gamma(c,1)=1\). Hence the
action of \(C_G(x)\) on \(M(x,\rho)_x\) is \(\rho\).

To prove that \(M(x,\rho)\) is simple, let
\(0\neq N\subseteq M(x,\rho)\) be a subobject and choose
\(y\in\mathcal O\) such that \(N_y\neq0\). Then
the action of \(g_y^{-1}\) is invertible by
\eqref{eq:YD-projective-action}. Hence
\(g_y^{-1}\rhd N_y\) is a nonzero subspace of \(N_x\). Since
\(N_x\) is \(C_G(x)\)-stable and \(\rho\) is irreducible, one has
\(N_x=W\). The transitivity of the \(G\)-action on \(\mathcal O\)
now gives \(N=M(x,\rho)\).
Finally, let \(\rho\) be the \(\Phi_x\)-projective representation of
\(C_G(x)\) on \(V_x\). For every \(y\in\mathcal O\), the action of
\(g_y\) is an isomorphism from \(V_x\) to \(V_y\). Therefore the
 map
\[
\theta:M(x,\rho)\longrightarrow V,
\qquad
\theta(g_y\otimes w)=g_y\rhd w,
\]
is an isomorphism of graded vector spaces. If
\(ag_y=g_{aya^{-1}}c_{a,y}\), then
\begin{align*}
\theta\bigl(a\rhd(g_y\otimes w)\bigr)
 &=
 \gamma(a,g_y)\,
 g_{aya^{-1}}\rhd(c_{a,y}\rhd w)  \\
 &=
 \Phi_x(a,g_y)(ag_y)\rhd w
 =
 a\rhd\theta(g_y\otimes w).
\end{align*}
Hence \(\theta\) is an isomorphism in \(\GG\).

If \(x,x'\) are fixed representatives and
\(M(x,\rho)\simeq M(x',\rho')\), equality of their supports gives
\(x=x'\). Restricting the isomorphism to degree \(x\) then gives
\(\rho\simeq\rho'\). This proves the parametrization.

\end{proof}

\subsection{Dual objects}
We record the duality conventions needed later. In particular, we
identify the projective representation defining the dual of a simple
object when the original projective representation is one-dimensional.

Let \(U\) be a finite-dimensional object of \(\GG\). For a homogeneous
basis \(\{u_j\}\), with \(u_j\in U_{x_j}\), let
\(\{u_j^\vee\}\) be the dual basis, so that
\(u_j^\vee\in(U^*)_{x_j^{-1}}\). We use the evaluation and
coevaluation maps
\[
\operatorname{ev}_U:U^*\otimes U\longrightarrow\mathbb C,
\qquad
\operatorname{coev}_U:\mathbb C\longrightarrow U\otimes U^*,
\qquad
\operatorname{coev}_U(1)
=\sum_j\Phi(x_j,x_j^{-1},x_j)u_j\otimes u_j^\vee.
\]
The  {triangle identities} follow from
\eqref{eq:cocycle-a-ainv-a}.
We shall use the following identity for duals.
\begin{lemma}
For every \(a\in G\), one has
\begin{equation}
\Phi(a,a^{-1},a)
\Phi(a^{-1},a,a^{-1})
=1.
\label{eq:cocycle-a-ainv-a}
\end{equation}
\end{lemma}

\begin{proof}
 {Applying} the normalized \(3\)-cocycle identity to
\((a,a^{-1},a,a^{-1})\) proves the assertion.
\end{proof}
\begin{lemma}
\label{lem:dual-projective-representation}
Let \(U=M(x,\chi)\), where \(\chi\) is one-dimensional, and choose
\(0\neq u\in U_x\). Let \(u^\vee\in(U^*)_{x^{-1}}\) satisfy 
\(u^\vee(u)=1\). For \(g\in C_G(x)\), write
\(g\rhd u=\chi(g)u\). Then
\(U^*\simeq M(x^{-1},\chi^*)\), where
\begin{equation}
\label{eq:dual-projective-representation}
\chi^*(g)
=
\frac{1}{\Phi^g(x^{-1},x)\chi(g)}.
\end{equation}
In particular, \(x^{-1}\rhd u^\vee=\chi(x)u^\vee\).
\end{lemma}

\begin{proof}
Since the evaluation is a morphism in \(\GG\), it gives
\begin{align*}
1
=
\operatorname{ev}_U(u^\vee\otimes u)
=
\operatorname{ev}_U\bigl(g\rhd(u^\vee\otimes u)\bigr)
=
\Phi^g(x^{-1},x)
\operatorname{ev}_U(g\rhd u^\vee\otimes g\rhd u).
\end{align*}
The component \((U^*)_{x^{-1}}\) is one-dimensional. Thus
\(g\rhd u^\vee=\chi^*(g)u^\vee\), and the preceding equality gives
\eqref{eq:dual-projective-representation}.
Duality preserves simplicity and reverses homogeneous degrees. Hence
\(U^*\) is simple with support \((x^{-1})^G\), and
Proposition~\ref{prop:simple-objects} identifies it with
\(M(x^{-1},\chi^*)\).
The relation \eqref{eq:YD-projective-action} on
\(U_x\) gives
\(
\chi(x^{-1})\chi(x)=\Phi_x(x^{-1},x).
\)
Hence
\[
\chi^*(x^{-1})
=\frac{\chi(x)}
{\Phi^{x^{-1}}(x^{-1},x)\Phi_x(x^{-1},x)}.
\]
By definition, the product in the denominator is
\(
\Phi(x^{-1},x,x^{-1})\Phi(x,x^{-1},x)=1
\)
by \eqref{eq:cocycle-a-ainv-a}. Therefore
\(\chi^*(x^{-1})=\chi(x)\).
\end{proof}

\subsection{Nichols algebras}
 {We recall the categorical definition of a Nichols algebra and its
behavior under braided monoidal equivalences.}

Let \(\mathcal C\) be a twisted Yetter--Drinfeld category over a
group, and let \(V\in\mathcal C\). The tensor algebra
\(T(V)=\bigoplus_{n\geq0}V^{\otimes n}\) is a graded braided Hopf
algebra. Its largest graded Hopf ideal
contained in \(\bigoplus_{n\geq2}T^n(V)\) is denoted by \(I(V)\), and
\[
\mathcal B(V)=T(V)/I(V).
\]
Equivalently, \(\mathcal B(V)\) is generated by
\(\mathcal B(V)_1=V\), has
\(\mathcal B(V)_0=\mathbf 1_{\mathcal C}\), and its primitive
subobject is \(V\), see \cite[Subsection~2.5]{reflection3}.

If \(F:\mathcal C\longrightarrow\mathcal D\) is a \(\Bbbk\)-linear
braided  monoidal equivalence of braided
monoidal abelian categories, then its tensor structure induces braided Hopf algebra
isomorphisms
\begin{equation}\label{eq:Nichols-under-equivalence}
T(F(V))\simeq F(T(V)),
\qquad
\mathcal B(F(V))\simeq F(\mathcal B(V)).
\end{equation}
See \cite[Lemma~2.16]{reflection1}.
\subsection{Braided decomposability}
We characterize when the Nichols algebra of a direct sum is  {isomorphic} to the
braided tensor product of the Nichols algebras of its two summands.
This criterion gives the notion of braided decomposability used
throughout the classification.

Let \(V,W\in\GG\) and let \(U=V\oplus W\). Denote by
\[
j_V:\mathcal B(V)\longrightarrow\mathcal B(U),
\qquad
j_W:\mathcal B(W)\longrightarrow\mathcal B(U)
\]
the canonical injective braided Hopf algebra morphisms.

\begin{lemma}\label{lem:canonical-nichols-factorization}
The following conditions are equivalent:

(1) The double braiding satisfies
\(c_{W,V}c_{V,W}=\operatorname{id}_{V\otimes W}\).

(2) The canonical morphisms \(j_V\) and \(j_W\) induce an
isomorphism of graded braided Hopf algebras
\[
m_{V,W}:
\mathcal B(V)\otimes\mathcal B(W)
\xrightarrow{\;\simeq\;}
\mathcal B(V\oplus W),
\qquad
a\otimes b\longmapsto j_V(a)j_W(b).
\]
\end{lemma}

\begin{proof}
Suppose first that
\(c_{W,V}c_{V,W}=\operatorname{id}_{V\otimes W}\).  {By the hexagon axioms, the double braiding between any tensor
power of \(V\) and any tensor power of \(W\) is the identity.
The same holds for \(\mathcal B(V)\) and \(\mathcal B(W)\).} Consequently,
\(A=\mathcal B(V)\otimes\mathcal B(W)\) is a connected
\(\mathbb N_0^2\)-graded braided Hopf algebra generated
by \(A_{(1,0)}\oplus A_{(0,1)}=V\oplus W\).

We claim that the primitive subobject of \(A\) is \(V\oplus W\).
Since the coproduct preserves the \(\mathbb N_0^2\)-grading, every
bihomogeneous component of a primitive element is primitive. Let
\(0\ne a\in A_{(p,q)}\). If \(p,q>0\), write
\(a=\sum_i u_i\otimes w_i\), where
\(u_i\in\mathcal B(V)_p\) and \(w_i\in\mathcal B(W)_q\). The
projection of the reduced coproduct of \(a\) onto
\(A_{(p,0)}\otimes A_{(0,q)}\) is
\[
\sum_i (u_i\otimes1)\otimes(1\otimes w_i).
\]
This is the image of \(a\) under the natural identification
\(A_{(p,q)}\simeq A_{(p,0)}\otimes A_{(0,q)}\), so it is nonzero.
Thus \(a\) is not primitive.

If \(q=0\), then \(a\) lies in the braided Hopf subalgebra
\(\mathcal B(V)\otimes1\), whose primitive subobject is \(V\).
Similarly, the primitive subobject of
\(1\otimes\mathcal B(W)\) is \(W\). Hence the primitive subobject of
\(A\) is \(V\oplus W\). Since \(A\) is connected, graded, and
generated by \(V\oplus W\), the characterization of Nichols algebras
 gives
\[
A\simeq\mathcal B(V\oplus W).
\]
This isomorphism restricts to the identity on \(V\oplus W\).
Therefore the isomorphism induced by \(j_V\) and \(j_W\) is
\(m_{V,W}\).

Conversely, suppose that \(m_{V,W}\) is an isomorphism. For
\(v\in V\) and \(w\in W\), multiplication in the braided tensor
product gives
\((1\otimes w)(v\otimes1)=c_{W,V}(w\otimes v)\). It follows that
\(E=w\otimes v-c_{W,V}(w\otimes v)\) belongs to the degree-two
component of the Nichols ideal \(I(U)\subseteq T(U)\). Since
\(I(U)_2=\ker(\operatorname{id}_{U\otimes U}+c_{U,U})\),
we obtain
\[
0=
\bigl(\operatorname{id}_{U\otimes U}+c_{U,U}\bigr)(E)
=
\bigl(
\operatorname{id}_{W\otimes V}
-c_{V,W}c_{W,V}
\bigr)(w\otimes v).
\]
Hence \(c_{V,W}c_{W,V}=\operatorname{id}_{W\otimes V}\), which is
equivalent to
\(c_{W,V}c_{V,W}=\operatorname{id}_{V\otimes W}\).
\end{proof}
\begin{definition}
We call \((V,W)\) \emph{braided-decomposable} if it satisfies the
equivalent conditions of Lemma~\ref{lem:canonical-nichols-factorization}.
Otherwise, it is \emph{braided-indecomposable}.
\end{definition}

\subsection{The braided adjoint action in Nichols algebras}
We describe the iterated braided adjoint objects by recursive
morphisms in the Nichols algebra. We then reduce their computation to
the images of a small number of homogeneous tensors.

Let \(\mathcal C=\GG\), let \(V,W\in\mathcal C\), and let
\(\mathcal B=\mathcal B(V\oplus W)\). Denote the canonical inclusions
by \(\iota_V:V\longrightarrow\mathcal B\) and
\(\iota_W:W\longrightarrow\mathcal B\), and denote by
\(\mu_{\mathcal B}\) the multiplication of \(\mathcal B\).
Since the elements of \(V\) are primitive, the restriction of the
left braided adjoint action of \(\mathcal B\) to
\(V\otimes\mathcal B\) is the morphism
\[
\operatorname{ad}_V
=
\mu_{\mathcal B}\circ
   (\iota_V\otimes\operatorname{id}_{\mathcal B})
-
\mu_{\mathcal B}\circ
   (\operatorname{id}_{\mathcal B}\otimes\iota_V)
   \circ c_{V,\mathcal B}.
\]

Set \(D_0=W\) and, recursively,
\(D_n=V\otimes D_{n-1}\) for \(n\geq1\). Define morphisms
\(\operatorname{ad}^{(n)}:D_n\longrightarrow\mathcal B\) by
\[
\operatorname{ad}^{(0)}=\iota_W,
\qquad
\operatorname{ad}^{(n)}
=
\operatorname{ad}_V\circ
\bigl(
\operatorname{id}_V\otimes\operatorname{ad}^{(n-1)}
\bigr),
\qquad n\geq1.
\]
All the maps involved are morphisms in \(\GG\), hence so is
\(\operatorname{ad}^{(n)}\). The \(n\)-th adjoint object of \(W\)
with respect to \(V\) is
\((\operatorname{ad}V)^n(W):=
\operatorname{Im}(\operatorname{ad}^{(n)})
\subseteq\mathcal B(V\oplus W)\).
Define morphisms
$\varphi_n^\Phi\in\operatorname{End}_{\mathcal C}(D_n)$ recursively.
For \(n=1\), set
\(\varphi_1^\Phi=\operatorname{id}_{V\otimes W}-c_{W,V}c_{V,W}\).
For \(n\geq2\), we define the
operator
\[
c_{1,2}^\Phi
=a_{V,V,D_{n-2}}\circ
 (c_{V,V}\otimes\operatorname{id}_{D_{n-2}})\circ
 a^{-1}_{V,V,D_{n-2}},
\]
and set
\[
\varphi_n^\Phi
=\operatorname{id}_{V\otimes D_{n-1}}
-c_{D_{n-1},V}c_{V,D_{n-1}}
+(\operatorname{id}_V\otimes\varphi_{n-1}^\Phi)\circ c_{1,2}^\Phi.
\]
Starting with \(X_0^{V,W}=W\), we regard
\(V\otimes X_{n-1}^{V,W}\) as a subobject of \(D_n\) and define
\[
X_n^{V,W}
=
\operatorname{Im}
\left(
\left.
\varphi_n^\Phi
\right|_{V\otimes X_{n-1}^{V,W}}
\right),
\qquad n\geq1.
\]

\begin{lemma}\label{lem:adjoint-object-realization}
For every \(n\in\mathbb N_0\), one has
\(X_n^{V,W}\cong(\operatorname{ad}V)^n(W)\) in \(\GG\).
\end{lemma}

\begin{proof}
By  \cite[Section~1]{JoyalStreet1993}, we can choose a strict braided monoidal category
\(\mathcal C^{\mathrm{str}}\) and a braided monoidal equivalence
\[
F:\mathcal C\longrightarrow\mathcal C^{\mathrm{str}}.
\] The braided structure is
transported along the equivalence as in
\cite[Example~2.4]{JoyalStreet1993}.
We transport the \(\Bbbk\)-linear abelian structure and countable
coproducts of \(\mathcal C\) to \(\mathcal C^{\mathrm{str}}\). The
tensor product in \(\mathcal C^{\mathrm{str}}\) then preserves these
coproducts, and \(F\) is \(\Bbbk\)-linear and exact.
Let
\(\xi_{A,B}:F(A)\otimes F(B)\to F(A\otimes B)\)
denote its tensor structure.

Set \(E_0=F(W)\) and \(E_n=F(V)\otimes E_{n-1}\). Define the canonical
isomorphisms \(\Xi_n:E_n\to F(D_n)\) recursively by
\(\Xi_0=\operatorname{id}_{F(W)}\) and
\(\Xi_n=\xi_{V,D_{n-1}}
(\operatorname{id}_{F(V)}\otimes\Xi_{n-1})\).
Let \(\varphi_n^{\mathrm{str}}\) be the operator on \(E_n\) defined by
the same recursion as \(\varphi_n^\Phi\) in the strict category.
These maps satisfy
\begin{equation}\label{eq:transport-varphi}
F(\varphi_n^\Phi)\circ\Xi_n
=
\Xi_n\circ\varphi_n^{\mathrm{str}},
\qquad n\geq1.
\end{equation}
For \(n=1\), this follows from the compatibility of \(F\) with the
braiding, applied twice. Suppose that
\eqref{eq:transport-varphi} holds for \(n-1\), and write
\(c^{\mathrm{str}}\) for the braiding of
\(\mathcal C^{\mathrm{str}}\). Naturality of \(\xi\), braidedness of
\(F\), and strictness give
\begin{align*}
&\Xi_n^{-1}
 F\!\left(c_{D_{n-1},V}\circ c_{V,D_{n-1}}\right)
 \Xi_n
 =
 c^{\mathrm{str}}_{E_{n-1},F(V)}
 \circ c^{\mathrm{str}}_{F(V),E_{n-1}},\\
&\Xi_n^{-1}F(c_{1,2}^{\Phi})\Xi_n
 =
 c^{\mathrm{str}}_{F(V),F(V)}
 \otimes\operatorname{id}_{E_{n-2}},\\
&\Xi_n^{-1}
 F\!\left(\operatorname{id}_V\otimes\varphi_{n-1}^{\Phi}\right)
 \Xi_n
 =
 \operatorname{id}_{F(V)}
 \otimes\varphi_{n-1}^{\mathrm{str}}.
\end{align*}
Substituting the three identities into the recursive
formula for \(\varphi_n^\Phi\) proves
\eqref{eq:transport-varphi}.

Let \(X_n^{\mathrm{str}}\subseteq E_n\) be defined by the same
recursion in \(\mathcal C^{\mathrm{str}}\). An induction
shows that \(\Xi_n\) identifies \(X_n^{\mathrm{str}}\) with
\(F(X_n^{V,W})\). The claim holds for \(n=0\) by the definitions. If it holds for \(n-1\),
then the tensor structure identifies
\(F(V)\otimes X_{n-1}^{\mathrm{str}}\) with
\(F(V\otimes X_{n-1}^{V,W})\), while
\eqref{eq:transport-varphi} identifies the two restricted morphisms
whose images define \(X_n^{\mathrm{str}}\) and \(F(X_n^{V,W})\).
Exactness of \(F\) therefore identifies their images. Thus
\(\Xi_n\) restricts to an isomorphism
\(X_n^{\mathrm{str}}\xrightarrow{\sim}F(X_n^{V,W})\).

The proof of \cite[Theorem~13.3.1]{rootsys} uses only the strict
braided structure. Applied in \(\mathcal C^{\mathrm{str}}\), it gives
\(X_n^{\mathrm{str}}\cong
(\operatorname{ad}F(V))^n(F(W))\).
Moreover, the universal property of Nichols algebras yields a braided
Hopf algebra isomorphism
\(F(\mathcal B_{\mathcal C}(V\oplus W))
 \cong
 \mathcal B_{\mathcal C^{\mathrm{str}}}(F(V)\oplus F(W))\).
Since \(F\) is monoidal and braided, it carries the multiplication and
braiding in each adjoint morphism to the corresponding maps in
\(\mathcal C^{\mathrm{str}}\).  Since an equivalence preserves images,
\[
F(X_n^{V,W})
\cong X_n^{\mathrm{str}}
\cong(\operatorname{ad}F(V))^n(F(W))
\cong F\bigl((\operatorname{ad}V)^n(W)\bigr).
\]
Since \(F\) is an equivalence, the assertion follows.
\end{proof}

Since
$
X_1^{V,W}
=
\operatorname{Im}\bigl(
\operatorname{id}_{V\otimes W}-c_{W,V}c_{V,W}
\bigr),
$
Lemma~\ref{lem:adjoint-object-realization} gives
\begin{equation}\label{eq:indecomposable-first-adjoint}
(V,W)\text{ is braided-indecomposable}
\quad\Longleftrightarrow\quad
(\operatorname{ad}V)(W)\neq0.
\end{equation}

For a subgroup \(H\leq G\) and a homogeneous vector \(u\), write
\[
\Bbbk H\rhd u
:=
\operatorname{span}_{\Bbbk}\{a\rhd u\mid a\in H\}.
\]
The relation \eqref{eq:YD-projective-action} implies that
\(\Bbbk G\rhd u\) is \(G\)-stable. The following results reduce the
computations in Sections~4--7 to one or a few homogeneous vectors.
\begin{lemma}\label{lem:homogeneous-orbit-tensor-generation}
Let \(V,X\in\GG\), and let \(0\neq x\in X_s\) be homogeneous with
\(X=\Bbbk G\rhd x\). Suppose that \(H\leq G\) and that
\(0\neq v_j\in V_{r_j}\), \(j\in J\), are homogeneous vectors
satisfying \(H\rhd x\subseteq\Bbbk x\) and
\(V=\sum_{j\in J}\Bbbk H\rhd v_j\). Then
\[
V\otimes X=\sum_{j\in J}\Bbbk G\rhd(v_j\otimes x).
\]
In particular, if \(V=\Bbbk H\rhd v\), then
\(V\otimes X=\Bbbk G\rhd(v\otimes x)\).
\end{lemma}

\begin{proof}
Set \(L=\sum_{j\in J}\Bbbk G\rhd(v_j\otimes x)\). It follows from
\eqref{eq:YD-projective-action} that \(L\) is \(G\)-stable.
For \(f\in H\), write \(f\rhd x=\chi_f x\), where
\(\chi_f\in \Bbbk^\times\). The relation \eqref{eq:tensor-product} gives
\(f\rhd(v_j\otimes x)
 =\Phi^f(r_j,s)\chi_f(f\rhd v_j)\otimes x\).
Thus \((f\rhd v_j)\otimes x\in L\), and consequently
\(V\otimes x\subseteq L\).

Now let \(u\in V_t\) be homogeneous and \(g\in G\). Since
\((g^{-1}\rhd u)\otimes x\in L\) and \(L\) is \(G\)-stable,
\[
g\rhd\bigl((g^{-1}\rhd u)\otimes x\bigr)
 =
\Phi^g(g^{-1}tg,s)\Phi_t(g,g^{-1})
u\otimes(g\rhd x)\in L.
\]
The coefficient is nonzero, so \(u\otimes(g\rhd x)\in L\).
Since homogeneous vectors span \(V\) and \(X=\Bbbk G\rhd x\), this proves
\(V\otimes X\subseteq L\). Each summand defining \(L\) lies in
\(V\otimes X\), which gives the reverse inclusion.
\end{proof}

Write \(X_n=X_n^{V,W}\). Regard
\(V\otimes X_{n-1}\) as a subobject of \(D_n\) and denote the
restriction of \(\varphi_n^\Phi\) to this subobject by
\(\psi_n\). Thus \(X_n=\operatorname{Im}\psi_n\).

\begin{prop}\label{prop:adjoint-orbit-generation}
Let \(n\geq1\), and choose homogeneous vectors
\(0\neq v\in V_r\) and
\(0\neq x_{n-1}\in(X_{n-1})_s\).
Suppose that, for some subgroup \(H\leq G\),
\(V=\Bbbk H\rhd v\), \(X_{n-1}=\Bbbk G\rhd x_{n-1}\), and
\(H\rhd x_{n-1}\subseteq\Bbbk x_{n-1}\).
Set \(x_n:=\psi_n(v\otimes x_{n-1})\). Then
\(X_n=\Bbbk G\rhd x_n\). Consequently,
\(X_n=0\) if and only if \(x_n=0\). Moreover, if \(x_n\neq0\), then
\(x_n\in(X_n)_{rs}\).
\end{prop}

\begin{proof}
Lemma~\ref{lem:homogeneous-orbit-tensor-generation} gives
\(V\otimes X_{n-1}=\Bbbk G\rhd(v\otimes x_{n-1})\).
Since \(\psi_n\) is a morphism in \(\GG\), it follows that
\[
X_n
 =\psi_n(V\otimes X_{n-1})
 =\Bbbk G\rhd\psi_n(v\otimes x_{n-1})
 =\Bbbk G\rhd x_n.
\]
Hence, \(X_n=0\) if and only if \(x_n=0\). Finally,
\(v\otimes x_{n-1}\) has degree \(rs\), and \(\psi_n\), being a
comodule morphism, preserves the \(G\)-grading.
\end{proof}

\subsection{Racks and enveloping groups}
We recall the relation between conjugation quandles and groups. In
particular, the support of a nonzero twisted Yetter--Drinfeld module
is a quandle under conjugation.
\begin{definition}
A rack is a nonempty set $X$ with a binary operation
\(\blacktriangleright\) such that each map
\(f_a(x)=a\blacktriangleright x\) is bijective and
\[
a\blacktriangleright(b\blacktriangleright c)
=(a\blacktriangleright b)\blacktriangleright
 (a\blacktriangleright c)
\]
for all $a,b,c\in X$. A quandle is a rack satisfying
\(a\blacktriangleright a=a\) for every $a\in X$.
\end{definition}
For racks $X$ and $Y$, a map $f:X\to Y$ is a rack morphism if
\(f(a\blacktriangleright b)=f(a)\blacktriangleright f(b)\) for all
\(a,b\in X\).
For a finite quandle $X=\{1,\ldots,n\}$, the notation
\(X:f_1\cdots f_n\) records its left translations. The conjugacy
class of $g$ in a group $G$, with its conjugation quandle
structure, is denoted by $g^G$.  {The inner automorphism group of $X$ is}
\(\operatorname{Inn}(X)=\langle f_a\mid a\in X\rangle\).
The quandle is indecomposable if this group acts transitively on $X$.

\begin{definition}
The enveloping group of a quandle $X$ is
\[
G_X=\langle x\in X\mid xyx^{-1}=x\blacktriangleright y
\text{ for all }x,y\in X\rangle.
\]
\end{definition}
Any group $G$ defines the conjugation quandle $\operatorname{Conj}(G)$,
with $a\brhd b=aba^{-1}$. This gives the conjugation functor from
groups to quandles, and the enveloping-group construction is its left
adjoint.
\begin{lemma}
There is a natural bijection
\[
\operatorname{Hom}_{\mathbf{Qnd}}(X,\operatorname{Conj}(G))
\cong \operatorname{Hom}_{\mathbf{Grp}}(G_X,G).
\]
\end{lemma}
\begin{proof}
This is the universal property of the displayed presentation of
\(G_X\).
\end{proof}

\begin{lemma}
Let \(0\neq V\in\GG\). Then \(\supp V\) is a quandle
under $g\brhd h=ghg^{-1}$.
\end{lemma}
\begin{proof}
The Yetter--Drinfeld compatibility makes $\supp V$ a union of
conjugacy classes. The quandle axioms are the corresponding identities
for conjugation in $G$.
\end{proof}
\subsection{Cartan graphs of Nichols algebras}
We recall the rank-two reflection theory used throughout the
classification. {
The main results collected here are the tensor decompositions,
the finite-dimensionality criterion, and the invariance of
finite-dimensionality and dimension under reflections.
} We use the notation of
\cite[Sections~2.1 and~2.6]{reflection3}.
 {The results in this subsection hold for every finite group
\(G\) and every normalized \(3\)-cocycle \(\Phi\) on \(G\).}

Let \(\mathbb I=\{1,2\}\), and let \(P=(P_1,P_2)\) be a tuple of
finite-dimensional simple objects in \(\GG\). Write
\(\mathcal B(P)=\mathcal B(P_1\oplus P_2)\), with its standard
\(\mathbb N_0^2\)-grading \(\deg P_i=\alpha_i\), where
\(\alpha_1,\alpha_2\) are the standard basis vectors of \(\mathbb Z^2\).
For tuples
\(Q=(Q_1,Q_2)\) and \(Q'=(Q'_1,Q'_2)\), write \(Q\simeq Q'\) if
\(Q_i\simeq Q'_i\) for every \(i\in\mathbb I\), and denote the
componentwise isomorphism class of \(Q\) by \([Q]\).

\begin{definition}
For \(i\ne j\), the tuple \(P\) admits the \(i\)-th reflection if
there exists an integer \(m_{ij}^P\geq0\) such that
$
(\operatorname{ad}P_i)^{m_{ij}^P}(P_j)
\text{ is nonzero}
$ and 
$(\operatorname{ad}P_i)^{m_{ij}^P+1}(P_j)=0.$
In this case, define
\[
R_i(P)_j =
\begin{cases}
P_i^*, & \text{if } j = i, \\
(\operatorname{ad}P_i)^{m_{ij}^P}(P_j), & \text{if } j \neq i.
\end{cases}
\]
If \(P\) admits both reflections, define
$a_{ii}^{[P]}=2,$
$a_{ij}^{[P]}=-m_{ij}^P.$
\end{definition}

\begin{lemma}\label{lem:reflection-basic-properties}
Suppose that \(P\) admits the \(i\)-th reflection. Then \(R_i(P)\) is
a tuple of finite-dimensional simple objects, and its isomorphism
class depends only on \([P]\). Moreover, \(R_i(P)\) admits the
\(i\)-th reflection,
\[
R_i^2(P)\simeq P,
\qquad
m_{ij}^{R_i(P)}=m_{ij}^P\quad(j\ne i).
\]
Consequently, whenever both Cartan matrices are defined,
\(A^{[P]}\) and \(A^{[R_i(P)]}\) have the same \(i\)-th row.
\end{lemma}

\begin{proof}
The  {last nonzero} adjoint object is simple by
\cite[Lemma~5.8]{reflection2}. Functoriality of the adjoint
construction gives the assertion about isomorphism classes. The
remaining assertions are
\cite[Lemma~5.9\textnormal{(2)} and
Corollary~5.15\textnormal{(1)}]{reflection2}.
\end{proof}

Moreover, \cite[Remark~2.26]{reflection3} shows that
\(A^{[P]}=(a_{ij}^{[P]})_{i,j\in\mathbb I}\) is a generalized Cartan
matrix whenever \(P\) admits both reflections.

\begin{definition}
The tuple \(P\) admits the empty sequence. For \(l\geq1\), it admits
\((i_1,\ldots,i_l)\) if it admits the \(i_1\)-th reflection and
\(R_{i_1}(P)\) admits
\((i_2,\ldots,i_l)\). The tuple \(P\) admits all reflections if it
admits every finite sequence in \(\mathbb I\). For such a tuple, let
\[
\mathcal F_2(P)
=\bigl\{R_{i_l}(\cdots R_{i_1}(P)\cdots)
\mid l\in\mathbb N_0,\ i_1,\ldots,i_l\in\mathbb I\bigr\}.
\]
Define
\[
\mathcal X(P)=\{[Q]\mid Q\in\mathcal F_2(P)\},
\qquad
r_i([Q])=[R_i(Q)],
\]
and set
\[
\mathcal G(P)=
\bigl(\mathbb I,\mathcal X(P),(r_i)_{i\in\mathbb I},
(A^{[Q]})_{[Q]\in\mathcal X(P)}\bigr).
\]
\end{definition}

\begin{prop}\label{prop:associated-Cartan-graph}
If \(P\) admits all reflections, then \(\mathcal G(P)\) is a connected
Cartan graph.
\end{prop}

\begin{proof}
This follows from
\cite[Remark~2.26 and Theorem~4.6]{reflection3}.
\end{proof}

\begin{definition}
Assume that \(P\) admits all reflections. For
\([Q]\in\mathcal X(P)\), define
$
s_i^{[Q]}(\alpha_j)=\alpha_j-a_{ij}^{[Q]}\alpha_i.
$
The Weyl groupoid \(\mathcal W(\mathcal G(P))\) is generated by the
morphisms \(s_i^{[Q]}:[Q]\to r_i([Q])\). The real roots at \([Q]\)
and the positive real roots are
\[
\Delta^{[Q],\operatorname{re}}
=\{w(\alpha_i)\mid i\in\mathbb I,
\ w:[Q']\to[Q]\text{ for some }[Q']\in\mathcal X(P)\},
\qquad
\Delta_+^{[Q],\operatorname{re}}
=\Delta^{[Q],\operatorname{re}}\cap\mathbb N_0^2.
\]
The Cartan graph is finite if \(\Delta^{[Q],\operatorname{re}}\) is
finite for every \([Q]\in\mathcal X(P)\). It is standard if its
Cartan matrix is the same at every object.
\end{definition}

\begin{prop}\label{prop:positive-root-factorization}
Assume that \(P\) admits all reflections and that \(\mathcal G(P)\) is
finite. Fix \(Q\in\mathcal F_2(P)\). Let \(w_0\) be a longest
morphism with target \([Q]\), choose a reduced expression
\(w_0=\operatorname{id}_{[Q]}s_{i_1}\cdots s_{i_l}\), and let
\[
\beta_k=
\operatorname{id}_{[Q]}s_{i_1}\cdots s_{i_{k-1}}(\alpha_{i_k}),
\qquad 1\leq k\leq l.
\]
Then
$\{\beta_1,\ldots,\beta_l\}
=\Delta_+^{[Q],\operatorname{re}},
$
 {and there are finite-dimensional simple objects
\(Q_{\beta_1},\ldots,Q_{\beta_l}\), with \(Q_{\beta_k}\)
having \(\mathbb N_0^2\)-degree \(\beta_k\), such that}
\[
\mathcal B(Q_{\beta_l})\otimes
\bigl(\cdots(\mathcal B(Q_{\beta_2})\otimes
\mathcal B(Q_{\beta_1}))\bigr)
\xrightarrow{\ \sim\ }\mathcal B(Q)
\]
as \(\mathbb N_0^2\)-graded objects.
\end{prop}

\begin{proof}
This is \cite[Corollary~5.9]{reflection3}.
\end{proof}

\begin{lemma}\label{lem:root-factors-and-reflected-components}
Under the hypotheses and notation of
Proposition~\ref{prop:positive-root-factorization}, let
\(S\in\mathcal F_2(P)\), \(i\in\mathbb I\), and let
\(w:[S]\to[Q]\) be a morphism in \(\mathcal W(\mathcal G(P))\). If
\(\beta_k=w(\alpha_i)\), then \(Q_{\beta_k}\simeq S_i\).
\end{lemma}

\begin{proof}
This  is \cite[Lemma~5.10]{reflection3}. 
\end{proof}

\begin{lemma}\label{lem:root-string-adjoint-simplicity}
Under the hypotheses and notation of
Proposition~\ref{prop:positive-root-factorization}, let
\(i\ne j\) and \(0\leq t\leq-a_{ij}^{[Q]}\). Then there exists
\(k\in\{1,\ldots,l\}\) such that
\[
\beta_k=\alpha_j+t\alpha_i,
\qquad
(\operatorname{ad}Q_i)^t(Q_j)\simeq Q_{\beta_k}.
\]
In particular, \((\operatorname{ad}Q_i)^t(Q_j)\) is simple.
\end{lemma}

\begin{proof}
This is \cite[Proposition~5.11]{reflection3}.
\end{proof}

\begin{thm}\label{thm:Nichols-finiteness-criterion}
Let \(P=(P_1,P_2)\) be a tuple of finite-dimensional simple objects in
\(\GG\). Then \(\mathcal B(P)\) is finite-dimensional if and only if
the following conditions hold:

 (1)
 \(P\) admits all reflections,
 
 (2) \(\mathcal G(P)\) is finite and \(\mathcal B(Q_i)\) is finite-dimensional for every
      \(Q\in\mathcal F_2(P)\) and \(i\in\mathbb I\).

\end{thm}

\begin{proof}
This is \cite[Theorem~5.12]{reflection3}.
\end{proof}

\begin{lemma}\label{lem:reflection-dimension-invariance}
Let \(P=(P_1,P_2)\) be a tuple of finite-dimensional simple objects in
\(\GG\). If \(P\) admits the \(i\)-th reflection, then
\begin{equation*}
\dim\mathcal B(R_i(P))<\infty
\quad\Longleftrightarrow\quad
\dim\mathcal B(P)<\infty,
\qquad
\dim\mathcal B(R_i(P))=\dim\mathcal B(P)
\quad\text{in this case}.
\end{equation*}
\end{lemma}

\begin{proof}
By \cite[Lemma~6.2]{reflection1}, if \(\mathcal B(P)\) is
finite-dimensional, then so is \(\mathcal B(R_i(P))\), and their
dimensions are equal. Moreover,
Lemma~\ref{lem:reflection-basic-properties} shows that \(R_i(P)\)
admits the \(i\)-th reflection and that \(R_i^2(P)\simeq P\).
Applying the same implication to \(R_i(P)\)
gives the converse.
\end{proof}
\subsection{Finite Cartan graphs of rank two}
We record two consequences used in the case-by-case classification.
Let $G$ be a finite non-abelian group and let $\Phi$ be a normalized
$3$-cocycle on
$G$. Let $V$ and $W$ be finite-dimensional simple Yetter--Drinfeld
modules in $\GG$, and let $M=(V,W)$. If $\bB(V\oplus W)$ is
finite-dimensional, then
Theorem~\ref{thm:Nichols-finiteness-criterion} shows that $M$ admits
all reflections and $\mathcal{G}(M)$ is a finite Cartan graph.
It is connected by Proposition~\ref{prop:associated-Cartan-graph}.

A connected Cartan graph is called indecomposable if one of its
Cartan matrices is indecomposable. In rank two, this is equivalent to
every Cartan matrix being indecomposable. Indeed, the \(i\)-th row is
unchanged along an \(i\)-edge, and \(a_{ij}=0\) if and only if
\(a_{ji}=0\).
If, in addition, \(M\) is braided-indecomposable, then by
\eqref{eq:indecomposable-first-adjoint}, \(m_{12}^M>0\). Since
\(A^{[M]}\) is a generalized Cartan matrix, also \(m_{21}^M>0\).
Thus \(A^{[M]}\), and hence \(\mathcal G(M)\), is indecomposable.

\begin{lemma}\label{lem:finite-type-object}
Let $\mathcal{G}$ be a connected indecomposable finite Cartan
graph. Then there exists $X\in \mathcal{X}$ such that $A^X$ is of finite type.
\end{lemma}
\begin{proof}
This is \cite[Theorem~4.2]{HV17}.
\end{proof}
\begin{prop}\label{prop:pair-with-finite-Cartan-matrix}
Let \(M=(V,W)\) be a tuple of finite-dimensional simple objects in
\(\GG\).
Assume that \(M\) is braided-indecomposable, admits all reflections,
and has a finite Cartan graph \(\mathcal G(M)\). Then there exists
\(N=(N_1,N_2)\in\mathcal F_2(M)\) such that
\[
1\leq a_{12}^{[N]}a_{21}^{[N]}\leq3.
\]
\end{prop}
\begin{proof}
Since \(M\) is braided-indecomposable, \(\mathcal{G}(M)\) is
indecomposable. Lemma~\ref{lem:finite-type-object} therefore gives an object
\(X\in\mathcal X(M)\) for which \(A^X\) is of finite type. By the
definition of \(\mathcal X(M)\), write \(X=[N]\) for some
\(N\in\mathcal F_2(M)\). Connectedness and indecomposability imply
that \(A^{[N]}\) is indecomposable. The indecomposable rank-two
generalized Cartan matrices of finite type are those of types
\(A_2\), \(B_2\), and \(G_2\), together with their transposes. Hence
\[
a_{12}^{[N]}a_{21}^{[N]}\in\{1,2,3\},
\]
as claimed.
\end{proof}
For a tuple \(P=(P_1,P_2)\) in \(\GG\), set
\(G_P=\langle\supp P_1\cup\supp P_2\rangle\).

\begin{lemma}\label{lem:reflections-preserve-support-group}
Let \(P=(P_1,P_2)\) be a tuple of finite-dimensional simple objects in
\(\GG\). If \(P\) admits the \(i\)-th reflection, then
\[
G_{R_i(P)}=G_P.
\]
If \(P\) admits all reflections, then consequently
\(G_Q=G_P\) for every \(Q\in\mathcal F_2(P)\).
\end{lemma}

\begin{proof}
Write \(R_i(P)=(P'_1,P'_2)\). The reflected \(i\)-th component is
\(P'_i=P_i^*\), and therefore
\(\supp P'_i=(\supp P_i)^{-1}\subseteq G_P\). Let \(j\neq i\) and let
\(m=m_{ij}^P\). Then \(P'_j=(\operatorname{ad}P_i)^m(P_j)\) is
the image of a grading-preserving morphism from
\(P_i^{\otimes m}\otimes P_j\). Hence
\[
\supp P'_j
\subseteq
(\supp P_i)^m\supp P_j
\subseteq G_P.
\]
It follows that \(G_{R_i(P)}\subseteq G_P\).

The tuple \(R_i(P)\) again admits the \(i\)-th reflection, and
\(R_i^2(P)\simeq P\) by
Lemma~\ref{lem:reflection-basic-properties}. Applying the
preceding inclusion to \(R_i(P)\) gives
\(G_P=G_{R_i^2(P)}\subseteq G_{R_i(P)}\). Thus
\(G_{R_i(P)}=G_P\). The equality \(G_Q=G_P\) for every
\(Q\in\mathcal F_2(P)\) follows by induction on the length of a
reflection sequence.
\end{proof}
\section{Support quandles and their enveloping groups}
\label{sec:support-quandles}
We first describe five quandles and their enveloping groups.
We then establish the support classification in
Theorem~\ref{thm:support-classification} under the stated adjoint
vanishing conditions. The resulting enveloping groups are
\(T,\Gamma_2,\Gamma_3\), and \(\Gamma_4\), which determine the
four cases considered in the following sections. The five quandles are
\begin{align*}
Z_T^{4,1} & : (243) \ (134) \ (142) \ (123) \ \text{id}, \\
Z_2^{2,2} & : (24) \ (13) \ (24) \ (13), \\
Z_3^{3,1} & : (23) \ (13) \ (12) \ \text{id}, \\
Z_3^{3,2} & : (23)(45) \ (13)(45) \ (12)(45) \ (123) \ (132), \\
Z_4^{4,2} & : (24)(56) \ (13)(56) \ (24)(56) \ (13)(56) \ (1234) \ (1432).
\end{align*}

\subsection{The quandle \texorpdfstring{$Z_T^{4,1}$}{ZT(4,1)}
and its enveloping group}
\begin{lemma}\label{lem:ZT-enveloping-group}
The enveloping group of \(Z_T^{4,1}\) has the presentation
\[
T=
\left\langle x_1,x_2,z\ \middle|\
zx_1=x_1z,\quad
zx_2=x_2z,\quad
x_1x_2x_1=x_2x_1x_2,\quad
x_1^3=x_2^3
\right\rangle.
\]
Set
$x_3=x_2x_1x_2^{-1}$,
$x_4=x_1x_2x_1^{-1}.$
Then \(z\) is central, \(x_1^T=\{x_1,x_2,x_3,x_4\}\), and
\[
z^T\cup x_1^T\cong Z_T^{4,1}
\]
as quandles.
\end{lemma}
\begin{proof}
This is the presentation in \cite[Section~2.1]{rank2-3}.
\end{proof}

\subsection{The groups \texorpdfstring{$\Gamma_n$}{Gamma n}
and their support quandles}
\label{subsec:enveloping-groups-Gamma}
Let \(n\in\mathbb N_{\geq2}\). Recall from
\cite[Section~3]{rank2-1} that
\[
\Gamma_n=
\left\langle g,h,\varepsilon\ \middle|\
hg=\varepsilon gh,\quad
g\varepsilon=\varepsilon^{-1}g,\quad
h\varepsilon=\varepsilon h,\quad
\varepsilon^n=1
\right\rangle.
\]
Every element of \(\Gamma_n\) has a unique expression
\(\varepsilon^ih^jg^k\), where \(0\leq i<n\) and
\(j,k\in\mathbb Z\). For
\(z\in Z(\Gamma_n)=\langle\varepsilon^{-1}h^2,h^n,g^2\rangle\)
and \(1\leq j\leq\lfloor n/2\rfloor\), the conjugacy classes are
\begin{align*}
z^{\Gamma_n}&=\{z\},&
(gz)^{\Gamma_n}&=\{\varepsilon^mgz\mid0\leq m<n\},\\
(h^jz)^{\Gamma_n}&=\{h^jz,\varepsilon^{-j}h^jz\},&
(hgz)^{\Gamma_n}&=\{\varepsilon^mhgz\mid0\leq m<n\}.
\end{align*}
The centralizers of the noncentral representatives are
\begin{align*}
C_{\Gamma_n}(gz)
 &=\langle\varepsilon^{-1}h^2,g,h^n\rangle,&
C_{\Gamma_n}(hgz)
 &=\langle\varepsilon^{-1}h^2,hg,h^n\rangle,\\
C_{\Gamma_n}(h^jz)
 &=\langle\varepsilon,h,g^2\rangle.
\end{align*}
\begin{prop}
\label{prop:support-quandle-enveloping-groups}
Let \(Z_2^{2,2}=h^{\Gamma_2}\cup g^{\Gamma_2}\). Then
\(Z_2^{2,2}\cong D_4\), the dihedral quandle of four elements. Its
enveloping group is
\[
\left\langle x_1,x_2,x_3,x_4\ \middle|\
x_ix_j=x_{2i-j\ (\mathrm{mod}\ 4)}x_i,
\quad i,j\in\{1,2,3,4\}\right\rangle\cong\Gamma_2.
\]
The isomorphism sends \(x_1\) to \(g\) and \(x_2\) to \(h\).

Let \(Z_3^{3,1}=g^{\Gamma_3}\cup\{\varepsilon h\}\). The element
\(\varepsilon h\) is central, and the enveloping group of
\(Z_3^{3,1}\) is
\[
\langle z\rangle\times
\left\langle x_1,x_2,x_3\ \middle|\
x_ix_j=x_{2i-j\ (\mathrm{mod}\ 3)}x_i,
\quad i,j\in\{1,2,3\}\right\rangle\cong\Gamma_3.
\]
The latter isomorphism sends \(z\) to \(\varepsilon h\), \(x_1\) to
\(g\), \(x_2\) to \(\varepsilon g\), and \(x_3\) to
\(\varepsilon^2g\).

For \(Z_3^{3,2}=g^{\Gamma_3}\cup h^{\Gamma_3}\), the enveloping
group is \(\Gamma_3\).

For \(Z_4^{4,2}=g^{\Gamma_4}\cup h^{\Gamma_4}\), the enveloping
group is \(\Gamma_4\).
\end{prop}

\begin{proof}
These are the enveloping-group identifications in
\cite[Examples~2.4--2.7]{rank2-2}.
\end{proof}

\subsection{Classification of the possible support quandles}
\label{subsec:support-restrictions}

Throughout this subsection, let \(G\) be a non-abelian group, let
\(\Phi\) be a normalized \(3\)-cocycle on \(G\), and let
\(V,W\in{}_G^G\mathcal{YD}^{\Phi}\). Set \(D_0=W\) and
\(D_m=V\otimes D_{m-1}\) for \(m\geq1\).

For \(s\in\supp W\), let \(Q_0(s)=W_s\). If \(m\geq1\) and
\(r_1,\ldots,r_m\in\supp V\), define recursively
\[
Q_m(r_1,\ldots,r_m;s)
 :=
\varphi_m^\Phi\bigl(
V_{r_1}\otimes Q_{m-1}(r_2,\ldots,r_m;s)
\bigr)
\subseteq D_m.
\]
Thus \(X_m^{V,W}\) is the sum of the spaces
\(Q_m(r_1,\ldots,r_m;s)\).
 For \(t\in D_m\), let \(\supp t\) be the set of multidegrees
of its nonzero components. For a subspace \(Q\subseteq D_m\),
set \(\supp Q=\bigcup_{t\in Q}\supp t\). Since
\(X_m^{V,W}\cong(\operatorname{ad}V)^m(W)\), we have
\begin{equation}\label{eq:Qm-vanishing}
(\operatorname{ad}V)^m(W)=0
\quad\Longleftrightarrow\quad
Q_m(r_1,\ldots,r_m;s)=0
\end{equation}
for all \(r_1,\ldots,r_m\in\supp V\) and \(s\in\supp W\).

\begin{prop}\label{prop:multidegree-insertion}
Let \(m\in\mathbb N_0\). Suppose that
\(p_1,\ldots,p_m,r_1,\ldots,r_m\in\supp V\),
\(p_{m+1},s\in\supp W\), and
\[
(p_1,\ldots,p_m,p_{m+1})
 \in\supp Q_m(r_1,\ldots,r_m;s).
\]
Let \(p\in\supp V\) and \(i\in\{1,\ldots,m+1\}\). Assume that
\begin{align}
&p_i\brhd p\neq p,
\ \ \
p_j\brhd p=p
\ \text{for } \ i<j\leq m+1,
\label{eq:multidegree-insertion-1}\\
&p\notin
\{p_j:1\leq j\leq m\}
\cup
\{(p_{j+1}\cdots p_{m+1})^{-1}\brhd p_j:1\leq j<i\}.
\label{eq:multidegree-insertion-2}
\end{align}
Then
\[
(p\brhd p_1,\ldots,p\brhd p_{i-1},
 p,p_i,\ldots,p_m,p_{m+1})
\in
\supp Q_{m+1}(p,r_1,\ldots,r_m;s).
\]
\end{prop}

\begin{proof}
Choose \(t\in Q_m(r_1,\ldots,r_m;s)\) whose component of
multidegree \((p_1,\ldots,p_m,p_{m+1})\) is nonzero, and take
\(0\neq v\in V_p\). By the recursive definition of \(Q_{m+1}\),
the element \(\varphi_{m+1}^{\Phi}(v\otimes t)\) belongs to
\(Q_{m+1}(p,r_1,\ldots,r_m;s)\).

Associators preserve multidegrees and multiply homogeneous tensors by
nonzero scalars. The braiding is an isomorphism and  {replaces
each adjacent pair of degrees \((a,b)\) by \((aba^{-1},a)\)}. Hence each individual summand arising from a
nonzero homogeneous component is nonzero. For
\(1\leq j\leq m+1\), its multidegree has one of the forms
\[
\begin{aligned}
&(p\brhd p'_1,\ldots,p\brhd p'_{j-1},
  p,p'_j,\ldots,p'_m,p'_{m+1}),\\
&(p\brhd p'_1,\ldots,p\brhd p'_{j-1},
  (pp'_j\cdots p'_{m+1})\brhd p,
  p\brhd p'_j,\ldots,p\brhd p'_m,p\brhd p'_{m+1}),
\end{aligned}
\]
where
\((p'_1,\ldots,p'_m,p'_{m+1})\in\supp t\).

The required multidegree occurs in the first family with \(j=i\) and
\((p'_1,\ldots,p'_{m+1})=(p_1,\ldots,p_{m+1})\). In the first family,
equality with the required multidegree gives \(p=p_j\) if \(j<i\),
\(p'_\ell=p_\ell\) for all \(1\leq\ell\leq m+1\) if \(j=i\), and
\(p=p_{j-1}\) if \(j>i\). In the second family, it gives
\(p=p_{j-1}\) if \(j>i\), \(p_i\brhd p=p\) if \(j=i\), and
\(p=(p_{j+1}\cdots p_{m+1})^{-1}\brhd p_j\) if \(j<i\). Conditions
\eqref{eq:multidegree-insertion-1} and
\eqref{eq:multidegree-insertion-2} exclude every case except the first
family with \(j=i\). This is the comparison used in
\cite[Proposition~5.5]{rank2-2}. Thus the required component is the
image of the chosen nonzero component under a single isomorphism and
is therefore nonzero.
\end{proof}
 {The preceding results on adjoint objects and multidegrees give
the following classification of support quandles.}

\begin{thm}\label{thm:support-classification}
Let \(G\) be a non-abelian group and let
\(V,W\in{}_G^G\mathcal{YD}^{\Phi}\) be simple objects. Assume that
\(G=\langle\supp V\cup\supp W\rangle\), that \((V,W)\) is
braided-indecomposable, and that
\[
(\operatorname{ad}V)^2(W)=0,
\qquad
(\operatorname{ad}W)^4(V)=0.
\]
Then \(\supp(V\oplus W)\), with its  quandle structure, is
isomorphic to one of
\[
Z_T^{4,1},\qquad
Z_2^{2,2},\qquad
Z_3^{3,1},\qquad
Z_3^{3,2},\qquad
Z_4^{4,2}.
\]
Accordingly, \(G\) is an epimorphic image of
\(T,\Gamma_2,\Gamma_3,\Gamma_3\), or \(\Gamma_4\), respectively.
\end{thm}
\begin{proof}
By Proposition~\ref{prop:simple-objects}, \(\supp V\) and \(\supp W\)
are conjugacy classes. Moreover,
\eqref{eq:indecomposable-first-adjoint} gives
\((\operatorname{ad}V)(W)\neq0\). Thus the group-theoretic hypotheses
in the proof of \cite[Theorem~4.4]{rank2-2} are satisfied.

The group-theoretic Lemma~5.1 of \cite{rank2-2} applies unchanged.
Lemma~\ref{lem:adjoint-object-realization} and
\eqref{eq:Qm-vanishing} replace
\cite[Lemmas~5.2--5.3 and Remark~5.4]{rank2-2}, while
Proposition~\ref{prop:multidegree-insertion} replaces
\cite[Proposition~5.5]{rank2-2}.  Its case \(m=0\) also gives
\cite[Remark~5.6]{rank2-2}. The arguments in
\cite[Corollaries~5.7--5.11, Lemmas~5.12--5.13 and
5.15--5.24]{rank2-2} therefore remain valid, since they use only
these results and the conjugation relations in
\(\supp V\cup\supp W\). 

If \(\supp V\) and \(\supp W\) commute elementwise,
\cite[Proposition~5.14]{rank2-2} gives
\(\supp(V\oplus W)\cong Z_3^{3,1}\) or \(Z_T^{4,1}\).
Otherwise, \cite[Proposition~5.25]{rank2-2} gives
\(\supp(V\oplus W)\cong Z_2^{2,2}\), \(Z_3^{3,2}\), or
\(Z_4^{4,2}\).

Let \(X=\supp(V\oplus W)\).  The inclusion
\(X\hookrightarrow\operatorname{Conj}(G)\) induces a homomorphism
\(G_X\to G\) from the enveloping group of \(X\). It is surjective
because \(G=\langle X\rangle\). Lemma~\ref{lem:ZT-enveloping-group}
and Proposition~\ref{prop:support-quandle-enveloping-groups}
identify \(G_X\) in the five cases and give the asserted epimorphic
images.
\end{proof}
\section{The \texorpdfstring{$\Gamma_2$}{Gamma2} case}
\label{sec:classification-Gamma2}

Let
\[
\Gamma_2
=
\left\langle
g,h,\varepsilon
\ \middle|\
hg=\varepsilon gh,\quad
\varepsilon^2=1,\quad
\varepsilon g=g\varepsilon, \quad \varepsilon h=h \varepsilon
\right\rangle,
\]
and let \(G\) be a finite non-abelian quotient of \(\Gamma_2\).
The images of \(g,h,\varepsilon\) in \(G\) are denoted by the same
letters.
Let \(\Phi\) be a normalized \(3\)-cocycle on \(G\).

Throughout this section, let \(V=M(g,\rho)\) be simple, and let
\(W=M(h,\sigma)\) be simple.  The support of \(V\) is
\(\{g,\varepsilon g\}\), and the support of \(W\) is
\(\{h,\varepsilon h\}\).  Assume that
\((V,W)\) is braided-indecomposable.

Lemma~\ref{lem:Gamma2-projective-representations-one-dimensional}
below shows that \(\rho\) and \(\sigma\) are one-dimensional under
the standing assumptions.  The parameters used below are defined in
\eqref{eq:Gamma2-parameters}, \eqref{eq:Delta-definition}, and
\eqref{eq:Theta-Phi-definition}.  Since their definitions depend on
the order of \(V\) and \(W\), they must be computed separately for
the two orders.  The classification is as follows.

\begin{thm}
\label{thm:Gamma2-classification}
Under the preceding assumptions,
\(\dim\mathcal B(V\oplus W)<\infty\) if and only if, for one of the
two orders of \(V\) and \(W\), one of the following two systems of
equations holds:
\begin{align}
&\Delta=1,
\qquad \lambda=\lambda'=-1,
\tag{A$_2$}\label{eq:Gamma2-classification-A2}
\\
&\Delta=1,
\qquad \lambda'=-1,
\qquad \lambda^2=\Theta_\Phi(g,h)=-1.
\tag{$\Gamma_2$-G$_2$}\label{eq:Gamma2-classification-G2}
\end{align}
\end{thm}

\medskip
\noindent\textit{Outline of the proof.}
\begin{itemize}[leftmargin=2em]
\item
Subsections~\ref{subsec:Gamma2-projective-representations}
and~\ref{subsec:Gamma2-adjoint-objects} compute the first adjoint
objects and show that the relevant projective representations are
one-dimensional.

\item
Subsection~\ref{subsec:Gamma2-duals-reflections} computes the duals
and the parameters after each reflection.

\item
Subsection~\ref{subsec:Gamma2-standard-A2} proves the \(A_2\) case,
and Subsection~\ref{subsec:Gamma2-local-B2-obstruction} shows that
type \(B_2\) cannot occur.

\item
Subsection~\ref{subsec:Gamma2-standard-G2-parameters}  {computes} the remaining
adjoint objects in the \(G_2\) case and  {proves} that its conditions are
preserved by reflections.

\item
Subsection~\ref{subsec:Gamma2-finite-Cartan-graphs} shows that every
finite-dimensional case has a standard Cartan graph of type \(A_2\)
or \(G_2\).

\item
Subsections~\ref{subsec:Gamma2-G2-simple-factors}
and~\ref{subsec:Gamma2-classification} prove finite-dimensionality
and compute the dimensions.

\item
Subsection~\ref{subsec:order16-G2-example} gives an explicit example
of the \(G_2\) case.
\end{itemize}
\medskip
\medskip
\noindent\textit{Parameters used in the theorem.}
The following table records where the parameters are defined and how
they are used.

\begingroup
\small
\setlength{\tabcolsep}{4pt}
\renewcommand{\arraystretch}{1.08}
\begin{center}
\begin{tabular}{@{}L{0.17\textwidth}L{0.25\textwidth}L{0.50\textwidth}@{}}
\hline
Parameter & Defined in & Use \\
\hline

\(\lambda\)
& \eqref{eq:Gamma2-parameters}
& If \(\Delta=1\), then \(X_2=0\) if and only if
  \(\lambda=-1\). \\

\(\lambda'\)
& \eqref{eq:Gamma2-parameters}
& If \(\Delta'=1\), then \(Y_2=0\) if and only if
  \(\lambda'=-1\). \\

\(\Delta,\Delta'\)
& \eqref{eq:Delta-definition},
  \eqref{eq:Delta-WV-definition}
& The object \(X_1\) is simple if and only if \(\Delta=1\), and
  \(Y_1\) is simple if and only if \(\Delta'=1\).  These two
  conditions are equivalent. \\

\(\Theta_\Phi(g,h)\)
& \eqref{eq:Theta-Phi-definition}
& If \(\Delta=1\) and \(X_2\neq0\), then \(X_2\) is simple if and only if
  \(\lambda^2=\Theta_\Phi(g,h)\).  This condition occurs in the
  \(G_2\) case. \\
\hline
\end{tabular}
\end{center}
\endgroup

\subsection{Projective representations}
\label{subsec:Gamma2-projective-representations}

The support representatives that arise below have the form \(g^m h\)
or \(h^m g\).  By Proposition~\ref{prop:simple-objects}, a simple
object with support \(x^G\) is determined by an irreducible
\(\Phi_x\)-projective representation of \(C_G(x)\).  We therefore
begin by computing the relevant centralizers.   {Corollary~\ref{cor:Gamma2-change-of-generators} will allow us to
apply the same argument to the projective representations defining
the higher adjoint objects.}

\begin{lemma}
\label{lem:Gamma2-normal-forms-centralizers}
The center of \(G\) is
$
Z(G)=\langle\varepsilon,g^2,h^2\rangle.
$
The centralizers needed below are
\begin{align*}
C_G(g)&=\langle\varepsilon,g,h^2\rangle,\\
C_G(h)&=\langle\varepsilon,h,g^2\rangle,\\
C_G(gh)=C_G(hg)
&=\langle\varepsilon,g^2,g^{-1}h\rangle
 =\langle\varepsilon,g^2,h^2,gh\rangle.
\end{align*}
  More generally, for every
\(m\in\mathbb Z\),
\[
C_G(g^m h)=
\begin{cases}
C_G(h),&m\text{ even},\\
C_G(gh),&m\text{ odd},
\end{cases}
\qquad
C_G(h^m g)=
\begin{cases}
C_G(g),&m\text{ even},\\
C_G(gh),&m\text{ odd}.
\end{cases}
\]
\end{lemma}

\begin{proof}
Since \(G\) is non-abelian, \(\varepsilon\neq1\), and hence
\(\varepsilon\) has order \(2\).  Every element of \(G\) can be
written as
\(\varepsilon^e g^a h^b\), where \(e\in\{0,1\}\) and
\(a,b\in\mathbb Z\).
Let
$u=\varepsilon^e g^a h^b,\ 
v=\varepsilon^f g^c h^d.$
Since \(h^b g^c=\varepsilon^{bc}g^c h^b\), one has
$
uv=\varepsilon^{e+f+bc}g^{a+c}h^{b+d}$,
$
vu=\varepsilon^{e+f+ad}g^{a+c}h^{b+d}$.
Therefore
$
uv=\varepsilon^{bc-ad}vu.
$
Since \(\varepsilon\neq1\), the elements \(u\) and \(v\) commute if
and only if \(bc-ad\) is even.
Taking \(v=g\) shows that \(u\in C_G(g)\) if and only if \(b\) is
even.  Taking \(v=h\) shows that \(u\in C_G(h)\) if and only if \(a\)
is even.  Thus
\[
C_G(g)=\langle\varepsilon,g,h^2\rangle,
\qquad
C_G(h)=\langle\varepsilon,h,g^2\rangle.
\]
An element is central if and only if both \(a\) and \(b\) are even,
which gives
$
Z(G)=\langle\varepsilon,g^2,h^2\rangle.
$

Taking \(v=gh\) shows that \(u\in C_G(gh)\) if and only if
\(a\equiv b\pmod 2\).  Hence
\[
C_G(gh)
=\langle\varepsilon,g^2,h^2,gh\rangle.
\]
If \(k=g^{-1}h\), then
$
k^2=\varepsilon g^{-2}h^2$,
$
gh=g^2k$,
and therefore
\[
C_G(gh)=\langle\varepsilon,g^2,g^{-1}h\rangle.
\]
The three  centralizers are each generated by central
elements and one additional element, so they are abelian.
Finally, applying the same commutation criterion to \(g^m h\) gives
\[
C_G(g^m h)=
\begin{cases}
C_G(h),&m\text{ even},\\
C_G(gh),&m\text{ odd}.
\end{cases}
\]
Applying it to \(h^m g\) gives the second formula.
\end{proof}

\begin{cor}
\label{cor:Gamma2-change-of-generators}
For every \(m\in\mathbb Z\), each of the pairs
$
(g^m h,g)$,
$
(h^m g,h)
$
generates \(G\).  If \((x,y)\) denotes either pair, then
\[
yx=\varepsilon xy,
\qquad
C_G(x)=\langle\varepsilon,x,y^2\rangle.
\]
Thus \(x,y,\varepsilon\) satisfy the defining relations of
\(\Gamma_2\).
\end{cor}

\begin{proof}
First let \(x=g^m h\) and \(y=g\).  Then
\(yx=\varepsilon xy\) and \(h=y^{-m}x\).  Next let
\(x=h^m g\) and \(y=h\).  Again \(yx=\varepsilon xy\), and now
\(g=y^{-m}x\).  In either case,
\[
\varepsilon=yxy^{-1}x^{-1},
\]
so \(x\) and \(y\) generate \(G\).  The formula for \(C_G(x)\)
follows from Lemma~\ref{lem:Gamma2-normal-forms-centralizers}.
\end{proof}

The dimensions of the projective representations \(\rho\) and
\(\sigma\) are controlled by the following alternator of \(\Phi\):
\[
f_\Phi(a,b,c):=
\frac{
 \Phi(a,b,c)\Phi(b,c,a)\Phi(c,a,b)
}{
 \Phi(a,c,b)\Phi(c,b,a)\Phi(b,a,c)
}.
\]

\begin{lemma}
\label{lem:alternator}
\label{lem:central-double-slant-character}
Let \(A\leq G\) be an abelian subgroup. Then
\(f_\Phi|_{A\times A\times A}\) is an alternating tricharacter. In
particular,
\[
f_\Phi(ab,c,d)=f_\Phi(a,c,d)f_\Phi(b,c,d),
\]
and similarly in the second and third variables, while
$
f_\Phi(a,a,b)=1,
\
f_\Phi(a,b,c)=f_\Phi(b,a,c)^{-1}.
$
Moreover, if \(a,b\in Z(G)\), then the map
\[ \chi:
G\longrightarrow\mathbb C^\times,
\qquad
x\longmapsto f_\Phi(x,a,b),
\]
is a group character.
\end{lemma}

\begin{proof}
Suppose that \(a,b\in G\) commute and that
\(x,y\in C_G(a)\cap C_G(b)\). Direct computation gives
\[
\Phi_{xy}(a,b)
=
\Phi_x(a,b)\Phi_y(a,b)
\frac{\Phi_a(x,y)\Phi_b(x,y)}
     {\Phi_{ab}(x,y)}.
\]
Interchanging \(a\) and \(b\) leaves the last factor unchanged.
Taking the quotient of the two identities therefore gives
\[
f_\Phi(xy,a,b)
=
f_\Phi(x,a,b)f_\Phi(y,a,b).
\]

Taking \(x,y,a,b\in A\) proves that \(f_\Phi|_{A^3}\) is
multiplicative in its first variable. Its defining formula gives
\[
f_\Phi(a,b,c)=f_\Phi(b,c,a)
             =f_\Phi(b,a,c)^{-1}.
\]
Hence it is multiplicative in all three variables. The same formula
gives \(f_\Phi(a,a,c)=1\), so \(f_\Phi|_{A^3}\) is an alternating
tricharacter.

Finally, if \(a,b\in Z(G)\), then
\(C_G(a)\cap C_G(b)=G\). The multiplicativity proved above, together
with \(f_\Phi(1,a,b)=1\), shows that
\(x\mapsto f_\Phi(x,a,b)\) is a group character.
\end{proof}
\begin{lemma}\label{lem:Gamma2-projective-representations-one-dimensional}
Every irreducible $\Phi_g$-projective representation of $C_G(g)$
is one-dimensional. More generally, for every
\(m\in\mathbb Z\) and
\(x\in\{g^m h,h^m g\}\), every irreducible
\(\Phi_x\)-projective representation of \(C_G(x)\) is
one-dimensional.
\end{lemma}

\begin{proof}
Put \(\vartheta_g=f_\Phi(g,\varepsilon,h^2)\) and
\(c_g(a,b)=\Phi_g(a,b)/\Phi_g(b,a)\) for \(a,b\in C_G(g)\).
By Lemma~\ref{lem:Gamma2-normal-forms-centralizers}, this centralizer
is abelian. Writing
\(a=\varepsilon^{x_1}g^{y_1}(h^2)^{z_1}\) and
\(b=\varepsilon^{x_2}g^{y_2}(h^2)^{z_2}\),
Lemma~\ref{lem:alternator} gives
\(c_g(a,b)=f_\Phi(g,a,b)=\vartheta_g^{\,x_1z_2-x_2z_1}\).
Since \(\varepsilon^2=1\), one has
\(\vartheta_g^2=f_\Phi(g,\varepsilon^2,h^2)=1\).

Assume, to the contrary, that \(\vartheta_g=-1\).
For $a \in C_G(g)$,
one has \(c_g(g,a)=f_\Phi(g,g,a)=1\).
The element corresponding to \(g\) is therefore central in the twisted
group algebra \(\mathbb C^{\Phi_g}C_G(g)\).  Schur's lemma gives
$
g\rhd v=\lambda v
$
for every $v\in V_g$, where $\lambda\in\mathbb C^\times$.
Since
\[
\varepsilon \rhd (\varepsilon \rhd v)
=\Phi_g(\varepsilon,\varepsilon)(\varepsilon^2 \rhd v)
=\Phi_g(\varepsilon,\varepsilon)v,
\]
there is a nonzero eigenvector, denoted by \(v\), and a scalar
\(\eta\ne0\) such that \(\varepsilon\rhd v=\eta v\).  Let
\(v_-:=h^2\rhd v\).  Then
\[
\begin{aligned}
\varepsilon\rhd v_-
=
\varepsilon\rhd(h^2\rhd v)
=
\frac{\Phi_g(\varepsilon,h^2)}
     {\Phi_g(h^2,\varepsilon)}
h^2\rhd(\varepsilon\rhd v)
=
\vartheta_g\eta v_-
=
-\eta v_-.
\end{aligned}
\]
Thus \(v\) and \(v_-\) are linearly independent. Set
\(\kappa:=\lambda\eta/\Phi_g(\varepsilon,g)\).
The projective action gives
\((\varepsilon g)\rhd v=\kappa v\) and
\((\varepsilon g)\rhd v_-=-\kappa v_-\). Since
\(gh=h\varepsilon g\), the action of \(g\) and the double braiding
\(c^2_{V,V}:=c_{V,V}\circ c_{V,V}\) satisfy
\begin{align*}
g\rhd(hv)
&=\frac{\Phi_g(g,h)}{\Phi_g(h,\varepsilon g)}\kappa hv,\\
c^2_{V,V}(v\otimes hv)
&=\frac{\Phi_g(g,h)}{\Phi_g(h,\varepsilon g)}
\kappa^2\,v\otimes hv,\\
c^2_{V,V}(hv\otimes v_-)
&=-\frac{\Phi_g(g,h)}{\Phi_g(h,\varepsilon g)}
\kappa^2\,hv\otimes v_-.
\end{align*}
The relation \eqref{eq:tensor-product} gives
\[
\begin{aligned}
h\rhd(v\otimes hv)
=
\Phi^h(g,\varepsilon g)
(hv)\otimes\bigl(h\rhd(hv)\bigr)
=
\Phi^h(g,\varepsilon g)\Phi_g(h,h)
(hv)\otimes v_-.
\end{aligned}
\]
Since the double braiding is a morphism in
\({}_G^G\mathcal{YD}^{\Phi}\), it commutes with the \(G\)-action.
Therefore,
\[
\begin{aligned}
c^2_{V,V}\bigl(h\rhd(v\otimes hv)\bigr)
=
h\rhd c^2_{V,V}(v\otimes hv)
=
\frac{\Phi_g(g,h)}
     {\Phi_g(h,\varepsilon g)}
\kappa^2\,h\rhd(v\otimes hv).
\end{aligned}
\]
The formula for \(c^2_{V,V}(hv\otimes v_-)\) above gives instead
\[
 c^2_{V,V}\bigl(h\rhd(v\otimes hv)\bigr)
=
-\frac{\Phi_g(g,h)}
     {\Phi_g(h,\varepsilon g)}
\kappa^2\,h\rhd(v\otimes hv).
\]
This is impossible because all  scalars and the
vector $h\rhd(v\otimes hv)$ are nonzero.
Thus $\vartheta_g\neq-1$, and hence $\vartheta_g=1$.
It follows that \(c_g(a,b)=1\) for all \(a,b\in C_G(g)\).
Thus \(\mathbb C^{\Phi_g}C_G(g)\) is commutative, and every
irreducible \(\Phi_g\)-projective representation of \(C_G(g)\)
is one-dimensional.

Let \(x\in\{g^m h,h^m g\}\), let \(y\) be \(g\) or \(h\),
respectively, and let \(\tau\) be an irreducible
\(\Phi_x\)-projective representation of \(C_G(x)\).
Corollary~\ref{cor:Gamma2-change-of-generators} shows that
\((x,y,\varepsilon)\) satisfies the defining \(\Gamma_2\) relations
and that \(C_G(x)=\langle\varepsilon,x,y^2\rangle\).
Replacing \((g,h,\rho)\) by \((x,y,\tau)\) in the preceding argument
shows that \(\mathbb C^{\Phi_x}C_G(x)\) is commutative.
Hence \(\dim\tau=1\).
\end{proof}

Interchanging \(g\) and \(h\) gives the same conclusion for
\(\sigma\).  Thus both \(\rho\) and \(\sigma\) are one-dimensional
throughout the remainder of this section.

\subsection{Adjoint objects}
\label{subsec:Gamma2-adjoint-objects}

By the preceding subsection, \(\rho\) and \(\sigma\) are
one-dimensional.  Choose nonzero homogeneous vectors
\(v\in V_g\) and \(w\in W_h\).  Since
\(hgh^{-1}=\varepsilon g\) and
\(ghg^{-1}=\varepsilon h\), one has
\(
hv\in V_{\varepsilon g}
\)
and
\(
gw\in W_{\varepsilon h}.
\)
Hence
$
    V=\Bbbk v\oplus\Bbbk hv,
    W=\Bbbk w\oplus\Bbbk gw
$
as $G$-graded vector spaces.

Define nonzero scalars
\begin{equation}
\begin{aligned}
    \varepsilon\rhd v&=\zeta v,
    &
    g\rhd v&=\lambda v,
    &
    h^2\rhd v&=\mu v,\\
    \varepsilon\rhd w&=\zeta' w,
    &
    h\rhd w&=\lambda' w,
    &
    g^2\rhd w&=\mu' w.
\end{aligned}
\label{eq:Gamma2-parameters}
\end{equation}
Thus $\zeta,\lambda,\mu$ describe the one-dimensional
$\Phi_g$-projective action of $C_G(g)$ on $V_g$, while
$\zeta',\lambda',\mu'$ describe the one-dimensional
$\Phi_h$-projective action of $C_G(h)$ on $W_h$.
Set
\[
X_n=X_n^{V,W}\cong(\operatorname{ad}V)^n(W),
\qquad
Y_n=X_n^{W,V}\cong(\operatorname{ad}W)^n(V).
\]
Here $X_n$ and $Y_n$ are the restricted images defined using
the operators $\varphi_n^\Phi$.

For the higher adjoint objects, the following observation will be used
repeatedly.  Let \(n\geq2\), assume that \(X_{n-1}\) is nonzero
and simple, and choose
\(0\neq x_{n-1}\in(X_{n-1})_{g^{n-1}h}\).
Lemma~\ref{lem:Gamma2-projective-representations-one-dimensional}
shows that this homogeneous component is \(\mathbb Cx_{n-1}\).
The subgroup \(\langle g^{n-1}h\rangle\) preserves this line, and
\((g^{n-1}h)g(g^{n-1}h)^{-1}=\varepsilon g\), so
\(V=\mathbb C\langle g^{n-1}h\rangle\rhd v\).
Since simplicity also gives
\(X_{n-1}=\mathbb CG\rhd x_{n-1}\),
Proposition~\ref{prop:adjoint-orbit-generation} yields
\(X_n=\mathbb CG\rhd x_n\), where
\(x_n=\varphi_n^\Phi(v\otimes x_{n-1})\).
Thus \(X_n=0\) if and only if \(x_n=0\).

\subsubsection{The first adjoint object}

\begin{lemma}
\label{lem:Gamma2-first-adjoint-formula}
Let
$x_1:=\varphi_1^\Phi(v\otimes w).$
Then
\[
x_1
=
v\otimes w
-
\frac{\zeta}{\Phi_g(h,\varepsilon)}
\,hv\otimes gw.
\]
Moreover,
$x_1\neq0, \ 
X_1=\mathbb{C}G\rhd x_1.$
\end{lemma}

\begin{proof}
Since \(W=\Bbbk G\rhd w\),
\(\langle h\rangle\rhd w\subseteq\Bbbk w\), and
\(V=\Bbbk\langle h\rangle\rhd v\),
Proposition~\ref{prop:adjoint-orbit-generation}, applied with
\(n=1\), \(H=\langle h\rangle\), and \(x_0=w\), gives
\(X_1=\Bbbk G\rhd x_1\).

 {By definition,}
$
\varphi_1^\Phi
=
\operatorname{id}_{V\otimes W}
-
c_{W,V}c_{V,W}.
$
Applying the double braiding to \(v\otimes w\) gives
\(c_{W,V}c_{V,W}(v\otimes w)
=(\varepsilon h)\rhd v\otimes gw\).
 {Since \(h\varepsilon=\varepsilon h\), we have
$
h\rhd(\varepsilon\rhd v)
=
\Phi_g(h,\varepsilon)
(\varepsilon h)\rhd v
$.} 
Consequently,
$
(\varepsilon h)\rhd v
=
\frac{\zeta}{\Phi_g(h,\varepsilon)}\,hv.
$
Substitution gives the asserted formula for \(x_1\).

Finally,
$
v\otimes w\in V_g\otimes W_h,\
hv\otimes gw\in
V_{\varepsilon g}\otimes W_{\varepsilon h}.
$
These are distinct bi-homogeneous components. Hence the two
summands cannot cancel, and therefore \(x_1\neq0\).
\end{proof}

 {To characterize the simplicity of \(X_1\), set}
\begin{equation}
\label{eq:Delta-definition}
\Delta
:=
\frac{
\zeta\zeta'\mu\mu'\,
\Phi_h(\varepsilon g,g)\Phi_g(h,h)
}{
\Phi_g(h,\varepsilon)\Phi_h(\varepsilon,g^2)
}.
\end{equation}

\begin{lemma}
\label{lem:Gamma2-X1-simple}
The object \(X_1\) is simple if and only if \(\Delta=1\).
In this case,
$
X_1\simeq M(gh,\sigma_1)
$
for a one-dimensional projective representation
\(\sigma_1\) of \(C_G(gh)\).
\end{lemma}

\begin{proof}
Let
$
a:=\frac{\zeta}{\Phi_g(h,\varepsilon)}
$
and introduce the vectors
\[
e_0:=v\otimes w,\quad
e_1:=hv\otimes gw, \quad
f_0:=v\otimes gw,\quad
f_1:=hv\otimes w.
\]
Thus
$
x_1=e_0-ae_1.
$
The actions of \(g\) and \(h\) on \(x_1\) are as follows.
Relation \eqref{eq:YD-projective-action} gives
$
(\varepsilon g)\rhd v
=
\frac{\lambda\zeta}{\Phi_g(\varepsilon,g)}\,v.
$
Since
$
gh=h(\varepsilon g),
$
we have
$(gh)\rhd v
=
\frac{1}{\Phi_g(h,\varepsilon g)}
h\rhd\bigl((\varepsilon g)\rhd v\bigr).$
It follows that
\begin{equation}
\label{eq:g-action-on-hv}
g\rhd(hv)
=
\frac{
\Phi_g(g,h)\lambda\zeta
}{
\Phi_g(\varepsilon,g)\Phi_g(h,\varepsilon g)
}\,hv.
\end{equation}
Similarly,
\begin{equation}
\label{eq:h-action-on-gw}
h\rhd(gw)
=
\frac{
\Phi_h(h,g)\lambda'\zeta'
}{
\Phi_h(\varepsilon,h)\Phi_h(g,\varepsilon h)
}\,gw.
\end{equation}
Moreover,
$
h\rhd e_0
=
\lambda'\Phi^h(g,h)f_1.
$
By \eqref{eq:h-action-on-gw} and
$
h\rhd(hv)=\Phi_g(h,h)\mu v,
$
we also have
$h\rhd e_1=A_1f_0
$,
where
\begin{equation}
\label{eq:A1-definition}
A_1
:=
\Phi^h(\varepsilon g,\varepsilon h)
\Phi_g(h,h)\mu
\frac{
\Phi_h(h,g)\lambda'\zeta'
}{
\Phi_h(\varepsilon,h)\Phi_h(g,\varepsilon h)
}.
\end{equation}
Therefore
\begin{equation}
\label{eq:h-action-x1}
h\rhd x_1
=
\lambda'\Phi^h(g,h)f_1-aA_1f_0.
\end{equation}

On the other hand,
$
g\rhd e_0
=
\lambda\Phi^g(g,h)f_0.
$
Using \eqref{eq:g-action-on-hv} and
$
g\rhd(gw)=\Phi_h(g,g)\mu'w,$
we get
$g\rhd e_1=\lambda B_1f_1,$
where
\begin{equation}
\label{eq:B1-definition}
B_1
:=
\Phi^g(\varepsilon g,\varepsilon h)
\frac{
\Phi_g(g,h)\zeta
}{
\Phi_g(\varepsilon,g)\Phi_g(h,\varepsilon g)
}
\Phi_h(g,g)\mu'.
\end{equation}
Consequently,
\begin{equation}
\label{eq:g-action-x1}
g\rhd x_1
=
\lambda\Phi^g(g,h)f_0
-
a\lambda B_1f_1.
\end{equation}

Both \(h\rhd x_1\) and \(g\rhd x_1\) are homogeneous of degree
\(\varepsilon gh\).  {Comparing the coefficients of both \(f_0\) and \(f_1\) in
\eqref{eq:h-action-x1} and \eqref{eq:g-action-x1}}, we see that
$h\rhd x_1 \in \mathbb{C}(g\rhd x_1)$
if and only if
\begin{equation}
\label{eq:first-proportionality-condition}
\frac{
a^2A_1B_1
}{
\lambda'\Phi^h(g,h)\Phi^g(g,h)
}
=1.
\end{equation}
By Lemma~\ref{lem:first-coefficient-reduction} in
Appendix~\ref{app:Gamma2-calculations}, the left-hand side
of \eqref{eq:first-proportionality-condition} is precisely
\(\Delta\). Hence \(h\rhd x_1\in\mathbb C(g\rhd x_1)\) if and only
if \(\Delta=1\).

{
The elements \(\varepsilon\) and \(g^2\) preserve
\(\mathbb C(v\otimes w)\). Since \(\varphi_1^\Phi\) commutes
with the \(G\)-action, they also preserve \(\mathbb Cx_1\).
By Lemma~\ref{lem:Gamma2-normal-forms-centralizers},
\(C_G(gh)=\langle\varepsilon,g^2,g^{-1}h\rangle\).
Thus \(\mathbb Cx_1\) is stable under \(C_G(gh)\) if and only if
\(h\rhd x_1\in\mathbb C(g\rhd x_1)\), which is equivalent to
\(\Delta=1\).
}
{
In this case, \(\mathbb Cx_1\) affords a one-dimensional
projective representation \(\sigma_1\) of \(C_G(gh)\).
Since \(X_1=\mathbb CG\rhd x_1\), Proposition~\ref{prop:simple-objects}
gives a nonzero surjection from the simple object
\(M(gh,\sigma_1)\) to \(X_1\), and hence an isomorphism.
Conversely, if \(X_1\) is simple,
Lemma~\ref{lem:Gamma2-projective-representations-one-dimensional}
gives \((X_1)_{gh}=\mathbb Cx_1\). This line is stable under
\(C_G(gh)\), so \(\Delta=1\).
}
\end{proof}

\subsubsection{The second adjoint object}

\begin{lemma}
\label{lem:second-adjoint-formula}
Assume that \(\Delta=1\), and let
$x_2:=\varphi_2^\Phi(v\otimes x_1).$
Then
\[
x_2
=
(1+\lambda)
\left[
v\otimes x_1
-
\frac{1}{\Phi_g(h,g)}
\,hv\otimes(g\rhd x_1)
\right].
\]
Moreover,
$X_2=\mathbb{C}G\rhd x_2.$
\end{lemma}

\begin{proof}
The  argument at the beginning of this subsection
gives \(X_2=\mathbb CG\rhd x_2\).
For the explicit calculation, let
\(T_0:=v\otimes x_1\) and \(T_1:=hv\otimes(g\rhd x_1)\).
Recall that \(x_1=v\otimes w-a\,hv\otimes gw\).
By the  definition of \(\varphi_2^\Phi\),
\begin{equation}
\label{eq:recursive-varphi-2}
\varphi_2^\Phi
=
\operatorname{id}_{V\otimes X_1}
-
c_{X_1,V}c_{V,X_1}
+
(\operatorname{id}_V\otimes\varphi_1^\Phi)
c_{1,2}^\Phi.
\end{equation}
First, \(c_{V,X_1}(v\otimes x_1)=(g\rhd x_1)\otimes v\), and then
\(c_{X_1,V}((g\rhd x_1)\otimes v)
=(hg)\rhd v\otimes(g\rhd x_1)\). Thus
\((hg)\rhd v=\lambda hv/\Phi_g(h,g)\).
Consequently,
\begin{equation}
\label{eq:double-braiding-T0}
c_{X_1,V}c_{V,X_1}(T_0)
=
\frac{\lambda}{\Phi_g(h,g)}\,T_1.
\end{equation}

Consider
\(
(\operatorname{id}_V\otimes\varphi_1^\Phi)
c_{1,2}^\Phi(T_0).
\)
For the first summand of \(T_0\), the definitions of
\(c_{1,2}^\Phi\) and \(\varphi_1^\Phi\) give
\((\operatorname{id}_V\otimes\varphi_1^\Phi)
c_{1,2}^\Phi(v\otimes(v\otimes w))=\lambda T_0\).
For the second summand, using \eqref{eq:g-action-on-hv} gives
\begin{equation}
\label{eq:second-c12-contribution-before-varphi1}
\begin{aligned}
c_{1,2}^\Phi
\bigl(v\otimes(hv\otimes gw)\bigr)=
\frac{
\Phi(g,\varepsilon g,\varepsilon h)
\Phi_g(g,h)\lambda\zeta
}{
\Phi(\varepsilon g,g,\varepsilon h)
\Phi_g(\varepsilon,g)
\Phi_g(h,\varepsilon g)
}
\,hv\otimes(v\otimes gw).
\end{aligned}
\end{equation}
Since \(\varphi_1^\Phi\) is a morphism in $\GG$,
\(g\rhd x_1=\varphi_1^\Phi(g\rhd(v\otimes w))\).  Moreover,
the relation \eqref{eq:tensor-product} gives 
\(g\rhd(v\otimes w)=
\lambda\Phi^g(g,h)\,v\otimes gw\).  Thus
\begin{equation}
\label{eq:varphi1-v-gw}
\varphi_1^\Phi(v\otimes gw)
=
\frac{1}{\lambda\Phi^g(g,h)}
\,g\rhd x_1.
\end{equation}
Combining
\eqref{eq:second-c12-contribution-before-varphi1} and
\eqref{eq:varphi1-v-gw}, and recalling that
$
T_0
=
v\otimes(v\otimes w)
-
a\,v\otimes(hv\otimes gw),
$
we obtain
\begin{align}
&(\operatorname{id}_V\otimes\varphi_1^\Phi)
c_{1,2}^\Phi(T_0)
=
\lambda T_0-\Xi_VT_1, \nonumber
\\
&\Xi_V
:=
\frac{
\zeta^2
\Phi(g,\varepsilon g,\varepsilon h)
\Phi_g(g,h)
}{
\Phi_g(h,\varepsilon)
\Phi(\varepsilon g,g,\varepsilon h)
\Phi_g(\varepsilon,g)
\Phi_g(h,\varepsilon g)
\Phi^g(g,h)
}.
\label{eq:Xi-V-definition}
\end{align}
Lemma~\ref{lem:Xi-V-Omega-V} gives
\(\Xi_V=\Phi_g(h,g)^{-1}\).  Hence
\begin{equation*}
(\operatorname{id}_V\otimes\varphi_1^\Phi)
c_{1,2}^\Phi(T_0)
=
\lambda T_0-\frac{1}{\Phi_g(h,g)}T_1.
\end{equation*}
Substituting this identity and \eqref{eq:double-braiding-T0} into
\eqref{eq:recursive-varphi-2} gives
\(x_2=(1+\lambda)\bigl(T_0-\Phi_g(h,g)^{-1}T_1\bigr)\),
as required.
\end{proof}

\begin{lemma}
\label{lem:second-adjoint-vanishing}
Assume that \(\Delta=1\). Then \(X_2=0\) if and only if
\(\lambda=-1\).
\end{lemma}

\begin{proof}
By Lemma~\ref{lem:second-adjoint-formula},
$X_2=\mathbb{C}G\rhd x_2.$
Hence
$X_2=0$ is equivalent to $x_2=0.$
Now
$v\otimes x_1
\in
V_g\otimes(X_1)_{gh},$
whereas
$hv\otimes(g\rhd x_1)
\in
V_{\varepsilon g}\otimes(X_1)_{\varepsilon gh}.$
Thus the two tensors inside the brackets in
Lemma~\ref{lem:second-adjoint-formula} belong to distinct
bi-homogeneous components. In particular,
\[
v\otimes x_1
-
\frac{1}{\Phi_g(h,g)}
hv\otimes(g\rhd x_1)
\neq0.
\]
It follows that
\(x_2=0\) if and only if \(1+\lambda=0\),
which proves the assertion.
\end{proof}

 {To characterize the simplicity of a nonzero \(X_2\), define}
\begin{equation}
\label{eq:Theta-Phi-definition}
\begin{aligned}
\Theta_\Phi
:={}&
\frac{
\Phi_h(\varepsilon,\varepsilon)
\Phi_h(\varepsilon g,g)
\Phi^h(g,gh)
\Phi^h(\varepsilon g,\varepsilon h)
\Phi_g(h,g)^2
}{
\Phi_g(h,\varepsilon)
\Phi_h(\varepsilon,g^2)
\Phi_h(\varepsilon,h)
\Phi_h(g,\varepsilon h)
\Phi^g(g,h)
}
\\
&\qquad\times
\frac{
\Phi_h(h,g)
\Phi_g(\varepsilon,g)
\Phi_g(h,\varepsilon g)
}{
\Phi^h(\varepsilon g,\varepsilon gh)
\Phi^g(\varepsilon g,\varepsilon h)
\Phi^h(\varepsilon g,h)
\Phi_g(g,h)
\Phi_h(g,g)
}.
\end{aligned}
\end{equation}
Whenever \(x,y\in G\) satisfy \(yx=\varepsilon xy\), we write
\(\Theta_\Phi(x,y)\) for the expression in
\eqref{eq:Theta-Phi-definition} after the simultaneous substitution
\(g\mapsto x\), \(h\mapsto y\).  Thus
\(\Theta_\Phi=\Theta_\Phi(g,h)\).

\begin{lemma}
\label{lem:Gamma2-X2-simple}
Assume that $X_2 \neq 0$ and $X_1$ is simple. Then \(X_2\) is
simple if and only if \(\lambda^2=\Theta_\Phi\).
\end{lemma}

\begin{proof}
Let
$p:=\frac{1}{\Phi_g(h,g)}$
and
$b_2:=T_0-pT_1.$
Since $X_2\neq 0$ is equivalent to  \(\lambda\neq-1\), we have
$x_2=(1+\lambda)b_2,$
so that
$\mathbb Cx_2=\mathbb Cb_2.$
We retain the notation \(a,A_1,B_1\) from
\eqref{eq:A1-definition} and \eqref{eq:B1-definition}.
Since \(\Delta=1\), equations \eqref{eq:h-action-x1},
\eqref{eq:g-action-x1}, and \eqref{eq:first-proportionality-condition}
give
\begin{equation}
\label{eq:h-action-x1-proportional}
h\rhd x_1
=
-\frac{aA_1}{\lambda\Phi^g(g,h)}
\,(g\rhd x_1).
\end{equation}
{%
Since \(X_1\) is simple and \(h\rhd(g\rhd x_1)\) has degree
\(gh\), one has
\(h\rhd(g\rhd x_1)\in(X_1)_{gh}=\mathbb Cx_1\).
Equation (\ref{eq:tensor-product}) gives
\(h\rhd f_0\in\mathbb Ce_1\) and
\(h\rhd f_1=D_1e_0\), where
}
$
{D_1=\Phi^h(\varepsilon g,h)\Phi_g(h,h)\mu\lambda'.}
$
{
Comparing the coefficient of \(e_0\) in
\eqref{eq:g-action-x1}, after applying \(h\), gives
}
\begin{equation}
\label{eq:h-action-gx1}
{h\rhd(g\rhd x_1)=-a\lambda B_1D_1x_1.}
\end{equation}
By \eqref{eq:h-action-x1-proportional},
\[
\begin{aligned}
h\rhd T_0
=
\Phi^h(g,gh)
(hv)\otimes(h\rhd x_1)
=
-a
\frac{
\Phi^h(g,gh)A_1
}{
\lambda\Phi^g(g,h)
}
T_1.
\end{aligned}
\]
On the other hand,
$h\rhd(hv)=\Phi_g(h,h)\mu v$.
Together with \eqref{eq:h-action-gx1}, this gives
\[
\begin{aligned}
h\rhd T_1
&=
\Phi^h(\varepsilon g,\varepsilon gh)
\bigl(h\rhd(hv)\bigr)
\otimes
\bigl(h\rhd(g\rhd x_1)\bigr)
\\
&=
-a\lambda
\Phi^h(\varepsilon g,\varepsilon gh)
\Phi_g(h,h)\mu B_1D_1
T_0.
\end{aligned}
\]
For brevity, let
$r_0
:=
-a
\frac{
\Phi^h(g,gh)A_1
}{
\lambda\Phi^g(g,h)
}$
and
$r_1
:=
-a\lambda
\Phi^h(\varepsilon g,\varepsilon gh)
\Phi_g(h,h)\mu B_1D_1$.
Then
$h\rhd T_0=r_0T_1,\ 
h\rhd T_1=r_1T_0,$
and hence \(h\rhd b_2=r_0T_1-pr_1T_0\).
 {The condition
$h\rhd b_2\in\mathbb Cb_2$
is equivalent to \(r_0=p^2r_1\).}
After cancelling the common nonzero factor \(-a\), this becomes
\[
\lambda^2
=
\frac{
\Phi^h(g,gh)A_1
}{
p^2
\Phi^g(g,h)
\Phi^h(\varepsilon g,\varepsilon gh)
\Phi_g(h,h)\mu B_1D_1
}.
\]
By Lemma~\ref{lem:Theta-Phi-reduction}, the right-hand side is
exactly \(\Theta_\Phi\). We have therefore proved
that \(h\rhd x_2\in\mathbb Cx_2\) if and only if
\(\lambda^2=\Theta_\Phi\).

{
The elements \(\varepsilon\) and \(g^2\) preserve
\(\mathbb C(v\otimes x_1)\). Since \(\varphi_2^\Phi\) commutes
with the \(G\)-action, they also preserve \(\mathbb Cx_2\).
By Lemma~\ref{lem:Gamma2-normal-forms-centralizers},
\(C_G(g^2h)=\langle\varepsilon,h,g^2\rangle\).
Hence \(\mathbb Cx_2\) is stable under this centralizer if and only
if \(h\rhd x_2\in\mathbb Cx_2\), equivalently
\(\lambda^2=\Theta_\Phi\).
}
{
As in the proof of Lemma~\ref{lem:Gamma2-X1-simple},
\(X_2=\mathbb CG\rhd x_2\) and
Lemma~\ref{lem:Gamma2-projective-representations-one-dimensional}
show that this line is stable under the centralizer if and only if
\(X_2\) is simple. In that case,
\(X_2\simeq M(g^2h,\sigma_2)\), where \(\sigma_2\) is the
one-dimensional projective representation on \(\mathbb Cx_2\).
}
\end{proof}

\subsubsection{The third adjoint object}

For the calculation of the third adjoint object, define
\begin{equation}
\label{eq:kappa-3-definition}
\kappa_3
:=
\frac{
\zeta\,
\Phi(g,\varepsilon g,\varepsilon gh)
\Phi_g(g,h)
}{
\Phi_g(h,g)
\Phi(\varepsilon g,g,\varepsilon gh)
\Phi_g(\varepsilon,g)
\Phi_g(h,\varepsilon g)
\Phi^g(g,gh)
}.
\end{equation}

\begin{lemma}
\label{lem:Gamma2-X3-formula}
Assume that \(X_1\) and \(X_2\) are simple and that \(X_2\ne0\).
Let \(x_2=\varphi_2^\Phi(v\otimes x_1)\) be as in
Lemma~\ref{lem:second-adjoint-formula}, and set
$
x_3:=\varphi_3^\Phi(v\otimes x_2).
$
Then
\begin{equation}
\label{eq:x3-explicit}
x_3
=
(1+\lambda+\lambda^2)
\left(
v\otimes x_2
-\kappa_3\,hv\otimes(g\rhd x_2)
\right).
\end{equation}
Moreover,
$
X_3=\mathbb CG\rhd x_3.
$
Consequently,
\begin{equation}
\label{eq:X3-vanishing-criterion}
X_3=0
\quad\Longleftrightarrow\quad
1+\lambda+\lambda^2=0.
\end{equation}
\end{lemma}

\begin{proof}
 {The argument at the beginning of
Subsection~\ref{subsec:Gamma2-adjoint-objects} gives
\(X_3=\mathbb CG\rhd x_3\).}
Set
$
S_0:=v\otimes x_2,
S_1:=hv\otimes(g\rhd x_2).
$
On \(S_0\), the three terms in the recursive definition of
\(\varphi_3^\Phi\) are as follows. First,
\begin{equation}
\label{eq:x3-double-braiding}
c_{X_2,V}c_{V,X_2}(S_0)=d_3S_1,
\qquad
d_3
=
\frac{
\zeta\lambda^2
}{
\Phi_g(\varepsilon,g)
\Phi_g(\varepsilon g,g)
\Phi_g(h,\varepsilon g^2)
}.
\end{equation}
Write $p=\Phi_g(h,g)^{-1}$ and use the formula for \(x_2\) in
Lemma~\ref{lem:second-adjoint-formula}.
The first summand gives
$
(\operatorname{id}_V\otimes\varphi_2^\Phi)
c_{1,2}^\Phi
\bigl(v\otimes(v\otimes x_1)\bigr)
=
\lambda v\otimes x_2.$
For the second summand, the associator corrections in
$c_{1,2}^\Phi$ give
\[
c_{1,2}^\Phi
\bigl(v\otimes(hv\otimes(g\rhd x_1))\bigr)
=
\frac{
\Phi(g,\varepsilon g,\varepsilon gh)
}{
\Phi(\varepsilon g,g,\varepsilon gh)
}
(g\rhd hv)\otimes
\bigl(v\otimes(g\rhd x_1)\bigr).
\]
Since $\varphi_2^\Phi$ is a morphism and
$
g\rhd(v\otimes x_1)
=
\lambda\Phi^g(g,gh)
v\otimes(g\rhd x_1),
$
we have
\[
\varphi_2^\Phi
\bigl(v\otimes(g\rhd x_1)\bigr)
=
\frac{1}{\lambda\Phi^g(g,gh)}
g\rhd x_2.
\]
Using also
$
g\rhd hv
=
\frac{
\Phi_g(g,h)\lambda\zeta
}{
\Phi_g(\varepsilon,g)\Phi_g(h,\varepsilon g)
}
hv,
$
we obtain
\begin{equation}
\label{eq:x3-c12-term}
(\operatorname{id}_V\otimes\varphi_2^\Phi)
c_{1,2}^\Phi(S_0)
=
\lambda(1+\lambda)S_0
-(1+\lambda)\kappa_3S_1.
\end{equation}
The cocycle identity in
Lemma~\ref{lem:x3-cocycle-reduction} gives
\(d_3=\lambda^2\kappa_3\).
Substituting
\eqref{eq:x3-double-braiding},
\eqref{eq:x3-c12-term}, and this equality into the recursive formula for
$\varphi_3^\Phi$ proves \eqref{eq:x3-explicit}.

Finally, $S_0$ and $S_1$ belong to the distinct 
bi-homogeneous components
$V_g\otimes(X_2)_{g^2h}$
and
$V_{\varepsilon g}\otimes(X_2)_{\varepsilon g^2h}$,
respectively.  Hence  {the vector in parentheses in
\eqref{eq:x3-explicit}} is nonzero.  Equations
\(X_3=\mathbb CG\rhd\varphi_3^\Phi(v\otimes x_2)\) and
\eqref{eq:x3-explicit} now give
\eqref{eq:X3-vanishing-criterion}.
\end{proof}

\subsubsection{Adjoint objects in the opposite order}

Since \(hg=\varepsilon gh\) and \(\varepsilon^2=1\), one also has
\(gh=\varepsilon hg\). The preceding formulas therefore apply to
\((W,V)\), with support representatives \((h,g)\).
Set
\begin{equation}
\label{eq:Delta-WV-definition}
\Delta'
:=
\frac{
\zeta\zeta'\mu\mu'\,
\Phi_g(\varepsilon h,h)\Phi_h(g,g)
}{
\Phi_h(g,\varepsilon)\Phi_g(\varepsilon,h^2)
}.
\end{equation}

\begin{lemma}
\label{lem:opposite-first-adjoint-simple}
\label{lem:opposite-second-adjoint-vanishing}
The object \(Y_1\) is simple if and only if \(\Delta'=1\).
In this case, \(Y_1\simeq M(hg,\tau_1)\) for a one-dimensional
projective representation \(\tau_1\) of \(C_G(hg)\), and
\(Y_2=0\) if and only if \(\lambda'=-1\).
\end{lemma}

\begin{proof}
Apply Lemmas~\ref{lem:Gamma2-X1-simple} and
\ref{lem:second-adjoint-vanishing} to the pair \((W,V)\).
Under the simultaneous substitutions
\(g\leftrightarrow h\), \(\zeta\leftrightarrow\zeta'\),
\(\lambda\leftrightarrow\lambda'\), and \(\mu\leftrightarrow\mu'\),
the parameter \(\Delta\) becomes \(\Delta'\).
\end{proof}

\subsection{Duals and reflections}
\label{subsec:Gamma2-duals-reflections}
By Lemma~\ref{lem:dual-projective-representation},
\(c_{V^*,V^*}\) acts by \(\lambda\) on
\((V^*)_{g^{-1}}\otimes(V^*)_{g^{-1}}\), and
\(c_{W^*,W^*}\) acts by \(\lambda'\) on
\((W^*)_{h^{-1}}\otimes(W^*)_{h^{-1}}\).
For the two reflections, let
\begin{equation*}
m_{12}:=\max\{m\geq0\mid X_m\neq0\},
\qquad
m_{21}:=\max\{n\geq0\mid Y_n\neq0\}.
\end{equation*}
Suppose that \(m_{12}\) and \(m_{21}\) are finite.
By Lemma~\ref{lem:reflection-basic-properties}, the last nonzero
objects $X_{m_{12}}$ and $Y_{m_{21}}$ are simple, and the
corresponding reflections are
\begin{equation*}
R_1(V,W)=(V^*,X_{m_{12}}),
\qquad
R_2(V,W)=(Y_{m_{21}},W^*).
\end{equation*}
For brevity, write $m=m_{12}$ and $n=m_{21}$. Since
$X_m\subseteq V^{\otimes m}\otimes W$, its support is contained in
$\{g^mh,\varepsilon g^mh\}$. These two elements are conjugate by
$g$. Since $X_m\ne0$ and its support is stable under conjugation,
its support equals this set. Similarly,
$\operatorname{supp}Y_n=\{h^ng,\varepsilon h^ng\}$.

\begin{lemma}
\label{lem:reflections-remain-Gamma2}
Write \(m=m_{12}\) and \(n=m_{21}\).  For the two reflected tuples,
choose support representatives as follows:
\[
\begin{array}{c|c|c}
\hline
\text{tuple} & x & y\\
\hline
R_1(V,W)=(V^*,X_m) & g^{-1} & g^m h\\
R_2(V,W)=(Y_n,W^*) & h^n g & h^{-1}\\
\hline
\end{array}
\]
For either row, the representatives \(x\) and \(y\) satisfy
\[
yx=\varepsilon xy,
\qquad
G=\langle x,y\rangle.
\]
Moreover, the supports of the first and second entries are
$
\{x,\varepsilon x\},\
\{y,\varepsilon y\},
$
respectively.  Thus both reflected tuples satisfy the same assumptions
on their supports as \((V,W)\).
\end{lemma}

\begin{proof}
For \(R_1(V,W)\), the relation \(hg=\varepsilon gh\) gives
$
(g^m h)g^{-1}
=
\varepsilon g^{-1}(g^m h).
$
Since
$
g=(g^{-1})^{-1},
$
$h=g^{-m}(g^m h),
$
the elements \(g^{-1}\) and \(g^m h\) generate \(G\).

For \(R_2(V,W)\), one similarly obtains
$
h^{-1}(h^n g)
=
\varepsilon(h^n g)h^{-1}.
$
Since
$
h=(h^{-1})^{-1},
$
$g=h^{-n}(h^n g),$
the elements \(h^n g\) and \(h^{-1}\) also generate \(G\).

The relation \(yx=\varepsilon xy\), together with
\(\varepsilon^2=1\), gives
\[
x^G=\{x,\varepsilon x\},
\qquad
y^G=\{y,\varepsilon y\}.
\]
Since \(V^*,X_m,Y_n\), and \(W^*\) are simple, their supports are the
corresponding conjugacy classes.
\end{proof}

\begin{lemma}
\label{lem:Gamma2-X1-Y1-inducing-values}
Whenever the one-dimensional projective representations
\(\sigma_1\) and \(\tau_1\) are defined, one has
\begin{equation}
\label{eq:Gamma2-X1-Y1-support-values}
\sigma_1(gh)=-\lambda\lambda',
\qquad
\tau_1(hg)=-\lambda\lambda'.
\end{equation}
\end{lemma}

\begin{proof}
Let $a=\zeta/\Phi_g(h,\varepsilon)$, so that
$x_1=v\otimes w-a\,hv\otimes gw$.  Acting by $gh$ on the first
summand and comparing the coefficient of $hv\otimes gw$ gives
\(\sigma_1(gh)=-\lambda\lambda'K_1\),
where
\begin{equation}
\label{eq:K1-reflected-character}
K_1
:=
\frac{
\Phi^{gh}(g,h)\Phi_g(h,\varepsilon)
}{
\Phi_h(g,h)
\Phi_g(h,\varepsilon g)
\Phi_g(\varepsilon,g)
}.
\end{equation}
 {Interchanging \((V,g,\rho)\) and \((W,h,\sigma)\) gives}
\(\tau_1(hg)=-\lambda\lambda'K_1^{\mathrm{op}}\),
with
\begin{equation}
\label{eq:K1-op-reflected-character}
K_1^{\mathrm{op}}
:=
\frac{
\Phi^{hg}(h,g)\Phi_h(g,\varepsilon)
}{
\Phi_h(g,\varepsilon h)
\Phi_h(\varepsilon,h)
\Phi_g(h,g)
}.
\end{equation}
By Lemma~\ref{lem:reflected-character-cocycle-reduction},
$K_1=K_1^{\mathrm{op}}=1$.  This proves
\eqref{eq:Gamma2-X1-Y1-support-values}.
\end{proof}

\subsection{The standard case of type \texorpdfstring{$A_2$}{A2}}
\label{subsec:Gamma2-standard-A2}

If $P=(P_1,P_2)$ is obtained from $(V,W)$ by a sequence of reflections,
we write
\[
m_{ij}^P
=
\max\{r\geq0\mid(\operatorname{ad}P_i)^r(P_j)\neq0\},
\qquad i\neq j.
\]
Thus the generalized Cartan matrix at $[P]$ is
\begin{equation}
\label{eq:Cartan-matrix-at-P}
A^{[P]}
=
\begin{pmatrix}
2&-m_{12}^P\\
-m_{21}^P&2
\end{pmatrix}.
\end{equation}
For any two objects $U,T$, the relation
\[
c_{U,T}(\operatorname{id}_{U\otimes T}-c_{T,U}c_{U,T})
=(\operatorname{id}_{T\otimes U}-c_{U,T}c_{T,U})c_{U,T}
\]
identifies the images of the two first adjoint maps and gives
\begin{equation}
\label{eq:first-adjoint-braiding-symmetry}
(\operatorname{ad}U)(T)
\simeq
(\operatorname{ad}T)(U).
\end{equation}

\begin{prop}
\label{prop:Gamma2-standard-A2}
Assume that $X_1$ and $Y_1$ are simple and that
$X_2=Y_2=0$.  Then $(V,W)$ admits all reflections and
$\mathcal G(V,W)$ is a standard Cartan graph of type $A_2$.
In particular, at every object $[P]$ one has
\begin{equation}
\label{eq:standard-A2-matrix}
A^{[P]}
=
\begin{pmatrix}
2&-1\\
-1&2
\end{pmatrix}.
\end{equation}
\end{prop}

\begin{proof}
The vanishing criteria give \(\lambda=\lambda'=-1\).  The two
reflections are \(R_1(V,W)=(V^*,X_1)\) and
\(R_2(V,W)=(Y_1,W^*)\), and their supports satisfy the same assumptions
by Lemma~\ref{lem:reflections-remain-Gamma2}.
For \(R_1(V,W)\), \cite[Lemma~3.8]{reflection3} gives
\((\operatorname{ad}V^*)(X_1)\simeq W\) and
\((\operatorname{ad}V^*)^2(X_1)=0\).
Equation~\eqref{eq:first-adjoint-braiding-symmetry} then gives
\((\operatorname{ad}X_1)(V^*)\simeq W\), so this object is simple.
By Lemma~\ref{lem:Gamma2-X1-Y1-inducing-values},
\(\sigma_1(gh)=-\lambda\lambda'=-1\).
Lemma~\ref{lem:second-adjoint-vanishing}, applied to \((X_1,V^*)\),
therefore gives \((\operatorname{ad}X_1)^2(V^*)=0\).

The same argument applied to \(R_2(V,W)\), using
\(\tau_1(hg)=-\lambda\lambda'=-1\), shows that its first adjoint
objects are isomorphic to \(V\) and its second adjoint objects vanish.
Thus both reflected tuples satisfy the hypotheses of the proposition.
Induction on the length of a reflection sequence proves that all
reflections are defined and that \(m_{12}^P=m_{21}^P=1\) at every
object \([P]\).  Equation~\eqref{eq:Cartan-matrix-at-P} gives
\eqref{eq:standard-A2-matrix}.
\end{proof}

By Lemmas~\ref{lem:Gamma2-X1-simple},
\ref{lem:second-adjoint-vanishing}, and
\ref{lem:opposite-first-adjoint-simple}, the hypotheses of
Proposition~\ref{prop:Gamma2-standard-A2} may equivalently be written as
\begin{equation}
\label{eq:standard-A2-scalar-conditions}
\Delta=1,
\qquad
\Delta'=1,
\qquad
\lambda=\lambda'=-1.
\end{equation}

\subsection{Exclusion of type \texorpdfstring{$B_2$}{B2}}
\label{subsec:Gamma2-local-B2-obstruction}

The Sylow decomposition of \(G\) restricts the order of
\(\Theta_\Phi\) and excludes type \(B_2\). For a positive integer \(N\), let
\(\mu_N=\{z\in\mathbb C^\times\mid z^N=1\}\).

\begin{lemma}
\label{lem:Gamma2-primary-cocycle}
Let \(P\) be the Sylow \(2\)-subgroup of \(G\), and let \(A\) be
the product of its Sylow subgroups of odd order. Then
\(G=P\times A\), \(\varepsilon\in P\), and \(A\) is abelian.
Every normalized \(3\)-cocycle on \(G\) is cohomologous to
\begin{equation}
\label{eq:Gamma2-split-cocycle}
\widetilde\Phi=p_P^*(\Phi_P)p_A^*(\Phi_A),
\end{equation}
where \(\Phi_P\) and \(\Phi_A\) are normalized \(3\)-cocycles on
\(P\) and \(A\), respectively, and \(\Phi_P\) takes values in
\(\mu_{|P|}\). Here \(p_P^*\) and \(p_A^*\) denote pullback along
the canonical projections \(p_P:G\to P\) and \(p_A:G\to A\).
\end{lemma}

\begin{proof}
The defining relations give
\([G,G]=\langle\varepsilon\rangle\subseteq Z(G)\).  Thus \(G\) is
nilpotent.  Since \(G\) is finite, it is the direct product of its
Sylow subgroups, and hence \(G=P\times A\).  The group \(A\) has odd
order, whereas \([G,G]\) has order two.  Therefore
\([A,A]\subseteq [G,G]\cap A=1\), so \(A\) is abelian.

Restriction to \(P\) and \(A\), together with inflation along the two
projections, gives
$$H^3(G,\mathbb C^\times)\simeq
H^3(P,\mathbb C^\times)\oplus H^3(A,\mathbb C^\times).$$
To prove this, let \(\alpha\in H^3(G,\mathbb C^\times)\), and write the
cohomology groups additively. Set
\[
\beta
=\alpha
-p_P^*(\operatorname{res}_P^G\alpha)
-p_A^*(\operatorname{res}_A^G\alpha).
\]
Since \(p_P|_P=\operatorname{id}_P\) and \(p_A|_P\) is the
trivial homomorphism, the restrictions to \(P\) of the two
pullback classes are \(\operatorname{res}_P^G\alpha\) and zero,
respectively. Hence \(\operatorname{res}_P^G\beta=0\).
The same argument gives \(\operatorname{res}_A^G\beta=0\).
For every subgroup \(H\leq G\), the composite
\(\operatorname{cor}_H^G\circ\operatorname{res}_H^G\)
is multiplication by \([G:H]\).
Applying this to \(P\) and \(A\), we obtain
\[
|A|\beta
=\operatorname{cor}_P^G(\operatorname{res}_P^G\beta)=0,
\qquad
|P|\beta
=\operatorname{cor}_A^G(\operatorname{res}_A^G\beta)=0.
\]
Since \(|P|\) and \(|A|\) are relatively prime, it follows that
\(\beta=0\). Thus \(\alpha\) is the sum of the two pullback
classes, which proves the decomposition
\eqref{eq:Gamma2-split-cocycle}. 

Finally, \(H^3(P,\mathbb C^\times)\) is annihilated by \(|P|\).
The long exact sequence associated with
\[
1\longrightarrow\mu_{|P|}
\longrightarrow\mathbb C^\times
\xrightarrow{\,z\mapsto z^{|P|}\,}
\mathbb C^\times
\longrightarrow1
\]
therefore shows that every class in \(H^3(P,\mathbb C^\times)\) has a
representative with values in \(\mu_{|P|}\).  Hence \(\Phi_P\) may be
chosen to be \(\mu_{|P|}\)-valued.
\end{proof}

\begin{lemma}
\label{lem:Gamma2-Theta-two-primary}
For the  representative in
\eqref{eq:Gamma2-split-cocycle}, one has
\begin{equation}
\label{eq:Gamma2-Theta-two-primary}
\Theta_{\widetilde\Phi}\in\mu_{|P|}.
\end{equation}
In particular, the order of $\Theta_{\widetilde\Phi}$ is a power of
two.
\end{lemma}

\begin{proof}
The definitions of $\Phi_x$, $\Phi^x$, and $\Theta_\Phi$ are
multiplicative in the $3$-cocycle.  Hence
\(\Theta_{\widetilde\Phi}
=\Theta_{p_P^*(\Phi_P)}\Theta_{p_A^*(\Phi_A)}\).
The first factor belongs to $\mu_{|P|}$.  For the contribution of
$A$, write $g_A,h_A$ for the projections of $g,h$ to
$A$.
Since $A$ is abelian and the projection of $\varepsilon$ to $A$ is
trivial, direct substitution in the definition of $\Theta_\Phi$
gives
\[
\Theta_{\Phi_A}
=
\left(
\frac{(\Phi_A)_{g_A}(h_A,g_A)}
     {(\Phi_A)_{g_A}(g_A,h_A)}
\right)^3
\frac{(\Phi_A)_{h_A}(h_A,g_A)}
     {(\Phi_A)_{h_A}(g_A,h_A)}.
\]
By Lemma~\ref{lem:alternator}, the right-hand side is
\(f_{\Phi_A}(g_A,h_A,g_A)^3f_{\Phi_A}(h_A,h_A,g_A)=1\).
Thus $\Theta_{\widetilde\Phi}=\Theta_{\Phi_P}$, which proves
\eqref{eq:Gamma2-Theta-two-primary}.
\end{proof}

\begin{lemma}
\label{lem:Gamma2-X3-nonvanishing}
Assume that $X_1$ and $X_2$ are simple and that $X_2\neq0$.  Then
\(X_3\neq0\).
\end{lemma}

\begin{proof}
Suppose that $X_3=0$.  Replace $\Phi$ by the representative
$\widetilde\Phi$ from Lemma~\ref{lem:Gamma2-primary-cocycle}, and
transport $(V,W)$ through the induced braided monoidal equivalence.
Write \(\widetilde V,\widetilde W\) for the transported objects and
\(\widetilde X_n=(\operatorname{ad}\widetilde V)^n(\widetilde W)\).
Let \(\widetilde\lambda\) be the scalar by which
\(c_{\widetilde V,\widetilde V}\) acts on
\(\widetilde V_g\otimes\widetilde V_g\). The
equivalence preserves adjoint objects, simplicity and vanishing by [\citealp{reflection3}, Proposition 6.3], so
\[
\widetilde X_1,\widetilde X_2\text{ are simple},
\qquad
\widetilde X_2\neq0,
\qquad
\widetilde X_3=0.
\]
Lemma~\ref{lem:Gamma2-X2-simple} and
\eqref{eq:X3-vanishing-criterion} give, respectively,
\[
\widetilde\lambda^2=\Theta_{\widetilde\Phi},
\qquad
1+\widetilde\lambda+\widetilde\lambda^2=0.
\]
Consequently, $\widetilde\lambda^2$ has order three.  This contradicts
Lemma~\ref{lem:Gamma2-Theta-two-primary}.  Hence $X_3\neq0$.
\end{proof}

\begin{prop}
\label{prop:Gamma2-no-local-B2}
If $\mathcal B(V\oplus W)$ is finite-dimensional, then no object of
$\mathcal G(V,W)$ has Cartan matrix of type $B_2$, up to
transposition.
\end{prop}

\begin{proof}
Suppose that $P=(P_1,P_2)$ is such an object.  After interchanging
the two components if necessary, its Cartan matrix is
\(\left(\begin{smallmatrix}2&-2\\-1&2\end{smallmatrix}\right)\).
By Lemma~\ref{lem:reflections-remain-Gamma2}, the supports of \(P\)
again have the form required in the \(\Gamma _2\) case.
Theorem~\ref{thm:Nichols-finiteness-criterion}, applied to
\((V,W)\), shows that its Cartan graph is finite. Since \(P\)
represents an object of this graph,
Lemma~\ref{lem:root-string-adjoint-simplicity} shows that
$(\operatorname{ad}P_1)(P_2)$ and
$(\operatorname{ad}P_1)^2(P_2)$ are simple and nonzero, whereas
$(\operatorname{ad}P_1)^3(P_2)=0$.
This contradicts
Lemma~\ref{lem:Gamma2-X3-nonvanishing}, applied to $P$.
\end{proof}

\subsection{ {The  case of type \texorpdfstring{$G_2$}{G2}}}
\label{subsec:Gamma2-standard-G2-parameters}

Define
\begin{equation}
\label{eq:Gamma2-kappa4}
\kappa_4
:=
\kappa_3
\frac{
\Phi(g,\varepsilon g,\varepsilon g^2h)
\Phi_g(g,h)\zeta
}{
\Phi(\varepsilon g,g,\varepsilon g^2h)
\Phi_g(\varepsilon,g)
\Phi_g(h,\varepsilon g)
\Phi^g(g,g^2h)
}.
\end{equation}

\begin{lemma}
\label{lem:Gamma2-fourth-adjoint-full}
Assume that \(X_1,X_2\), and \(X_3\) are simple and nonzero.  Let
\(x_3=\varphi_3^\Phi(v\otimes x_2)\) be as in
Lemma~\ref{lem:Gamma2-X3-formula}, and set
\(x_4:=\varphi_4^\Phi(v\otimes x_3)\).  Then
\begin{equation}
\label{eq:Gamma2-x4-full}
x_4
=
(1+\lambda+\lambda^2+\lambda^3)
\left(
v\otimes x_3
-\kappa_4\,hv\otimes(g\rhd x_3)
\right).
\end{equation}
Moreover, \(X_4=\mathbb CG\rhd x_4\).  Therefore \(X_4=0\) if and
only if $\lambda^2=-1$.  
\end{lemma}

\begin{proof}
 {The argument at the beginning of
Subsection~\ref{subsec:Gamma2-adjoint-objects} gives}
\begin{equation*}
X_4=\mathbb CG\rhd\varphi_4^\Phi(v\otimes x_3).
\end{equation*}
Set \(U_0=v\otimes x_3\) and
\(U_1=hv\otimes(g\rhd x_3)\).
   {Using the formula for $\varphi_4^\Phi$ gives}
\begin{align}
&c_{X_3,V}c_{V,X_3}(U_0)=d_4U_1,\quad
d_4=
\frac{\lambda^3}
{\Phi_g(h,g^3)\Phi_g(g,g)\Phi_g(g^2,g)},
\label{eq:Gamma2-x4-double-braiding}
\\
&(\operatorname{id}_V\otimes\varphi_3^\Phi)c_{1,2}^\Phi(U_0)
=
\lambda(1+\lambda+\lambda^2)U_0
-(1+\lambda+\lambda^2)\kappa_4U_1.
\nonumber
\end{align}
The cocycle reduction in
Lemma~\ref{lem:Gamma2-G2-cocycle-reductions} gives
$d_4=\lambda^3\kappa_4$.  Substitution into the recursive formula
yields \eqref{eq:Gamma2-x4-full}, because
\(1+\lambda(1+\lambda+\lambda^2)
=\lambda^3+(1+\lambda+\lambda^2)
=1+\lambda+\lambda^2+\lambda^3\).
The first tensor factors have distinct degrees $g$ and
$\varepsilon g$, so $U_0$ and $U_1$ are linearly independent.
 {Hence \(X_4=0\) if and only if
\((1+\lambda)(1+\lambda^2)=0\).
Since \(X_2\ne0\) implies \(\lambda\ne-1\), this is equivalent to
\(\lambda^2=-1\)}.
\end{proof}

\begin{lemma}
\label{lem:Gamma2-X3-support-value}
Assume that $X_1,X_2$ and $X_3$ are simple and nonzero, and write
$X_3\simeq M(g^3h,\sigma_3)$.  Then
\begin{equation}
\label{eq:Gamma2-sigma3}
\sigma_3(g^3h)=-\lambda^6\lambda'.
\end{equation}
\end{lemma}

\begin{proof}
Let $q=g^2h$.  Acting by $g^3h$ on the two  components in
\eqref{eq:x3-explicit} and using 
relation \eqref{eq:YD-projective-action} gives
\((g^3h)\rhd x_3=-\lambda^6\lambda'\mathcal K_3x_3\), where
\begin{equation}
\label{eq:Gamma2-K3}
\begin{aligned}
\mathcal K_3
={}&
\frac{
  \Phi_g(h,g)^2\,\Phi_g(h,\varepsilon)\,
  \Phi(\varepsilon g,g,\varepsilon gh)\,
  \Phi^g(g,gh)
}{
  \Phi_h(g,h)\,\Phi_g(h,g^2)\,
  \Phi_g(h,\varepsilon g^3)\,
  \Phi_g(g,h)\,\Phi_g(g,g)\,
  \Phi_g(\varepsilon,g)
}
\\
&\times
\frac{
  \Phi^{gh}(g,h)\,
  \Phi^q(g,gh)\,
  \Phi^{g^3h}(g,q)
}{
  \Phi_g(\varepsilon g,g)\,
  \Phi_g(\varepsilon g^2,g)\,
  \Phi_{gh}(g,gh)\,
  \Phi_q(g,q)\,
  \Phi(g,\varepsilon g,\varepsilon gh)
}.
\end{aligned}
\end{equation}
Lemma~\ref{lem:Gamma2-G2-cocycle-reductions} gives
$\mathcal K_3=1$, proving \eqref{eq:Gamma2-sigma3}.
\end{proof}

For \(\mathfrak s_\Phi(x,y)\) defined in
\eqref{eq:Gamma2-X3-cocycle-factor},
Lemma~\ref{lem:Gamma2-X3-cocycle-factor-identity} gives
\(\mathfrak s_\Phi(x,y)=\Theta_\Phi(x,y)^2\).

\begin{lemma}
\label{lem:Gamma2-X3-simplicity}
Assume \(\Delta=1\) and \(\lambda^2=\Theta_\Phi=-1\). Then
$X_1,X_2,X_3$ are simple and nonzero.
\end{lemma}

\begin{proof}
 {Lemmas~\ref{lem:Gamma2-X1-simple} and~\ref{lem:Gamma2-X2-simple}
show that $X_1$ and $X_2$ are simple.}  Since $\lambda^2=-1$, one has
$1+\lambda+\lambda^2=\lambda\neq0$.
Hence  {\eqref{eq:x3-explicit}} gives
\(x_3=(1+\lambda+\lambda^2)(S_0-\kappa_3S_1)\), where
\(S_0=v\otimes x_2\) and
\(S_1=hv\otimes(g\rhd x_2)\). Moreover,
$X_3=\mathbb C G\rhd x_3$.
Write $q=g^2h$.
 {Equations~\eqref{eq:tensor-product} and
\eqref{eq:YD-projective-action} give}
\begin{align}
 g\rhd x_3
 &= (1+\lambda+\lambda^2)
 \bigl(A_g\,v\otimes(g\rhd x_2)-\kappa_3B_g\,hv\otimes x_2\bigr),
 \label{eq:Gamma2-g-action-x3-line}
 \\
 h\rhd x_3
 &= (1+\lambda+\lambda^2)
 \bigl(A_h\,hv\otimes x_2-\kappa_3B_h\,v\otimes(g\rhd x_2)\bigr),
 \label{eq:Gamma2-h-action-x3-line}
\end{align}
where
\begin{align*}
 A_g&=\Phi^g(g,q)\lambda,\quad
 A_h=\Phi^h(g,q)\sigma_2(h),
 \\
 B_g&=
 \Phi^g(\varepsilon g,\varepsilon q)
 \frac{\Phi_g(g,h)\lambda\zeta}
      {\Phi_g(\varepsilon,g)\Phi_g(h,\varepsilon g)}
 \Phi_q(g,g)\sigma_2(g^2),
 \\
 B_h&=
 \Phi^h(\varepsilon g,\varepsilon q)\Phi_g(h,h)\mu
 \frac{
  \Phi_q(h,g)\sigma_2(h)\sigma_2(\varepsilon)
 }{
  \Phi_q(\varepsilon g,h)\Phi_q(g,\varepsilon)
 }.
\end{align*}
Here $X_2\simeq M(q,\sigma_2)$ and
\begin{align}
 \sigma_2(\varepsilon)
 &=
 \Phi^\varepsilon(g,gh)\Phi^\varepsilon(g,h)
 \zeta^2\zeta',
 \label{eq:Gamma2-sigma2-epsilon-parameter}
 \\
 \sigma_2(g^2)
 &=
 \Phi^{g^2}(g,gh)\Phi^{g^2}(g,h)
 \frac{\lambda^4\mu'}{\Phi_g(g,g)^2}.
 \label{eq:Gamma2-sigma2-g2-parameter}
\end{align}

 {By Lemma~\ref{lem:Gamma2-normal-forms-centralizers},
$C_G(g^3h)=\langle\varepsilon,g^2,h^2,g^{-1}h\rangle$.}
The elements $\varepsilon,g^2,h^2$ belong to
$C_G(g)\cap C_G(q)$.  The line $\mathbb C(v\otimes x_2)$ is therefore
stable under these three elements.    The map \(\varphi_3^\Phi\) commutes
with the \(G\)-action.  Hence, for
\(a\in\{\varepsilon,g^2,h^2\}\),
\[
a\rhd x_3
=\varphi_3^\Phi\bigl(a\rhd(v\otimes x_2)\bigr)
\in\mathbb Cx_3.
\]
Thus \(\mathbb Cx_3\) is stable under these three elements.  It remains
to check stability under \(g^{-1}h\).  By the invertibility of the projective action, this is
equivalent to
\(\mathbb C(g\rhd x_3)=\mathbb C(h\rhd x_3)\).
Equations \eqref{eq:Gamma2-g-action-x3-line} and
\eqref{eq:Gamma2-h-action-x3-line} show that this holds if and only if
\begin{equation}
\label{eq:Gamma2-X3-determinant-condition}
 \kappa_3^2B_gB_h=A_gA_h.
\end{equation}
Lemma~\ref{lem:Gamma2-projective-representations-one-dimensional} now
shows that the invariance of \(\mathbb Cx_3\) under
\(C_G(g^3h)\) is equivalent to the simplicity of the object
$X_3=\mathbb C G\rhd x_3$.

Substituting \eqref{eq:Gamma2-sigma2-epsilon-parameter} and
\eqref{eq:Gamma2-sigma2-g2-parameter} into
\(\kappa_3^2B_gB_h/(A_gA_h)\), we obtain
\(\zeta^5\zeta'\mu\mu'\lambda^4\) times a scalar depending only
on \(\Phi\).  Using
\(\lambda^4=1\), \(\zeta^2=\Phi_g(\varepsilon,\varepsilon)\),
\(\Delta=1\), and the definition of \(\Delta\), the quotient becomes
$\mathfrak s_\Phi(g,h)$.  Thus
\eqref{eq:Gamma2-X3-determinant-condition} is equivalent to
$\mathfrak s_\Phi(g,h)=1$.
Lemma~\ref{lem:Gamma2-X3-cocycle-factor-identity} gives
\(\mathfrak s_\Phi(g,h)=\Theta_\Phi(g,h)^2=1\). Hence the determinant
condition holds and $X_3$ is simple.
\end{proof}

\begin{lemma}
\label{lem:Gamma2-G2-parameter-identities}
The parameters defined above satisfy
\begin{align}
&\Delta=1
\quad\Longleftrightarrow\quad
\Delta'=1,
\label{eq:Gamma2-Delta-unit-loci}\\
&\Theta_\Phi(hg,h^{-1})
=
\Theta_\Phi(g,h).
\label{eq:Gamma2-R2-Theta-invariance}
\end{align}

\end{lemma}

\begin{proof}
Equation~\eqref{eq:first-adjoint-braiding-symmetry} gives
\(X_1\simeq Y_1\). Hence Lemmas~\ref{lem:Gamma2-X1-simple}
and~\ref{lem:opposite-first-adjoint-simple} yield
\eqref{eq:Gamma2-Delta-unit-loci}.
Equation \eqref{eq:Gamma2-R2-Theta-invariance} is the cocycle identity
proved in
Lemma~\ref{lem:Gamma2-R2-Theta-cocycle}.
\end{proof}

\begin{prop}
\label{prop:Gamma2-standard-G2-stability-under-reflections}
Assume
\begin{equation}
\label{eq:Gamma2-standard-G2-locus}
 \Delta=1,
 \qquad
 \lambda'=-1,
 \qquad
 \lambda^2=\Theta_\Phi(g,h)=-1.
\end{equation}
Then $(V,W)$ admits all reflections and its Cartan graph is standard
of type $G_2$.
\end{prop}

\begin{proof}
At the initial object, the first two conditions and
\eqref{eq:Gamma2-Delta-unit-loci} make $X_1$ and $Y_1$ simple and give
$Y_2=0$.  Lemma \ref{lem:Gamma2-X2-simple} and
Lemma~\ref{lem:Gamma2-X3-simplicity} make
$X_2$ and $X_3$ simple and nonzero.  Since
\(1+\lambda+\lambda^2+\lambda^3=0\),
\eqref{eq:Gamma2-x4-full} gives $X_4=0$.  Thus both initial
reflections are defined and the initial Cartan matrix is of type $G_2$.

Let \(Q_1=R_1(V,W)\), so that \(Q_1=(V^*,X_3)\). Use the support
representatives \((g^{-1},g^3h)\) from
Lemma~\ref{lem:reflections-remain-Gamma2}. By
\cite[Lemma~3.8]{reflection3},
\[
 (\operatorname{ad}V^*)(X_3)\simeq X_2,
 \qquad
 (\operatorname{ad}V^*)^2(X_3)\simeq X_1.
\]
Since \(X_2\) is simple,
Lemma~\ref{lem:Gamma2-X1-simple}, applied to \(Q_1\), gives
\(\Delta_{Q_1}=1\).
Lemma~\ref{lem:dual-projective-representation} gives
\(\lambda_{Q_1}=\lambda\). In particular,
\(\lambda_{Q_1}\ne-1\), because \(\lambda^2=-1\). The second
displayed isomorphism and Lemma~\ref{lem:Gamma2-X2-simple} therefore
give \(\Theta_{Q_1}=\lambda_{Q_1}^2=\lambda^2=-1\). On the other
hand, Lemma~\ref{lem:Gamma2-X3-support-value} gives
\(\lambda'_{Q_1}=\sigma_3(g^3h)=-\lambda^6\lambda'=-1\).
Thus \(Q_1\) satisfies all three equations in
\eqref{eq:Gamma2-standard-G2-locus}.

Let $Q_2=R_2(V,W)=(Y_1,W^*)$.
By \cite[Lemma~3.8]{reflection3} and
\eqref{eq:first-adjoint-braiding-symmetry},
\((\operatorname{ad}Y_1)(W^*)\simeq V\), so
\(\Delta_{Q_2}=1\).
Lemma~\ref{lem:Gamma2-X1-Y1-inducing-values} and the cocycle identity
in Lemma~\ref{lem:Gamma2-G2-parameter-identities} give
\(\lambda_{Q_2}=\tau_1(hg)=-\lambda\lambda'=\lambda\),
\(\lambda'_{Q_2}=\lambda'=-1\), and
\(\Theta_{Q_2}=\Theta_\Phi(g,h)=-1\).
Thus \(Q_2\) also satisfies \eqref{eq:Gamma2-standard-G2-locus}.

After either reflection, the support assumptions and all three
parameter equations are preserved.  The first paragraph of the proof
therefore applies to each reflected tuple.  Induction proves that all
reflections are defined and that the Cartan graph is standard of type
\(G_2\).
\end{proof}

\begin{cor}
\label{cor:Gamma2-standard-G2-conditions}
The following conditions are equivalent:

(1)
The objects $X_1,X_2,X_3,Y_1$ are simple and nonzero, while
$X_4=Y_2=0$.

(2)
The equations in \eqref{eq:Gamma2-standard-G2-locus} hold.
\end{cor}

\begin{proof}
Assume (1).  Lemmas~\ref{lem:Gamma2-X1-simple} and
\ref{lem:opposite-first-adjoint-simple} give
\(\Delta=\Delta'=1\).  Since \(Y_2=0\),
Lemma~\ref{lem:opposite-second-adjoint-vanishing} gives
\(\lambda'=-1\).
Since \(X_2\neq0\), Lemma~\ref{lem:second-adjoint-vanishing} gives
\(\lambda\neq-1\).  The simplicity of \(X_2\) and
Lemma~\ref{lem:Gamma2-X2-simple} then give
\(\lambda^2=\Theta_\Phi(g,h)\).  Finally,
Lemma~\ref{lem:Gamma2-fourth-adjoint-full} and \(X_4=0\) give
\((1+\lambda)(1+\lambda^2)=0\).  Since \(\lambda\neq-1\), it follows
that \(\lambda^2=-1\).  Thus
\(\lambda^2=\Theta_\Phi(g,h)=-1\), and (2) holds.

 {The implication (2)\(\Rightarrow\)(1) was proved in the first
paragraph of the proof of
Proposition~\ref{prop:Gamma2-standard-G2-stability-under-reflections}.}
\end{proof}

\subsection{Finite Cartan graphs}
\label{subsec:Gamma2-finite-Cartan-graphs}

\begin{prop}
\label{prop:Gamma2-reduction-to-A2-or-G2}
Assume that \((V,W)\) is braided-indecomposable and that
\(\dim\mathcal B(V\oplus W)<\infty\).
Then its Cartan graph is standard
of type $A_2$ or standard of type $G_2$, up to interchanging the two
indices.
\end{prop}

\begin{proof}
 {By Theorem~\ref{thm:Nichols-finiteness-criterion},
$(V,W)$ admits all reflections and its Cartan graph is finite.} It is connected by
Proposition~\ref{prop:associated-Cartan-graph}. By
Lemma~\ref{lem:finite-type-object}, there exists an object $P=(P_1,P_2)$ whose Cartan
matrix is of finite type. The tuple \(P\) again satisfies the standing
\(\Gamma _2\) support assumptions, and all
nonzero objects obtained by the two iterated braided adjoint actions
are simple by
Lemma~\ref{lem:root-string-adjoint-simplicity}. The possible indecomposable
rank-two Cartan matrices of finite type are \(A_2,B_2\), and \(G_2\), up to
transposition.

If the matrix at $P$ is of type $A_2$, Proposition
\ref{prop:Gamma2-standard-A2}, applied to $P$, shows that the entire
connected graph is standard of type $A_2$.  Type $B_2$ is impossible
by Proposition~\ref{prop:Gamma2-no-local-B2}.  Finally, if the matrix
at $P$ is of type $G_2$, interchange the entries of \(P\), if
necessary, so that its matrix is
\(\left(\begin{smallmatrix}2&-3\\-1&2\end{smallmatrix}\right)\).
 {Corollary~\ref{cor:Gamma2-standard-G2-conditions}, applied
to \(P\), gives \eqref{eq:Gamma2-standard-G2-locus}.
Proposition~\ref{prop:Gamma2-standard-G2-stability-under-reflections}
then shows that the  connected Cartan graph is standard of type
\(G_2\).}
\end{proof}

\subsection{ {Nichols algebras corresponding to positive roots of type
\texorpdfstring{$G_2$}{G2}}}
\label{subsec:Gamma2-G2-simple-factors}
\begin{lemma}
\label{lem:Gamma4-support-group-restriction}
Let \(H\leq G\), and let \(R\in{}_G^G\mathcal{YD}^{\Phi}\) satisfy
\(\operatorname{supp}R\subseteq H\).  Restricting the action to
\(H\) and viewing the grading as an \(H\)-grading gives an object
\(R|_H\in{}_H^H\mathcal{YD}^{\Phi|_H}\).  After restriction to
\(H\), the identity maps on the tensor powers identify
\(\mathcal B(R)\) with \(\mathcal B(R|_H)\).  In particular, the two
Nichols algebras have the same dimension.
\end{lemma}

\begin{proof}
The defining identities for \(R|_H\) are obtained by restricting
those for \(R\), so \(R|_H\) is well defined.  On every tensor power
of \(R\), the grading and the \(H\)-action agree in the two
categories.  Since all homogeneous degrees belong to \(H\), the
associativity constraints and braidings also agree.  Hence the
quantum symmetrizers, and therefore the Nichols ideals, coincide in
every degree.  The identity maps on the tensor powers induce the
asserted isomorphism.
\end{proof}

Let \(R=M(x,\tau)\) be one of the simple objects corresponding
to the positive roots considered below.
 {Choose nonzero vectors \(u_0\in R_x\) and
\(u_1\in R_{\varepsilon x}\), and define \(q_{ij}^R\) by}
\[
 {c_{R,R}(u_i\otimes u_j)=q_{ij}^R u_j\otimes u_i
\qquad(i,j\in\{0,1\}).}
\]
After restriction to the abelian subgroup generated by its support,
\(R\) is of diagonal type. By
\cite[Proposition~2.11]{huang2024classification},
\(\mathcal B(R)\) has the same  {dimension and generalized
Dynkin diagram} as an ordinary Nichols algebra of diagonal type.
We may therefore apply the classification of ordinary Nichols
algebras of diagonal type.
\begin{lemma}
\label{lem:Gamma2-diagonal-braiding}
 {Let $P=(R,S)$ be a braided-indecomposable pair of simple objects
with support representatives $(x,y)$ such that
$G=\langle x,y\rangle$ and $yx=\varepsilon xy$, and write
$R=M(x,\tau)$.}  If
$\Theta_P$ denotes the scalar obtained from
\eqref{eq:Theta-Phi-definition} using the support representatives
\((x,y)\) in this order, then
\begin{equation}
\label{eq:Gamma2-diagonal-braiding}
q_{00}^R=q_{11}^R=\tau(x),
\qquad
q_{01}^Rq_{10}^R=\frac{\tau(x)^2}{\Theta_P}.
\end{equation}
\end{lemma}

\begin{proof}
Choose \(0\neq u_0\in R_x\) and let
\(u_1=y\rhd u_0\in R_{\varepsilon x}\).  Then \(u_0,u_1\) is a
homogeneous basis of \(R\).  Since
\(x\rhd u_0=\tau(x)u_0\), the definition of the braiding gives
\(q_{00}^R=\tau(x)\).
The braiding commutes with the \(G\)-action.  Apply \(y\) to
\(c_{R,R}(u_0\otimes u_0)=\tau(x)u_0\otimes u_0\).
Since \(y\rhd(u_0\otimes u_0)\) is a nonzero scalar multiple of
\(u_1\otimes u_1\), it follows that
\(c_{R,R}(u_1\otimes u_1)=\tau(x)u_1\otimes u_1\).
Thus \(q_{11}^R=q_{00}^R=\tau(x)\).
Then \eqref{eq:YD-projective-action} gives
\[
q_{10}^R
=\frac{\tau(x)\tau(\varepsilon)}{\Phi_x(\varepsilon,x)},
\qquad
q_{01}^R
=\frac{
\tau(x)\tau(\varepsilon)\Phi_x(x,y)
}{
\Phi_x(\varepsilon,x)\Phi_x(y,\varepsilon x)
}.
\]
Since $\tau(\varepsilon)^2=\Phi_x(\varepsilon,\varepsilon)$, their product is
\[
q_{01}^Rq_{10}^R
=\tau(x)^2
\frac{
\Phi_x(\varepsilon,\varepsilon)\Phi_x(x,y)
}{
\Phi_x(\varepsilon,x)^2\Phi_x(y,\varepsilon x)
}.
\]
Lemma~\ref{lem:Theta-diagonal-braiding-cocycle} identifies the last
fraction with $\Theta_P^{-1}$, proving
\eqref{eq:Gamma2-diagonal-braiding}.
\end{proof}

\begin{lemma}
\label{lem:Gamma2-X2-support-value}
Assume that $X_1$ and $X_2$ are simple and nonzero, and write
$X_2\simeq M(g^2h,\sigma_2)$.  Then
\begin{equation}
\label{eq:Gamma2-sigma2-support-value}
\sigma_2(g^2h)=\lambda^3\lambda'.
\end{equation}
\end{lemma}

\begin{proof}
Acting by $g^2h$ on the two  bi-homogeneous components in the
formula for $x_2$ gives
\[
\sigma_2(g^2h)=\lambda^3\lambda'K_2,
\]
where
\begin{equation}
\label{eq:Gamma2-K2-support-value}
K_2
:=
\frac{
\Phi_g(h,g)\Phi_g(h,\varepsilon)\Phi^{gh}(g,h)
\Phi^{g^2h}(g,gh)
}{
\Phi_h(g,h)\Phi_g(h,g^2)\Phi_g(h,\varepsilon g)
\Phi_g(g,g)\Phi_g(\varepsilon,g)\Phi_{gh}(g,gh)
}.
\end{equation}
Lemma~\ref{lem:Gamma2-K2-cocycle-reduction} gives $K_2=1$, proving
\eqref{eq:Gamma2-sigma2-support-value}.
\end{proof}

 {The simple objects whose Nichols algebras occur in the tensor
decomposition are given explicitly below.}

\begin{prop}
\label{prop:standard-G2-simple-factors}
Assume that \((V,W)\) admits all reflections and that its Cartan graph
is standard with Cartan matrix
\[
\begin{pmatrix}
2&-3\\
-1&2
\end{pmatrix}.
\]
Set \(H=(\operatorname{ad}Y_1)^3(W^*)\).  The simple objects
corresponding to the six positive roots of \(G_2\) have the following
support representatives and braiding coefficients:
\begin{equation}
\label{eq:standard-G2-root-data}
\begin{array}{c|c|c|c|c}
\textnormal{root}
&\textnormal{object}
&\textnormal{support representative}
&q_{00}=q_{11}
&q_{01}q_{10}
\\ \hline
\alpha _1
&V
&g
&\lambda
&1
\\
\alpha _2
&W
&h
&-1
&-1
\\
\alpha _1+\alpha _2
&X_1
&gh
&\lambda
&1
\\
2\alpha _1+\alpha _2
&X_2
&g^2h
&\lambda
&1
\\
3\alpha _1+\alpha _2
&X_3
&g^3h
&-1
&-1
\\
3\alpha _1+2\alpha _2
&H
&g^3h^2
&-1
&-1
\end{array}
\end{equation}
Consequently,
\begin{equation}
\label{eq:G2-rank-one-dimensions}
\begin{aligned}
\dim\mathcal B(R)&=16
&&\text{for }R\in\{V,X_1,X_2\},\\
\dim\mathcal B(R)&=8
&&\text{for }R\in\{W,X_3,H\}.
\end{aligned}
\end{equation}
\end{prop}

\begin{proof}
Proposition~\ref{prop:positive-root-factorization} gives a tensor
decomposition of \(\mathcal B(V\oplus W)\) indexed by the six positive
roots of \(G_2\). The corresponding simple objects are identified by
Lemmas~\ref{lem:root-factors-and-reflected-components} and
\ref{lem:root-string-adjoint-simplicity}.

For the chosen Cartan matrix, \(X_1,X_2,X_3\), and \(Y_1\) are simple
and nonzero, while \(X_4=Y_2=0\).  Corollary
\ref{cor:Gamma2-standard-G2-conditions} therefore gives
\(\lambda^2=\Theta_\Phi(g,h)=-1\) and \(\lambda'=-1\).
In particular, \(\lambda\) is a primitive fourth root of unity.
We use \(gh\), \(g^2h\), and \(g^3h\) as support representatives of
\(X_1,X_2\), and \(X_3\).  Equations
\eqref{eq:Gamma2-X1-Y1-support-values},
\eqref{eq:Gamma2-sigma2-support-value}, and
\eqref{eq:Gamma2-sigma3} give, respectively,
\(\sigma_1(gh)=-\lambda\lambda'=\lambda\),
\(\sigma_2(g^2h)=\lambda^3\lambda'=\lambda\), and
\(\sigma_3(g^3h)=-\lambda^6\lambda'=-1\).
Lemma~\ref{lem:Gamma2-diagonal-braiding} shows that each of these
values is equal to both diagonal coefficients \(q_{00}\) and
\(q_{11}\).  Together with the values for \(V\) and \(W\), this proves
the fourth column of \eqref{eq:standard-G2-root-data}, except for the
highest root.

Let \(Q=R_2(V,W)=(Y_1,W^*)\).  By
Lemma~\ref{lem:reflections-remain-Gamma2}, we may use
\((hg,h^{-1})\) as the support representatives of \(Q\).  Let
\(\lambda_Q\) and \(\lambda'_Q\) denote the corresponding values of
the two projective representations.  Equation
\eqref{eq:Gamma2-X1-Y1-support-values} and
Lemma~\ref{lem:dual-projective-representation} give
\(\lambda_Q=\tau_1(hg)=\lambda\) and \(\lambda'_Q=-1\).

At \(Q\), the third adjoint object in the first direction is
\(H=(\operatorname{ad}Y_1)^3(W^*)\), and
\((hg)^3h^{-1}=g^3h^2\) is a support representative of \(H\).
Applying \eqref{eq:Gamma2-sigma3} to \(Q\) gives
\(\sigma_3^Q(g^3h^2)=-\lambda_Q^6\lambda'_Q=-1\).
Moreover,
\(s_2(3\alpha_1+\alpha_2)=3\alpha_1+2\alpha_2\).
Thus \(H\) corresponds to the highest root.  Lemma
\ref{lem:Gamma2-diagonal-braiding} gives
\(q_{00}=q_{11}=-1\) for \(H\).

It remains to compute the last column of the table.  Let
\(P=(P_s,P_\ell)\) be any tuple obtained from \((V,W)\) by
reflections, where \(P_s\) and \(P_\ell\) correspond to a short and a
long simple root, respectively.  By
Lemma~\ref{lem:reflections-remain-Gamma2}, the results proved above for
the \(\Gamma_2\) case also apply to \(P\).  Write
\(\lambda_P,\lambda'_P\), and \(\Theta_P\) for its corresponding
parameters.  Corollary~\ref{cor:Gamma2-standard-G2-conditions} gives
\(\lambda_P^2=\Theta_P=-1\) and \(\lambda'_P=-1\).
Lemma~\ref{lem:Theta-opposite-symmetry} shows that \(\Theta_P\) is
still \(-1\) when the two entries are interchanged.  Formula
\eqref{eq:Gamma2-diagonal-braiding}, applied in the two orders, now
gives
\(q_{01}^{P_s}q_{10}^{P_s}=1\) and
\(q_{01}^{P_\ell}q_{10}^{P_\ell}=-1\).
Every object in the table occurs in such a reflected tuple by
Lemmas~\ref{lem:root-factors-and-reflected-components} and
\ref{lem:root-string-adjoint-simplicity}. This proves the last column.

Finally, for
\(V,X_1\), and \(X_2\),  {the generalized Dynkin diagram consists of two disconnected
vertices labelled by \(\lambda\).}  Since \(\lambda\) has order four,
the corresponding Nichols algebra has dimension \(4^2=16\).
For \(W,X_3\), and \(H\), both vertex labels and the product of the
off-diagonal coefficients are \(-1\).  This is the diagonal case of
type \(A_2\). Its three root vectors square to zero, so the dimension
is \(2^3=8\).  This proves
\eqref{eq:G2-rank-one-dimensions}.
\end{proof}
\begin{cor}
\label{cor:standard-G2-Nichols-dimension}
Under the hypotheses of
Proposition~\ref{prop:standard-G2-simple-factors}, one has
\begin{equation*}
\dim\mathcal B(V\oplus W)
=16^3 8^3
=2^{21}.
\end{equation*}
\end{cor}

\begin{proof}
A standard Cartan graph of type \(G_2\) is finite.  Proposition
\ref{prop:positive-root-factorization} therefore gives a tensor
decomposition  {whose factors are the Nichols algebras of the
six simple objects in \eqref{eq:standard-G2-root-data}.
Their dimensions are given by \eqref{eq:G2-rank-one-dimensions}.}  Multiplying these
dimensions gives
\(\dim\mathcal B(V\oplus W)=16^3\,8^3=2^{21}\).
\end{proof}

\begin{rmk}\upshape
\label{rmk:Gamma2-G2-no-Hopf-counterpart}
The \(\Gamma _2\)-\(G_2\) case obtained above does not occur when
\(\Phi=1\).  Indeed, in this case all the cocycle terms in
\eqref{eq:Theta-Phi-definition} are equal to \(1\), and hence
\(\Theta_1(g,h)=1\).  On the other hand, the \(G_2\) conditions require
\(\Theta_\Phi(g,h)=-1\).  Thus this case has no counterpart in the
ordinary Yetter--Drinfeld category
\({}_G^G\mathcal{YD}\) over the Hopf algebra \(\mathbb C G\).
In fact, in the ordinary \(\Gamma _2\) case, a finite-dimensional
Nichols algebra under the standing assumptions has Cartan type
\(A_2\), see \cite[Theorem~4.6]{rank2-1}.
\end{rmk}

\subsection{The \texorpdfstring{$A_2$}{A2} case and completion of the classification}
\label{subsec:Gamma2-classification}

 {For type \(A_2\), it remains to prove that the Nichols algebras
of the three simple objects \(V,X_1,W\) are finite-dimensional.}

\begin{lemma}
\label{lem:Gamma2-A2-simple-factors-finite}
If \(\Delta=\Delta'=1\) and \(\lambda=\lambda'=-1\),
then $\mathcal B(V)$, $\mathcal B(W)$ and $\mathcal B(X_1)$ are
finite-dimensional.
\end{lemma}

\begin{proof}
Let \(R\in\{V,W,X_1\}\), and write
\(\operatorname{supp}R=\{x,\varepsilon x\}\).
Both homogeneous components of \(R\) are one-dimensional. Since
\(\langle\varepsilon,x\rangle\) is abelian, the reduction in
Subsection~\ref{subsec:Gamma2-G2-simple-factors} gives  {an ordinary
Nichols algebra of diagonal type} with the same  {diagonal coefficients,
edge label, and dimension}.

The two diagonal coefficients are both \(-1\).  For \(V\) and \(W\),
this follows from \(\lambda=\lambda'=-1\).  For \(X_1\), it follows
from \(\sigma_1(gh)=-\lambda\lambda'=-1\).  Let
\(Q_R=q_{01}^Rq_{10}^R\).
{
We first show that \(Q_R\) is a root of unity. Let \(n=|G|\).
As in the last paragraph of the proof of
Lemma~\ref{lem:Gamma2-primary-cocycle}, the class of \(\Phi\)
has a normalized representative \(\widetilde\Phi\) with values
in \(\mu_n\). The associated braided monoidal equivalence
preserves the eigenvalues of the double braiding and Nichols
algebra dimensions. We may therefore work with
\(\widetilde\Phi\).
}

Let \(\widetilde\chi\) be the one-dimensional
\((\widetilde\Phi)_x\)-projective representation on the homogeneous
component of degree \(x\).  If \(c\in C_G(x)\) has order \(m\), repeated
use of the equation \eqref{eq:YD-projective-action} gives
\[
\widetilde\chi(c)^m
=\prod_{j=1}^{m-1}(\widetilde\Phi)_x(c,c^j).
\]
Every factor on the right belongs to \(\mu_n\).  Hence
\(\widetilde\chi(c)^{mn}=1\), so \(\widetilde\chi(c)\) is a root of
unity.  The coefficients of the diagonal braiding are products of
values of \(\widetilde\Phi\) and values of these projective
representations.  Thus \(Q_R\) is a root of unity.

If \(Q_R=1\), the diagram is of type \(A_1\times A_1\).  Both
generators square to zero, and hence
\(\dim\mathcal B(R)=2^2=4\).
Suppose that \(Q_R\neq1\), and let
\(N_R=\operatorname{ord}(Q_R)\).  Since both diagonal coefficients are
\(-1\), the generalized Cartan matrix is of type \(A_2\).  The
braiding coefficients corresponding to its three positive roots are
\(-1,-1\), and \(Q_R\).  Their orders are therefore \(2,2\), and
\(N_R\), respectively.  Hence
\(\dim\mathcal B(R)=2\cdot2\cdot N_R=4N_R\).
Thus \(\mathcal B(R)\) is finite-dimensional in both cases.
\end{proof}

\begin{cor}
\label{cor:Gamma2-A2-Nichols-dimension}
Assume \eqref{eq:Gamma2-classification-A2}, and let
\(X_1=(\operatorname{ad}V)(W)\).   {For \(R\in\{V,X_1,W\}\), set
\(Q_R=q_{01}^Rq_{10}^R\) and \(N_R=\operatorname{ord}(Q_R)\).}  Then
\[
\dim\mathcal B(V\oplus W)
=64N_VN_{X_1}N_W.
\]
\end{cor}

\begin{proof}
The proof of Lemma~\ref{lem:Gamma2-A2-simple-factors-finite} gives
\(\dim\mathcal B(R)=4N_R\) for
\(R\in\{V,X_1,W\}\).  Equation
\eqref{eq:Gamma2-Delta-unit-loci} gives \(\Delta'=1\), so
\eqref{eq:standard-A2-scalar-conditions} holds.  Proposition
\ref{prop:Gamma2-standard-A2} gives the standard \(A_2\) Cartan graph,
while Proposition~\ref{prop:positive-root-factorization} and
Lemma~\ref{lem:root-string-adjoint-simplicity} give the corresponding
tensor decomposition with factors  {\(\mathcal B(V)\),
\(\mathcal B(X_1)\), and \(\mathcal B(W)\)}. Multiplying their
dimensions proves the formula.
\end{proof}

\begin{proof}[Proof of Theorem~\ref{thm:Gamma2-classification}]

Assume that \(\mathcal B(V\oplus W)\) is finite-dimensional.
By Proposition~\ref{prop:Gamma2-reduction-to-A2-or-G2}, after
interchanging the two entries if necessary, its Cartan graph is
standard of type \(A_2\) or \(G_2\).

Suppose first that it is of type \(A_2\).  Then \(X_1\) and \(Y_1\)
are simple, while \(X_2=Y_2=0\).
Lemmas~\ref{lem:Gamma2-X1-simple} and
\ref{lem:opposite-first-adjoint-simple} give
\(\Delta=\Delta'=1\).  Lemmas~\ref{lem:second-adjoint-vanishing} and
\ref{lem:opposite-second-adjoint-vanishing} give
\(\lambda=\lambda'=-1\).  Hence
\eqref{eq:Gamma2-classification-A2} holds.

Suppose next that the Cartan graph is of type \(G_2\).
Interchange \((g,V,\rho)\) and \((h,W,\sigma)\), if necessary, so that
the Cartan matrix is
\(\left(\begin{smallmatrix}2&-3\\-1&2\end{smallmatrix}\right)\).
Then \(X_1,X_2,X_3\), and \(Y_1\) are simple and nonzero, while
\(X_4=Y_2=0\).  Corollary
\ref{cor:Gamma2-standard-G2-conditions} now gives
\eqref{eq:Gamma2-classification-G2}.

Conversely, assume first that \eqref{eq:Gamma2-classification-A2} holds.
Corollary~\ref{cor:Gamma2-A2-Nichols-dimension} shows that
\(\mathcal B(V\oplus W)\) is finite-dimensional.

Finally, assume that \eqref{eq:Gamma2-classification-G2} holds.
Proposition~\ref{prop:Gamma2-standard-G2-stability-under-reflections}
shows that the tuple admits all reflections and that its Cartan graph
is standard of type \(G_2\).  Corollary
\ref{cor:standard-G2-Nichols-dimension} then gives
\(\dim\mathcal B(V\oplus W)=2^{21}\).  In particular,
\(\mathcal B(V\oplus W)\) is finite-dimensional.
\end{proof}

\subsection{An explicit \texorpdfstring{$G_2$}{G2} example over a group of order
\texorpdfstring{$16$}{16}}
\label{subsec:order16-G2-example}

The following data give an explicit finite quotient of \(\Gamma_2\), a normalized
\(3\)-cocycle, and two simple objects satisfying
\eqref{eq:Gamma2-classification-G2}.  The resulting Cartan graph is
standard of type \(G_2\), and the Nichols algebra has dimension
\(2^{21}\).

For \(b\in\mathbb Z/4\mathbb Z\), let
\(\overline b\in\mathbb Z/2\mathbb Z\) denote its reduction modulo
two.  For later use, also let
\(\widehat b=(b-\overline b)/2\in\mathbb Z/2\mathbb Z\), using
the representatives \(b\in\{0,1,2,3\}\) and
\(\overline b\in\{0,1\}\).
Let
\[
G=\{[e,a,b]\mid e,a\in\mathbb Z/2\mathbb Z,\ 
b\in\mathbb Z/4\mathbb Z\}.
\]
Define multiplication by
\begin{equation}
\label{eq:order16-group-law}
[e,a,b][f,c,d]
=
[e+f+\overline b\,c,\ a+c,\ b+d],
\end{equation}
where the first two coordinates are computed modulo two and the third
modulo four.

This multiplication is associative.  Indeed, only the first
coordinate requires verification.  For a third element \([r,u,v]\),
associativity follows from
\(\overline b\,c+\overline{b+d}\,u
=\overline d\,u+\overline b(c+u)\) in \(\mathbb F_2\), since
\(\overline{b+d}=\overline b+\overline d\).
The identity element is \([0,0,0]\), and
\([e,a,b]^{-1}=[e+\overline b\,a,a,-b]\).
Thus \(G\) is a group of order \(2\cdot2\cdot4=16\).

Set
\(\varepsilon=[1,0,0]\), \(g=[0,1,0]\), and \(h=[0,0,1]\).
A direct calculation using \eqref{eq:order16-group-law} gives
\(hg=\varepsilon gh\),
\(\varepsilon^2=g^2=h^4=1\), and
\(\varepsilon\in Z(G)\).
Moreover,
\(\varepsilon=hgh^{-1}g^{-1}\), and every element of \(G\) can be
written as
\([e,a,b]=\varepsilon^e g^a h^b\).
Hence \(g\) and \(h\) generate \(G\).
Since \(\varepsilon\neq1\), the relation
\(hg=\varepsilon gh\) also shows that \(G\) is non-abelian.
Therefore the assignment of the generators of \(\Gamma_2\) to
\(\varepsilon,g,h\) defines a surjective homomorphism
\(\Gamma_2\twoheadrightarrow G\).
For
$x=[e,a,b]$, $y=[f,c,d]$ and $z=[r,u,t]$, define
\begin{equation}
\label{eq:order16-G2-cocycle}
\Phi(x,y,z)
=
(-1)^{\omega(x,y,z)},
\qquad
\omega(x,y,z)
=
c\bigl(au+\overline t(e+a\overline b+a\overline d)
+a\widehat t\bigr).
\end{equation}
All exponents in the definition of $\omega$ are evaluated in
$\mathbb F_2$.

\begin{lemma}
\label{lem:order16-G2-cocycle}
The map $\Phi$ in \eqref{eq:order16-G2-cocycle} is a normalized
$3$-cocycle on $G$.
\end{lemma}

\begin{proof}
All exponents in this proof are computed in \(\mathbb F_2\).
If one of the three arguments is the identity element, every term in
the formula for \(\omega\) is zero.  Hence
\(\Phi=(-1)^\omega\) is normalized.

For
\(x=[e_x,a_x,b_x]\) and \(y=[e_y,a_y,b_y]\), write
\(xy=[e_{xy},a_{xy},b_{xy}]\).  The multiplication in
\eqref{eq:order16-group-law} gives
\begin{equation}
\label{eq:order16-coordinate-products}
\begin{aligned}
e_{xy}&=e_x+e_y+\overline b_xa_y,
&
a_{xy}&=a_x+a_y,\\
\overline b_{xy}&=\overline b_x+\overline b_y,
&
\widehat b_{xy}&=\widehat b_x+\widehat b_y
+\overline b_x\overline b_y.
\end{aligned}
\end{equation}
Write \(x_j=[e_j,a_j,b_j]\) for \(0\leq j\leq3\).
By \eqref{eq:order16-G2-cocycle}, one has
\(\omega=\omega_1+\cdots+\omega_5\), where
\[
\begin{aligned}
\omega_1(x_0,x_1,x_2)&=a_0a_1a_2,
&
\omega_2(x_0,x_1,x_2)&=e_0a_1\overline b_2,\\
\omega_3(x_0,x_1,x_2)&=a_0\overline b_0a_1\overline b_2,
&
\omega_4(x_0,x_1,x_2)&=a_0a_1\overline b_1\overline b_2,\\
\omega_5(x_0,x_1,x_2)&=a_0a_1\widehat b_2.
\end{aligned}
\]
Substitution of \eqref{eq:order16-coordinate-products} into the
inhomogeneous coboundary formula gives
\[
\begin{aligned}
\delta\omega_1&=0,\\
\delta\omega_2&=a_1a_2\overline b_0\overline b_3,\\
\delta\omega_3&=a_0a_2\overline b_1\overline b_3
+a_1a_2\overline b_0\overline b_3,\\
\delta\omega_4&=a_0a_1\overline b_2\overline b_3
+a_0a_2\overline b_1\overline b_3,\\
\delta\omega_5&=a_0a_1\overline b_2\overline b_3,
\end{aligned}
\]
where each coboundary is evaluated at \((x_0,x_1,x_2,x_3)\).
Adding these identities gives \(\delta\omega=0\).  Thus
\(\Phi=(-1)^\omega\) is a normalized \(3\)-cocycle.
\end{proof}

The relevant centralizers are
\begin{equation*}
C_G(g)=\{[e,a,2j]\},
\qquad
C_G(h)=\{[e,0,b]\}.
\end{equation*}
Here \(e,a,j\in\mathbb Z/2\mathbb Z\) and
\(b\in\mathbb Z/4\mathbb Z\).
Define
\begin{equation}
\label{eq:order16-projective-representations}
\rho([e,a,2j])=i^a,
\qquad
\sigma([e,0,b])=(-1)^{e+\overline b}.
\end{equation}
Direct substitution into  {the definition of $\Phi_x$} gives
\begin{equation*}
\begin{aligned}
\Phi_g([e,a,2j],[f,c,2k])&=(-1)^{ac},\\
\Phi_h([e,0,b],[f,0,d])&=1.
\end{aligned}
\end{equation*}
It follows at once that $\rho$ and $\sigma$ are respectively
one-dimensional $\Phi_g$- and $\Phi_h$-projective representations.
Set \(V=M(g,\rho)\) and \(W=M(h,\sigma)\).
 {The parameters of \((V,W)\) are}
\begin{align}
(\zeta,\lambda,\mu)&=(1,i,1),
& (\zeta',\lambda',\mu')&=(-1,-1,1),
\label{eq:order16-initial-parameters}
\\
\Delta&=\Delta'=1,
& \Theta_\Phi&=-1=\lambda^2.
\label{eq:order16-initial-reductions}
\end{align}
\begin{prop}
\label{prop:order16-standard-G2}

The pair \((V,W)\) defined by
\eqref{eq:order16-G2-cocycle} and
\eqref{eq:order16-projective-representations} admits all reflections
and has a standard Cartan graph of type \(G_2\). Moreover,
\(\dim\mathcal B(V\oplus W)=2^{21}\).
\end{prop}

\begin{proof}
{
Equations~\eqref{eq:order16-initial-parameters} and
\eqref{eq:order16-initial-reductions} give
\eqref{eq:Gamma2-standard-G2-locus}. The assertions follow from
Proposition~\ref{prop:Gamma2-standard-G2-stability-under-reflections}
and Corollary~\ref{cor:standard-G2-Nichols-dimension}.
}
\end{proof}

\section{The \texorpdfstring{$\Gamma_4$}{Gamma4} case}
\label{sec:classification-Gamma4}

Let
\[
\Gamma _4=
\left\langle g,h,\varepsilon\ \middle|\
hg=\varepsilon gh,\quad
g\varepsilon=\varepsilon^{-1}g,\quad
h\varepsilon=\varepsilon h,\quad
\varepsilon^4=1
\right\rangle,
\]
and let \(G\) be a finite non-abelian quotient of \(\Gamma _4\) in
which the image of \(\varepsilon\) has order four.
 {Write \(\pi:\Gamma _4\twoheadrightarrow G\) for the quotient map.}
The images of the generators in \(G\) are denoted by the same letters.  Let \(\Phi\) be a
normalized \(3\)-cocycle on \(G\).

Throughout this section, let \(V=M(h,\rho)\) be simple, and let
\(W=M(g,\sigma)\) be simple.  The support of \(V\) is
\(\{h,\varepsilon^{-1}h\}\), and the support of \(W\) is
\(\{g,\varepsilon g,\varepsilon^2g,\varepsilon^3g\}\).  Assume that
\((V,W)\) is braided-indecomposable.
{
Lemma~\ref{lem:Gamma4-projective-representations-one-dimensional}
below shows that \(\rho\) and \(\sigma\) are one-dimensional.
We use the scalar \(\Delta_4\) and the seven scalars
\(\Theta_{2,j}\), defined in
\eqref{eq:Gamma4-alpha-beta-Delta} and
\eqref{eq:Gamma4-Y2-H-Theta}, respectively.
} 
The classification for quotients of \(\Gamma_4\) is as follows.

\begin{thm}
\label{thm:Gamma4-classification}
Under the preceding assumptions,
\(\dim\mathcal B(V\oplus W)<\infty\) if and only if
{
\begin{equation*}
\tag{B$_2$}
\label{eq:Gamma4-classification-B2}
\rho(h)=\sigma(g)=-1,\qquad
\Delta_4=1,\qquad
\Theta_{2,j}=1\quad(1\leq j\leq7).
\end{equation*}
}
In this case, the Cartan graph is standard of type \(B_2\), with
matrix
\(\left(\begin{smallmatrix}2&-1\\-2&2\end{smallmatrix}\right)\)
at every object.
 {
One has \(\dim\mathcal B(V\oplus W)=2^{18}\).
}
\end{thm}

 {%
\medskip
\noindent\textit{Outline of the proof.}
\begin{itemize}[leftmargin=2em]
\item
{
Subsection~\ref{subsec:Gamma4-projective-representations} computes
\(H^3(\Gamma_4,\mathbb C^\times)\) and proves that both
\(\rho\) and \(\sigma\) are one-dimensional.
}
\item
Subsection~\ref{subsec:Gamma4-adjoint-objects} computes
\(X_1,X_2,Y_1,Y_2\) and obtains the parameter conditions in
\eqref{eq:Gamma4-classification-B2}.
\item
Subsection~\ref{subsec:Gamma4-sufficiency} proves that these
conditions make the pullback of \(\Phi\) a coboundary. The ordinary
classification then gives sufficiency and the dimension \(2^{18}\).
It also gives necessity when the initial Cartan matrix has finite type.
\item
Subsection~\ref{subsec:Gamma4-reflections} proves that the parameter
conditions are preserved by both reflections.
\item
Subsection~\ref{subsec:Gamma4-finite-Cartan-graphs} proves necessity
without the additional assumption on the initial Cartan matrix.
The final subsection completes the classification theorem.
\end{itemize}
}
\medskip
\noindent\textit{Notation used in this section.}
\begingroup
\small
\setlength{\tabcolsep}{4pt}
\renewcommand{\arraystretch}{1.08}
\begin{center}
\begin{tabular}{@{}L{0.19\textwidth}L{0.24\textwidth}L{0.49\textwidth}@{}}
\hline
Notation & Defined in & Used to prove \\
\hline
\(\rho(h),\sigma(g)\)
& Values of the inducing projective representations at the chosen
  support representatives
&  {%
The values \(\rho(h)=\sigma(g)=-1\), together with
  \(\Theta_{2,3}=1\), imply that \([\pi^*\Phi]=0\).
} \\
\(\Delta_4\)
& \eqref{eq:Gamma4-alpha-beta-Delta}
& The equality \(\Delta_4=1\) is equivalent to the simplicity of
  \(X_1\). \\
\(\Theta_{2,j}\), \(1\leq j\leq7\)
& \eqref{eq:Gamma4-Y2-H-Theta}
& Once \(\Delta_4=1\) and \(\sigma(g)=-1\), the seven equalities
  \(\Theta_{2,j}=1\) are exactly the remaining conditions for the
  simplicity of \(Y_2\). \\
\hline
\end{tabular}
\end{center}
\endgroup

\subsection{{Cohomology and projective representations}}
\label{subsec:Gamma4-projective-representations}

{
 {We compute \(C_G(h)\) and \(C_G(g)\). We then use
\(H^3(\Gamma_4,\mathbb C^\times)\) to show that \(\rho\) and
\(\sigma\) are one-dimensional.}
}

\begin{lemma}
\label{lem:Gamma4-centralizers}
The centralizers \(C_G(h)\) and \(C_G(g)\) are abelian.  More
precisely, \(C_G(h)=\langle\varepsilon,h,g^2\rangle\) and
\(C_G(g)=\langle \varepsilon^2,\varepsilon^{-1}h^2,g\rangle\).
\end{lemma}

\begin{proof}
The defining relations show that the two subgroups in the statement
are abelian and are contained in \(C_G(h)\) and \(C_G(g)\), respectively.
Every element of \(G\) has a representative of the form
\(\varepsilon^ih^jg^r\).  Hence the first subgroup has index at most
two.  The second has index at most four, since \(G\) is the union of
its cosets represented by \(1,\varepsilon,h,\varepsilon h\).  Since \(|h^G|=2\) and \(|g^G|=4\), the preceding inclusions
and index bounds give
\[
\begin{aligned}
2=[G:C_G(h)]
&\leq [G:\langle\varepsilon,h,g^2\rangle]\leq2,\\
4=[G:C_G(g)]
&\leq [G:\langle\varepsilon^2,\varepsilon^{-1}h^2,g\rangle]
\leq4.
\end{aligned}
\]
Thus each subgroup has the same index in \(G\) as the centralizer
containing it, so both inclusions are equalities.
\end{proof}

\begin{lemma}
\label{lem:Gamma4-third-cohomology}
{
Write every element of \(\Gamma _4\) uniquely as
\(x=\varepsilon^{a_x}h^{b_x}g^{c_x}\), where
\(0\leq a_x<4\) and \(b_x,c_x\in\mathbb Z\). Put
\(s_x=(-1)^{c_x}\), let \(A_x\in\{0,\ldots,7\}\) be congruent to
\(2a_x+b_x\) modulo eight, and set \(\zeta_8=e^{2\pi i/8}\).
Then
\[
\kappa(x,y,z)=\zeta_8^{\,A_xs_x\left\lfloor(A_y+s_yA_z)/8\right\rfloor}
\]
is a normalized \(3\)-cocycle, and
\(H^3(\Gamma _4,\mathbb C^\times)=\langle[\kappa]\rangle
\simeq\mathbb Z/8\mathbb Z\).
}
\end{lemma}

\begin{proof}
{
Let \(A=\langle\varepsilon,h\rangle\simeq C_4\times\mathbb Z\).
The presentation gives \(\Gamma _4=A\rtimes\langle g\rangle\), where
\(T(q)=gqg^{-1}\) satisfies
\(T(\varepsilon)=\varepsilon^{-1}\) and
\(T(h)=\varepsilon^{-1}h\). In particular,
\[
(a,b,c)(d,e,f)=
\bigl([a+(-1)^cd-\tfrac{1-(-1)^c}{2}e]_4,b+e,c+f\bigr).
\]
Here \([\cdot]_4\) denotes the representative modulo four in
\(\{0,1,2,3\}\). The cyclic resolution and the K\"unneth formula give
\(H^2(A,\mathbb C^\times)\simeq C_4\) and
\(H^3(A,\mathbb C^\times)\simeq C_4\).
 {The first group is identified with the group of fourth roots of unity by}
\([\beta]\mapsto\beta(\varepsilon,h)/\beta(h,\varepsilon)\),
so \(T^*=-1\) on it. Restriction identifies the second group with
\(H^3(\langle\varepsilon\rangle,\mathbb C^\times)\), on which
inversion acts trivially. Hence \(T^*=1\) on this group.
The Lyndon--Hochschild--Serre spectral sequence has only columns
zero and one because \(\langle g\rangle\simeq\mathbb Z\), and gives
\[
0\longrightarrow C_2\longrightarrow
H^3(\Gamma _4,\mathbb C^\times)
\longrightarrow C_4\longrightarrow0;
\]
see \cite[Chapter~VII, Section~6]{Brown1982}.
Thus this cohomology group has order eight.
\endgraf
We verify that \([\kappa]\) has order eight. {
Let \(d_s\) denote the cochain differential for the action
\(x\cdot a=s_xa\) on \(\mathbb R\) and \(\mathbb Z\).
For a one-cochain \(f\), this means
\((d_sf)(x,y)=f(x)+s_xf(y)-f(xy)\).
Since \(A_{xy}\) is the representative of \(A_x+s_xA_y\)
modulo eight in \(\{0,\ldots,7\}\), the cochain
\(u(x)=A_x/8\) satisfies
\[
(d_su)(x,y)
=\frac{A_x+s_xA_y-A_{xy}}8
=\left\lfloor\frac{A_x+s_xA_y}{8}\right\rfloor
\in\mathbb Z.
\]
The defining formula for \(\kappa\) can therefore be written as
$
\kappa(x,y,z)
=\exp\bigl(2\pi i\,u(x)s_x(d_su)(y,z)\bigr).
$
Using \(d_s^2u=0\), expansion of its coboundary gives
\[
(\delta\kappa)(x,y,z,t)
=\exp\bigl(
2\pi i\,s_xs_y(d_su)(x,y)(d_su)(z,t)
\bigr)=1,
\]
where \(\delta\) is the multiplicative coboundary with trivial
coefficient action. Moreover, \(u(1)=0\) and
\((d_su)(1,y)=(d_su)(x,1)=0\), so \(\kappa\) is normalized.
We next replace \(u\) by
\(v(x)=a_x/4+b_x/8\).
The difference
\[
n(x)=v(x)-u(x)=\frac{2a_x+b_x-A_x}{8}
\]
is integer-valued. Put \(r_x=(1-s_x)/2\). The multiplication
formula in \(\Gamma_4\) gives
\[
L(x,y):=(d_sv)(x,y)
=\frac{a_x+s_xa_y-r_xb_y-a_{xy}}4
\in\mathbb Z,
\qquad
L=d_su+d_sn.
\]
Define the normalized two-cochain
$
J_0(x,y)=\exp\bigl(-2\pi i\,u(x)s_xn(y)\bigr).
$
Its coboundary is
{
\[
\begin{aligned}
(\delta J_0)(x,y,z)
&=\exp\Bigl(2\pi i\bigl[
u(x)s_x(d_sn)(y,z)
-s_xs_y(d_su)(x,y)n(z)
\bigr]\Bigr)\\
&=\exp\bigl(2\pi i\,u(x)s_x(d_sn)(y,z)\bigr),
\end{aligned}
\]}
because \((d_su)(x,y)n(z)\) is an integer.
It follows that
\[
\begin{aligned}
\kappa(x,y,z)(\delta J_0)(x,y,z)
=\exp\bigl(2\pi i\,u(x)s_xL(y,z)\bigr)
=\exp\bigl(2\pi i\,v(x)s_xL(y,z)\bigr),
\end{aligned}
\]
where the last equality uses \(n(x)L(y,z)\in\mathbb Z\).
Taking fourth powers and using \(4v(x)=a_x+b_x/2\), we obtain
\[
\begin{aligned}
\bigl(\kappa^4\delta(J_0^4)\bigr)(x,y,z)
=\exp\bigl(2\pi i(a_x+b_x/2)s_xL(y,z)\bigr)
=(-1)^{b_xL(y,z)}.
\end{aligned}
\]
Thus \(\kappa^4\) is cohomologous to
\(\omega(x,y,z)=(-1)^{b_xL(y,z)}\).
}
\endgraf
{%
For the normalized two-cochain
\(J(x,y)=\zeta_8^{-b_xa_y}\), the identities
\(b_{xy}=b_x+b_y\) and
\(a_{yz}=a_y+s_ya_z-r_yb_z-4L(y,z)\) give
$
(\delta J)(x,y,z)
=\zeta_8^{\,b_x(2r_ya_z+r_yb_z+4L(y,z))}.
$
Thus
\[
R(x,y,z):=\frac{\omega(x,y,z)}{(\delta J)(x,y,z)}
=i^{-b_xr_ya_z}\zeta_8^{-b_xr_yb_z}.
\]
If \(y\in A\), then \(r_y=0\), so \(R|_{A^3}=1\).
We shall prove that \(R\) is not a coboundary.
For \(q,r\in A\), define
\[
\beta(q,r)
=\frac{R(g,q,r)R(Tq,Tr,g)}{R(Tq,g,r)}.
\]
Since \(r_g=1\), \(r_q=r_{Tr}=0\), and \(b_{Tq}=b_q\),
the formula for \(R\) gives
\[
R(g,q,r)=R(Tq,Tr,g)=1,
\qquad
\beta(q,r)=i^{b_qa_r}\zeta_8^{b_qb_r}.
\]
The \(b\)-coordinates are additive on \(A\), and the
\(a\)-coordinates are additive modulo four. Since \(i^4=1\),
the displayed expression is multiplicative in each argument.
Hence \(\beta\) is a normalized two-cocycle.
Moreover,
\[
\beta(\varepsilon,h)=1,\qquad
\beta(h,\varepsilon)=i,\qquad
\frac{\beta(\varepsilon,h)}{\beta(h,\varepsilon)}=-i.
\]
Under the identification of \(H^2(A,\mathbb C^\times)\)
with the fourth roots of unity established above,
\([\beta]\) therefore has order four.
}
\endgraf
{%
Suppose, to the contrary, that \(R=\delta F\).
The equality \(R|_{A^3}=1\) implies that
\(F|_{A^2}\) is a two-cocycle.
Using \(gq=T(q)g\) and \(gr=T(r)g\), expansion of the
three coboundaries in the definition of \(\beta\) gives
\[
\beta(q,r)
=
\frac{F(q,r)}{F(Tq,Tr)}
\frac{F(Tq,g)F(Tr,g)F(g,qr)}
     {F(g,q)F(g,r)F(T(qr),g)}.
\]
For \(\psi(q)=F(Tq,g)/F(g,q)\), the second quotient is
\((\delta\psi)(q,r)\). Passing to cohomology therefore gives
\[
[\beta]
=
[F|_{A^2}]
\bigl(T^*[F|_{A^2}]\bigr)^{-1}
=
[F|_{A^2}]^2,
\]
because \(T^*\) acts by inversion on
\(H^2(A,\mathbb C^\times)\).
Since this group is cyclic of order four, every square
has order at most two, contradicting the order of \([\beta]\).
Thus \(R\) is not a coboundary.
As \([R]=[\omega]=[\kappa]^4\), we obtain
\([\kappa]^4\neq1\).
The group \(H^3(\Gamma_4,\mathbb C^\times)\) has order eight,
so \([\kappa]\) has order eight and generates it.
}
}
\end{proof}

For commuting \(x,a,b\in G\), let
\(c_x(a,b):=\Phi_x(a,b)/\Phi_x(b,a)\).  By the definition of
\(\Phi_x\) and Lemma~\ref{lem:alternator},
\(c_x(a,b)=f_\Phi(x,a,b)\).
Put \(\eta_h=f_\Phi(h,\varepsilon,g^2)\) and
\(\eta_g=f_\Phi(g,\varepsilon^2,\varepsilon^{-1}h^2)\).
For \(a=\varepsilon^ih^jg^{2r}\) and
\(b=\varepsilon^{i'}h^{j'}g^{2r'}\) in \(C_G(h)\), and for
\(a'=(\varepsilon^2)^i(\varepsilon^{-1}h^2)^jg^r\) and
\(b'=(\varepsilon^2)^{i'}(\varepsilon^{-1}h^2)^{j'}g^{r'}\)
in \(C_G(g)\), Lemma~\ref{lem:alternator} gives
\begin{equation}
\label{eq:Gamma4-commutator-bicharacters}
c_h(a,b)=\eta_h^{\,ir'-i'r},
\qquad
c_g(a',b')=\eta_g^{\,ij'-i'j}.
\end{equation}
\begin{lemma}
\label{lem:Gamma4-projective-representations-one-dimensional}
{
One has \(\eta_h=\eta_g=1\), and hence
\(\dim\rho=\dim\sigma=1\).
Consequently, \(\dim V=2\) and \(\dim W=4\).
}
\end{lemma}

\begin{proof}
{
By Lemma~\ref{lem:Gamma4-third-cohomology}, write
\(\pi^*\Phi=\kappa^k\delta J\).
The triples \((h,\varepsilon,g^2)\) and
\((g,\varepsilon^2,\varepsilon^{-1}h^2)\) consist of commuting
elements of \(\Gamma_4\). The alternator of a coboundary is one
on commuting triples, so
\[
\eta_h=f_{\kappa^k}(h,\varepsilon,g^2),
\qquad
\eta_g=f_{\kappa^k}(g,\varepsilon^2,\varepsilon^{-1}h^2).
\]
For the first triple, the values of \((A_x,s_x)\) are
\((1,1)\), \((2,1)\), and \((0,1)\); for the second they are
\((0,-1)\), \((4,1)\), and \((0,1)\).
The formula for \(\kappa\) gives the value one on
every permutation of either triple. Thus \(\eta_h=\eta_g=1\).
Equation~\eqref{eq:Gamma4-commutator-bicharacters} shows that
\(\mathbb C^{\Phi_h}C_G(h)\) and
\(\mathbb C^{\Phi_g}C_G(g)\) are commutative.
Their irreducible representations \(\rho\) and \(\sigma\) are
therefore one-dimensional. The dimensions of \(V\) and \(W\)
follow from \(|h^G|=2\) and \(|g^G|=4\).
}
\end{proof}

\subsection{Adjoint objects}
\label{subsec:Gamma4-adjoint-objects}

{
We now compute the first and second adjoint objects.
}
The elements
\(\varepsilon^2\), \(g^2\), and \(\varepsilon^{-1}h^2\) are central
in \(G\).  Choose
\(0\neq v\in V_h\) and \(0\neq w\in W_g\), and set
\(v_1=g\rhd v\), \(w_1=h\rhd w\), and
\(w_2=\varepsilon\rhd w\).

\subsubsection{The objects \texorpdfstring{$X_1$ and $X_2$}{X1 and X2}}

 {We first determine \(C_G(hg)\) and its irreducible
\(\Phi_{hg}\)-projective representations.}

\begin{lemma}
\label{lem:Gamma4-hg-centralizer}
One has \(C_G(hg)=\langle \varepsilon^2,\varepsilon^{-1}h^2,hg\rangle\), and this group is
abelian.  Moreover,  every
irreducible \(\Phi_{hg}\)-projective representation of \(C_G(hg)\)
is one-dimensional.
\end{lemma}

\begin{proof}
The defining relations give
\((hg)^G=\{\varepsilon^i hg\mid 0\leq i<4\}\) and show that
\(\langle \varepsilon^2,\varepsilon^{-1}h^2,hg\rangle\) is an abelian subgroup of \(C_G(hg)\).
The identities \(g^2=\varepsilon h^{-2}(hg)^2\),
\(g=h^{-1}(hg)\), and
\(h^2=\varepsilon(\varepsilon^{-1}h^2)\), together with
\(\varepsilon^2\in\langle\varepsilon^2,\varepsilon^{-1}h^2,hg\rangle\),
show that every
normal-form representative lies in one of its four cosets represented
by \(1,\varepsilon,h,\varepsilon h\).  Hence
\([G:\langle \varepsilon^2,\varepsilon^{-1}h^2,hg\rangle]\leq4\), whereas
\([G:C_G(hg)]=|(hg)^G|=4\).  Since
\(\langle \varepsilon^2,\varepsilon^{-1}h^2,hg\rangle\subseteq C_G(hg)\), equality follows, and
therefore \(C_G(hg)=\langle \varepsilon^2,\varepsilon^{-1}h^2,hg\rangle\).

Since \(\varepsilon^2\) and \(\varepsilon^{-1}h^2\) are central,
Lemma~\ref{lem:alternator} and
\eqref{eq:Gamma4-commutator-bicharacters} give
\[
\begin{aligned}
c_{hg}(\varepsilon^2,\varepsilon^{-1}h^2)
&=f_\Phi(hg,\varepsilon^2,\varepsilon^{-1}h^2)
=f_\Phi(h,\varepsilon^2,\varepsilon^{-1}h^2)
  f_\Phi(g,\varepsilon^2,\varepsilon^{-1}h^2)\\
&=c_h(\varepsilon^2,\varepsilon^{-1}h^2)
  c_g(\varepsilon^2,\varepsilon^{-1}h^2)=1.
\end{aligned}
\]
Alternation also gives
\(c_{hg}(hg,\varepsilon^2)=c_{hg}(hg,\varepsilon^{-1}h^2)=1\).  Since \(c_{hg}\) is multiplicative
in both variables and \(\varepsilon^2,\varepsilon^{-1}h^2,hg\) generate \(C_G(hg)\), it is trivial
on \(C_G(hg)\).  Thus
\(\mathbb C^{\Phi_{hg}}C_G(hg)\) is commutative, and every irreducible
\(\Phi_{hg}\)-projective representation of \(C_G(hg)\) is
one-dimensional.
\end{proof}

Define
\begin{equation}
\label{eq:Gamma4-alpha-beta-Delta}
\begin{aligned}
&\alpha
=\frac{\rho(\varepsilon^{-1})}
        {\Phi_h(g,\varepsilon^{-1})},
\quad
\beta
=\alpha\Phi_g(\varepsilon^{-1}h,h)\sigma(\varepsilon^{-1}h^2),\\
&\Delta_4
=\beta\Phi_h(g,g)\rho(g^2)
=\Phi_h(g,\varepsilon g)\Phi_g(\varepsilon^{-1}h,h)
  \rho(\varepsilon^{-1}g^2)\sigma(\varepsilon^{-1}h^2).
\end{aligned}
\end{equation}

\begin{lemma}
\label{lem:Gamma4-first-adjoint}
The vector
\begin{equation}
\label{eq:Gamma4-x1}
x_1:=\varphi_1^\Phi(v\otimes w)
=v\otimes w-\alpha v_1\otimes w_1
\end{equation}
is nonzero, and
\(X_1=(\operatorname{ad}V)(W)=\mathbb CG\rhd x_1\), with
\(\operatorname{supp}X_1=(hg)^G\).  Moreover,
\[
X_1\text{ is simple}
\quad\Longleftrightarrow\quad
\Delta_4=1.
\]
When these conditions hold,
\(X_1\simeq M(hg,\tau_1)\) for a one-dimensional
\(\Phi_{hg}\)-projective representation \(\tau_1\), and
\(\dim X_1=4\).
\end{lemma}

\begin{proof}
Let \(T=c_{W,V}c_{V,W}\).  Since
\(hgh^{-1}=\varepsilon g\), the braiding formula gives
$
T(v\otimes w)
=((\varepsilon g)\rhd v)\otimes(h\rhd w).
$
The relation \(g\varepsilon^{-1}=\varepsilon g\), together with
\eqref{eq:YD-projective-action}, gives
\[
(\varepsilon g)\rhd v
=\frac{\rho(\varepsilon^{-1})}
       {\Phi_h(g,\varepsilon^{-1})}v_1.
\]
Thus \(x_1=(\operatorname{id}-T)(v\otimes w)\) is given by
\eqref{eq:Gamma4-x1}.  Its two summands have distinct bidegrees
\((h,g)\) and \((\varepsilon^{-1}h,\varepsilon g)\), and hence
\(x_1\neq0\).
The group \(C_G(g)\) stabilizes \(\mathbb Cw\), while \(g\) sends the
degree \(h\) to \(\varepsilon^{-1}h\).  Since both homogeneous
components of \(V\) are one-dimensional, it follows that
\(V=\mathbb C C_G(g)\rhd v\).  Proposition~\ref{prop:adjoint-orbit-generation}, applied
with \(X_0=W\), now gives
\[
X_1=\mathbb CG\rhd x_1.
\]
Both summands of \(x_1\) have total degree \(hg\).  Therefore
\(\operatorname{supp}X_1=(hg)^G\).

It remains to determine when \(X_1\) is simple.  The component
\((V\otimes W)_{hg}\) has basis
\(v\otimes w,v_1\otimes w_1\), and \(T\) interchanges
 {the lines spanned by these basis vectors}.   {Combining} \eqref{eq:YD-projective-action}  {with}
Lemma~\ref{lem:local-cocycle-coherence}  {gives}
\begin{align*}
&\rho(\varepsilon^{-1})\rho(g^2)
=\Phi_h(\varepsilon^{-1},g^2)\rho(\varepsilon^{-1}g^2),\\
&\Phi_h(g,\varepsilon^{-1})\Phi_h(g,\varepsilon g)
=\Phi_h(g,g)\Phi_h(g^2,\varepsilon^{-1}).
\end{align*}
Using also \(c_h(g^2,\varepsilon^{-1})=1\), a second application of
\(T\) gives
\[
\begin{aligned}
T^2(v\otimes w)
=\Phi_h(g,\varepsilon g)
  \Phi_g(\varepsilon^{-1}h,h)
  \rho(\varepsilon^{-1}g^2)\sigma(\varepsilon^{-1}h^2)\,v\otimes w
=\Delta_4\,v\otimes w.
\end{aligned}
\]
Since \(T\) interchanges  {the two lines}, the same scalar occurs
on \(v_1\otimes w_1\).  Hence
\(T^2=\Delta_4\operatorname{id}\) on
\((V\otimes W)_{hg}\).
Since \(\varphi_1^\Phi=\operatorname{id}-T\), one has
$
(X_1)_{hg}
=\operatorname{Im}\bigl(
(\operatorname{id}-T)|_{(V\otimes W)_{hg}}\bigr).
$
If \(\Delta_4\neq1\), then
\((\operatorname{id}-T)(\operatorname{id}+T)
=(1-\Delta_4)\operatorname{id}\), so
\(\dim(X_1)_{hg}=2\).  This is incompatible with simplicity, because
Lemma~\ref{lem:Gamma4-hg-centralizer} shows that every irreducible
\(\Phi_{hg}\)-projective representation of \(C_G(hg)\) is
one-dimensional.

Suppose now that \(\Delta_4=1\).  Then \(T^2=\operatorname{id}\),
and, since \(T\) interchanges  {the two lines},
\(\operatorname{id}-T\) has rank one.  Thus
\((X_1)_{hg}=\mathbb Cx_1\).  This line is stable under \(C_G(hg)\)
and affords a one-dimensional projective representation \(\tau_1\).
 {Since \(X_1=\mathbb CG\rhd x_1\), there is a surjection}
\[
M(hg,\tau_1)\longrightarrow X_1.
\]
The source is simple, so this map is an isomorphism.  Therefore
\(X_1\) is simple and
\(\dim X_1=|(hg)^G|=4\).
\end{proof}

For the second adjoint object, let
\[
\kappa_X=\frac{\rho(\varepsilon^2)}
{\Phi_h(\varepsilon^2,h)\Phi_h(g,\varepsilon^2h)}.
\]

\begin{lemma}
\label{lem:Gamma4-second-adjoint}
Assume that \(\Delta_4=1\).  Then
\begin{equation}
\label{eq:Gamma4-x2}
x_2:=\varphi_2^\Phi(v\otimes x_1)
=(1+\rho(h))
 \bigl(v\otimes x_1
       -\kappa_Xv_1\otimes(h\rhd x_1)\bigr).
\end{equation}
Moreover, \(X_2=(\operatorname{ad}V)^2(W)=\mathbb CG\rhd x_2\),
and hence \(X_2=0\) if and only if \(\rho(h)=-1\).
\end{lemma}

\begin{proof}
Let \(T_0=v\otimes x_1\) and
\(T_1=v_1\otimes(h\rhd x_1)\).
The two nonidentity terms in the recursive formula
\eqref{eq:recursive-varphi-2} are
\begin{align*}
&c_{X_1,V}c_{V,X_1}(T_0)
=\rho(h)\kappa_XT_1,\\
&(\operatorname{id}_V\otimes\varphi_1^\Phi)c_{1,2}^\Phi(T_0)
=\rho(h)T_0-\Xi_XT_1,
\end{align*}
where \(\Xi_X\) is the scalar defined in
Lemma~\ref{lem:Gamma4-X2-cocycle-reduction}.  That lemma proves
\(\Xi_X=\kappa_X\).  Therefore
\[
\begin{aligned}
x_2
=T_0-\rho(h)\kappa_XT_1+\rho(h)T_0-\Xi_XT_1
=(1+\rho(h))(T_0-\kappa_XT_1),
\end{aligned}
\]
which is \eqref{eq:Gamma4-x2}.

Since \(\Delta_4=1\), the line \(\mathbb Cx_1\) is stable under
\(C_G(hg)\).  The element \(hg\in C_G(hg)\) conjugates \(h\) to
\(\varepsilon^{-1}h\), and hence
\(V=\mathbb C C_G(hg)\rhd v\).
Proposition~\ref{prop:adjoint-orbit-generation} now gives \(X_2=\mathbb CG\rhd x_2\).
The degrees of the first \(V\)-factors in \(T_0\) and \(T_1\) are
\(h\) and \(\varepsilon^{-1}h\), respectively, so
\(T_0-\kappa_XT_1\neq0\).  The last
assertion follows from \eqref{eq:Gamma4-x2}.
\end{proof}

\subsubsection{The objects \texorpdfstring{$Y_1$ and $Y_2$}{Y1 and Y2}}

\begin{lemma}
\label{lem:Gamma4-opposite-first-adjoint}
The vector
\begin{equation}
\label{eq:Gamma4-y1}
y_1:=c_{V,W}(x_1)
=w_1\otimes v-\beta w\otimes v_1
=\varphi_1^\Phi(w_1\otimes v)
\end{equation}
generates \(Y_1=(\operatorname{ad}W)(V)\).  Moreover,
\(Y_1\simeq X_1\) and \(\operatorname{supp}Y_1=(hg)^G\).  In
particular,
\[
X_1\text{ is simple}
\quad\Longleftrightarrow\quad
Y_1\text{ is simple}
\quad\Longleftrightarrow\quad
\Delta_4=1.
\]
If these equivalent conditions hold, then
\(X_1\simeq Y_1\simeq M(hg,\tau_1)\), for the same
one-dimensional \(\Phi_{hg}\)-projective representation \(\tau_1\),
and \(\dim X_1=\dim Y_1=4\).
\end{lemma}

\begin{proof}
By \eqref{eq:Gamma4-x1} and the braiding formula,
\(c_{V,W}(v\otimes w)=w_1\otimes v\), while
\[
\begin{aligned}
c_{V,W}(v_1\otimes w_1)
=((\varepsilon^{-1}h)\rhd w_1)\otimes v_1
=\Phi_g(\varepsilon^{-1}h,h)\sigma(\varepsilon^{-1}h^2)
  \,w\otimes v_1.
\end{aligned}
\]
Since
\(\beta=\alpha\Phi_g(\varepsilon^{-1}h,h)\sigma(\varepsilon^{-1}h^2)\), it follows that
$
c_{V,W}(x_1)
=w_1\otimes v-\beta w\otimes v_1.
$
The identity preceding \eqref{eq:first-adjoint-braiding-symmetry}
gives \(c_{V,W}(x_1)=\varphi_1^\Phi(w_1\otimes v)\) and
identifies \(X_1\) with \(Y_1\).  Thus \(y_1\) generates \(Y_1\).
Since this identification is graded and  {commutes with the $G$-action}, it preserves
the support and, when \(\Delta_4=1\), the inducing representation
\(\tau_1\).  The remaining assertions follow from
Lemma~\ref{lem:Gamma4-first-adjoint}.
\end{proof}

\begin{lemma}
\label{lem:Gamma4-general-Y2-formula}
For every homogeneous vector \(a\in W_x\), one has
\begin{equation}
\label{eq:Gamma4-general-Y2-formula}
\begin{aligned}
\varphi_2^\Phi(a\otimes y_1)
={}&a\otimes y_1
-\bigl(x(hg)x^{-1}\rhd a\bigr)\otimes(x\rhd y_1)\\
&+\frac{\Phi(x,\varepsilon g,h)}
        {\Phi(x(\varepsilon g)x^{-1},x,h)}
  (x\rhd w_1)\otimes\varphi_1^\Phi(a\otimes v)\\
&-\beta
  \frac{\Phi(x,g,\varepsilon^{-1}h)}
       {\Phi(xgx^{-1},x,\varepsilon^{-1}h)}
  (x\rhd w)\otimes\varphi_1^\Phi(a\otimes v_1).
\end{aligned}
\end{equation}
\end{lemma}

\begin{proof}
Here the maps \(\varphi_n^\Phi\) are those associated with the ordered
pair \((W,V)\).  Since
\[
w_1\otimes v\in(W\otimes V)_{(\varepsilon g)h},
\qquad
w\otimes v_1\in(W\otimes V)_{g(\varepsilon^{-1}h)},
\]
and
$
(\varepsilon g)h=g(\varepsilon^{-1}h)=hg,
$
the vector \(y_1\) is homogeneous of degree \(hg\).  Hence, for
\(a\in W_x\), the braiding formula gives
\[
c_{Y_1,W}c_{W,Y_1}(a\otimes y_1)
=
\bigl(x(hg)x^{-1}\rhd a\bigr)\otimes(x\rhd y_1).
\]
Here \(x\rhd y_1\) denotes the tensor-product action on
\(W\otimes V\).

Let \(b\in W_y\) and \(r\in V_\ell\) be homogeneous.  By the
definition of \(c_{1,2}^\Phi\) and the associator convention,
\[
c_{1,2}^\Phi\bigl(a\otimes(b\otimes r)\bigr)
=
\frac{\Phi(x,y,\ell)}
     {\Phi(xyx^{-1},x,\ell)}
(x\rhd b)\otimes(a\otimes r).
\]
Now substitute
$
y_1=w_1\otimes v-\beta w\otimes v_1
$
into the recursive formula
$
\varphi_2^\Phi
=
\operatorname{id}_{W\otimes Y_1}
-c_{Y_1,W}c_{W,Y_1}
+(\operatorname{id}_W\otimes\varphi_1^\Phi)c_{1,2}^\Phi.
$
Applying the preceding formula for \(c_{1,2}^\Phi\) to the two
summands of \(y_1\) gives
\eqref{eq:Gamma4-general-Y2-formula}.
\end{proof}
Recall \(w_2=\varepsilon\rhd w\). Let
$
z_0=\varphi_2^\Phi(w\otimes y_1),
$ and $
z_1=\varphi_2^\Phi(w_2\otimes y_1).
$

\begin{lemma}
\label{lem:Gamma4-Y2-generators}
Assume that \(\Delta_4=1\).  Then
\begin{align}
Y_2
&=\mathbb CG\rhd z_0\oplus\mathbb CG\rhd z_1,
\nonumber\\
z_1&\neq0.
\label{eq:Gamma4-z1-nonzero}
\end{align}
Moreover, \(z_0\in(Y_2)_{\varepsilon^{-1}hg^2}\) and
\(0\neq z_1\in(Y_2)_{\varepsilon hg^2}\).  The conjugacy classes
\((\varepsilon^{-1}hg^2)^G
=\{\varepsilon^{-1}hg^2,hg^2\}\) and
\((\varepsilon hg^2)^G=\{\varepsilon hg^2,\varepsilon^2hg^2\}\)
are disjoint, so the two summands
have disjoint supports.
\end{lemma}

\begin{proof}
Lemma~\ref{lem:Gamma4-opposite-first-adjoint} shows that \(Y_1\) is
simple and that \(C_G(hg)\) stabilizes \(\mathbb Cy_1\).  This
centralizer has two orbits on \(\operatorname{supp}W\), represented
by \(g\) and \(\varepsilon^2g\).  {Conjugation by \(hg\) interchanges \(g\) and
\(\varepsilon g\), and also \(\varepsilon^2g\) and \(\varepsilon^{-1}g\).}  Lemma
\ref{lem:homogeneous-orbit-tensor-generation}, followed by the \(G\)-linearity of
\(\varphi_2^\Phi\), therefore gives
\(Y_2=\mathbb CG\rhd z_0+\mathbb CG\rhd z_1\).
The map \(\varphi_2^\Phi\) preserves the grading, so
\(z_0\in(Y_2)_{\varepsilon^{-1}hg^2}\) and
\(z_1\in(Y_2)_{\varepsilon hg^2}\).  The two displayed conjugacy
classes are disjoint, and hence the sum is direct.

In Lemma~\ref{lem:Gamma4-general-Y2-formula} with \(a=w_2\), the
\(W\)-degrees of the four terms are
\(\varepsilon^2g,\varepsilon g,\varepsilon^{-1}g,g\), respectively.  Thus the
component of
multidegree \((\varepsilon^2g,\varepsilon g,h)\) occurs only in
\(w_2\otimes y_1\), where it is
\(w_2\otimes(w_1\otimes v)\neq0\).  Hence \(z_1\neq0\).
\end{proof}

\subsubsection{ {Simplicity of \texorpdfstring{$Y_2$}{Y2}}}

\begin{lemma}
\label{lem:Gamma4-z0-criterion}
Assume that \(\Delta_4=1\).  Then
\[
z_0=0
\quad\Longleftrightarrow\quad
\sigma(g)=-1.
\]
\end{lemma}

\begin{proof}
Lemma~\ref{lem:Gamma4-z0-factorization} shows that the coefficient of
\(w\otimes(w\otimes v_1)\) in \(z_0\) is
\(-\beta(1+\sigma(g))\).  Since \(\beta\neq0\), the equality
\(z_0=0\) implies \(\sigma(g)=-1\).  Conversely, if
\(\sigma(g)=-1\), the same lemma gives \(z_0=0\), since
\(\Delta_4=1\).
\end{proof}

\begin{lemma}
\label{lem:Gamma4-epsilon-ht-centralizer}
One has
\[
C_G(\varepsilon hg^2)
=\langle\varepsilon,h,g^2\rangle
=\langle \varepsilon^2,\varepsilon^{-1}h^2,g^2,h\rangle.
\]
Moreover, every irreducible
\(\Phi_{\varepsilon hg^2}\)-projective representation of this
centralizer is one-dimensional.
\end{lemma}

\begin{proof}
Let \(H=\langle\varepsilon,h,g^2\rangle\).  The defining relations
show that \(g^2\) commutes with both
\(\varepsilon\) and \(h\). Thus \(H\) is abelian and
\(H\subseteq C_G(\varepsilon hg^2)\).
Every element of \(G\) has a representative
\(\varepsilon^ih^jg^r\). According as \(r\) is even or odd, this
representative belongs to \(H\) or \(Hg\). Hence \([G:H]\leq2\).
Moreover, conjugation by \(g\) interchanges \(\varepsilon hg^2\) and
\(\varepsilon^2hg^2\). Since \(\varepsilon\) has order four, these elements are
distinct. Hence \([G:C_G(\varepsilon hg^2)]=2\), and consequently
\(C_G(\varepsilon hg^2)=H\).

Clearly,
\(\langle\varepsilon^2,\varepsilon^{-1}h^2,g^2,h\rangle\subseteq H\).
The reverse inclusion follows from
\(\varepsilon=h^2(\varepsilon^{-1}h^2)^{-1}\).  Hence
\(H=\langle\varepsilon^2,\varepsilon^{-1}h^2,g^2,h\rangle\).
{%
Lemma~\ref{lem:Gamma4-projective-representations-one-dimensional}
gives \(\eta_h=1\).
}
By Lemma~\ref{lem:alternator},
\[
c_{\varepsilon hg^2}(\varepsilon,h)=\eta_h^{-1},\qquad
c_{\varepsilon hg^2}(\varepsilon,g^2)=\eta_h,\qquad
c_{\varepsilon hg^2}(h,g^2)=\eta_h^{-1}.
\]
Since \(c_{\varepsilon hg^2}\) is a bicharacter and
\(\varepsilon,h,g^2\) generate \(H\), it is trivial on \(H\).
Therefore \(\mathbb C^{\Phi_{\varepsilon hg^2}}H\) is commutative,
and every irreducible \(\Phi_{\varepsilon hg^2}\)-projective
representation of \(H\) is one-dimensional.
\end{proof}

 {%
Write \(z_1=\sum_{j=0}^7z_{1,j}\), where \(z_{1,j}\) is  {the component of
\(z_1\) of the multidegree indicated below}:
\[
\begin{array}{c|c@{\qquad}c|c}
j&\text{multidegree of }z_{1,j}
&j&\text{multidegree of }z_{1,j}\\ \hline
0&(g,\varepsilon^2g,\varepsilon^{-1}h)
&4&(g,\varepsilon^{-1}g,h)\\
1&(\varepsilon g,\varepsilon^{-1}g,\varepsilon^{-1}h)
&5&(\varepsilon g,g,h)\\
2&(\varepsilon^2g,g,\varepsilon^{-1}h)
&6&(\varepsilon^2g,\varepsilon g,h)\\
3&(\varepsilon^{-1}g,\varepsilon g,\varepsilon^{-1}h)
&7&(\varepsilon^{-1}g,\varepsilon^2g,h).
\end{array}
\]
Put \(\nu=(0\ 1\ 2\ 3)(4\ 5\ 6\ 7)\).
\endgraf
For the conjugacy classes of \(h\) and \(g\), choose
\(p_h=p_g=1\), \(p_{\varepsilon^{-1}h}=g\), \(p_{\varepsilon g}=h\),
\(p_{\varepsilon^2g}=\varepsilon\), and \(p_{\varepsilon^{-1}g}=\varepsilon h\).  For
\(M(x,\tau)\), set
\[
d_x(a;y)=p_{aya^{-1}}^{-1}ap_y,
\qquad
m_{x,\tau}(a;y)=
\frac{\Phi_x(a,p_y)}
     {\Phi_x(p_{aya^{-1}},d_x(a;y))}\tau(d_x(a;y)).
\]
By \eqref{eq:induced-action}, \(m_{x,\tau}(a;y)\) is the action
coefficient in the standard homogeneous basis. Write
\(\mathsf R(a;y)=m_{h,\rho}(a;y)\) and
\(\mathsf S(a;x)=m_{g,\sigma}(a;x)\).
\endgraf
Let \(v_y\in V_y\) and \(w_x\in W_x\) be the corresponding standard
basis vectors, chosen so that \(v_h=v\), \(w_g=w\),
\(w_{\varepsilon g}=w_1\), and \(w_{\varepsilon^2g}=w_2\).  For
\(0\leq j\leq7\),  {let \(E_j\) be the basis tensor of the
multidegree assigned to \(z_{1,j}\) above.} Define
\begin{align}
\mathcal A_0={}&-\frac{\Phi(\varepsilon^2g,g,\varepsilon^{-1}h)}{\Phi(g,\varepsilon^2g,\varepsilon^{-1}h)}
 \mathsf R(\varepsilon g;h)\mathsf S(\varepsilon^2g;g)
 \mathsf S(\varepsilon^{-1}h;\varepsilon g),
\label{eq:Gamma4-Y2-coefficients}\\
\mathcal A_1={}&-\Phi^{\varepsilon^2g}(\varepsilon g,h)\mathsf R(\varepsilon^2g;h)
 \mathsf S(\varepsilon^2g;\varepsilon g)\mathsf S(\varepsilon^{-1}hg;\varepsilon^2g),
\nonumber\\
\mathcal A_2={}&-\mathsf R(\varepsilon g;h)\mathsf S(\varepsilon^{-1}h;\varepsilon g),
\nonumber\\
\mathcal A_3={}&-\frac{\Phi(\varepsilon^2g,\varepsilon g,h)}{\Phi(\varepsilon^{-1}g,\varepsilon^2g,h)}
 \mathsf R(\varepsilon^2g;h)\mathsf S(\varepsilon^2g;\varepsilon g)\mathsf S(\varepsilon^{-1}h;\varepsilon^2g),
\nonumber\\
\mathcal A_4={}&\frac{\Phi(\varepsilon^2g,g,\varepsilon^{-1}h)}{\Phi(g,\varepsilon^2g,\varepsilon^{-1}h)}
 \mathsf R(\varepsilon g;h)\mathsf R(\varepsilon^2g;\varepsilon^{-1}h)
 \mathsf S(h;\varepsilon^2g)\mathsf S(\varepsilon^2g;g)\mathsf S(\varepsilon^{-1}h;\varepsilon g),
\nonumber\\
\mathcal A_5={}&\Phi^{\varepsilon^2g}(g,\varepsilon^{-1}h)
 \mathsf R(\varepsilon g;h)\mathsf R(\varepsilon^2g;\varepsilon^{-1}h)
 \mathsf S(\varepsilon^2g;g)\mathsf S(\varepsilon^{-1}h;\varepsilon g)\mathsf S(\varepsilon^{-1}hg;\varepsilon^2g),
\nonumber\\
\mathcal A_6={}&1,
\nonumber\\
\mathcal A_7={}&\frac{\Phi(\varepsilon^2g,\varepsilon g,h)}{\Phi(\varepsilon^{-1}g,\varepsilon^2g,h)}
 \mathsf S(\varepsilon^2g;\varepsilon g).
\nonumber
\end{align}
If \(E_j=w_{x_j}\otimes(w_{y_j}\otimes v_{\ell_j})\), put
\begin{equation}
\label{eq:Gamma4-Y2-H}
H_j=
 \Phi^h(x_j,y_j\ell_j)\Phi^h(y_j,\ell_j)
 \mathsf S(h;x_j)\mathsf S(h;y_j)\mathsf R(h;\ell_j).
\end{equation}
\endgraf
\begin{lemma}
\label{lem:Gamma4-Y2-coordinate-calculation}
One has \(z_1=\sum_{j=0}^7\mathcal A_jE_j\), with all
\(\mathcal A_j\neq0\), and
\(h\rhd E_j=H_jE_{\nu(j)}\), where
\(\nu=(0\ 1\ 2\ 3)(4\ 5\ 6\ 7)\).
\end{lemma}
\endgraf
\begin{proof}
 {Take \(a=w_2\) in
\eqref{eq:Gamma4-general-Y2-formula} and substitute the action
coefficients \(\mathsf R\) and \(\mathsf S\).}
The four terms contribute, in order, to the pairs
\((E_6,E_2)\), \((E_5,E_1)\), \((E_7,E_3)\), and
\((E_0,E_4)\).  Expanding the corresponding action coefficients
gives exactly \eqref{eq:Gamma4-Y2-coefficients}.  Every
\(\mathcal A_j\) is nonzero because it is a product or quotient of
nonzero cocycle values and action coefficients.
\endgraf
The defining relations of \(\Gamma_4\) show that conjugation by \(h\)
permutes the eight displayed multidegrees according to
\(\nu=(0\ 1\ 2\ 3)(4\ 5\ 6\ 7)\).  Applying
\eqref{eq:tensor-product} twice to each \(E_j\) gives
\(h\rhd E_j=H_jE_{\nu(j)}\), with \(H_j\) as in
\eqref{eq:Gamma4-Y2-H}.
\endgraf
These formulas are also verified in
\gaplink{Gamma_4/reflected_y2/gamma4_reflected_y2_certificate.g}.
\end{proof}
\endgraf
\begin{lemma}
\label{lem:Gamma4-Y2-line-stability}
There are unique scalars
\(\lambda_j\in\mathbb C^\times\) such that
\[
h\rhd z_{1,j}=\lambda_jz_{1,\nu(j)}
\qquad(0\leq j\leq7).
\]
For \(1\leq j\leq7\), define
\begin{equation}
\label{eq:Gamma4-Y2-H-Theta}
\Theta_{2,j}=\frac{\lambda_j}{\lambda_0}.
\end{equation}
Explicitly,
\begin{equation}
\label{eq:Gamma4-Y2-H-Theta-explicit}
\Theta_{2,j}
=
\frac{H_j\mathcal A_j\mathcal A_1}
     {H_0\mathcal A_0\mathcal A_{\nu(j)}}
\qquad(1\leq j\leq7).
\end{equation}
If \(\Delta_4=1\), then
\[
\mathbb Cz_1\text{ is \(C_G(\varepsilon hg^2)\)-stable}
\quad\Longleftrightarrow\quad
\Theta_{2,j}=1
\quad(1\leq j\leq7).
\]
\end{lemma}
\endgraf
\begin{proof}
Lemma~\ref{lem:Gamma4-Y2-coordinate-calculation} gives
\[
z_{1,j}=\mathcal A_jE_j,\qquad
h\rhd z_{1,j}
=\frac{H_j\mathcal A_j}{\mathcal A_{\nu(j)}}z_{1,\nu(j)}.
\]
Since every \(\mathcal A_j\) is nonzero, this proves the existence
and uniqueness of \(\lambda_j\). The equality \(\nu(0)=1\) gives
\eqref{eq:Gamma4-Y2-H-Theta-explicit}.  Replacing \(v\) and \(w\) by
\(av\) and \(bw\), respectively, multiplies every \(z_{1,j}\) by
\(ab^2\).  Hence the scalars
\(\lambda_j\) and the ratios \(\Theta_{2,j}\) are unchanged.
\endgraf
Assume now that \(\Delta_4=1\).
Then
\(\mathbb Cy_1\) is stable under \(C_G(hg)\).  The elements \(\varepsilon^2,\varepsilon^{-1}h^2,g^2\)
belong to this centralizer and are central in \(G\).  Hence each of
them acts by a scalar on both \(W_{\varepsilon^2g}=\mathbb Cw_2\) and
\((Y_1)_{hg}=\mathbb Cy_1\).  Since \(\varphi_2^\Phi\) is a
morphism, \(\varepsilon^2,\varepsilon^{-1}h^2,g^2\) stabilize \(\mathbb Cz_1\).
\endgraf
By Lemma~\ref{lem:Gamma4-epsilon-ht-centralizer},
\(C_G(\varepsilon hg^2)=\langle \varepsilon^2,\varepsilon^{-1}h^2,g^2,h\rangle\).  It remains only to
decide whether \(h\rhd z_1\) belongs to \(\mathbb Cz_1\).  The eight
multidegrees are distinct, and therefore
\[
h\rhd z_1\in\mathbb Cz_1
\quad\Longleftrightarrow\quad
\lambda_0=\lambda_1=\cdots=\lambda_7.
\]
By \eqref{eq:Gamma4-Y2-H-Theta}, this is equivalent to
\(\Theta_{2,j}=1\) for \(1\leq j\leq7\).
\end{proof}
}

\begin{lemma}
\label{lem:Gamma4-Y2-explicit-simplicity}
{
Assume that \(\Delta_4=1\). Then
}
\[
Y_2\text{ is simple}
\quad\Longleftrightarrow\quad
\sigma(g)=-1
\quad\text{and}\quad
\Theta_{2,j}=1
\quad(1\leq j\leq7).
\]
If these conditions hold, then
\(Y_2\simeq M(\varepsilon hg^2,\tau_2)\), where \(\tau_2\) is
one-dimensional, and \(\dim Y_2=2\).
\end{lemma}

\begin{proof}
By Lemma~\ref{lem:Gamma4-Y2-generators}, \(Y_2\) is the direct sum
of the subobjects generated by \(z_0\) and \(z_1\), whose supports
are disjoint, and \(z_1\neq0\). Hence, if \(Y_2\) is simple, then
\(z_0=0\). By Lemma~\ref{lem:Gamma4-z0-criterion}, this is
equivalent to \(\sigma(g)=-1\).

Assume that \(\sigma(g)=-1\). Then \(z_0=0\), so
\(Y_2=\mathbb CG\rhd z_1\), with
\(0\neq z_1\in(Y_2)_{\varepsilon hg^2}\). If \(Y_2\) is simple,
Proposition~\ref{prop:simple-objects} shows that
\((Y_2)_{\varepsilon hg^2}\) is an irreducible
\(\Phi_{\varepsilon hg^2}\)-projective representation of
\(C_G(\varepsilon hg^2)\). By
Lemma~\ref{lem:Gamma4-epsilon-ht-centralizer}, it is
one-dimensional. Hence
\((Y_2)_{\varepsilon hg^2}=\mathbb Cz_1\), and this line is
\(C_G(\varepsilon hg^2)\)-stable.

Conversely, suppose that \(\mathbb Cz_1\) is
\(C_G(\varepsilon hg^2)\)-stable. It then affords a one-dimensional
\(\Phi_{\varepsilon hg^2}\)-projective representation \(\tau_2\).
 {The induced construction gives a morphism
\(M(\varepsilon hg^2,\tau_2)\to Y_2\) with image
\(\mathbb CG\rhd z_1=Y_2\). Since its source is simple by
Proposition~\ref{prop:simple-objects} and \(Y_2\neq0\), this morphism
is an isomorphism.}

Lemma~\ref{lem:Gamma4-Y2-line-stability} now gives the stated
conditions on the \(\Theta_{2,j}\). Finally, the conjugacy class
\((\varepsilon hg^2)^G=\{\varepsilon hg^2,\varepsilon^2hg^2\}\) has two elements.
Since \(\tau_2\) is one-dimensional, \(\dim Y_2=2\).
\end{proof}

{%
\subsection{{Sufficiency}}
\label{subsec:Gamma4-sufficiency}
\endgraf
{
We show that the parameter conditions force the pullback of
\(\Phi\) to be a coboundary.  {Sufficiency then follows from the classification for
\(\Gamma_4\) in the untwisted case.}
}
\endgraf

\endgraf
\begin{lemma}
\label{lem:Gamma4-trivial-pullback}
{
{
Assume that \(\rho(h)=\sigma(g)=-1\) and
\(\Theta_{2,3}=1\). Then \(\pi^*\Phi\) is a coboundary.
}
{
Moreover, \((V,W)\) can be replaced by a pair in
\({}_{\Gamma_4}^{\Gamma_4}\mathcal{YD}\), preserving the dimensions
of its Nichols algebra, all its iterated adjoint objects, and their
Nichols algebras.} The values of
\(\rho(h),\sigma(g),\Delta_4\), and \(\Theta_{2,j}\) are preserved,
and the resulting ordinary character satisfies
\(\rho(\varepsilon)^2=-1\).
}
\end{lemma}

\begin{proof}
{%
By Lemma~\ref{lem:Gamma4-third-cohomology}, write
\(\pi^*\Phi=\kappa^k\delta J\), where \(0\leq k<8\) and
\(J\) is normalized. We prove that \(k=0\).
Since the commutator subgroup of \(\Gamma_4\) is
\(\langle\varepsilon\rangle\), on which \(\pi\) is injective,
\[
\pi^{-1}\bigl(C_G(\pi(x))\bigr)=C_{\Gamma_4}(x)
\qquad(x\in\Gamma_4).
\]
Thus \(\pi\) maps each conjugacy class bijectively onto its image.
Lifting the support degrees of \(V,W\) to the conjugacy classes
of \(h,g\), and letting \(\Gamma_4\) act through \(\pi\), gives
simple objects with cocycle \(\pi^*\Phi\) on the same vector spaces.
}

The two-cochain \(J\) then gives a braided monoidal equivalence
to the category with cocycle \(\kappa^k\), with tensor structure
\(u_x\otimes v_y\mapsto J(x,y)^{-1}u_x\otimes v_y\) and action
\[
a\rhd_Jv_x
=\frac{J(axa^{-1},a)}{J(a,x)}\,a\rhd v_x.
\]
In particular, \(\rho(h)\) and \(\sigma(g)\) are unchanged.
The iterated tensor structure maps \(z_1\) and its eight
multidegree components to a common nonzero multiple of their
counterparts. All these components have total degree
\(d=\varepsilon hg^2\), which commutes with \(h\). Hence
\[
\lambda_j^J=\frac{J(d,h)}{J(h,d)}\lambda_j
\quad(0\leq j\leq7),
\qquad
\Theta_{2,j}^J=\Theta_{2,j}.
\]

Retain \(\rho,\sigma\) for the resulting inducing representations,
and put
\(r=\rho(\varepsilon)\),
\(s=\sigma(\varepsilon^2)\), and
\(t=\sigma(\varepsilon^{-1}h^2)\).
Equation~\eqref{eq:YD-projective-action} gives
\(r^4=\zeta_8^k\) and \(s^2=(-1)^k\). More explicitly, on the
two centralizers,
\[
\begin{aligned}
&\rho(\varepsilon^ah^bg^{2c})
=r^a(-1)^b\rho(g^2)^c
  \zeta_8^{-k\lfloor(2a+b)/8\rfloor}
&& (0\leq a<4),\\
&\sigma\bigl((\varepsilon^2)^a
           (\varepsilon^{-1}h^2)^bg^c\bigr)
=s^at^b(-1)^c
&& (0\leq a<2),
\end{aligned}
\]
where \(b,c\in\mathbb Z\).
Substitution in \eqref{eq:Gamma4-Y2-coefficients} and
\eqref{eq:Gamma4-Y2-H} gives
\[
\begin{array}{c|ccccc}
&\mathcal A_0&\mathcal A_1&\mathcal A_3&H_0&H_3\\ \hline
&r^3st\zeta_8^{3k}&-r^2s&r^2\zeta_8^{5k}
&-r^3\zeta_8^{4k}&-r^3st^2\zeta_8^{5k}.
\end{array}
\]
Therefore
\[
\lambda_0=\frac{H_0\mathcal A_0}{\mathcal A_1}=t,
\qquad
\lambda_3=\frac{H_3\mathcal A_3}{\mathcal A_0}=-r^{-2}t,
\qquad
\Theta_{2,3}=-r^{-2}.
\]
The assumption \(\Theta_{2,3}=1\) gives \(r^2=-1\).
Consequently, \(\zeta_8^k=r^4=1\), so \(k=0\) and
\(\pi^*\Phi=\delta J\).

{%
The resulting pair therefore lies in the 
Yetter--Drinfeld category over \(\Gamma_4\).
The lift leaves the associators, braidings, tensor actions,
and duality maps unchanged, while the equivalence defined by
\(J\) preserves Nichols algebras and the recursively defined
adjoint maps. This proves the asserted dimension equalities.
The bijections of conjugacy classes also preserve simplicity:
on an object supported on one class, they identify the graded
subobjects, and an object supported on several classes
decomposes into the corresponding subobjects.
These comparisons therefore apply along subsequent reflections.
Finally, the proof of Lemma~\ref{lem:Gamma4-first-adjoint} gives
\[
(c_{W,V}c_{V,W})^2
=\Delta_4\operatorname{id}
\quad\text{on }(V\otimes W)_{hg}.
\]
 {Conjugating this equality by the tensor structure
isomorphism shows that \(\Delta_4\) is preserved.}
}
\end{proof}
\begin{prop}
\label{prop:Gamma4-sufficiency}
Assume \eqref{eq:Gamma4-classification-B2}. Then
\(\dim\mathcal B(V\oplus W)<\infty\), and \((V,W)\) admits all
reflections and has a standard Cartan graph of type \(B_2\), with
matrix \(\left(\begin{smallmatrix}2&-1\\-2&2\end{smallmatrix}\right)\).
Moreover, there is an isomorphism of \(\mathbb N_0^2\)-graded
objects in \(\GG\),
\begin{equation}
\label{eq:Gamma4-B2-tensor-decomposition}
\mathcal B(V\oplus W)
\simeq
\mathcal B(V)\otimes
\bigl(\mathcal B(Y_1)\otimes
\bigl(\mathcal B(Y_2)\otimes\mathcal B(W)\bigr)\bigr).
\end{equation}
\end{prop}
\endgraf
\begin{proof}
{
{
Apply Lemma~\ref{lem:Gamma4-trivial-pullback}, and retain
\(\rho,\sigma\) for the resulting ordinary characters.
Then \(\rho(\varepsilon)^2=-1\), and the condition
\(\Delta_4=1\) becomes
\(\rho(g^2)\sigma(\varepsilon^{-1}h^2)=\rho(\varepsilon)\).
Together with \(\rho(h)=\sigma(g)=-1\), these are the conditions
of \cite[Theorem~5.5]{rank2-3}. Thus the ordinary pair has a
finite-dimensional Nichols algebra and a standard Cartan graph
of type \(B_2\).}
}

{
The dimension equality in Lemma~\ref{lem:Gamma4-trivial-pullback}
gives finite-dimensionality for the original pair, so it admits all
reflections by Theorem~\ref{thm:Nichols-finiteness-criterion}.}
 {Lifting the support degrees along \(\pi\) and twisting
the cocycle as in Lemma~\ref{lem:Gamma4-trivial-pullback} preserve
duals and the dimensions of iterated adjoint objects. Induction on
the length of a reflection sequence therefore shows that the
corresponding reflected pairs have the same Cartan matrices.}
Thus the original Cartan graph is also standard of type \(B_2\).
The positive roots in the associated convex order are
\(\alpha_2,\alpha_1+2\alpha_2,\alpha_1+\alpha_2,\alpha_1\), with
corresponding objects \(W,Y_2,Y_1,V\) by
Lemma~\ref{lem:root-string-adjoint-simplicity}.
Proposition~\ref{prop:positive-root-factorization} gives
\eqref{eq:Gamma4-B2-tensor-decomposition}.
\end{proof}
\endgraf
\begin{cor}
\label{cor:Gamma4-Nichols-dimension}
Under the hypotheses of Proposition~\ref{prop:Gamma4-sufficiency},
\[
\dim\mathcal B(V)=\dim\mathcal B(Y_2)=8,
\qquad
\dim\mathcal B(W)=\dim\mathcal B(Y_1)=64.
\]
Consequently, \(\dim\mathcal B(V\oplus W)=2^{18}\).
\end{cor}
\endgraf
\begin{proof}
For the ordinary pair in the preceding proof, the four dimensions
are given by \cite[Propositions~6.15 and~6.16]{rank2-3}.
Lemma~\ref{lem:Gamma4-trivial-pullback} preserves them, and
\eqref{eq:Gamma4-B2-tensor-decomposition} gives the last assertion.
\end{proof}
}

We first determine the parameters under the additional assumption
that the Cartan matrix is of finite type.
\begin{prop}
\label{prop:Gamma4-local-B2-Cartan-matrix}
Assume that \(\mathcal B(V\oplus W)\) is finite-dimensional and that
the Cartan matrix of \((V,W)\) is indecomposable
and of
finite type.  Then
\[
\rho(h)=\sigma(g)=-1,
\qquad
\Delta_4=1,
\qquad
\Theta_{2,j}=1\quad(1\leq j\leq7).
\]
The Cartan matrix of \((V,W)\) is
\[
\begin{pmatrix}2&-1\\-2&2\end{pmatrix}.
\]
\end{prop}

\begin{proof}
{
Lemma~\ref{lem:Gamma4-projective-representations-one-dimensional}
shows that \(\rho\) and \(\sigma\) are one-dimensional.
} By
Theorem~\ref{thm:Nichols-finiteness-criterion}, the pair \((V,W)\)
admits all reflections and has a finite Cartan graph.

Lemma~\ref{lem:Gamma4-first-adjoint} gives \(x_1\neq0\), and hence
\(X_1\neq0\) and \(m_{12}\geq1\). Therefore
Lemma~\ref{lem:root-string-adjoint-simplicity} shows that \(X_1\) is
simple.
Lemma~\ref{lem:Gamma4-first-adjoint} then gives \(\Delta_4=1\).
Under this condition, \eqref{eq:Gamma4-z1-nonzero} gives
\(0\neq z_1\in Y_2\).  Consequently, \(Y_2\neq0\) and
\(m_{21}\geq2\). Another application of
Lemma~\ref{lem:root-string-adjoint-simplicity} shows that \(Y_2\) is
simple.

Since the Cartan matrix of \((V,W)\) is an indecomposable rank-two
generalized Cartan matrix of finite type, one has
\(m_{12}m_{21}\in\{1,2,3\}\).  Together with \(m_{12}\geq1\) and
\(m_{21}\geq2\), this forces \(m_{12}=1\).  By the definition of
\(m_{12}\), it follows that \(X_2=0\).
Lemma~\ref{lem:Gamma4-second-adjoint}, which applies because
\(\Delta_4=1\), now gives \(\rho(h)=-1\).

{
Since \(Y_2\) is simple,
Lemma~\ref{lem:Gamma4-Y2-explicit-simplicity} gives
\(\sigma(g)=-1\) and \(\Theta_{2,j}=1\) for \(1\leq j\leq7\).
Proposition~\ref{prop:Gamma4-sufficiency} now gives
\((m_{12},m_{21})=(1,2)\), as asserted.
}
\end{proof}

\subsection{Reflections}
\label{subsec:Gamma4-reflections}

The following lemma gives support representatives for the two
reflected tuples.

\begin{lemma}
\label{lem:Gamma4-reflected-supports}
{
Assume that \(m_{12}=1\), \(m_{21}=2\), and that \(X_1\) and
\(Y_2\) are simple, so that both reflections are defined.
} Let
$
(\varepsilon_1,h_1,g_1)
=(\varepsilon^{-1},h^{-1},hg),
$
$(\varepsilon_2,h_2,g_2)
=(\varepsilon^{-1},\varepsilon hg^2,g^{-1})$.
Then
\[
R_1(V,W)=(V^*,X_1),\qquad
R_2(V,W)=(Y_2,W^*)
\]
have supports of the same form as \((V,W)\), with support
representatives \(h_1,g_1\) and \(h_2,g_2\), respectively.
 {For \(i=1,2\), the triple
\((\varepsilon_i,h_i,g_i)\) satisfies the defining relations of
\(\Gamma_4\) and generates \(G\).}
\end{lemma}

\begin{proof}
Since \(m_{12}=1\), \(m_{21}=2\), and \(X_1\) and \(Y_2\) are
simple, the definition of the reflections gives
\(R_1(V,W)=(V^*,X_1)\) and \(R_2(V,W)=(Y_2,W^*)\).
Since \(X_1\) is simple,
Lemma~\ref{lem:Gamma4-first-adjoint} gives \(\Delta_4=1\) and
\(\operatorname{supp}X_1
=\{\varepsilon^ihg\mid0\leq i<4\}\).
The simplicity of \(Y_2\), together with
Lemma~\ref{lem:Gamma4-Y2-explicit-simplicity}, then gives
\(\operatorname{supp}Y_2=\{\varepsilon hg^2,\varepsilon^2hg^2\}\).
Recall that \(g^2\) is central.  The defining relations of
\(\Gamma_4\) give
\[
\begin{aligned}
h_1g_1
 &=g
  =\varepsilon^{-1}(g_1h_1),
&
g_1\varepsilon^{-1}
 &=\varepsilon g_1,
&
h_1\varepsilon^{-1}
 &=\varepsilon^{-1}h_1,\\
h_2g_2
 &=\varepsilon hg
  =\varepsilon^{-1}(g_2h_2),
&
g_2\varepsilon^{-1}
 &=\varepsilon g_2,
&
h_2\varepsilon^{-1}
 &=\varepsilon^{-1}h_2.
\end{aligned}
\]
For the second row, one uses
\(g_2h_2=g^{-1}\varepsilon hg^2=\varepsilon^2 hg\).
Thus both triples satisfy
\(h_ig_i=\varepsilon_i g_ih_i\),
\(g_i\varepsilon_i=\varepsilon_i^{-1}g_i\), and
\(h_i\varepsilon_i=\varepsilon_i h_i\).
Since \(\varepsilon\) has order four,
\(\varepsilon_1=\varepsilon_2=\varepsilon^{-1}\) also has order four.

For both triples, the first relation gives
\(\varepsilon^{-1}=h_ig_ih_i^{-1}g_i^{-1}\), so
\(\varepsilon^{-1}\in\langle h_i,g_i\rangle\).
For the first triple, \(h=h_1^{-1}\) and \(g=h_1g_1\).
For the second triple, \(g=g_2^{-1}\) and
\(h=\varepsilon^{-1}h_2g_2^2\).
Consequently, both pairs \(h_i,g_i\) generate \(G\).
The grading on a dual object is obtained by taking inverses of the
original degrees.  In particular,
\((\varepsilon^ig)^{-1}
=g^{-1}\varepsilon^{-i}
=\varepsilon^ig^{-1}\).
Therefore
\[
\begin{aligned}
&\operatorname{supp}V^*
 =\{h^{-1},\varepsilon h^{-1}\}
   =\{h_1,\varepsilon_1^{-1}h_1\},\\
&\operatorname{supp}X_1
 =\{\varepsilon^ihg\mid0\leq i<4\}
   =\{\varepsilon_1^ig_1\mid0\leq i<4\},\\
&\operatorname{supp}Y_2
 =\{\varepsilon hg^2,\varepsilon^2hg^2\}
   =\{h_2,\varepsilon_2^{-1}h_2\},\\
&\operatorname{supp}W^*
 =\{\varepsilon^ig^{-1}\mid0\leq i<4\}
   =\{\varepsilon_2^ig_2\mid0\leq i<4\}.
\end{aligned}
\]
Thus both reflected tuples have supports of the required form.
\end{proof}

\begin{prop}
\label{prop:Gamma4-explicit-B2-stability-under-reflections}
{
Assume \eqref{eq:Gamma4-classification-B2}. With support
representatives chosen successively as in
Lemma~\ref{lem:Gamma4-reflected-supports}, every reflected tuple
satisfies \eqref{eq:Gamma4-classification-B2}.}
\end{prop}

{%
\begin{proof}
{
 {By Proposition~\ref{prop:Gamma4-sufficiency},
\(\mathcal B(V\oplus W)\) is finite-dimensional, and \((V,W)\)
admits all reflections and has a standard Cartan graph of type \(B_2\).}} In particular,
\(m_{12}=1\), \(m_{21}=2\), and \(X_1,Y_2\) are simple by
Lemma~\ref{lem:root-string-adjoint-simplicity}.
For \(i=1,2\), Lemma~\ref{lem:Gamma4-reflected-supports} shows that
\(R_i(V,W)\), with the indicated generators, satisfies the standing
support assumptions of this section. Its Nichols algebra is
finite-dimensional by Lemma~\ref{lem:reflection-dimension-invariance},
and its Cartan matrix has type \(B_2\). Proposition
\ref{prop:Gamma4-local-B2-Cartan-matrix} therefore gives
\eqref{eq:Gamma4-classification-B2} for this reflected tuple.
Induction on the length of a reflection sequence proves the assertion.
\end{proof}
}

\subsection{Finite Cartan graphs}
\label{subsec:Gamma4-finite-Cartan-graphs}

\begin{prop}
\label{prop:Gamma4-pair-with-finite-Cartan-matrix}
Assume that \((V,W)\) admits all reflections and that its Cartan
graph is finite.  Let
\(P=(P_1,P_2)\in\mathcal F_2(V,W)\) have an indecomposable Cartan
matrix of finite type.  After interchanging \(P_1\) and \(P_2\) if
necessary, there exist
\(\varepsilon_P,h_P,g_P\in G\) such that
\begin{align*}
h_Pg_P&=\varepsilon_Pg_Ph_P,&
g_P\varepsilon_P&=\varepsilon_P^{-1}g_P,&
h_P\varepsilon_P&=\varepsilon_Ph_P,&
\operatorname{ord}(\varepsilon_P)&=4,
\end{align*}
and
\[
\operatorname{supp}P_1
=\{h_P,\varepsilon_P^{-1}h_P\},
\qquad
\operatorname{supp}P_2
=\{\varepsilon_P^ig_P\mid0\leq i<4\}.
\]
Moreover, \(G=\langle \varepsilon_P, h_P,g_P\rangle\).
\end{prop}

\begin{proof}
The relation \(hg=\varepsilon gh\) gives
\(\varepsilon=hgh^{-1}g^{-1}\), so \(G=\langle h,g\rangle\).
Since \(h\) and \(g\) belong to the initial supports,
Lemma~\ref{lem:reflections-preserve-support-group} gives
\(G=\langle\operatorname{supp}P_1\cup
\operatorname{supp}P_2\rangle\).

An indecomposable Cartan matrix of rank two and finite type is of
type \(A_2\), \(B_2\), or \(G_2\), possibly with the two indices
interchanged.  Thus, in one of the two orderings,
\(-a_{12}^{[P]}=1\) and
\(1\leq-a_{21}^{[P]}\leq3\).  In this ordering,
\[
(\operatorname{ad}P_1)^2(P_2)=0,
\qquad
(\operatorname{ad}P_2)^4(P_1)=0.
\]
Moreover, \((\operatorname{ad}P_1)(P_2)\neq0\), so \(P\) is
braided-indecomposable.  Therefore
Theorem~\ref{thm:support-classification} applies.

If a union of conjugacy classes generates \(G\), then its inner
automorphism group is isomorphic to \(G/Z(G)\).  The initial support
is \(Z_4^{4,2}\) and generates \(G\). Hence \(G/Z(G)\) is isomorphic
to \(\operatorname{Inn}(Z_4^{4,2})\), which has order eight.
The support of \(P_1\oplus P_2\) also generates \(G\), so its inner
automorphism group has order eight as well.

The permutation descriptions in
Section~\ref{sec:support-quandles} show that the inner automorphism
groups of
\(Z_T^{4,1}\), \(Z_2^{2,2}\), \(Z_3^{3,1}\),
\(Z_3^{3,2}\), and \(Z_4^{4,2}\) have orders
\(12,4,6,6,8\), respectively.  Hence
Theorem~\ref{thm:support-classification} implies that
\(\operatorname{supp}(P_1\oplus P_2)\) is isomorphic to
\(Z_4^{4,2}\).

The two supports are the two  orbits of \(Z_4^{4,2}\).
After interchanging \(P_1\) and \(P_2\) if necessary, assume that
\(\operatorname{supp}P_1\) is the two-element orbit and
\(\operatorname{supp}P_2\) is the four-element orbit.  The
corresponding quandle isomorphism induces an epimorphism
\(\Gamma_4\twoheadrightarrow G\).  Let
\(\varepsilon_P,h_P,g_P\) be the images of the standard generators.
The defining relations and conjugacy classes of \(\Gamma_4\) give
the relations and supports in the statement.

The four elements of \(\operatorname{supp}P_2\) are distinct, while
\(\varepsilon_P^4=1\).  Therefore
\(\operatorname{ord}(\varepsilon_P)=4\).
Finally,
\(\varepsilon_P=h_Pg_Ph_P^{-1}g_P^{-1}\), and the epimorphism above
shows that \(G=\langle\varepsilon_P,h_P,g_P\rangle\).
\end{proof}

\begin{prop}
\label{prop:Gamma4-necessity}
Assume that \(\mathcal B(V\oplus W)\) is finite-dimensional.  Then
the Cartan graph of \((V,W)\) is standard of type \(B_2\).  Moreover,
\[ \dim\rho=\dim\sigma=1, \quad
\rho(h)=\sigma(g)=-1,
\quad
\Delta_4=1,
\quad
\Theta_{2,j}=1\quad(1\leq j\leq7).
\]
With the original ordering, the Cartan matrix at every object \([Q]\)
is
\[
A^{[Q]}=
\begin{pmatrix}2&-1\\-2&2\end{pmatrix}.
\]
\end{prop}

\begin{proof}
By Theorem~\ref{thm:Nichols-finiteness-criterion}, \((V,W)\)
admits all reflections and has a finite Cartan graph.
Proposition~\ref{prop:pair-with-finite-Cartan-matrix} provides
\(P\in\mathcal F_2(V,W)\) whose Cartan matrix is indecomposable
and of finite type. By
Proposition~\ref{prop:Gamma4-pair-with-finite-Cartan-matrix},
after interchanging its components if necessary, the resulting
tuple \(\widetilde P\) satisfies the standing assumptions of this
section. Its Nichols algebra is finite-dimensional by
Lemma~\ref{lem:reflection-dimension-invariance}.

{
By Lemma~\ref{lem:Gamma4-projective-representations-one-dimensional},
the inducing representations of \((V,W)\) and \(\widetilde P\)
are one-dimensional.
}
Proposition~\ref{prop:Gamma4-local-B2-Cartan-matrix}, applied to
\(\widetilde P\), gives
\eqref{eq:Gamma4-classification-B2}.
{%
Proposition~\ref{prop:Gamma4-sufficiency} therefore gives a standard
Cartan graph of type \(B_2\) for \(\widetilde P\), and hence for
\((V,W)\) up to interchanging the two indices.}

In particular, the initial Cartan matrix is of finite type.
Applying Proposition~\ref{prop:Gamma4-local-B2-Cartan-matrix}
to \((V,W)\) gives the asserted parameter equalities and the
displayed matrix in the original ordering. Since the Cartan
graph is standard, this matrix occurs at every object.
\end{proof}

\subsection{Proof of the classification theorem}

{%
\begin{proof}[Proof of Theorem~\ref{thm:Gamma4-classification}]
Proposition~\ref{prop:Gamma4-necessity} proves necessity, and
Proposition~\ref{prop:Gamma4-sufficiency} proves sufficiency and
the assertion about the Cartan graph. The dimension is given by
Corollary~\ref{cor:Gamma4-Nichols-dimension}.
\end{proof}
}
{%
\begin{rmk}\upshape
When \(\Phi=1\), Theorem~\ref{thm:Gamma4-classification}
specializes to \cite[Theorem~5.5]{rank2-3}.
 {In every case covered by the theorem,
\(\dim\mathcal B(V\oplus W)=2^{18}\).}
\end{rmk}
}
\section{The \texorpdfstring{$T$}{T} case}
\label{sec:classification-T}

Let \(G\) be a finite non-abelian quotient of
\[
T=\left\langle z,x_1,x_2\ \middle|\
zx_1=x_1z,\quad zx_2=x_2z,\quad
x_1x_2x_1=x_2x_1x_2,\quad x_1^3=x_2^3
\right\rangle .
\]
Write \(\pi:T\twoheadrightarrow G\) for the quotient map, and denote
the images of the generators by the same letters. The defining
relations imply that \(z\in Z(G)\).
Let \(x_3=x_2x_1x_2^{-1}\) and \(x_4=x_1x_2x_1^{-1}\).
Throughout this section, the elements \(x_1,x_2,x_3,x_4\) are distinct,
\(x_1^G=\{x_1,x_2,x_3,x_4\}\), and
\(G=\langle z,x_1^G\rangle\).
Let \(\Phi\) be a normalized \(3\)-cocycle on \(G\).  Let
\(V=M(z,\rho)\) and \(W=M(x_1,\sigma)\) be simple, and assume that
\((V,W)\) is braided-indecomposable.

The rack table gives
\(2\brhd3=4\), \(4\brhd2=3\), and \(3\brhd4=2\).  Hence
\(c_{W,W}^{3}\) preserves \(W_{x_2}\otimes W_{x_3}\).  By Proposition~\ref{prop:T-sigma-one-dimensional},
one has \(\dim\sigma=1\).
We define
\(\varepsilon_\Phi(W)\) by
\begin{equation}
\label{eq:T-epsilon-definition}
\left.c_{W,W}^{3}\right|_{W_{x_2}\otimes W_{x_3}}
=\sigma(x_1)^{-1}\varepsilon_\Phi(W)\operatorname{id}.
\end{equation}
This definition is independent of all basis choices.
 {A gauge transformation conjugates the restrictions of
\(c_{W,W}\) to \(W_{x_1}\otimes W_{x_1}\) and of
\(c_{W,W}^3\) to \(W_{x_2}\otimes W_{x_3}\).
Their scalars \(\sigma(x_1)\) and
\(\sigma(x_1)^{-1}\varepsilon_\Phi(W)\) are therefore unchanged,
and so is \(\varepsilon_\Phi(W)\).}

The classification for quotients of \(T\) is as follows.

\begin{thm}
\label{thm:T-classification}
Under the preceding assumptions,
\(\dim\mathcal B(V\oplus W)<\infty\) if and only if
\begin{equation*}
\tag{T-G$_2$}
\label{eq:T-classification-G2}
\begin{aligned}
&\dim\rho=1,\qquad \qquad \qquad \qquad \sigma(x_1)=-1,\\
&\bigl(\rho(x_1)\sigma(z)\bigr)^2
 -\rho(x_1)\sigma(z)+1=0,\\
&\rho(x_1)\rho(z)\sigma(z)=1,
\qquad \qquad \qquad \qquad \varepsilon_\Phi(W)=1,\\
&\Phi(x_1x_4^{-1},x_1x_3^{-1},x_1x_2^{-1})
 \Phi(x_1x_3^{-1},x_1x_2^{-1},x_1x_4^{-1})
 \Phi(x_1x_2^{-1},x_1x_4^{-1},x_1x_3^{-1})=1.
\end{aligned}
\end{equation*}
Whenever these conditions hold, \((V,W)\) admits all reflections and
its Cartan graph is standard of type \(G_2\), with matrix
\(\left(\begin{smallmatrix}2&-1\\-3&2\end{smallmatrix}\right)\)
at every object.  Moreover,
\(\dim\mathcal B(V\oplus W)=6^3\,72^3=80\,621\,568\).
\end{thm}

\medskip
\noindent\textit{Outline of the proof.}
\begin{itemize}[leftmargin=2em]
\item
Subsection~\ref{subsec:T-cocycle-normalization} computes \(H^3(T,\mathbb C^\times)\) and  {shows that the pullback of \(\Phi\) is cohomologous to one of eight fixed cocycles}. 

\item
Subsection~\ref{subsec:T-projective-representations} describes the
centralizers and the projective representations that occur.

\item
Subsection~\ref{subsec:T-adjoint-objects} proves \(\dim\rho=1\)
 {when \(\mathcal B(V\oplus W)\) is finite-dimensional
and the Cartan matrix is of type \(G_2\)}, and computes the adjoint
objects.  It gives \(\sigma(x_1)=-1\), the quadratic condition in
\eqref{eq:T-classification-G2}, and \(\varepsilon_\Phi(W)=1\).
{%
Under the last equality in \eqref{eq:T-classification-G2}, the
 calculations in \cite{rank2-3} give \(\dim Y_3=1\)
and \(Y_4=0\).
}

\item
Subsections~\ref{subsec:T-G2-parameters} and~\ref{subsec:T-reflections}
derive \(\rho(x_1)\rho(z)\sigma(z)=1\), exclude the remaining
nontrivial cocycle class, and prove that all parameter conditions are
preserved by both reflections.

\item
Subsection~\ref{subsec:T-finite-Cartan-graphs} proves necessity and
shows that the Cartan graph is standard of type \(G_2\).

\item
Subsections~\ref{subsec:T-G2-simple-factors}
and~\ref{subsec:T-proof-classification} prove sufficiency.  {Among the six simple objects corresponding to the positive roots,
three have dimension \(1\) and three have dimension \(4\).} Their Nichols algebras have dimensions \(6\)
and \(72\), respectively.
\end{itemize}

\subsection{The pullback of the cocycle to \texorpdfstring{$T$}{T}}
\label{subsec:T-cocycle-normalization}

We first relate the last equality in
\eqref{eq:T-classification-G2} to the pullback of \(\Phi\).

\begin{lemma}
\label{lem:T-pullback-cocycle}
One has \(H^3(T,\mathbb C^\times)\simeq\mathbb{Z}_8\).
The class of \(\pi^*\Phi\) is trivial if and only if
the last equality in \eqref{eq:T-classification-G2} holds.
The left-hand side of this equality is unchanged when \(\Phi\)
is replaced by a cohomologous normalized cocycle.
\end{lemma}

\begin{proof}
In \(T\), put
\[
i=x_1x_4^{-1},\qquad j=x_1x_3^{-1},\qquad k=x_1x_2^{-1}.
\]
 {The defining relations give \(j=x_2^{-1}x_1\).} Since the common cube
\(x_1^3=x_2^3\) is central,
\[
x_1ix_1^{-1}=j,\qquad
x_1jx_1^{-1}=k,\qquad
x_1kx_1^{-1}=i.
\]
Substituting \(x_2=k^{-1}x_1\) into the braid relation gives
\(i^{-1}=k^{-1}j^{-1}\). Its cyclic conjugates therefore give
\(ij=k\), \(jk=i\), and \(ki=j\). These relations present \(Q_8\):
eliminating \(k\) gives \(i^4=1\), \(i^2=j^2\), and
\(jij^{-1}=i^{-1}\).
Conversely, in \(Q_8\rtimes\langle x_1\rangle\), one has
\((k^{-1}x_1)^3=k^{-1}i^{-1}j^{-1}x_1^3=x_1^3\) and
\(x_1(k^{-1}x_1)x_1=i^{-1}x_1^3
=(k^{-1}x_1)x_1(k^{-1}x_1)\).
Thus the substitutions \(k=x_1x_2^{-1}\) and \(x_2=k^{-1}x_1\)
give mutually inverse homomorphisms.  {The subgroup \(\langle z\rangle\) is  infinite, so}
\(T\simeq(Q_8\rtimes\langle x_1\rangle)\times\langle z\rangle\),
where both displayed cyclic groups are infinite.

The periodic free resolution in
\cite[Proposition~3.2]{TomodaZvengrowski2008}, with trivial
coefficients \(\mathbb C^\times\), gives
\[
H^1(Q_8,\mathbb C^\times)\simeq(\mathbb Z/2)^2,\qquad
H^2(Q_8,\mathbb C^\times)=0,\qquad
H^3(Q_8,\mathbb C^\times)\simeq\mathbb{Z}_8.
\]
Conjugation by \(x_1\) cyclically permutes the three nontrivial
characters of \(Q_8\). Thus the invariant and coinvariant groups
of \(H^1(Q_8,\mathbb C^\times)\) under this action are zero.
The action on \(H^3(Q_8,\mathbb C^\times)\) is trivial, since
its order divides three whereas \(\operatorname{Aut}(\mathbb{Z}_8)\)
has order four.

Put \(S=Q_8\rtimes\langle x_1\rangle\). The
Lyndon--Hochschild--Serre spectral sequence for
\(1\to Q_8\to S\to\langle x_1\rangle\to1\) is
\[
E_2^{p,q}
=H^p\bigl(\langle x_1\rangle,H^q(Q_8,\mathbb C^\times)\bigr)
\Longrightarrow H^{p+q}(S,\mathbb C^\times);
\]
see \cite[Chapter~VII, Section~6]{Brown1982}.
Since \(\langle x_1\rangle\) is infinite cyclic, only columns
\(p=0,1\) can be nonzero. These columns consist of the invariant
and coinvariant groups, respectively, and all differentials
from the second page onward vanish.
In total degree two, both terms
\(E_2^{0,2}\) and \(E_2^{1,1}\) are zero, so
\(H^2(S,\mathbb C^\times)=0\).
In total degree three, \(E_2^{1,2}=0\), whereas
\(E_2^{0,3}=H^3(Q_8,\mathbb C^\times)\).
The edge homomorphism, which is restriction to \(Q_8\),
therefore gives an isomorphism
\(H^3(S,\mathbb C^\times)\simeq H^3(Q_8,\mathbb C^\times)\).

Apply the same spectral sequence to
\(1\to S\to T\to\langle z\rangle\to1\).
The action of \(\langle z\rangle\) on \(H^q(S,\mathbb C^\times)\)
is trivial because \(z\) is central.
The two possible terms in total degree three are
\(H^3(S,\mathbb C^\times)\) and
\(H^1(\langle z\rangle,H^2(S,\mathbb C^\times))\).
The latter is zero. Hence restriction from \(T\) to \(S\)
is also an isomorphism in degree three. Composing the two
restriction maps gives
\(H^3(T,\mathbb C^\times)\simeq
H^3(Q_8,\mathbb C^\times)\simeq\mathbb{Z}_8\).

For any two-cochain \(J\) on \(Q_8\), the relations
\(ij=k\), \(jk=i\), and \(ki=j\) give
\(\delta J(i,j,k)=J(j,k)J(i,i)/(J(k,k)J(i,j))\).
 {Taking the product of this expression and its two cyclic permutations
yields}
\(\delta J(i,j,k)\delta J(j,k,i)\delta J(k,i,j)=1\).
Consequently,
\[
[\Psi]\longmapsto
\Psi(i,j,k)\Psi(j,k,i)\Psi(k,i,j)
\]
defines a homomorphism
\(H^3(Q_8,\mathbb C^\times)\to\mathbb C^\times\).
The cocycle obtained by exponentiating the additive cocycle in
Lemma~\ref{lem:T-fixed-quaternion-cocycle} has product
\(e^{(3+3+3)\pi\mathrm i/4}=e^{\pi\mathrm i/4}\).
Since \(H^3(Q_8,\mathbb C^\times)\) has order eight,
this homomorphism is an isomorphism onto the group of eighth
roots of unity.

Together with the restriction isomorphism above, this shows that
\([\pi^*\Phi]\) is trivial precisely when
\((\pi^*\Phi)(i,j,k)(\pi^*\Phi)(j,k,i)
(\pi^*\Phi)(k,i,j)=1\).
This is the last equality in \eqref{eq:T-classification-G2}.
Replacing \(\Phi\) by a cohomologous normalized cocycle changes
its pullback by a coboundary, whose three values have product
one by the preceding calculation. This proves the final assertion.
\end{proof}

For the rest of this subsection, retain \(i,j,k\) from the preceding
proof, and let \(A\) be the automorphism of \(Q_8\) cycling them.
Let \(F:Q_8^3\to\mathbb Z_8\) be the normalized
\(A\)-invariant additive cocycle of
Lemma~\ref{lem:T-fixed-quaternion-cocycle}.  Define
\begin{equation}
\label{eq:T-cocycle-representative}
\Omega(ux_1^mz^a,vx_1^nz^b,wx_1^rz^c)
=\exp\!\left(\frac{\pi\mathrm i}{4}
 F(u,A^mv,A^{m+n}w)\right).
\end{equation}
The cocycle identity for \(F\), applied to
\(u,A^mv,A^{m+n}w,A^{m+n+r}t\), and its \(A\)-invariance show
that \(\Omega\) is a normalized cocycle on \(T\). Moreover,
\(\Omega(i,j,k)\Omega(j,k,i)\Omega(k,i,j)
=e^{\pi\mathrm i/4}\).
 {Thus there is a unique \(h\in\{0,\ldots,7\}\) for which a normalized
two-cochain \(J\) on \(T\) satisfies}
\begin{equation}
\label{eq:T-cocycle-normalization}
\pi^*\Phi=\Omega^h\delta J,
\qquad
\delta J(a,b,c)=\frac{J(b,c)J(a,bc)}{J(ab,c)J(a,b)}.
\end{equation}
In particular, the last equality in \eqref{eq:T-classification-G2}
is equivalent to \(h=0\).

\begin{lemma}
\label{lem:T-compatible-normalization}
 The pair \((V,W)\) can be replaced
by a pair in \({}_T^T\mathcal{YD}^{\Omega^h}\), with \(h\) as in
\eqref{eq:T-cocycle-normalization}, preserving the dimensions of
its Nichols algebra and all its iterated adjoint objects.  The scalars \(\rho(z)\) and \(\sigma(x_1)\) are preserved.
When \(\dim\sigma=1\), so is \(\varepsilon_\Phi(W)\).
When \(\dim\rho=\dim\sigma=1\), so is
\(\rho(x_1)\sigma(z)\).
\end{lemma}

\begin{proof}
The map \(\pi\) is bijective from \(x_1^T\) onto \(x_1^G\).
Lifting the support degrees and letting \(T\) act through \(\pi\)
therefore gives simple objects in
\({}_T^T\mathcal{YD}^{\pi^*\Phi}\).
On every tensor word, the associators and braidings agree with
the original ones.

By \eqref{eq:T-cocycle-normalization}, the two-cochain \(J\)
induces a braided monoidal equivalence to
\({}_T^T\mathcal{YD}^{\Omega^h}\), with tensor structure and action
\begin{equation}
\label{eq:T-cochain-tensor-structure}
u_a\otimes v_b\longmapsto J(a,b)^{-1}u_a\otimes v_b,
\qquad
g\rhd_Jv_x=
\frac{J(gxg^{-1},g)}{J(g,x)}\,g\rhd v_x.
\end{equation}
This equivalence preserves the dimensions of the homogeneous
spaces, Nichols algebras, and iterated adjoint objects.
The scalar assertions follow from the action formula and
the invariance of \(\varepsilon_\Phi(W)\) established after
\eqref{eq:T-epsilon-definition}.

Retain \(\rho,\sigma\) for the inducing representations after this
replacement. By \eqref{eq:T-cocycle-representative}, the multipliers
at \(z\) and \(x_1\) are trivial on \(T\) and
\(C_T(x_1)=\langle-1,x_1,z\rangle\simeq C_2\times\mathbb Z^2\),
respectively. Thus \(\rho\) and \(\sigma\) are ordinary irreducible
representations.
\end{proof}

\begin{lemma}
\label{lem:T-cocycle-two-branches}
Assume that \(\dim\sigma=1\), \(\sigma(x_1)=-1\), and
\(\varepsilon_\Phi(W)=1\).
After the replacement in
Lemma~\ref{lem:T-compatible-normalization}, denote the inducing
character of \(C_T(x_1)\) again by \(\sigma\).
Then
\[
\bigl(h,\sigma((x_1x_2^{-1})^2)\bigr)
=(0,1)\qquad\text{or}\qquad(4,-1).
\]
\end{lemma}

\begin{proof}
Work over \(T\) with cocycle \(\Omega^h\).
Choose the induced basis \(w_1,w_2,w_3,w_4\) corresponding to the
representatives \(1,i,k,j\) and the ordered support
\(x_1,x_2,x_3,x_4\). Write
\(x_a\rhd w_b=\beta_{ab}w_{a\brhd b}\),
\(q=\sigma(x_1)\), and
\(\eta=\sigma((x_1x_2^{-1})^2)\).
Since \(\sigma\) is an ordinary character after the replacement
and \((x_1x_2^{-1})^2\) has order two, one has \(\eta^2=1\).
Substituting the values of \(F\) from
Lemma~\ref{lem:T-fixed-quaternion-cocycle} into
\eqref{eq:induced-action} gives
\[
\beta_{23}=\beta_{42}=\beta_{34}
=q\eta\exp(7h\pi\mathrm i/4).
\]
 {For each of these three entries, the element \(c\) in the coset
decomposition used in \eqref{eq:induced-action} is \(-x_1\)}, and the quotient of the two multiplier factors in
\eqref{eq:induced-action} is \(\exp(7h\pi\mathrm i/4)\).
Equation~\eqref{eq:T-epsilon-definition} therefore gives
\[
\varepsilon_{\Omega^h}(W)
=q\beta_{23}\beta_{42}\beta_{34}
=q^4\eta\exp(5h\pi\mathrm i/4).
\]
Since \(q=-1\), its value is one exactly when
\((h,\eta)=(0,1)\) or \((4,-1)\).
\end{proof}

\subsection{Projective representations}
\label{subsec:T-projective-representations}

\begin{lemma}
\label{lem:T-group-data}
One has
\[
Z(G)=\langle z,x_1^3,x_1x_2x_3\rangle,\qquad
G/Z(G)\simeq A_4,\qquad
C_G(x_1)=\langle x_1,x_2x_3,z\rangle.
\]
The group \(C_G(x_1)\) is abelian.
\end{lemma}

\begin{proof}
The conjugation action of \(x_1,x_2,x_3,x_4\) on their common
conjugacy class is
\[
(243),\qquad(134),\qquad(142),\qquad(123).
\]
These permutations lie in \(A_4\).  The subgroup generated by the
first two is transitive.  Its order is divisible by both \(4\) and
\(3\), so it is \(A_4\).
An element lies in the kernel of this action exactly when it commutes
with all four \(x_i\).  Since \(z\) is central and \(G\) is generated
by \(z,x_1,x_2,x_3,x_4\), the kernel is \(Z(G)\).  Hence
\(G/Z(G)\simeq A_4\).

 {The defining relations show that
\(z,x_1^3,x_1x_2x_3\) are central, and}
\(G/\langle z,x_1^3,x_1x_2x_3\rangle\) is a quotient of \(A_4\).
This quotient maps onto \(G/Z(G)\simeq A_4\).  Its order is therefore
at most \(12\) and at least \(12\).  Thus
\(Z(G)=\langle z,x_1^3,x_1x_2x_3\rangle\).

The centralizer of \((243)\) in \(A_4\) is
\(\langle(243)\rangle\).  Its inverse image in \(G\) is
\(\langle x_1,Z(G)\rangle=\langle x_1,x_2x_3,z\rangle\).
Every element of \(C_G(x_1)\) lies in this inverse image, while its
generators commute with \(x_1\).  Hence it is \(C_G(x_1)\).
The group is abelian: \(z\) and \(x_1x_2x_3\) are central, and
\(x_2x_3=x_1^{-1}(x_1x_2x_3)\).

\end{proof}

\begin{prop}
\label{prop:T-sigma-one-dimensional}
{%
One has \(\dim\sigma=1\), and the restriction of \(f_\Phi\)
to \(C_G(x_1)^3\) is trivial.
}
\end{prop}

\begin{proof}
The defining relations of \(T\) give
\(x_1x_2x_3=x_2x_3x_1\), and \(z\) is central.
Thus \(x_1,x_2x_3,z\) commute in \(T\).
As before, the same symbols denote their images in \(G\).
For any commuting triple \(a,b,c\), substituting the formula for
\(\delta J\) in \eqref{eq:T-cocycle-normalization} into the
definition of the alternator gives \(f_{\delta J}(a,b,c)=1\).
Indeed, all two-cochain factors cancel because
\(ab=ba\), \(bc=cb\), and \(ca=ac\).
Moreover, \eqref{eq:T-cocycle-representative} and the normalization
of \(F\) give
\(\Omega(a,b,z)=\Omega(a,z,b)=\Omega(z,a,b)=1\)
for all \(a,b\in T\).
Hence all six cocycle values occurring in
\(f_{\Omega^h}(x_1,x_2x_3,z)\) are one.
Using \(\pi^*\Phi=\Omega^h\delta J\), we therefore obtain
\[
\begin{aligned}
f_\Phi(x_1,x_2x_3,z)
=f_{\pi^*\Phi}(x_1,x_2x_3,z)
=f_{\Omega^h}(x_1,x_2x_3,z)
  f_{\delta J}(x_1,x_2x_3,z)
=1.
\end{aligned}
\]
By Lemma~\ref{lem:T-group-data},
\(C_G(x_1)=\langle x_1,x_2x_3,z\rangle\) is abelian.
Lemma~\ref{lem:alternator} shows that the restriction of
\(f_\Phi\) to this group is an alternating tricharacter,
determined by its value on \((x_1,x_2x_3,z)\).
It is therefore trivial. Thus
\(\Phi_{x_1}(a,b)/\Phi_{x_1}(b,a)=f_\Phi(x_1,a,b)=1\)
for all \(a,b\in C_G(x_1)\).
The corresponding twisted group algebra is commutative,
so its irreducible module \(W_{x_1}\) is one-dimensional.
\end{proof}
\begin{lemma}
\label{lem:T-central-projective-reduction}
The homomorphism
\(g\mapsto[\rho(g)]\in\operatorname{PGL}(V_z)\) factors through
\(G/Z(G)\simeq A_4\), and the resulting projective representation is
irreducible.  In particular, \(\dim\rho\in\{1,2,3\}\).
\end{lemma}

\begin{proof}
 {%
Proposition~\ref{prop:T-sigma-one-dimensional} gives
\(f_\Phi(x_1,x_2x_3,z)=1\).
}
For each \(c\in\{x_1^3,x_1x_2x_3\}\),
Lemma~\ref{lem:alternator} shows that
\(g\mapsto f_\Phi(g,z,c)\) is a character of \(G\).
Its value at \(z\) is \(1\) by alternation.
For \(c=x_1^3\), its value at \(x_1\) is
\(f_\Phi(x_1,z,x_1)^3=1\).
For \(c=x_1x_2x_3\), its value at \(x_1\) is
\(f_\Phi(x_1,x_2x_3,z)^{-1}=1\).
A character has the same value on conjugate elements.  Since
\(G=\langle z,x_1^G\rangle\), both characters are trivial.
For \(c=z\), alternation gives \(f_\Phi(g,z,z)=1\) directly.

Let \(c\in\{z,x_1^3,x_1x_2x_3\}\) and \(g\in G\).
Since \(z\) and \(c\) are central,
Lemma~\ref{lem:alternator} gives
\( \frac{\Phi_z(c,g)}{\Phi_z(g,c)}
=f_\Phi(z,c,g)=f_\Phi(g,z,c)=1\).
The equation \eqref{eq:YD-projective-action} therefore yields
\[
\rho(c)\rho(g)
=\Phi_z(c,g)\rho(cg)
=\Phi_z(g,c)\rho(gc)
=\rho(g)\rho(c).
\]
By Schur's lemma, \(\rho(c)\) is a scalar.
These three elements generate \(Z(G)\), so
 {\eqref{eq:YD-projective-action}} shows that every element
of \(Z(G)\) acts by a scalar.
 {The same identity} also shows that
\(g\mapsto[\rho(g)]\) is a homomorphism from \(G\) to
\(\operatorname{PGL}(V_z)\).  It is trivial on \(Z(G)\), so it
factors through \(G/Z(G)\simeq A_4\).  The resulting projective
representation is irreducible.

The corresponding twisted group algebra of \(A_4\) is semisimple and
has dimension \(12\).  Hence the matrix block belonging to this
irreducible representation has dimension
\((\dim\rho)^2\leq12\).  Therefore
\(\dim\rho\in\{1,2,3\}\).
\end{proof}

\subsection{Adjoint objects}
\label{subsec:T-adjoint-objects}

Write \(Y_m=(\operatorname{ad}W)^m(V)\).  Let
\(\varphi_m^{Y,\Phi}\) be the   recursive
adjoint map from Lemma~\ref{lem:adjoint-object-realization}, with
\(V\) and \(W\) interchanged.  Let
\(Q_m^Y(r_1,\ldots,r_m;z)\) denote the spaces \(Q_m\) from
Subsection~\ref{subsec:support-restrictions}, with \(V\) and \(W\)
interchanged.  We retain the notation \(\supp Q\) from that subsection
for multidegree support.  By
Lemma~\ref{lem:adjoint-object-realization}, the sum of the spaces
\(Q_m^Y(r_1,\ldots,r_m;z)\), over all
\(r_1,\ldots,r_m\in\supp W\), is isomorphic to \(Y_m\).

\subsubsection{Dimension constraints from the adjoint objects}

\begin{lemma}
\label{lem:T-x1-action-minus-one}
Assume that
\(Y_3=(\operatorname{ad}W)^3(V)\) is nonzero and simple.  Then
\[
x_1\rhd w=-w\qquad(w\in W_{x_1}).
\]
\end{lemma}

\begin{proof}
If \(Y_1=0\), then the recursive definition gives \(Y_3=0\).
Hence \(Y_1\neq0\). By Lemma~\ref{lem:adjoint-object-realization}, at least one of
the spaces \(Q_1^Y(x_i;z)\) is nonzero.
The first adjoint map commutes with the \(G\)-action, so
\(g\rhd Q_1^Y(x_i;z)=Q_1^Y(gx_ig^{-1};z)\).
Since the \(x_i\) form a single conjugacy class, it follows that
\(Q_1^Y(x_4;z)\neq0\).
Moreover, centrality of \(z\) implies that the double braiding
preserves \(W_{x_4}\otimes V_z\).
Thus \(Q_1^Y(x_4;z)\subseteq W_{x_4}\otimes V_z\), and its
multidegree is \((x_4,z)\).
Since
\[
x_4\brhd x_1=x_2\neq x_1,\qquad
z\brhd x_1=x_1,\qquad x_1\neq x_4,
\]
Proposition~\ref{prop:multidegree-insertion}, applied by inserting
\(x_1\) in the first position, gives
\begin{equation*}
(x_1,x_4,z)
\in\supp Q_2^Y(x_1,x_4;z).
\end{equation*}
Insert \(x_2\) in the second position.  Here
\(x_2\brhd x_1=x_3\), \(x_4\brhd x_2=x_3\neq x_2\), and
\(z\brhd x_2=x_2\), while  {the set on the right-hand side of
\eqref{eq:multidegree-insertion-2}} is
\(\{x_1,x_4,x_3\}\).  Therefore
\begin{equation}
\label{eq:T-x1-action-second-tuple}
(x_3,x_2,x_4,z)
\in\supp Q_3^Y(x_2,x_1,x_4;z).
\end{equation}

The relations
\(x_3x_2=x_1x_3\) and \(x_3x_4=x_2x_3\) give
\(x_3x_2x_4=x_1x_2x_3\).  Hence the component in
\eqref{eq:T-x1-action-second-tuple} has degree
\(x_1x_2x_3z\), which is central by
Lemma~\ref{lem:T-group-data}.  Since \(Y_3\) is simple, its support
is one conjugacy class, and therefore
\(\supp Y_3=\{x_1x_2x_3z\}\).
On the other hand, every vector in
\(Q_3^Y(x_1,x_1,x_4;z)\) has degree \(x_1^2x_4z\).  This degree is
not central: since
\(x_1^2x_4=x_1^3x_2x_1^{-1}\) and \(x_1^3\) is central,
\[
(x_1^2x_4)\brhd x_1=x_2x_1x_2^{-1}=x_3\neq x_1.
\]
Consequently,
\begin{equation}
\label{eq:T-x1-action-repeated-vanishing}
Q_3^Y(x_1,x_1,x_4;z)=0.
\end{equation}

The degree formulas in the proof of
Proposition~\ref{prop:multidegree-insertion} show that the possible
multidegrees of \(Q_2^Y(x_1,x_4;z)\) are
$
(x_1,x_4,z),\ (x_3,x_1,z),\ (x_4,x_3,z).
$
Consider the \((x_1,x_1,x_4,z)\)-component of
\(\varphi_3^{Y,\Phi}\) on
\(W_{x_1}\otimes Q_2^Y(x_1,x_4;z)\).
The identity term contributes only from \((x_1,x_4,z)\).
 {In the image of the middle double braiding term, the first tensor
factor has degree \((x_1^2x_4z)\brhd x_1=x_3\).
For the other two multidegrees, the first tensor factor after applying
\(c_{1,2}^{\Phi}\) has degree \(x_2\) or \(x_3\), respectively.}
Thus the required component has exactly two contributions, acting
on the first two factors as
\(\operatorname{id}_{W\otimes W}\) and \(c_{W,W}\). The two
associativity factors in \(c_{1,2}^{\Phi}\) cancel.

Choose an element of \(Q_2^Y(x_1,x_4;z)\) with a nonzero
\((x_1,x_4,z)\)-component.  Applying a suitable linear functional
on \(W_{x_4}\otimes V_z\) gives \(0\neq w_0\in W_{x_1}\).
Equation~\eqref{eq:T-x1-action-repeated-vanishing} then yields
\[
w\otimes w_0+(x_1\rhd w_0)\otimes w=0
\qquad(w\in W_{x_1}).
\]
Taking \(w=w_0\) gives \(x_1\rhd w_0=-w_0\).
By Proposition~\ref{prop:T-sigma-one-dimensional},
\(W_{x_1}\) is one-dimensional, so \(x_1\) acts by \(-1\)
on all of \(W_{x_1}\).
\end{proof}

\begin{prop}
\label{prop:T-rho-one-dimensional}
Assume that \(\mathcal B(V\oplus W)\) is finite-dimensional and that
\((V,W)\) has Cartan matrix
\[
\begin{pmatrix}
2&-1\\
-3&2
\end{pmatrix}.
\]
Then \(\dim\rho=1\).
\end{prop}

\begin{proof}
Proposition~\ref{prop:T-sigma-one-dimensional} gives
\(\dim\sigma=1\).
By Theorem~\ref{thm:Nichols-finiteness-criterion}, the tuple \((V,W)\)
admits all reflections and has a finite Cartan graph. The Cartan matrix gives
\(X_1\neq0\), \(X_2=0\), and \(Y_3\neq0\).  Since
\(-a_{12}=1\) and \(-a_{21}=3\),
Lemma~\ref{lem:root-string-adjoint-simplicity}, applied with \(t=1\)
and \(t=3\), shows that \(X_1\) and \(Y_3\) are simple.
Lemma~\ref{lem:T-x1-action-minus-one} therefore gives
\(\sigma(x_1)=-1\).
Lemma~\ref{lem:T-central-projective-reduction} gives
\(\dim\rho\in\{1,2,3\}\).  Its proof also shows that every element of
\(Z(G)\) acts on \(V_z\) by a scalar and that the projective class of
\(\rho\) gives an irreducible projective representation of
\(G/Z(G)\simeq A_4\).  In particular, we regard \(\rho(z)\) as a
nonzero scalar.

Let \(T_0=\sigma(z)\rho(x_1)\in\operatorname{End}(V_z)\).
{
Since \(C_G(zx_1)=C_G(x_1)\) and \(zx_1\in C_G(x_1)\),
Proposition~\ref{prop:T-sigma-one-dimensional} gives
\[
\frac{\Phi_{zx_1}(a,b)}{\Phi_{zx_1}(b,a)}
=f_\Phi(zx_1,a,b)=1
\qquad(a,b\in C_G(x_1)).
\]
Thus the corresponding twisted group algebra is commutative.
}

The support of the nonzero simple object \(X_1\) is
\(\{zx_1,zx_2,zx_3,zx_4\}\).  By
Proposition~\ref{prop:simple-objects}, its component
\((X_1)_{zx_1}\) is an irreducible
\(\Phi_{zx_1}\)-projective representation of \(C_G(x_1)\).
Consequently, \(\dim(X_1)_{zx_1}=1\).
For \(v\in V_z\) and \(0\neq w\in W_{x_1}\), the double braiding
 {in \({}_G^G\mathcal{YD}^{\Phi}\)} gives
\(\varphi_1^\Phi(v\otimes w)
=(\operatorname{id}_{V_z}-T_0)v\otimes w\).
Lemma~\ref{lem:adjoint-object-realization} identifies its image with
\((X_1)_{zx_1}\).  Since \(W_{x_1}\) is one-dimensional, it follows
that
\begin{equation}
\label{eq:T-rank-one-T0}
\operatorname{rank}(\operatorname{id}_{V_z}-T_0)=1.
\end{equation}

{
Consider the reflected tuple \(R_1(V,W)=(V^*,X_1)\).
Lemma~\ref{lem:reflection-dimension-invariance} shows that its
Nichols algebra is finite-dimensional.  Since \(m_{12}=1\),
\cite[Lemma~3.8]{reflection3} gives
\((\operatorname{ad}V^*)(X_1)\simeq W\).
Equation~\eqref{eq:first-adjoint-braiding-symmetry} then gives
\((\operatorname{ad}X_1)(V^*)\simeq W\neq0\).
In particular, the reflected tuple is braided-indecomposable.
Its supports are
\(\{z^{-1}\}\) and
\(\{zx_1,zx_2,zx_3,zx_4\}\).
The element \(z^{-1}\) is central.  The elements
\(z^{-1},zx_1,zx_2\) satisfy the defining relations of \(T\).
They generate \(G\), and the four elements \(zx_i\) are distinct and
have the same conjugation relations as the \(x_i\).
}

{
The degree-\(x_4\) component of
\((\operatorname{ad}X_1)(V^*)\) is nonzero and  {arises from the component of multidegree \((zx_4,z^{-1})\)}.  The two applications of
Proposition~\ref{prop:multidegree-insertion} used in the proof of
Lemma~\ref{lem:T-x1-action-minus-one} give successively
\((zx_1,zx_4,z^{-1})\) and
\((zx_3,zx_2,zx_4,z^{-1})\).
 {Hence the map \(\varphi_3^\Phi\) for the pair \((X_1,V^*)\) has nonzero
image, and}
Lemma~\ref{lem:adjoint-object-realization} gives
\((\operatorname{ad}X_1)^3(V^*)\neq0\).
Thus the corresponding reflected Cartan entry satisfies
\(a_{21}^{[R_1(V,W)]}\leq -3\).  The reflected tuple lies in the same finite Cartan
graph, so Lemma~\ref{lem:root-string-adjoint-simplicity}, applied with
\(t=3\), shows that \((\operatorname{ad}X_1)^3(V^*)\) is simple.
Lemma~\ref{lem:T-x1-action-minus-one}, applied to the reflected pair,
therefore shows that \(zx_1\) acts as
\(-\operatorname{id}\) on \((X_1)_{zx_1}\).
}

Suppose first that \(\dim\rho=3\).
The associated twisted group algebra of \(A_4\) is semisimple
of dimension \(12\), and the simple module \(V_z\) contributes
a matrix block of dimension \(9\).
The remaining blocks have total dimension \(3\).
 {Since the dimension of each matrix block is a perfect square,
these must be three one-dimensional blocks.}
In particular, a one-dimensional projective representation exists.
Its multiplication rule expresses the defining \(2\)-cocycle
as a coboundary. Hence the quotient projective representation
on \(V_z\) can be rescaled to an ordinary irreducible
representation of \(A_4\).

The three-dimensional irreducible representation of \(A_4\)
is realized on the sum-zero subspace of its permutation
representation on four letters.
On this subspace, a three-cycle has eigenvalues
\(1,e^{2\pi\mathrm i/3},e^{4\pi\mathrm i/3}\).
Since \(x_1Z(G)\) is a three-cycle and \(T_0\) is a nonzero
scalar multiple of its operator, every eigenspace of \(T_0\)
is one-dimensional.
However, \eqref{eq:T-rank-one-T0} gives a two-dimensional
eigenspace for the eigenvalue \(1\), a contradiction.

Assume now that \(\dim\rho=2\).  Since \(x_1^3\in Z(G)\), projective
multiplication shows that  {\(\rho(x_1)^3\), and hence \(T_0^3\), is a
nonzero scalar multiple of the identity.}  Thus \(T_0\) is diagonalizable.
Equation~\eqref{eq:T-rank-one-T0} gives a one-dimensional eigenspace
for the eigenvalue \(1\).   {Therefore \(T_0^3\) is the identity operator.}  We may choose a basis \(e_0,e_r\) of \(V_z\) such that
\[
T_0e_0=e_0,\qquad
T_0e_r=re_r,\qquad
r^2+r+1=0.
\]
Choose \(0\neq w\in W_{x_1}\) and let
\(\xi=(1-r)e_r\otimes w\).  Then
\(\xi=\varphi_1^\Phi(e_r\otimes w)\), so \(\xi\) spans
\((X_1)_{zx_1}\).
We compute \(\varphi_2^\Phi\) directly in the original category.
 {For \(v\in V_z\), we have} 
\(z\rhd\xi=\Phi^z(z,x_1)\rho(z)\sigma(z)\xi\), while 
$$(zx_1)\rhd v
=\Phi_z(z,x_1)^{-1}\rho(z)\rho(x_1)(v).$$
Since \(z\) is central,
\(\Phi^z(z,x_1)=\Phi_z(z,x_1)=\Phi(z,x_1,z)\).
The two factors therefore cancel, and
\(c_{X_1,V}c_{V,X_1}(v\otimes\xi)
=\rho(z)^2T_0(v)\otimes\xi\).

Moreover,
\(c_{1,2}^\Phi
=a_{V,V,W}(c_{V,V}\otimes\operatorname{id}_W)a_{V,V,W}^{-1}\).
The two associativity constraints contribute
\(\Phi(z,z,x_1)\) and its inverse, while \(c_{V,V}\) contributes
\(\rho(z)\).  Thus the last term in the recursive formula contributes
\(\rho(z)(\operatorname{id}_{V_z}-T_0)e_r
\otimes((\operatorname{id}_{V_z}-T_0)v\otimes w)\).
Together with the identity and double-braiding terms, this gives
\[
\begin{aligned}
\varphi_2^\Phi(e_0\otimes\xi)
&=(1-\rho(z)^2)e_0\otimes\xi,\\
\varphi_2^\Phi(e_r\otimes\xi)
&=\bigl(1-\rho(z)^2r+\rho(z)(1-r)\bigr)e_r\otimes\xi.
\end{aligned}
\]
By Lemma~\ref{lem:adjoint-object-realization}, \(X_2=0\) forces both
expressions to vanish.  The first gives \(\rho(z)^2=1\).  If
\(\rho(z)=1\), the coefficient in the second expression is
\(2(1-r)\neq0\).  Hence \(\rho(z)=-1\).

Finally, note that
\(\rho(z)\rho(x_1)=\Phi_z(z,x_1)\rho(zx_1)\) and
\(\sigma(z)\sigma(x_1)
=\Phi_{x_1}(z,x_1)\sigma(zx_1)\).
The \(3\)-cocycle identity applied to
\((z,x_1,z,x_1)\) gives
\(\Phi^{zx_1}(z,x_1)
=\Phi_z(z,x_1)\Phi_{x_1}(z,x_1)\).
 {Using these identities and \(T_0e_r=re_r\), we obtain}
\[
(zx_1)\rhd\xi
=
\frac{\Phi^{zx_1}(z,x_1)}
     {\Phi_z(z,x_1)\Phi_{x_1}(z,x_1)}
\,\rho(z)\sigma(x_1)r\,\xi
=r\xi,
\]
because \(\rho(z)=\sigma(x_1)=-1\).
Comparison with \((zx_1)\rhd\xi=-\xi\) gives \(r=-1\), contrary to
\(r^2+r+1=0\).

Thus neither \(\dim\rho=2\) nor \(\dim\rho=3\) is possible.
Therefore \(\dim\rho=1\).
\end{proof}

\subsubsection{The first and second adjoint objects}

\begin{lemma}
\label{lem:T-central-support-adjoints}
Assume \(\dim\rho=1\).  Then
\(\rho(x_1)\sigma(z)\neq1\), and
\(X_1=(\operatorname{ad}V)(W)\) is simple of dimension four.
Moreover,
\begin{equation}
\label{eq:T-X2-condition}
X_2=0
\quad\Longleftrightarrow\quad
\bigl(1+\rho(z)\bigr)
\bigl(1-\rho(x_1)\rho(z)\sigma(z)\bigr)=0.
\end{equation}
\end{lemma}

\begin{proof}
By Lemma~\ref{lem:adjoint-object-realization}, we identify
\(X_1\) and \(X_2\) with the images of the recursive maps.
Choose \(0\neq v\in V_z\) and \(0\neq w_1\in W_{x_1}\). Then
\[
c_{W,V}c_{V,W}(v\otimes w_1)
=\rho(x_1)\sigma(z)v\otimes w_1.
\]
The double braiding commutes with the \(G\)-action, which
permutes the four homogeneous lines of \(V\otimes W\)
transitively. Hence it acts by the same scalar on all of
\(V\otimes W\).
Braided indecomposability gives \(\rho(x_1)\sigma(z)\neq1\), so
\(\varphi_1^\Phi=(1-\rho(x_1)\sigma(z))
\operatorname{id}_{V\otimes W}\) and \(X_1=V\otimes W\).
Its homogeneous components are one-dimensional, with degrees
\(zx_1,zx_2,zx_3,zx_4\) forming one conjugacy class.
Thus \(X_1\) is simple of dimension four.

Put \(e=v\otimes(v\otimes w_1)\).
The two associator factors in \(c_{1,2}^\Phi\) cancel, giving
\(c_{1,2}^\Phi(e)=\rho(z)e\).
Since \(z\) is central,
\(\Phi^z(z,x_1)=\Phi_z(z,x_1)=\Phi(z,x_1,z)\).
Therefore 
\[
\begin{aligned}
c_{V\otimes W,V}c_{V,V\otimes W}(e)
&=\Phi^z(z,x_1)\rho(z)\sigma(z)\rho(zx_1)e\\
&=\frac{\Phi^z(z,x_1)}{\Phi_z(z,x_1)}
  \rho(z)^2\rho(x_1)\sigma(z)e\\
&=\rho(z)^2\rho(x_1)\sigma(z)e.
\end{aligned}
\]
Substituting into the defining recursion yields
\[
\begin{aligned}
\varphi_2^\Phi(e)
&=\bigl(1-\rho(z)^2\rho(x_1)\sigma(z)
  +\rho(z)(1-\rho(x_1)\sigma(z))\bigr)e\\
&=(1+\rho(z))
  (1-\rho(x_1)\rho(z)\sigma(z))e.
\end{aligned}
\]
The map \(\varphi_2^\Phi\) commutes with the \(G\)-action, which
permutes the four homogeneous lines of \(V\otimes X_1\)
transitively. Hence \(X_2=0\) if and only if
\(\varphi_2^\Phi(e)=0\), proving \eqref{eq:T-X2-condition}.
\end{proof}

\begin{lemma}
\label{lem:T-Y2-cocycle-reduction}
Assume \(\dim\rho=1\) and \(\sigma(x_1)=-1\).  Then the
homogeneous component \((Y_2)_{x_2x_3z}\) is one-dimensional if and
only if
\(\bigl(\rho(x_1)\sigma(z)\bigr)^2
-\rho(x_1)\sigma(z)+1=0\) and \(\varepsilon_\Phi(W)=1\).
\end{lemma}

\begin{proof}
Lemma~\ref{lem:adjoint-object-realization} identifies \(Y_1\) and
\(Y_2\) with the images of the corresponding recursive maps.
Choose \(0\neq w_i\in W_{x_i}\) and \(0\neq v\in V_z\).
On \(W_{x_1}\otimes V_z\), the double braiding sends
\(w_1\otimes v\) to
\((z\rhd w_1)\otimes(x_1\rhd v)
=\rho(x_1)\sigma(z)w_1\otimes v\).
 {It commutes with the \(G\)-action}, so it acts by \(\rho(x_1)\sigma(z)\) on all four
homogeneous components of \(W\otimes V\).  Braided indecomposability
gives \(\rho(x_1)\sigma(z)\neq1\).  Hence
\(\varphi_1^{Y,\Phi}
=\bigl(1-\rho(x_1)\sigma(z)\bigr)
\operatorname{id}_{W\otimes V}\), and we identify \(Y_1\) with
\(W\otimes V\).
Define \(\beta_{ij}\in\mathbb C^\times\) by
\(c_{W,W}(w_i\otimes w_j)
=\beta_{ij}w_{i\brhd j}\otimes w_i\), and let
\(E_{ij}=\Phi(x_i,x_j,z)^{-1}w_i\otimes(w_j\otimes v)\).

Using the definition of \(c_{1,2}^{\Phi}\) and the associativity
constraint, we obtain
\[
\begin{aligned}
c_{1,2}^{\Phi}(E_{ij})
&=
\Phi(x_i,x_j,z)^{-1}
a_{W,W,V}
(c_{W,W}\otimes\operatorname{id}_V)
a_{W,W,V}^{-1}
\bigl(w_i\otimes(w_j\otimes v)\bigr)\\
&=
\Phi(x_i,x_j,z)^{-1}
\Phi(x_i,x_j,z)\,
\beta_{ij}\,
\Phi(x_{i\brhd j},x_i,z)^{-1}
w_{i\brhd j}\otimes(w_i\otimes v)\\
&=
\beta_{ij}E_{i\brhd j,i}.
\end{aligned}
\]
For the middle double braiding, the \(W\)-\(V\) double braiding
contributes \(\rho(x_1)\sigma(z)\), and the two \(W\)-\(W\)
braidings contribute
\(\beta_{ij}\beta_{i\brhd j,i}\).
Since \(z\) is central, the definitions of
\(\Phi^{x_i}\) and \(\Phi_{x_i}\) give
\[
\frac{\Phi^{x_i}(x_j,z)}
     {\Phi_{x_i}(x_{i\brhd j},z)}
\Phi(x_i,x_j,z)^{-1}
=
\Phi(x_{(i\brhd j)\brhd i},x_{i\brhd j},z)^{-1}.
\]
Hence
\[
\begin{aligned}
c_{W\otimes V,W}c_{W,W\otimes V}(E_{ij})=
\rho(x_1)\sigma(z)\,
\beta_{ij}\beta_{i\brhd j,i}\,
E_{(i\brhd j)\brhd i,i\brhd j}.
\end{aligned}
\]
The recursive formula now gives
\begin{equation}
\label{eq:T-Y2-general}
\begin{aligned}
\varphi_2^{Y,\Phi}(E_{ij})
={}E_{ij}
-\rho(x_1)\sigma(z)\,\beta_{ij}\beta_{i\brhd j,\,i}\,
E_{(i\brhd j)\brhd i,\,i\brhd j}
+\bigl(1-\rho(x_1)\sigma(z)\bigr)
\beta_{ij}E_{i\brhd j,\,i}.
\end{aligned}
\end{equation}
 {Since \(c_{W,W}\) commutes with the \(G\)-action,
which transitively permutes the lines \(W_{x_i}\otimes W_{x_i}\),}
\(\beta_{ii}=\beta_{11}=\sigma(x_1)=-1\) for every \(i\).
Thus \eqref{eq:T-Y2-general} gives
\(\varphi_2^{Y,\Phi}(E_{ii})=0\).
The relations of \(T\) give
\(x_2x_3z=x_3x_4z=x_4x_2z\), and  {the tensors \(E_{ij}\) of this degree with \(i\ne j\) are
\(E_{23},E_{34},E_{42}\).}
 {With respect to the ordered basis \(E_{23},E_{34},E_{42}\), the
coordinate vectors of their images are the columns of}
\begin{equation*}
M=
\begin{pmatrix}
1&\bigl(1-\rho(x_1)\sigma(z)\bigr)\beta_{34}
&-\rho(x_1)\sigma(z)\beta_{42}\beta_{34}\\
-\rho(x_1)\sigma(z)\beta_{23}\beta_{42}
&1&\bigl(1-\rho(x_1)\sigma(z)\bigr)\beta_{42}\\
\bigl(1-\rho(x_1)\sigma(z)\bigr)\beta_{23}
&-\rho(x_1)\sigma(z)\beta_{34}\beta_{23}&1
\end{pmatrix}.
\end{equation*} {The span of \(E_{23},E_{34},E_{42}\) is preserved by
\eqref{eq:T-Y2-general}, and every \(E_{ii}\) of this degree maps to zero.}
Lemma~\ref{lem:adjoint-object-realization} therefore gives
\(\dim(Y_2)_{x_2x_3z}=\operatorname{rank}M\).
Since the diagonal entries of \(M\) are one, this component is
one-dimensional if and only if \(\operatorname{rank}M=1\).

Suppose that \(\operatorname{rank}M=1\).  The following two minors
must vanish:
\begin{align}
\det M_{\{1,3\},\{1,2\}}
&=
-\beta_{23}\beta_{34}
\Bigl(\rho(x_1)\sigma(z)
+\bigl(1-\rho(x_1)\sigma(z)\bigr)^2\Bigr),
\label{eq:T-Y2-first-minor}\\
\det M_{\{1,2\},\{1,2\}}
&=
1+\rho(x_1)\sigma(z)
\bigl(1-\rho(x_1)\sigma(z)\bigr)
\beta_{23}\beta_{34}\beta_{42}\notag\\
&=
1-\rho(x_1)\sigma(z)
\bigl(1-\rho(x_1)\sigma(z)\bigr)\varepsilon_\Phi(W).
\label{eq:T-Y2-second-minor}
\end{align}
Indeed, three successive braidings on
\(W_{x_2}\otimes W_{x_3}\) contribute
\(\beta_{23}\beta_{42}\beta_{34}\).  Since \(\sigma(x_1)=-1\),
\eqref{eq:T-epsilon-definition} gives
\(\beta_{23}\beta_{34}\beta_{42}=-\varepsilon_\Phi(W)\).
As all \(\beta_{ij}\) are nonzero,
\eqref{eq:T-Y2-first-minor} gives
\(\bigl(\rho(x_1)\sigma(z)\bigr)^2
-\rho(x_1)\sigma(z)+1=0\), and hence
\(\rho(x_1)\sigma(z)
\bigl(1-\rho(x_1)\sigma(z)\bigr)=1\).
Equation~\eqref{eq:T-Y2-second-minor} then gives
\(\varepsilon_\Phi(W)=1\).

Conversely, assume
\(\bigl(\rho(x_1)\sigma(z)\bigr)^2
-\rho(x_1)\sigma(z)+1=0\) and \(\varepsilon_\Phi(W)=1\).  Then
the quadratic equality and \eqref{eq:T-epsilon-definition} give
\(\beta_{23}\beta_{34}\beta_{42}=-1\).  Substitution in \(M\) shows
that its second column is
\(\bigl(1-\rho(x_1)\sigma(z)\bigr)\beta_{34}\) times its first.
Its third column is
\(-\rho(x_1)\sigma(z)\beta_{42}\beta_{34}\) times its first.
Since the first column is nonzero, \(\operatorname{rank}M=1\).
Thus \((Y_2)_{x_2x_3z}\) is one-dimensional.
\end{proof}

\begin{lemma}
\label{lem:T-second-adjoint}
Assume \(\dim\rho=1\) and \(\sigma(x_1)=-1\).  Then
\[
\supp Y_2
=
\{(x_1x_2x_3z)x_i^{-1}\mid 1\leq i\leq4\}.
\]
The four corresponding homogeneous components are one-dimensional
if and only if
\begin{equation}
\label{eq:T-Y2-conditions}
\bigl(\rho(x_1)\sigma(z)\bigr)^2
-\rho(x_1)\sigma(z)+1=0,
\qquad \varepsilon_\Phi(W)=1.
\end{equation}
In this case, \(\dim Y_2=4\) and  $Y_2$ is simple.
\end{lemma}

\begin{proof}
Retain the notation \(E_{ij}\) and \(\beta_{ij}\) from the proof
of Lemma~\ref{lem:T-Y2-cocycle-reduction}.
Since \(Y_1=W\otimes V\), Lemma~\ref{lem:adjoint-object-realization}
identifies \(Y_2\) with the image of \(\varphi_2^{Y,\Phi}\).
By \(\beta_{ii}=-1\), equation~\eqref{eq:T-Y2-general} gives
\(\varphi_2^{Y,\Phi}(E_{ii})=0\), while
\[
\begin{aligned}
\varphi_2^{Y,\Phi}(E_{23})
={}E_{23}
-\rho(x_1)\sigma(z)\beta_{23}\beta_{42}E_{34}
+\bigl(1-\rho(x_1)\sigma(z)\bigr)\beta_{23}E_{42}
\neq0.
\end{aligned}
\]
The nonvanishing follows from the linear independence of
\(E_{23},E_{34},E_{42}\).

The relations of \(T\) give
\[
\begin{aligned}
x_2x_3=x_3x_4=x_4x_2
  &=(x_1x_2x_3)x_1^{-1},\\
x_1x_4=x_4x_3=x_3x_1
  &=(x_1x_2x_3)x_2^{-1},\\
x_1x_2=x_2x_4=x_4x_1
  &=(x_1x_2x_3)x_3^{-1},\\
x_1x_3=x_3x_2=x_2x_1
  &=(x_1x_2x_3)x_4^{-1}.
\end{aligned}
\]
Thus \(\supp Y_2\) is contained in the four degrees
\((x_1x_2x_3z)x_i^{-1}\).
By Lemma~\ref{lem:T-group-data}, \(x_1x_2x_3z\) is central,
so these degrees form one conjugacy class.
Since \((Y_2)_{x_2x_3z}\neq0\), all four degrees occur,
and their homogeneous components have equal dimension.

Lemma~\ref{lem:T-Y2-cocycle-reduction} now gives
\eqref{eq:T-Y2-conditions}.
Under these conditions, the four homogeneous components are
one-dimensional, so \(\dim Y_2=4\).
Their degrees form one conjugacy class, hence every nonzero
subobject contains all four components. Thus \(Y_2\) is simple.
\end{proof}

\subsubsection{Higher adjoint objects}

{
We prove that \(Y_3\) is one-dimensional with central support and
that \(Y_4=0\). Under the last equality in
\eqref{eq:T-classification-G2}, Lemma~\ref{lem:T-compatible-normalization}
reduces these assertions to the   calculations in
\cite{rank2-3}.
}

\begin{lemma}
\label{lem:T-tetrahedral-braiding}
Assume that \(\sigma(x_1)=-1\), \(\varepsilon_\Phi(W)=1\),
and the last equality in \eqref{eq:T-classification-G2} holds.
After the replacement in
Lemma~\ref{lem:T-compatible-normalization}, denote the ordinary
object again by \(W\).  It has a homogeneous basis \(w_i\) of
degree \(x_i\), \(1\leq i\leq4\), such that
\begin{equation}
\label{eq:T-constant-minus-one-braiding}
c_{W,W}(w_i\otimes w_j)
=-w_{i\brhd j}\otimes w_i.
\end{equation}
\end{lemma}

\begin{proof}
By Lemma~\ref{lem:T-pullback-cocycle}, one has
\(\pi^*\Phi=\delta J\).
Write \(\widetilde\sigma\) for the inducing character after the
replacement in Lemma~\ref{lem:T-compatible-normalization}.
Lemma~\ref{lem:T-cocycle-two-branches} gives
\(\widetilde\sigma((x_1x_2^{-1})^2)=1\).
Using
\(T=(Q_8\rtimes\langle x_1\rangle)\times\langle z\rangle\),
define a character of \(T\) by
\[
\chi(ux_1^mz^n)=(-1)^m\widetilde\sigma(z)^n
\qquad u\in Q_8,\ m,n\in\mathbb Z.
\]
The characters \(\chi|_{C_T(x_1)}\) and \(\widetilde\sigma\)
agree on the generators \((x_1x_2^{-1})^2,x_1,z\), and hence
\(\chi|_{C_T(x_1)}=\widetilde\sigma\).

Since \(T=Q_8C_T(x_1)\), choose \(t_i\in Q_8\) with
\(t_ix_1t_i^{-1}=x_i\), and put \(w_i=t_i\otimes1\).
For \(gt_i=t_jc\), with \(c\in C_T(x_1)\), one has
\[
g\rhd w_i
=\widetilde\sigma(c)w_j
=\chi(t_j^{-1}gt_i)w_j
=\chi(g)w_j.
\]
Since \(\chi(x_i)=-1\), this proves
\eqref{eq:T-constant-minus-one-braiding}.
\end{proof}

\begin{lemma}
\label{lem:T-third-fourth-adjoints}
Assume \(\dim\rho=1\), \(\sigma(x_1)=-1\),
\(\bigl(\rho(x_1)\sigma(z)\bigr)^2
-\rho(x_1)\sigma(z)+1=0\), and \(\varepsilon_\Phi(W)=1\).
Assume also the last equality in \eqref{eq:T-classification-G2}.
Then
\[
\dim Y_3=1,\qquad
\supp Y_3=\{x_1x_2x_3z\}\subseteq Z(G),
\qquad Y_4=0.
\]
\end{lemma}

\begin{proof}
{
Apply Lemma~\ref{lem:T-compatible-normalization}, and write
\(\widetilde V,\widetilde W\) for the ordinary objects over \(T\),
with inducing characters \(\widetilde\rho,\widetilde\sigma\).
The values of \(\sigma(x_1)\), \(\rho(x_1)\sigma(z)\), and
\(\varepsilon_\Phi(W)\) are preserved. The  table in
\cite[Remark~2.5]{rank2-3} gives
\[
\left.c_{\widetilde W,\widetilde W}^{3}\right|
 _{\widetilde W_{x_2}\otimes\widetilde W_{x_3}}
=-\widetilde\sigma(x_2x_3)\operatorname{id}.
\]
Comparing with \eqref{eq:T-epsilon-definition}, we obtain
\(\widetilde\sigma(x_2x_3)=\varepsilon_1(\widetilde W)=1\).
Thus \cite[Lemmas~4.14 and~4.16]{rank2-3} show that
\((\operatorname{ad}\widetilde W)^3(\widetilde V)\) is
one-dimensional of degree \(x_1x_2x_3z\), and that
\((\operatorname{ad}\widetilde W)^4(\widetilde V)=0\).
The degree is central by Lemma~\ref{lem:T-group-data}.
The replacement preserves the dimensions of adjoint objects,
and their degrees map to the original degrees under \(\pi\).
This proves the assertions.
}
\end{proof}

\subsection{Parameter conditions for type
\texorpdfstring{$G_2$}{G2}}
\label{subsec:T-G2-parameters}

 {When \(\mathcal B(V\oplus W)\) is finite-dimensional
and the Cartan matrix is of type \(G_2\), we prove the remaining
parameter conditions in \eqref{eq:T-classification-G2}.
We first exclude the nontrivial class left by
Lemma~\ref{lem:T-cocycle-two-branches} under the first five
conditions in \eqref{eq:T-classification-G2}.}
Retain the  elements \(i,j,k\) from
Lemma~\ref{lem:T-pullback-cocycle}, and let \(A\) cyclically
permute them. Write
\(u=(-1)^s i^{a(u)}j^{b(u)}\), with
\(a(u),b(u),s\in\mathbb F_2\), and put
\[
E(u,v,w)=\sum_{r=0}^2 a(A^ru)a(A^rv)b(A^rw)\in\mathbb F_2.
\]
Let \(F:Q_8^3\to\mathbb Z_8\) be the cocycle used in
\eqref{eq:T-cocycle-representative}.
Lemma~\ref{lem:T-fixed-quaternion-cocycle} gives
\(E(u,v,w)=F(u,v,w)\bmod2\) for all \(u,v,w\in Q_8\).
Reducing the cocycle identity for \(F\) modulo two shows
that \(E\) is a normalized \(A\)-invariant additive
three-cocycle.  

Let
\[
H=Q_8\rtimes\langle t\mid t^3=1\rangle,
\qquad
G_*=H\times\langle z\mid z^6=1\rangle,
\]
where \(t\) acts by \(A\), so that
\(H\simeq\mathrm{SL}_2(3)\). The following is a normalized $3$-cocycle
\begin{equation}
\label{eq:T-order-two-finite-cocycle}
\Phi_*(ut^rz^a,vt^sz^b,wt^lz^c)
=(-1)^{E(u,A^rv,A^{r+s}w)}.
\end{equation}
Using
\(T=(Q_8\rtimes\langle x_1\rangle)\times\langle z\rangle\),
define
\[
T\longrightarrow G_*,
\qquad
ux_1^mz^a\longmapsto (-1)^m u\,t^mz^a.
\]
This is a homomorphism: it restricts to the identity on
\(Q_8\), conjugation by \(-t\) induces \(A\) on \(Q_8\),
and \(z\) is central in \(G_*\).
Since \(A\) fixes \(-1\) and \(E\) is unchanged by multiplying
any argument by \(-1\), one has
\[
\begin{aligned}
&\Phi_*\bigl((-1)^m u\,t^mz^a,
             (-1)^n v\,t^nz^b,
             (-1)^r w\,t^rz^c\bigr)\\
&\quad=
(-1)^{E((-1)^m u,\,(-1)^n A^mv,\,(-1)^r A^{m+n}w)}\\
&\quad=
(-1)^{E(u,A^mv,A^{m+n}w)}\\
&\quad=
\Omega(ux_1^mz^a,vx_1^nz^b,wx_1^rz^c)^4.
\end{aligned}
\]
The last equality follows from
\eqref{eq:T-cocycle-representative} and \(F\equiv E\pmod2\).
Thus this homomorphism pulls \(\Phi_*\) back to \(\Omega^4\).
We also write \(x_i\) for the images of the support elements
in \(G_*\).

For each \(p\in\mathbb{C}^\times\) satisfying \(p^2-p+1=0\),
define
\begin{equation}
\label{eq:T-order-two-finite-pair}
\begin{aligned}
&V_0=M(z,\rho_0),\qquad
  \rho_0|_H=1,\qquad \rho_0(z)=p^{-1},\\
&W_0=M(-t,\sigma_0),\qquad
  \sigma_0(-t)=-1,\qquad \sigma_0(z)=p.
\end{aligned}
\end{equation}
Here
\(C_{G_*}(-t)=\langle-t\rangle\times\langle z\rangle
\simeq C_6\times C_6\).
The multipliers at \(z\) and \(-t\) are trivial on
\(G_*\) and \(C_{G_*}(-t)\), respectively.
Since \(p^3=-1\) and \(p^6=1\), the displayed values define
ordinary characters and hence simple objects \(V_0,W_0\).

\begin{lemma}
\label{lem:T-order-two-finite-model}
Assume that \(\dim\rho=1\), \(\sigma(x_1)=-1\),
\(p^2-p+1=0\), \(p\rho(z)=1\), and
\(\varepsilon_\Phi(W)=1\), where \(p=\rho(x_1)\sigma(z)\).
If the product in the last equality of
\eqref{eq:T-classification-G2} is \(-1\), then
\(\dim\mathcal B(V\oplus W)
=\dim\mathcal B(V_0\oplus W_0)\).
\end{lemma}

\begin{proof}
By Lemmas~\ref{lem:T-cocycle-two-branches}
and~\ref{lem:T-compatible-normalization}, we may replace
\((V,W)\) by a pair in \({}_T^T\mathcal{YD}^{\Omega^4}\)
without changing \(\dim\mathcal B(V\oplus W)\).
Retain \(\rho,\sigma\) for the ordinary inducing characters.
Since conjugation by \(x_1\) cycles \(i,j,k\) and \(ij=k\),
\[
\rho(i)=\rho(j)=\rho(k)=\rho(i)\rho(j).
\]
These values are nonzero, so \(\rho(i)=\rho(j)=1\), and hence
\(\rho|_{Q_8}=1\).
Moreover, Lemma~\ref{lem:T-cocycle-two-branches} gives
\(\sigma((x_1x_2^{-1})^2)=-1\).
Apply the  {braided monoidal equivalence}
defined by the normalized two-cocycle
\[
J(ux_1^mz^n,vx_1^rz^s)=\rho(x_1)^{ms}
\qquad u,v\in Q_8;\ m,n,r,s\in\mathbb Z.
\]
By \eqref{eq:T-cochain-tensor-structure}, this replaces
\(\rho(x_1)\) by \(1\) and \(\sigma(z)\) by \(p\),
while leaving \(\Omega^4\), \(\rho|_{Q_8}\),
\(\rho(z)=p^{-1}\), \(\sigma(x_1)=-1\), and
\(\sigma((x_1x_2^{-1})^2)=-1\) unchanged.

Apply Lemma~\ref{lem:T-compatible-normalization} to
\((V_0,W_0)\), using the homomorphism \(T\to G_*\).
Since the pullback of \(\Phi_*\) is \(\Omega^4\), the
resulting pair belongs to \({}_T^T\mathcal{YD}^{\Omega^4}\).
Its inducing characters have precisely the values obtained
above.
Thus the two pairs are isomorphic.
 {The replacement in Lemma~\ref{lem:T-compatible-normalization} and
the two braided monoidal equivalences preserve the dimensions of the
corresponding Nichols algebras}, giving
\[
\dim\mathcal B(V\oplus W)
=\dim\mathcal B(V_0\oplus W_0).
\]
\end{proof}

\begin{lemma}
\label{lem:T-order-two-Cartan-periodicity}
Assume that the pair \((V_0,W_0)\) in
\eqref{eq:T-order-two-finite-pair} admits all reflections.
Let \((V_r,W_r)\), \(r\geq0\), be obtained by successively
applying \(R_2,R_1,R_2,R_1,\ldots\).
Then
\(A^{[(V_{r+4},W_{r+4})]}
=A^{[(V_r,W_r)]}\) for every \(r\geq0\).
\end{lemma}

\begin{proof}
Define homomorphisms \(\pi_0,\pi_4:T\to G_*\) by
\[
\begin{aligned}
\pi_0(ux_1^mz^n)&=(-1)^m u\,t^mz^n,\\
\pi_4(ux_1^mz^n)&=(-1)^m u\,t^mz^{n+3m}.
\end{aligned}
\]
Since \(\Phi_*\) is independent of the \(z\)-coordinate,
\(\pi_0^*\Phi_*=\pi_4^*\Phi_*=\Omega^4\).
Using these maps, respectively, assign both pairs
\((V_0,W_0)\) and \((V_4,W_4)\) the \(T\)-degrees
\(z,x_1,\ldots,x_4\), and let \(T\) act through
\(\pi_0,\pi_4\).
Both pairs then belong to
\({}_T^T\mathcal{YD}^{\Omega^4}\).

By Lemma~\ref{lem:T-order-two-reflection-calculation}
and \(p^3=-1\), their inducing characters satisfy
\[
\rho|_{Q_8}=1,\qquad
\rho(z)=p^{-1},\qquad
\sigma(x_1)=-1,\qquad
\sigma((x_1x_2^{-1})^2)=-1,
\]
while the remaining values are
\[
\begin{array}{c|cc}
 & (V_0,W_0) & (V_4,W_4)\\ \hline
\rho(x_1)&1&-1\\
\sigma(z)&p&-p
\end{array}.
\]
Apply to the fourth pair the braided monoidal
autoequivalence defined by the normalized two-cocycle
\[
J(ux_1^mz^n,vx_1^rz^s)=(-1)^{ms}.
\]
By \eqref{eq:T-cochain-tensor-structure}, this multiplies
\(\rho(x_1)\) and \(\sigma(z)\) by \(-1\), leaving the
other displayed values unchanged.
Thus the transformed fourth pair is isomorphic to the
initial pair in \({}_T^T\mathcal{YD}^{\Omega^4}\).

Both changes of grading and the braided monoidal equivalence
preserve duals and braided adjoint objects.
Applying the same alternating reflections to the two
isomorphic pairs therefore gives
\[
A^{[(V_{r+4},W_{r+4})]}
=A^{[(V_r,W_r)]}
\qquad(r\geq0).
\]
\end{proof}

\begin{lemma}
\label{lem:T-nontrivial-cocycle-infinite}
Assume that \(\dim\rho=1\), \(\sigma(x_1)=-1\),
\(p^2-p+1=0\), \(p\rho(z)=1\), and
\(\varepsilon_\Phi(W)=1\), where \(p=\rho(x_1)\sigma(z)\).
If the product in the last equality of
\eqref{eq:T-classification-G2} is \(-1\), then
\(\mathcal B(V\oplus W)\) is infinite-dimensional.
\end{lemma}

\begin{proof}
Suppose, to the contrary, that
\(\mathcal B(V\oplus W)\) is finite-dimensional.
By Lemma~\ref{lem:T-order-two-finite-model},
\(\mathcal B(V_0\oplus W_0)\) is finite-dimensional.
Theorem~\ref{thm:Nichols-finiteness-criterion} therefore
implies that \((V_0,W_0)\) admits all reflections and has
a finite Cartan graph.

Starting from \((V_0,W_0)\), denote the successive pairs under
\(R_2,R_1,R_2,R_1,\ldots\) by \((V_r,W_r)\).
By Lemma~\ref{lem:T-order-two-Cartan-periodicity},
their Cartan matrices repeat after every four reflections.
Using the Cartan entries computed in
Lemma~\ref{lem:T-order-two-reflection-calculation},
the composite of the first four reflections transports
the root lattice at step four to that at step zero by
\begin{equation*}
P=
\begin{pmatrix}1&0\\3&-1\end{pmatrix}
\begin{pmatrix}-1&2\\0&1\end{pmatrix}
\begin{pmatrix}1&0\\3&-1\end{pmatrix}
\begin{pmatrix}-1&1\\0&1\end{pmatrix}
=\begin{pmatrix}-5&3\\-12&7\end{pmatrix},
\end{equation*}
where coordinates are written as columns.

The four-step periodicity implies that \(P^n\) transports
roots from step \(4n\) to the initial object.
Since \((P-I)^2=0\), one has \(P^n=I+n(P-I)\), and hence
\[
P^n\alpha_2=3n\alpha_1+(6n+1)\alpha_2
\qquad(n\geq0).
\]
Each vector is the image of the simple root \(\alpha_2\)
at step \(4n\), so each is a real root at the initial
object. These roots are pairwise distinct, contradicting
the finiteness of its Cartan graph.
\end{proof}

\begin{prop}
\label{prop:T-G2-adjoint-objects}
Assume \eqref{eq:T-classification-G2}.  Then \(X_1\) and \(Y_3\)
are nonzero and simple, \(\dim X_1=4\), \(\dim Y_3=1\), and
\(X_2=Y_4=0\).
Consequently, the Cartan matrix of \((V,W)\) is
\(\left(\begin{smallmatrix}2&-1\\-3&2\end{smallmatrix}\right)\).
\end{prop}

\begin{proof}
The quadratic condition in \eqref{eq:T-classification-G2} gives
\(\rho(x_1)\sigma(z)\neq1\), while the next condition gives
\(\rho(x_1)\rho(z)\sigma(z)=1\).
Lemma~\ref{lem:T-central-support-adjoints} therefore shows that
\(X_1\) is simple of dimension four and that \(X_2=0\).
Lemma~\ref{lem:T-third-fourth-adjoints} gives
\(\dim Y_3=1\) and \(Y_4=0\).  Thus \(Y_3\) is nonzero and simple.
Therefore \(m_{12}=1\) and \(m_{21}=3\), which gives the asserted
Cartan matrix.
\end{proof}

\begin{prop}
\label{prop:T-G2-parameters-at-finite-Cartan-matrix}
Assume that \(\mathcal B(V\oplus W)\) is finite-dimensional and that
\((V,W)\) has Cartan matrix
\(\left(\begin{smallmatrix}2&-1\\-3&2\end{smallmatrix}\right)\).
Then all conditions in \eqref{eq:T-classification-G2} hold.
\end{prop}

\begin{proof}
Propositions~\ref{prop:T-sigma-one-dimensional} and
\ref{prop:T-rho-one-dimensional} give
\(\dim\sigma=\dim\rho=1\).
By Theorem~\ref{thm:Nichols-finiteness-criterion}, the pair admits all
reflections and \(\mathcal G(V,W)\) is finite. The Cartan matrix gives
\(X_1,Y_2,Y_3\neq0\), and
Lemma~\ref{lem:root-string-adjoint-simplicity}, applied in the two
directions with \(t=1\) and \(t=2,3\), shows that these objects are
simple.
Therefore Lemma~\ref{lem:T-x1-action-minus-one} gives
\(\sigma(x_1)=-1\), while
Lemma~\ref{lem:T-central-support-adjoints} gives
\(\rho(x_1)\sigma(z)\neq1\).

By Lemma~\ref{lem:T-second-adjoint},
\[
\supp Y_2
=\{(x_1x_2x_3z)x_i^{-1}\mid1\leq i\leq4\}.
\]
Since \(x_2x_3z=(x_1x_2x_3z)x_1^{-1}\) and
\(x_1x_2x_3z\) is central, one has
\(C_G(x_2x_3z)=C_G(x_1)=\langle x_1,x_2x_3,z\rangle\).
{
Since \(x_2x_3z\in C_G(x_1)\),
Proposition~\ref{prop:T-sigma-one-dimensional} gives
\(\Phi_{x_2x_3z}(a,b)=\Phi_{x_2x_3z}(b,a)\)
for \(a,b\in C_G(x_1)\).
The corresponding twisted group algebra is therefore commutative.
Since \(Y_2\) is simple, Proposition~\ref{prop:simple-objects}
gives \(\dim(Y_2)_{x_2x_3z}=1\).
}
The \(G\)-action identifies the four homogeneous components, so
Lemma~\ref{lem:T-second-adjoint} gives
\eqref{eq:T-Y2-conditions}.

It remains to prove \(\rho(x_1)\rho(z)\sigma(z)=1\).
{
The proof of Proposition~\ref{prop:T-rho-one-dimensional}
shows that \(zx_1\) acts as \(-\operatorname{id}\) on
\((X_1)_{zx_1}\).
}

Choose \(0\neq v\in V_z\), \(0\neq w\in W_{x_1}\), and let
\(\xi=(1-\rho(x_1)\sigma(z))v\otimes w\).
This vector spans \((X_1)_{zx_1}\), and
\[
\begin{aligned}
(zx_1)\rhd\xi
&=
\frac{\Phi^{zx_1}(z,x_1)}
     {\Phi_z(z,x_1)\Phi_{x_1}(z,x_1)}
\,\sigma(x_1)\rho(x_1)\rho(z)\sigma(z)\,\xi\\
&=-\rho(x_1)\rho(z)\sigma(z)\,\xi.
\end{aligned}
\]
The cocycle identity at \((z,x_1,z,x_1)\) identifies
\(\Phi^{zx_1}(z,x_1)\) with
\(\Phi_z(z,x_1)\Phi_{x_1}(z,x_1)\).  Together with
\(\sigma(x_1)=-1\), this gives the second equality.
Comparison with \((zx_1)\rhd\xi=-\xi\) yields
\(\rho(x_1)\rho(z)\sigma(z)=1\).
Together with \eqref{eq:T-Y2-conditions}, this proves the first five
conditions in \eqref{eq:T-classification-G2}.  If its last equality
failed, Lemma~\ref{lem:T-cocycle-two-branches} would make its
left-hand side equal to \(-1\).  Lemma
\ref{lem:T-nontrivial-cocycle-infinite} would then imply
\(\dim\mathcal B(V\oplus W)=\infty\), contrary to the hypothesis.
Thus all conditions in \eqref{eq:T-classification-G2} hold.
\end{proof}

\subsection{Reflections}
\label{subsec:T-reflections}

\begin{lemma}
\label{lem:T-reflected-supports}
Assume \eqref{eq:T-classification-G2}.  The supports of
\(R_1(V,W)=(V^*,X_1)\) are
\[
\supp V^*=\{z^{-1}\},\qquad
\supp X_1=\{zx_1,zx_2,zx_3,zx_4\},
\]
and the supports of \(R_2(V,W)=(Y_3,W^*)\) are
\[
\supp Y_3=\{x_1x_2x_3z\},\qquad
\supp W^*
=\{x_1^{-1},x_2^{-1},x_4^{-1},x_3^{-1}\}.
\]
The assignments
\[
\begin{aligned}
(z,x_1,x_2)
&\longmapsto (z^{-1},zx_1,zx_2)
&&\text{for }R_1(V,W),\\
(z,x_1,x_2)
&\longmapsto
(x_1x_2x_3z,x_1^{-1},x_2^{-1})
&&\text{for }R_2(V,W)
\end{aligned}
\]
extend to epimorphisms \(T\twoheadrightarrow G\).
\end{lemma}

\begin{proof}
Proposition~\ref{prop:T-G2-adjoint-objects} shows that both
reflections exist.  The proof of
Lemma~\ref{lem:T-central-support-adjoints} identifies \(X_1\) with
\(V\otimes W\), so
\(\supp X_1=\{zx_1,zx_2,zx_3,zx_4\}\).
Lemma~\ref{lem:T-third-fourth-adjoints} gives
\(\supp Y_3=\{x_1x_2x_3z\}\), and duality replaces every degree by
its inverse.  This proves the displayed support formulas.

The elements \(z^{-1}\) and \(x_1x_2x_3z\) are central by
Lemma~\ref{lem:T-group-data}.  For the first assignment, the two
sides of the braid relation and the two cubes are obtained from the
original ones by multiplication by \(z^3\).  For the second
assignment, the braid and cube relations follow by taking inverses.
Thus both assignments define homomorphisms from \(T\) to \(G\).

Under the first assignment, the elements \(x_3\) and \(x_4\) defined
in \(T\) map to \(zx_3\) and \(zx_4\).  Under the second assignment,
the braid relation gives
\[
x_2^{-1}x_1^{-1}x_2=x_4^{-1},
\qquad
x_1^{-1}x_2^{-1}x_1=x_3^{-1},
\]
so they map to the third and fourth displayed elements.
 {By Lemma~\ref{lem:ZT-enveloping-group}, both ordered
four-element supports have the same conjugation table as
\((x_1,x_2,x_3,x_4)\).}

Finally, the first image contains
\(z=(z^{-1})^{-1}\) and
\(x_i=(zx_i)z^{-1}\) for every \(i\).
The second image contains all \(x_i\) and also
\[
z=(x_1x_2x_3)^{-1}(x_1x_2x_3z).
\]
Both homomorphisms are therefore surjective.
\end{proof}

\begin{lemma}
\label{lem:T-R1-parameter-conditions}
Assume \eqref{eq:T-classification-G2}, and use the ordered
supports in Lemma~\ref{lem:T-reflected-supports}.
The pair \(R_1(V,W)=(V^*,X_1)\) satisfies the first five
conditions in \eqref{eq:T-classification-G2}.
Moreover, \(z^{-1}\) acts by \(\rho(z)\) on
\((V^*)_{z^{-1}}\).
\end{lemma}

\begin{proof}
{
Apply Lemma~\ref{lem:T-compatible-normalization}, and denote
the ordinary inducing characters by
\(\widetilde\rho,\widetilde\sigma\). The comparison in the proof of
Lemma~\ref{lem:T-third-fourth-adjoints} gives
\(\widetilde\sigma(x_2x_3)=1\), and the product condition gives
\(\widetilde\rho(x_1z)\widetilde\sigma(z)=1\).
Thus \cite[Lemma~4.19]{rank2-3} gives the first five conditions
for the reflected pair, with the ordering in
Lemma~\ref{lem:T-reflected-supports}; here
\cite[Remark~2.5]{rank2-3} again identifies its value of
\(\varepsilon_1\) with the corresponding inducing character
evaluated at the product of the second and third support elements.
Moreover,
\[
\widetilde\rho^*(z^{-1})=\widetilde\rho(z)=\rho(z).
\]
The replacement preserves duals, adjoint objects, and the scalar
parameters in these conditions. Hence the assertions hold for
the original pair.
}
\end{proof}

\begin{lemma}
\label{lem:T-R2-parameter-conditions}
Assume \eqref{eq:T-classification-G2}, and use the ordered
supports in Lemma~\ref{lem:T-reflected-supports}.
The pair \(R_2(V,W)=(Y_3,W^*)\) satisfies the first five
conditions in \eqref{eq:T-classification-G2}.
Moreover, \(x_1x_2x_3z\) acts by \(\rho(z)\) on
\((Y_3)_{x_1x_2x_3z}\).
\end{lemma}

\begin{proof}
{
Apply Lemma~\ref{lem:T-compatible-normalization}, and denote
the ordinary inducing characters by
\(\widetilde\rho,\widetilde\sigma\).
As in the preceding proof, the ordinary parameters satisfy the
hypotheses of \cite[Lemma~4.21]{rank2-3}, which gives the first
five conditions for the reflected pair with the ordering in
Lemma~\ref{lem:T-reflected-supports}.
Put \(\beta=x_1x_2x_3z\), and choose a nonzero vector \(y\) in
the ordinary third adjoint object.
By \cite[Lemma~4.14]{rank2-3},
\[
\beta\rhd y
=\bigl(\widetilde\rho(x_1)^2\widetilde\sigma(z)\bigr)^3
  \widetilde\rho(z)\widetilde\sigma(z)^3y
=\widetilde\rho(z)y=\rho(z)y,
\]
since \((\widetilde\rho(x_1)\widetilde\sigma(z))^6=1\).
The replacement preserves duals, adjoint objects, and the scalar
parameters in these conditions, so the assertions hold for the
original pair.
}
\end{proof}

\begin{lemma}
\label{lem:T-reflection-cocycle-reductions}
Assume \eqref{eq:T-classification-G2}, and use the ordered
supports in Lemma~\ref{lem:T-reflected-supports}.
Both reflected pairs
\(R_1(V,W)=(V^*,X_1)\) and
\(R_2(V,W)=(Y_3,W^*)\) satisfy
\eqref{eq:T-classification-G2}.
Moreover, \(z^{-1}\) and \(x_1x_2x_3z\) act by
\(\rho(z)\) on \((V^*)_{z^{-1}}\) and
\((Y_3)_{x_1x_2x_3z}\), respectively.
\end{lemma}

\begin{proof}
Lemmas~\ref{lem:T-R1-parameter-conditions}
and~\ref{lem:T-R2-parameter-conditions} give the first
five conditions and the asserted actions.
It remains to verify the last equality in
\eqref{eq:T-classification-G2}.
Since \(x_1x_2x_3\in Z(T)\), the defining relations of
\(T\) give endomorphisms \(f_1,f_2:T\to T\) with
\[
\begin{aligned}
f_1(z)&=z^{-1},&
f_1(x_j)&=zx_j,\\
f_2(z)&=x_1x_2x_3z,&
f_2(x_j)&=x_j^{-1}
\end{aligned}
\qquad j=1,2.
\]
By Lemma~\ref{lem:T-reflected-supports}, the quotient
maps for the two reflected pairs are
\(\pi\circ f_1\) and \(\pi\circ f_2\).
Lemma~\ref{lem:T-pullback-cocycle} gives
\([\pi^*\Phi]=0\), and hence
\[
[(\pi\circ f_a)^*\Phi]
=f_a^*[\pi^*\Phi]=0
\qquad a=1,2.
\]
Applying the same lemma to these two quotient maps gives
the remaining equality.
\end{proof}

\begin{prop}
\label{prop:T-G2-stability-under-reflections}
Assume \eqref{eq:T-classification-G2}. Then \((V,W)\)
admits all reflections. With the ordered support
representatives chosen successively as in
Lemma~\ref{lem:T-reflected-supports}, every
\(Q\in\mathcal F_2(V,W)\) satisfies the assumptions
at the beginning of this section and
\eqref{eq:T-classification-G2}, with its one-element
support in the first coordinate.
The Cartan graph is finite and standard of type \(G_2\).
\end{prop}

\begin{proof}
Let \(Q\) satisfy the assumptions at the beginning of
this section and \eqref{eq:T-classification-G2}.
Proposition~\ref{prop:T-G2-adjoint-objects} shows that
both reflections are defined, have simple components,
and that
\[
A^{[Q]}=
\begin{pmatrix}
2&-1\\
-3&2
\end{pmatrix}.
\]
Lemma~\ref{lem:T-reflected-supports} gives the support
relations and generation of \(G\) for both reflected
pairs. Lemma~\ref{lem:T-reflection-cocycle-reductions}
gives \eqref{eq:T-classification-G2} for these pairs.
Its quadratic relation implies that their double
braidings are not the identity, so both pairs are
braided-indecomposable. Thus they again satisfy all
the assumptions at the beginning of this section.

Induction on the length of a reflection word proves
that \((V,W)\) admits all reflections and that every
\(Q\in\mathcal F_2(V,W)\) has the asserted properties.
The Cartan matrix is therefore constant, and its
reflection matrices generate the ordinary Weyl group
of type \(G_2\). Consequently, the real root sets are
finite, and the Cartan graph is finite and standard
of type \(G_2\).
\end{proof}

\subsection{Finite Cartan graphs}
\label{subsec:T-finite-Cartan-graphs}

 {We use the support classification to determine the Cartan graph when
\(\mathcal B(V\oplus W)\) is finite-dimensional.}

\begin{prop}
\label{prop:T-pair-with-finite-Cartan-matrix}
Assume that \((V,W)\) admits all reflections, and let
\(P=(P_1,P_2)\in\mathcal F_2(V,W)\) have an indecomposable Cartan
matrix of finite type.  After interchanging the two entries if
necessary, there is a quandle isomorphism
\[
Z_T^{4,1}\longrightarrow
\operatorname{supp}(P_1\oplus P_2)
\]
that sends the one-element orbit to \(\operatorname{supp}P_1\) and
the four-element orbit to \(\operatorname{supp}P_2\), where
\(\operatorname{supp}P_1\) consists of a single central element.
This quandle isomorphism induces an epimorphism
\(T\twoheadrightarrow G\).
\end{prop}

\begin{proof}
Lemma~\ref{lem:reflections-preserve-support-group} gives
\(G=\langle\operatorname{supp}P_1\cup\operatorname{supp}P_2\rangle\).
For the moment, order the entries so that
\(-a_{12}^{[P]}=1\) and
\(-a_{21}^{[P]}\in\{1,2,3\}\), as permitted by the classification of
indecomposable rank-two Cartan matrices of finite type.  By the
definition of the Cartan entries,
\((\operatorname{ad}P_1)(P_2)\neq0\),
\((\operatorname{ad}P_1)^2(P_2)=0\), and
\((\operatorname{ad}P_2)^4(P_1)=0\).  Hence \(P\) is
braided-indecomposable by
\eqref{eq:indecomposable-first-adjoint}.  Its entries are simple, so
Theorem~\ref{thm:support-classification} applies to
\(X=\operatorname{supp}(P_1\oplus P_2)\).

Because \(X\) generates \(G\), the conjugation action of \(G\) on
\(X\) has image \(\operatorname{Inn}(X)\) and kernel
\(C_G(X)=Z(G)\).  Thus Lemma~\ref{lem:T-group-data} gives
\(\operatorname{Inn}(X)\simeq A_4\).  For the five quandles in
Theorem~\ref{thm:support-classification}, the permutations displayed
in Section~\ref{sec:support-quandles} generate groups of orders
\(12,4,6,6,8\), respectively.  Only \(Z_T^{4,1}\) has inner
automorphism group of order \(12\), and therefore
\(X\simeq Z_T^{4,1}\).  By
Proposition~\ref{prop:simple-objects}, the supports of \(P_1\) and
\(P_2\) are \(G\)-orbits, hence the two
\(\operatorname{Inn}(X)\)-orbits of \(X\).  After interchanging the
entries if necessary, these are the singleton and the four-element
orbit, in that order.   {The element in the singleton orbit commutes with every element of
\(X\), and is therefore central} because \(X\) generates \(G\).
Finally, a quandle isomorphism \(Z_T^{4,1}\to X\) induces a
homomorphism from its enveloping group to \(G\).  Its image contains
\(X\), hence equals \(G\).  Lemma~\ref{lem:ZT-enveloping-group}
identifies the source with \(T\) and gives the required support
labeling.
\end{proof}

\begin{prop}
\label{prop:T-necessity}
Assume that \(\mathcal B(V\oplus W)\) is finite-dimensional.  Then
\eqref{eq:T-classification-G2} holds.
\end{prop}

\begin{proof}
By Theorem~\ref{thm:Nichols-finiteness-criterion}, the pair \((V,W)\)
admits all reflections and has a finite Cartan graph. Proposition
\ref{prop:pair-with-finite-Cartan-matrix} gives
\(P\in\mathcal F_2(V,W)\) such that
\(1\leq a_{12}^{[P]}a_{21}^{[P]}\leq3\).  Thus the Cartan matrix of
\(P\) is indecomposable and of finite type.  Proposition
\ref{prop:T-pair-with-finite-Cartan-matrix} provides
\(\tau\in S_2\) and support representatives for
\(P^\tau=(P_{\tau(1)},P_{\tau(2)})\), where the support of the first
component consists of a single central element and that of the
second has four elements.
Repeated use of Lemma~\ref{lem:reflection-dimension-invariance},
together with invariance under interchanging the two entries, shows that
\(\mathcal B((P^\tau)_1\oplus(P^\tau)_2)\) is finite-dimensional.
Moreover,
\(a_{12}^{[P^\tau]}a_{21}^{[P^\tau]}
 =a_{12}^{[P]}a_{21}^{[P]}\).

For this calculation, denote the support representatives
of \(P^\tau\) by \(z,x_1,\ldots,x_4\).
Since its Cartan matrix is indecomposable,
\((\operatorname{ad}(P^\tau)_2)((P^\tau)_1)\neq0\).
The two applications of
Proposition~\ref{prop:multidegree-insertion}
in the proof of Lemma~\ref{lem:T-x1-action-minus-one}
 {therefore show that the following multidegrees occur:}
\[
(x_4,z)\longmapsto(x_1,x_4,z)
\longmapsto(x_3,x_2,x_4,z).
\]
Hence
\((\operatorname{ad}(P^\tau)_2)^3((P^\tau)_1)\neq0\),
so \(-a_{21}^{[P^\tau]}\geq3\).
Together with
\(1\leq a_{12}^{[P^\tau]}a_{21}^{[P^\tau]}\leq3\),
this gives
\[
A^{[P^\tau]}=
\begin{pmatrix}
2&-1\\
-3&2
\end{pmatrix}.
\]
Proposition~\ref{prop:T-G2-parameters-at-finite-Cartan-matrix}
now gives \eqref{eq:T-classification-G2} for \(P^\tau\).

The involutivity of reflections and the identity
\((R_iQ)^\tau\simeq R_{\tau(i)}(Q^\tau)\) give
\((V,W)^\tau\in\mathcal F_2(P^\tau)\).
By Proposition~\ref{prop:T-G2-stability-under-reflections},
this pair satisfies \eqref{eq:T-classification-G2}
and its first component has one-element support.
Since \(\operatorname{supp}W\) has four elements,
\(\tau=\operatorname{id}\).
The resulting support ordering differs from the original
one by a rack automorphism.
The scalar conditions are unchanged: conjugation by \(G\)
is transitive on ordered pairs of distinct elements of
\(x_1^G\), and the braidings commute with the \(G\)-action.
The rack automorphism also induces an automorphism of \(T\)
fixing \(z\), so the vanishing of the pullback cocycle
class is unchanged.
Lemma~\ref{lem:T-pullback-cocycle} therefore gives the
last equality in \eqref{eq:T-classification-G2} for the
original ordering as well.
\end{proof}

\subsection{Nichols algebras  {corresponding to positive roots} of type
\texorpdfstring{$G_2$}{G2}}
\label{subsec:T-G2-simple-factors}

Under \eqref{eq:T-classification-G2}, Proposition
\ref{prop:T-G2-stability-under-reflections} gives a finite standard
Cartan graph of type \(G_2\).  Its six positive real roots index
 {six simple objects}.  We first determine the dimensions
of their Nichols algebras and then describe their supports and
 {the tensor decomposition}.

\begin{lemma}
\label{lem:T-rank-one-simple-factors}
Assume \eqref{eq:T-classification-G2}.  Every simple  {object} indexed by
a positive real root of \(G_2\) has a Nichols algebra of dimension
\(6\) when its support consists of a single central element, and of dimension
\(72\) when its support consists of four elements.
\end{lemma}

\begin{proof}
Let \(R\) be  {such an object}.
By Lemma~\ref{lem:root-factors-and-reflected-components},
\(R\) is isomorphic to a component of some pair in
\(\mathcal F_2(V,W)\).
Proposition~\ref{prop:T-G2-stability-under-reflections}
shows that this pair satisfies
\eqref{eq:T-classification-G2}.
If \(R\) has one-element support, it is one-dimensional.
Write \(c_{R,R}=q\,\operatorname{id}_{R\otimes R}\).
The quadratic and product relations in
\eqref{eq:T-classification-G2} give \(q^2-q+1=0\).
Thus \(q\) is a primitive sixth root of unity, and
\(\dim\mathcal B(R)=6\).

If \(R\) has four-element support, apply
Lemma~\ref{lem:T-compatible-normalization} to the corresponding
reflected pair.  The ordinary object has the braiding
\eqref{eq:T-constant-minus-one-braiding} by
Lemma~\ref{lem:T-tetrahedral-braiding}.
Its Nichols algebra has dimension \(72\) by
\cite[Theorem~6.15]{AGn03}.
 {Lemma~\ref{lem:T-compatible-normalization} preserves
the dimensions of Nichols algebras}, so \(\dim\mathcal B(R)=72\).
\end{proof}
\begin{prop}
\label{prop:T-G2-simple-factors}
Assume \eqref{eq:T-classification-G2}.  For the reduced expression
\((2,1,2,1,2,1)\) of a longest morphism with target \([(V,W)]\), let
\(\beta_1,\ldots,\beta_6\) and
\(M_{\beta_1},\ldots,M_{\beta_6}\) be the corresponding positive real
roots and simple  {objects}. 
Multiplication induces an isomorphism of
\(\mathbb N_0^2\)-graded objects in \(\GG\),
\[
\mathcal B(M_{\beta_6})\otimes
\bigl(\cdots(
\mathcal B(M_{\beta_2})\otimes
\mathcal B(M_{\beta_1})
)\bigr)
\xrightarrow{\ \sim\ }
\mathcal B(V\oplus W).
\]
\end{prop}

\begin{proof}
Write \((i_1,\ldots,i_6)=(2,1,2,1,2,1)\).
 {By Proposition~\ref{prop:T-G2-stability-under-reflections}, the pair
admits all reflections, and the Cartan matrix at every object is}
\(\left(\begin{smallmatrix}2&-1\\-3&2\end{smallmatrix}\right)\).
Using \(s_i(\alpha_j)=\alpha_j-a_{ij}\alpha_i\), one obtains
\[
(\beta_1,\ldots,\beta_6)
=
\bigl(
\alpha_2,\,
\alpha_1+3\alpha_2,\,
\alpha_1+2\alpha_2,\,
2\alpha_1+3\alpha_2,\,
\alpha_1+\alpha_2,\,
\alpha_1
\bigr).
\]
These are precisely the six positive roots of \(G_2\), so the chosen
word is reduced and represents a longest morphism.

For \(1\leq k\leq6\), consider the pair obtained after
the first \(k-1\) reflections in the sequence
\((2,1,2,1,2,1)\).
By Lemma~\ref{lem:root-factors-and-reflected-components},
\(M_{\beta_k}\) is isomorphic to the second component
of this pair when \(k\) is odd, and to the first
component when \(k\) is even.
By Proposition~\ref{prop:T-G2-stability-under-reflections},
 {the support of the first component consists of a single central
element, and the support of the second has four elements.} Thus
\[
|\operatorname{supp}M_{\beta_k}|=
\begin{cases}
4,&k=1,3,5,\\
1,&k=2,4,6.
\end{cases}
\]
Finally, the same proposition ensures that all reflections
exist and that the Cartan graph is finite.
Proposition~\ref{prop:positive-root-factorization}
therefore gives the displayed multiplication isomorphism.
\end{proof}

\begin{cor}
\label{cor:T-Nichols-dimension}
Under \eqref{eq:T-classification-G2},
\[
\dim\mathcal B(V\oplus W)
=6^3\,72^3
=80\,621\,568.
\]
\end{cor}

\begin{proof}
Among the six simple objects \(M_{\beta_k}\) in
Proposition~\ref{prop:T-G2-simple-factors}, three are one-dimensional
and three are four-dimensional. By
Lemma~\ref{lem:T-rank-one-simple-factors}, their Nichols algebras have
dimensions \(6\) and \(72\), respectively.
The tensor decomposition gives the asserted dimension.
\end{proof}

\subsection{Proof of the classification theorem}
\label{subsec:T-proof-classification}

\begin{proof}[Proof of Theorem~\ref{thm:T-classification}]
If \(\mathcal B(V\oplus W)\) is finite-dimensional, then
Proposition~\ref{prop:T-necessity} gives
\eqref{eq:T-classification-G2}.

Conversely, assume \eqref{eq:T-classification-G2}.
Proposition~\ref{prop:T-G2-stability-under-reflections} shows that
\((V,W)\) admits all reflections and that its Cartan graph is standard
of type \(G_2\), with the stated matrix at every object.
Corollary~\ref{cor:T-Nichols-dimension} gives the stated dimension.
In particular, \(\mathcal B(V\oplus W)\) is finite-dimensional.
\end{proof}
\begin{rmk}\upshape
When \(\Phi=1\), the category
\({}_G^G\mathcal{YD}^{\Phi}\) is the ordinary Yetter--Drinfeld
category over the Hopf algebra \(\mathbb CG\).  Under the conditions
of Theorem~\ref{thm:T-classification}, one has
\(\varepsilon_1(W)=\sigma(x_2x_3)\) and
\(\rho(x_1)\rho(z)=\rho(x_1z)\).  Hence
Theorem~\ref{thm:T-classification} specializes to
\cite[Theorem~2.9]{rank2-3}.
\end{rmk}

\section{The \texorpdfstring{$\Gamma _3$}{Gamma3} case}
\label{sec:classification-Gamma3}

In this section, \(\Gamma _3\) is given by the presentation
\begin{equation}
\label{eq:Gamma3-presentation}
\Gamma _3=
\left\langle
\varepsilon,g,h
\ \middle|\
\varepsilon^3=1,
\quad g\varepsilon g^{-1}=\varepsilon^{-1},
\quad [h,g]=[h,\varepsilon]=1
\right\rangle .
\end{equation}
This is the presentation from
Subsection~\ref{subsec:enveloping-groups-Gamma} after replacing its
generator \(h\) by \(\varepsilon^{-1}h\).
Let \(\pi:\Gamma _3\twoheadrightarrow G\) be an epimorphism onto a
finite non-abelian group.  The images in
\(G\) are denoted by the same letters \(\varepsilon,g,h\).  Let
\(\Phi\in Z^3(G,\mathbb C^\times)\) be normalized.

Throughout this section,
\(V,W\in{}_G^G\mathcal{YD}^{\Phi}\) are finite-dimensional simple
objects satisfying
\begin{equation}
\label{eq:Gamma3-standing-assumptions}
G=\langle\operatorname{supp}(V\oplus W)\rangle,
\qquad
(V,W)\text{ is braided-indecomposable}.
\end{equation}
We use the following notation for the two ordered support pairs that
occur in the classification:
\begin{align}
\label{eq:Gamma3-support-32}
(3,2):\quad&
V=M(g,\rho),
&W=M(\varepsilon h,\sigma),
\end{align}
and
\begin{align}
\label{eq:Gamma3-support-31}
(3,1):\quad&
V=M(g,\rho),
&W=M(h,\tau).
\end{align}
Their support quandles are \(Z_3^{3,2}\) and \(Z_3^{3,1}\),
respectively.

Lemma~\ref{lem:Gamma3-projective-dimensions} below shows that the
inducing representations in \eqref{eq:Gamma3-support-32} are
one-dimensional and that \(\dim\rho=1\) in
\eqref{eq:Gamma3-support-31}.   { In the \((3,1)\) case, we distinguish
\(\dim\tau=1\) from \(\dim\tau=2\).}

\subsection{The four parameter families}
\label{subsec:Gamma3-Phi-parameters}

The parameters occurring in the four families below, namely
\(\Delta _1,s\), \(\Delta _3\), and
\(\kappa_\Phi(\varepsilon)\), are computed from the fixed cocycle
\(\Phi\) and the indicated inducing projective representations. They
are defined in \eqref{eq:Gamma3-Delta1-s},
\eqref{eq:Gamma3-Delta3}, and \eqref{eq:Gamma3-kappa}, respectively.
In the two-dimensional case,
\(I=\operatorname{id}_{W_h}\).
The four parameter families are as follows.
\begin{align}
\label{eq:Gamma3-P1}
\mathcal P_1^\Phi:\quad&
\begin{gathered}
V=M(g,\rho),\quad W=M(\varepsilon h,\sigma),\\
\rho(g)=\sigma(\varepsilon h)=-1,
\quad \Delta _1=1,
\quad 1+s+s^2=0;
\end{gathered}
\\[1mm]
\label{eq:Gamma3-P2}
\mathcal P_2^\Phi:\quad&
\begin{gathered}
V=M(g,\rho),\quad W=M(\varepsilon h,\sigma),\\
\rho(g)=\sigma(\varepsilon h)=-1,
\quad \Delta _1=1,
\quad s=1;
\end{gathered}
\\[1mm]
\label{eq:Gamma3-P3}
\mathcal P_3^\Phi:\quad&
\begin{gathered}
V=M(g,\rho),\quad W=M(h,\tau),
\quad\dim\tau=2,\\
\kappa_\Phi(\varepsilon)=1,
\quad \rho(g)=-1,
\quad \tau(h)=-I,
\quad \Delta _3=I;
\end{gathered}
\\[1mm]
\label{eq:Gamma3-P4}
\mathcal P_4^\Phi:\quad&
\begin{gathered}
V=M(g,\rho),\quad W=M(h,\tau),
\quad\dim\tau=1,\\
\kappa_\Phi(\varepsilon)=1,
\quad \rho(g)=-1,\\
1-\rho(h)\tau(g)+\rho(h)^2\tau(g)^2=0,
\quad \rho(h)\tau(g)\tau(h)=1.
\end{gathered}
\end{align}
In \(\mathcal P_1^\Phi\) and \(\mathcal P_2^\Phi\), the displayed
conditions and
Lemma~\ref{lem:Gamma3-32-second-adjoint-coefficients} give
\(\kappa_\Phi(\varepsilon)=s^3=1\).  Thus every pair in the four
families satisfies \(\kappa_\Phi(\varepsilon)=1\), although this
condition is explicit only in \(\mathcal P_3^\Phi\) and
\(\mathcal P_4^\Phi\).
The classification for finite quotients of \(\Gamma _3\) is as
follows.
\begin{thm}
\label{thm:Gamma3-classification}
Under the standing assumptions
\eqref{eq:Gamma3-presentation}--\eqref{eq:Gamma3-standing-assumptions},
the Nichols algebra \(\mathcal B(V\oplus W)\) is finite-dimensional
 {if and only if, after possibly interchanging \(V\) and \(W\) and
retaining the notation \((V,W)\) for the resulting ordered pair,
there are generators \(\varepsilon,g,h\in G\) inducing the
epimorphism from \eqref{eq:Gamma3-presentation} such that
\((V,W)\) has one of the forms
\eqref{eq:Gamma3-support-32} or \eqref{eq:Gamma3-support-31}
and belongs to one of}
\[
\mathcal P_1^\Phi,
\qquad
\mathcal P_2^\Phi,
\qquad
\mathcal P_3^\Phi,
\qquad
\mathcal P_4^\Phi.
\]
 {Here the parameters are computed using these generators and the fixed
cocycle \(\Phi\).}
In every case the Cartan graph is standard of type \(B_2\), with
matrix
$\begin{pmatrix}2&-2\\-1&2\end{pmatrix}$
at every object.  Moreover,
\begin{equation}
\label{eq:Gamma3-dimensions}
\dim\mathcal B(V\oplus W)=
\begin{cases}
10368,&(V,W)\in\mathcal P_1^\Phi\cup\mathcal P_4^\Phi,\\
2304,&(V,W)\in\mathcal P_2^\Phi\cup\mathcal P_3^\Phi.
\end{cases}
\end{equation}
\end{thm}

\medskip
\noindent\textit{Outline of the proof.}
\begin{itemize}[leftmargin=2em]
\item
Subsection~\ref{subsec:Gamma3-supports-dimensions} determines the
possible supports and  {dimensions of the inducing
representations}.  This reduces the proof
to the \((3,2)\) case and the two \((3,1)\) cases with
\(\dim\tau=2\) or \(1\).

\item
Subsection~\ref{subsec:Gamma3-third-homology} identifies
\(\kappa_\Phi(\varepsilon)\) as the obstruction to the  {cohomological triviality} of
the inflated cocycle.  Subsections~\ref{subsec:Gamma3-adjoint-objects}
and~\ref{subsec:Gamma3-branch-obstructions} compute the required
adjoint objects, derive the parameter conditions, and show that this
obstruction vanishes in every finite-dimensional case.

\item
Subsection~\ref{subsec:Gamma3-obstruction-untwist}  {lifts the
pair to \(\Gamma_3\), trivializes the pulled-back cocycle, and
identifies the four families \(\mathcal P_i^\Phi\) with their ordinary
counterparts}.

\item
Subsection~\ref{subsec:Gamma3-necessity-reflections} proves necessity
and shows that the four families are closed under reflections.  It
also proves that their Cartan graphs are standard of type \(B_2\).

\item
Subsection~\ref{subsec:Gamma3-sufficiency} transfers finite-dimensionality
and the dimension formulas from the ordinary classification.
\end{itemize}
\medskip

\medskip
\noindent\textit{Notation used in this section.}
\begingroup
\small
\setlength{\tabcolsep}{4pt}
\renewcommand{\arraystretch}{1.08}
\begin{center}
\begin{tabular}{@{}L{0.18\textwidth}L{0.22\textwidth}L{0.52\textwidth}@{}}
\hline
Notation & Defined in &  {Role} \\
\hline
\(\Delta_1\)
& \eqref{eq:Gamma3-Delta1-s}
& In the \((3,2)\) case, \(X_1\) is simple if and only if
  \(\Delta_1=1\). \\

\(s\)
& \eqref{eq:Gamma3-Delta1-s}
& The conditions \(1+s+s^2=0\) and \(s=1\) distinguish
  \(\mathcal P_1^\Phi\) from \(\mathcal P_2^\Phi\). \\

\(\Delta_3\)
& \eqref{eq:Gamma3-Delta3}
& In the \((3,1)\) case with \(\dim\tau=2\), \(X_1\) is simple if
  and only if \(\Delta_3=I\). \\

\(\kappa_\Phi(\varepsilon)\)
& \eqref{eq:Gamma3-kappa}
& The equality \(\kappa_\Phi(\varepsilon)=1\) is equivalent to the
  triviality of \([\pi^*\Phi]\).  It follows from the other conditions
  in \(\mathcal P_1^\Phi,\mathcal P_2^\Phi\) and is imposed in
  \(\mathcal P_3^\Phi,\mathcal P_4^\Phi\). \\
\hline
\end{tabular}
\end{center}
\endgroup

\subsection{ {Supports and dimensions of the inducing representations}}
\label{subsec:Gamma3-supports-dimensions}

\begin{lemma}
\label{lem:Gamma3-kernel-centralizers}
For every \(x\in\Gamma _3\),
\begin{equation}
\label{eq:Gamma3-centralizer-lift}
 \pi^{-1}\bigl(C_G(\pi(x))\bigr):= \{a \in \Gamma_3:  \pi(a)\pi(x)=\pi(x)\pi(a)\}=C_{\Gamma _3}(x).
\end{equation}
Consequently, \(\pi\) induces a bijection
\(x^{\Gamma _3}\simeq\pi(x)^G\).  Moreover,
\(C_G(g)=\langle g,h\rangle\),
\(C_G(\varepsilon h)=\langle\varepsilon,g^2,h\rangle\), and
\(C_G(h)=G\).
Every conjugacy class has one of the following forms, for some
\(\delta\in Z(G)\):
\begin{equation}
\label{eq:Gamma3-class-list}
 \{\delta\}\qquad
 \{\varepsilon\delta,\varepsilon^2\delta\}\qquad
 \{g\delta,g\varepsilon\delta,g\varepsilon^2\delta\}.
\end{equation}
\end{lemma}

\begin{proof}
{
The presentation identifies \(\Gamma_3\) with
\((\langle\varepsilon\rangle\rtimes\langle g\rangle)
\times\langle h\rangle\), where \(g\) acts by inversion.
Thus every element has a unique normal form
\(\varepsilon^i g^m h^n\), with \(0\leq i<3\) and
\(m,n\in\mathbb Z\), and
\([\Gamma_3,\Gamma_3]=\langle\varepsilon\rangle\).
Since \(G\) is non-abelian, \(\pi(\varepsilon)\ne1\), so
\(\ker\pi\cap[\Gamma_3,\Gamma_3]=1\).
If \(\pi(a)\) centralizes \(\pi(x)\), then \([a,x]\) belongs
to this intersection and is trivial. The reverse inclusion is
immediate, proving \eqref{eq:Gamma3-centralizer-lift}.
Together with the surjectivity of \(\pi\), this gives a bijection
\[
\Gamma_3/C_{\Gamma_3}(x)
\longrightarrow G/C_G(\pi(x)),
\]
and hence the asserted bijection of conjugacy classes.}

{
 {A direct calculation gives}
\[
C_{\Gamma_3}(g)=\langle g,h\rangle,\qquad
C_{\Gamma_3}(\varepsilon h)=\langle\varepsilon,g^2,h\rangle,
\qquad C_{\Gamma_3}(h)=\Gamma_3,
\]
because \(\varepsilon^i g^m h^n\) commutes with \(g\) exactly
when \(i=0\), and with \(\varepsilon\) exactly when \(m\) is
even. In particular, \(Z(\Gamma_3)=\langle g^2,h\rangle\).
Equation~\eqref{eq:Gamma3-centralizer-lift} gives the stated
centralizers in \(G\).
If \(m\) is even, then \(g^m h^n\) is central, and conjugation
by \(g\) interchanges
\(\varepsilon g^m h^n\) and \(\varepsilon^2g^m h^n\).
If \(m\) is odd, then \(g^{m-1}h^n\) is central, and
conjugation by \(\varepsilon\) cyclically permutes
\(g^m h^n\), \(\varepsilon^2g^m h^n\), and
\(\varepsilon g^m h^n\), in this order. The bijections of conjugacy classes therefore
give \eqref{eq:Gamma3-class-list}.}
\end{proof}

For commuting \(a,b,c\), recall that \(f_\Phi(a,b,c)=\frac{\Phi(a,b,c)\Phi(b,c,a)\Phi(c,a,b)}{
\Phi(a,c,b)\Phi(c,b,a)\Phi(b,a,c)}\).
On an abelian subgroup this is an alternating tricharacter, and
\begin{equation}
\label{eq:Gamma3-local-commutator}
 \frac{\Phi_x(a,b)}{\Phi_x(b,a)}=f_\Phi(x,a,b)
 \qquad a,b\in C_G(x),\ ab=ba.
\end{equation}
\begin{lemma}
\label{lem:Gamma3-projective-dimensions}
The restriction of \(f_\Phi\) to each of
\(\langle g,h\rangle\) and
\(\langle\varepsilon,g^2,h\rangle\) is trivial.  Consequently, in
the case \eqref{eq:Gamma3-support-32} one has
\(\dim\rho=\dim\sigma=1\), whereas in the case
\eqref{eq:Gamma3-support-31} one has
\(\dim\rho=1\) and \(\dim\tau\le2\).
\end{lemma}

\begin{proof}
The abelian group \(\langle g,h\rangle\) is generated by two
elements, so every alternating tricharacter on it is trivial.
Conjugation by \(g\) sends \(\varepsilon\) to \(\varepsilon^2\) and
fixes \(g^2,h\).  Hence  {invariance under simultaneous conjugation} and
trilinearity give
\[
f_\Phi(\varepsilon,g^2,h)
=f_\Phi(\varepsilon^2,g^2,h)
=f_\Phi(\varepsilon,g^2,h)^2.
\]
Thus \(f_\Phi(\varepsilon,g^2,h)=1\).  Since the abelian subgroup
\(\langle\varepsilon,g^2,h\rangle\) is generated by these three
elements, alternation and trilinearity show that \(f_\Phi\) is
trivial on \(\langle\varepsilon,g^2,h\rangle\).  By
Lemma~\ref{lem:Gamma3-kernel-centralizers} and
\eqref{eq:Gamma3-local-commutator}, the twisted group algebras
\(\mathbb C^{\Phi_g}C_G(g)\) and
\(\mathbb C^{\Phi_{\varepsilon h}}C_G(\varepsilon h)\) are
commutative.  Since \(G\) is finite, their irreducible modules are
one-dimensional, proving \(\dim\rho=\dim\sigma=1\).

Finally, consider \eqref{eq:Gamma3-support-31}, and let
\[
A=\langle\varepsilon,g^2,h\rangle.
\]
The subgroup \(A\) is abelian and normal in \(G\).  Moreover,
\(G=\langle A,g\rangle\), \(g^2\in A\), and \(A\ne G\), since \(G\)
is non-abelian.  Hence \([G:A]=2\).
Since \(h\in A\) and \(f_\Phi\) is trivial on \(A^3\),
\eqref{eq:Gamma3-local-commutator} shows that
\(\mathbb C^{\Phi_h}A\) is commutative.  Therefore the operators
\(\tau(a)\), \(a\in A\), have a common eigenline \(L\subseteq W_h\).
For \(a\in A\), normality gives \(g^{-1}ag\in A\), so the projective
multiplication law shows that \(a\) preserves \(g\rhd L\).
Moreover,
\[
g\rhd(g\rhd L)=g^2\rhd L\subseteq L
\]
as subspaces.   {Thus \(L+g\rhd L\) is a nonzero
subspace of \(W_h\) invariant under \(\tau\).
Since \(\tau\) is irreducible,}
\(W_h=L+g\rhd L\), and therefore \(\dim\tau\leq2\).
\end{proof}

\begin{cor}
\label{cor:Gamma3-no-33-support}
Let \(U,Z\in{}_G^G\mathcal{YD}^{\Phi}\) be simple, and suppose that
\((\operatorname{ad}U)(Z)\) is nonzero and simple.  Then the supports
of \(U\) and \(Z\) cannot both have three elements.
\end{cor}

\begin{proof}
Assume otherwise.  By \eqref{eq:Gamma3-class-list}, choose
\(\delta_1,\delta_2\in Z(G)\) such that
\(x=g\delta_1\in\operatorname{supp}U\) and
\(y=g\varepsilon\delta_2\in\operatorname{supp}Z\), and take nonzero
\(u\in U_x\), \(w\in Z_y\).  By
Lemma~\ref{lem:adjoint-object-realization}, identify
\(X=(\operatorname{ad}U)(Z)\) with the image of
\(\operatorname{id}-c_{Z,U}c_{U,Z}\), and let
\[
\xi=(\operatorname{id}-c_{Z,U}c_{U,Z})(u\otimes w)\in X_{xy}.
\]
Its two summands are nonzero and have ordered degrees
\((g\delta_1,g\varepsilon\delta_2)\) and
\((g\varepsilon\delta_1,g\varepsilon^2\delta_2)\), respectively.
Since \(xy=\varepsilon g^2\delta_1\delta_2\), conjugation by \(xy\)
cycles these two pairs and
\((g\varepsilon^2\delta_1,g\delta_2)\), in that order.  Thus
\((xy)\rhd\xi\) has a nonzero component of the last ordered degree,
whereas \(\xi\) does not, so \((xy)\rhd\xi\notin\mathbb C\xi\).

On the other hand, \(X_{xy}\) is an irreducible
\(\Phi_{xy}\)-projective \(C_G(xy)\)-module by
Proposition~\ref{prop:simple-objects}.  For every \(a\in C_G(xy)\),
\eqref{eq:Gamma3-local-commutator} gives
\(\Phi_{xy}(xy,a)/\Phi_{xy}(a,xy)=f_\Phi(xy,xy,a)=1\).
Hence the action of \(xy\) commutes with that of \(C_G(xy)\), and
Schur's lemma shows that \(xy\) acts by a scalar on \(X_{xy}\).
This is a contradiction.
\end{proof}

\begin{prop}
\label{prop:Gamma3-support-and-dimensions}
Assume that \(\mathcal B(V\oplus W)\) is finite-dimensional.  After
possibly interchanging \(V\) and \(W\), there is a choice of generators
\(\varepsilon,g,h\in G\) satisfying the relations in
\eqref{eq:Gamma3-presentation} and inducing an epimorphism
\(\Gamma _3\twoheadrightarrow G\), for which the pair has one of the forms
\eqref{eq:Gamma3-support-32} or \eqref{eq:Gamma3-support-31}.  In the
case \((3,2)\), one has \(\dim\rho=\dim\sigma=1\). In the case
\((3,1)\), one has \(\dim\rho=1\) and \(\dim\tau\in\{1,2\}\).
\end{prop}

\begin{proof}
Braided indecomposability and
\eqref{eq:indecomposable-first-adjoint} give \(X_1\ne0\).
By Theorem~\ref{thm:Nichols-finiteness-criterion}, the pair admits
all reflections and has a finite Cartan graph.  Lemma~\ref{lem:root-string-adjoint-simplicity} therefore shows
that \(X_1\) is simple.  Corollary~\ref{cor:Gamma3-no-33-support}
excludes two three-element supports.

Equation~\eqref{eq:Gamma3-centralizer-lift} gives
\(Z(G)=\langle g^2,h\rangle\).  Hence
\eqref{eq:Gamma3-class-list} shows that every conjugacy class with at
most two elements is contained in the abelian subgroup
\(\langle\varepsilon,g^2,h\rangle\).   {Since the union of the two supports generates the non-abelian group \(G\)}, at least one of them has three elements.
Thus exactly one support has three elements.

After interchanging the two entries if necessary,
\eqref{eq:Gamma3-class-list} therefore gives, for some
\(\delta_1,\delta_2\in Z(G)\),
\[
\operatorname{supp}V
=\{g\delta_1,g\varepsilon\delta_1,g\varepsilon^2\delta_1\},
\qquad
\operatorname{supp}W
=\begin{cases}
\{\varepsilon\delta_2,\varepsilon^2\delta_2\},&(3,2),\\
\{\delta_2\},&(3,1).
\end{cases}
\]
 {The subgroup generated by the three-element support contains
\(g\delta_1\) and \(\varepsilon\). The subgroup generated by both
supports also contains \(\delta_2\).}  Since
\(\delta_1,\delta_2\) are central, replacing \(g\) by
\(g\delta_1\) and \(h\) by \(\delta_2\) preserves the relations in
\eqref{eq:Gamma3-presentation}. The resulting elements generate
\(G\), so they induce the required epimorphism and give
\eqref{eq:Gamma3-support-32} or \eqref{eq:Gamma3-support-31}.  The
dimension assertions follow from
Lemma~\ref{lem:Gamma3-projective-dimensions}, with \(\tau\ne0\) giving
\(\dim\tau\in\{1,2\}\).
\end{proof}

\subsection{The third homology of \texorpdfstring{$\Gamma _3$}{Gamma3}}
\label{subsec:Gamma3-third-homology}

We use the normalized integral bar complex of \(\Gamma _3\).  Its
degree-three chains are integral sums of symbols \([a|b|c]\), with any
symbol containing the identity understood to be zero.  For a normalized
\(\mathbb C^\times\)-valued three-cochain \(\Psi\) on \(\Gamma _3\), our
conventions are
\begin{align*}
&\partial[a|b|c]
=[b|c]-[ab|c]+[a|bc]-[a|b],\\
&\left\langle
 \Psi,\sum_j n_j[a_j|b_j|c_j]
\right\rangle
=\prod_j\Psi(a_j,b_j,c_j)^{n_j}.
\end{align*}
Recall that \(\pi:\Gamma _3\twoheadrightarrow G\) is an epimorphism onto a
finite non-abelian group.
The pullback of \(\Phi\) is
\((\pi^*\Phi)(a,b,c)=\Phi(\pi(a),\pi(b),\pi(c))\).
Define
\begin{equation}
\label{eq:Gamma3-kappa}
\kappa_\Phi(\varepsilon)=
\Phi(\varepsilon,\varepsilon,\varepsilon)
\Phi(\varepsilon,\varepsilon^2,\varepsilon).
\end{equation}
In the bar complex, let
\begin{equation}
\label{eq:Gamma3-kappa-cycle}
 \mathsf K=[\varepsilon|\varepsilon|\varepsilon]
 +[\varepsilon|\varepsilon^2|\varepsilon].
\end{equation}
The boundaries of the two summands are
\(-[\varepsilon^2|\varepsilon]+[\varepsilon|\varepsilon^2]\) and its
negative.  Hence \(\mathsf K\) is a cycle, and
\(\langle\pi^*\Phi,\mathsf K\rangle
=\kappa_\Phi(\varepsilon)\).

\begin{lemma}
\label{lem:Gamma3-third-homology}
One has
\[
 H_3(\Gamma _3,\mathbb Z)\simeq\mathbb Z/3,
\]
and \(\mathsf K\) represents a generator.  Moreover, for
\[
 \varpi:\Gamma _3\longrightarrow S_3,\qquad
 \varpi(\varepsilon)=(123),\quad
 \varpi(g)=(12),\quad \varpi(h)=1,
\]
the map
\(\varpi_*:H_3(\Gamma _3,\mathbb Z)\to H_3(S_3,\mathbb Z)\)
is injective.
\end{lemma}

\begin{proof}
Let \(D=\langle\varepsilon,g\rangle\).
The presentation shows that
\(D\simeq C_3\rtimes\mathbb Z\), where the generator of
\(\mathbb Z\) acts on \(C_3\) by inversion, and that
\(\Gamma _3=D\times\langle h\rangle\).  For the extension
\(1\to C_3\to D\to\mathbb Z\to1\), the
Lyndon--Hochschild--Serre homology spectral sequence is
\[
E^2_{p,q}
=H_p\bigl(\mathbb Z,H_q(C_3,\mathbb Z)\bigr)
\Longrightarrow H_{p+q}(D,\mathbb Z),
\]
see \cite[Chapter~VII, Section~6]{Brown1982}.
 {The standard two-periodic resolution of the trivial
\(\mathbb ZC_3\)-module \(\mathbb Z\) gives
\(H_1(C_3,\mathbb Z)\simeq H_3(C_3,\mathbb Z)\simeq\mathbb Z/3\)
and \(H_2(C_3,\mathbb Z)=0\).}
The same resolution shows that inversion acts as \(-1\)
on \(H_1\) and as \(+1\) on \(H_3\).
Since \(\mathbb Z\) has homological dimension one,
only the columns \(p=0,1\) occur.
In total degrees two and three, the only terms that
remain to be computed are
\[
\begin{aligned}
E^2_{1,1}
&=\ker(-2:\mathbb Z/3\longrightarrow\mathbb Z/3)=0,\\
E^2_{0,3}
&=H_0\bigl(\mathbb Z,\mathbb Z/3\bigr)
 =\mathbb Z/3,
\end{aligned}
\]
where the action in the second line is trivial.
Thus \(H_2(D,\mathbb Z)=0\) and
\(H_3(D,\mathbb Z)\simeq\mathbb Z/3\).  After tensoring the two-periodic resolution with the
trivial module \(\mathbb Z\), the differentials in
degrees three and four are zero and multiplication
by three, respectively. Thus the degree-three basis
vector represents a generator of
\(H_3(C_3,\mathbb Z)\simeq\mathbb Z/3\).
A chain map from the two-periodic resolution to the
normalized bar resolution, lifting the identity of
\(\mathbb Z\), may be chosen to send the degree-three
basis vector to
\[
[\varepsilon|\varepsilon|\varepsilon]
+[\varepsilon|\varepsilon^2|\varepsilon]
=\mathsf K.
\]
After tensoring with the trivial module \(\mathbb Z\),
this map induces an isomorphism on homology.
Hence \([\mathsf K]\) generates \(H_3(C_3,\mathbb Z)\).  The edge map in
the above spectral sequence is induced by the inclusion
\(C_3\hookrightarrow D\).  Since \(E^\infty_{0,3}\cong\mathbb Z/3\) is
the only nonzero term of total degree \(3\), this edge map is an
isomorphism.  Hence \([\mathsf K]\) also generates
\(H_3(D,\mathbb Z)\).
The K\"unneth formula has no Tor terms because the homology of
\(\langle h\rangle\simeq\mathbb Z\) is free.  It follows that
\(H_3(\Gamma _3,\mathbb Z)\simeq
H_3(D,\mathbb Z)\oplus H_2(D,\mathbb Z)\simeq\mathbb Z/3\), and
\(\mathsf K\) represents a generator.

It remains to prove the assertion about \(\varpi_*\). Let \(i:C_3\hookrightarrow S_3\) send \(\varepsilon\)
to \((123)\), and let
\(\operatorname{tr}:H_3(S_3,\mathbb Z)\to
H_3(C_3,\mathbb Z)\)
be the associated homology transfer.
Since \(C_3\) is normal of index two,
\cite[Proposition~III.9.5(iii)]{Brown1982} gives
\[
\operatorname{tr}\circ i_*=\operatorname{id}+\iota_*,
\]
where \(\iota:C_3\to C_3\) sends
\(\varepsilon\) to \(\varepsilon^{-1}\).
As shown above, \(\iota_*\) is the identity on
\(H_3(C_3,\mathbb Z)\).
Thus \(\operatorname{tr}\circ i_*=2\operatorname{id}\).
Since \(H_3(C_3,\mathbb Z)\simeq\mathbb Z/3\),
it follows that \(i_*\) is injective.  Finally, \(\varpi_*\) sends
\([\mathsf K]\) to \(i_*[\mathsf K]\). Since \([\mathsf K]\) generates
\(H_3(\Gamma _3,\mathbb Z)\), the map \(\varpi_*\) is injective.
\end{proof}

\subsection{Adjoint calculations in the three cases}
\label{subsec:Gamma3-adjoint-objects}

We now compute the parts of \(X_1=(\operatorname{ad}V)(W)\) and
\(X_2=(\operatorname{ad}V)^2(W)\) that determine simplicity and
support.  Appendix~\ref{app:Gamma3-calculations} contains the longer
coefficient and bar-complex calculations.

\subsubsection{The \texorpdfstring{$(3,2)$}{(3,2)} case}

Let \(V=M(g,\rho)\) and \(W=M(\varepsilon h,\sigma)\), with
\(\dim\rho=\dim\sigma=1\).  Define
\begin{align}
&a={}
\frac{
 \Phi_g(\varepsilon^2h,\varepsilon^2)
}{
 \Phi_g(\varepsilon,h)\Phi_{\varepsilon h}(g,\varepsilon^2)
}
\rho(h)\sigma(\varepsilon^2),
\label{eq:Gamma3-a}
\\
&b_0={}
\Phi^g(g\varepsilon^2,\varepsilon h)
\frac{\Phi_g(g,\varepsilon^2)}{\Phi_g(\varepsilon,g)}
\rho(g),
\label{eq:Gamma3-b0-b1}
\\
&b_1={}
\Phi^g(g\varepsilon,\varepsilon^2h)
\frac{\Phi_g(g,\varepsilon)}{\Phi_g(\varepsilon^2,g)}
\rho(g)\Phi_{\varepsilon h}(g,g)\sigma(g^2).
\label{eq:Gamma3-b1}
\end{align}
Set
\begin{equation}
\label{eq:Gamma3-Delta1-s}
\Delta _1=\frac{a^2b_1}{b_0},
\qquad
\chi=-\frac{b_0}{a},
\qquad
s=-\chi\rho(gh).
\end{equation}
Choose nonzero vectors
\(v_0\in V_g\) and \(w_0\in W_{\varepsilon h}\), and let
\[
v_i=\varepsilon^i\rhd v_0\quad i\in\mathbb Z/3\mathbb Z,
\qquad w_1=g\rhd w_0.
\]

\begin{lemma}
\label{lem:Gamma3-32-X1}
Let
$
E_0=v_2\otimes w_0,\
E_1=v_1\otimes w_1,\
 u_0=\varphi_1^\Phi(E_0).
$
Then $u_0=E_0-aE_1\ne 0$, and \(X_1\) is simple if and only if
\(\Delta_1=1\).  Under this condition, its support is \((gh)^G\),
and \(g\) acts by \(\chi\) on the line
\(\mathbb Cu_0=(X_1)_{gh}\).
\end{lemma}

\begin{proof}
Equation~\eqref{eq:YD-projective-action} gives
\[
 (g\varepsilon^2)\rhd w_0
 =\frac{\sigma(\varepsilon^2)}
        {\Phi_{\varepsilon h}(g,\varepsilon^2)}w_1,
 \qquad
 (\varepsilon^2h)\rhd v_2
 =\frac{\Phi_g(\varepsilon^2h,\varepsilon^2)}
        {\Phi_g(\varepsilon,h)}\rho(h)v_1.
\]
Using
\(\varphi_1^\Phi=\operatorname{id}-c_{W,V}c_{V,W}\) therefore gives
\(u_0=E_0-aE_1\).  The tensors \(E_0\) and \(E_1\) lie in the
distinct direct summands
\(V_{g\varepsilon^2}\otimes W_{\varepsilon h}\) and
\(V_{g\varepsilon}\otimes W_{\varepsilon^2h}\), respectively, so
\(u_0\ne0\).  Equations~\eqref{eq:tensor-product} and
\eqref{eq:YD-projective-action} also give
\[
g\rhd E_0=b_0E_1,
\qquad
g\rhd E_1=b_1E_0.
\]
Consequently,
\(g\rhd u_0=b_0E_1-ab_1E_0\).  This vector is proportional to
\(u_0\) if and only if \(a^2b_1=b_0\), or equivalently
\(\Delta_1=1\).  In that case
\[
g\rhd u_0=-\frac{b_0}{a}u_0=\chi u_0.
\]

The degree-\(gh\) component of \(V\otimes W\) is
\(\mathbb CE_0\oplus\mathbb CE_1\).  Its image is
\(\mathbb Cu_0+\mathbb C(g\rhd u_0)\).  This space has dimension one
exactly when \(\Delta_1=1\).  Since \(u_0\ne0\) and \(X_1\) is
\(G\)-stable, while the homogeneous degrees of \(V\otimes W\) are
\(gh,g\varepsilon h,g\varepsilon^2h\), its support is exactly
\((gh)^G\).  By
Lemma~\ref{lem:Gamma3-kernel-centralizers},
\(C_G(gh)=\langle g,h\rangle\).  This group is abelian and generated
by two elements, so \(f_\Phi\) is trivial on \(C_G(gh)^3\).
Equation~\eqref{eq:Gamma3-local-commutator} therefore shows that the
\(\Phi_{gh}\)-twisted group algebra of \(C_G(gh)\) is commutative.
In particular, all its irreducible modules are one-dimensional.  The
description of simple objects in
Proposition~\ref{prop:simple-objects} now shows that \(X_1\) is simple
exactly when \((X_1)_{gh}\) is one-dimensional.  This proves the
assertions.
\end{proof}

We next record the part of the second-adjoint calculation used below.
The full coefficient calculation is given in
Appendix~\ref{app:Gamma3-32-obstruction}.

\begin{lemma}
\label{lem:Gamma3-32-second-adjoint-coefficients}
Assume \(\Delta_1=1\), and let
\(u_i=\varepsilon^i\rhd u_0\) and \(P_i=v_i\otimes u_i\).  Then the
following statements hold.

(1)
There are nonzero scalars \(c_1,c_2,t_0,t_1,t_2\) such that
\begin{equation}
\label{eq:Gamma3-32-X2-summary-central}
 \varphi_2^\Phi(P_0)=(1+s)P_0+c_1P_1+c_2P_2,
 \qquad
 \varepsilon\rhd P_i=t_iP_{i+1},
\end{equation}
where the indices are taken modulo three.

(2)
The component \((X_2)_{g^2h}\) is nonzero.  If
\((X_2)_{\varepsilon g^2h}=0\), then \(\rho(g)=-1\).

(3)
For the scalars in (1), if \(\rho(g)=-1\), let
\begin{equation}
\label{eq:Gamma3-32-summary-lambda-r}
 \lambda=\frac{c_1t_1}{c_2},
 \qquad
 r=\frac{c_1^2t_1}{c_2t_0}.
\end{equation}
Then
\begin{equation}
\label{eq:Gamma3-32-summary-identities}
 s^3=\kappa_\Phi(\varepsilon),
 \qquad
 \lambda^3=t_0t_1t_2,
 \qquad
 r^2=s.
\end{equation}
\end{lemma}

\begin{proof}
(1)  
Lemma~\ref{lem:Gamma3-32-second-adjoint-expansion} gives the first
formula in \eqref{eq:Gamma3-32-X2-summary-central} and
\(c_1c_2\ne0\).  Equations~\eqref{eq:Gamma3-32-t0}--%
\eqref{eq:Gamma3-32-t2} give the second formula and show that
\(t_0,t_1,t_2\) are nonzero.

(2) The tensors \(P_0,P_1,P_2\) lie in distinct
direct summands.  Hence the first formula in
\eqref{eq:Gamma3-32-X2-summary-central}, together with
\(c_1c_2\ne0\), shows that
\(\varphi_2^\Phi(P_0)\ne0\).  Since this vector has degree \(g^2h\),
we obtain \((X_2)_{g^2h}\ne0\).  Moreover,
\(\varphi_2^\Phi(v_2\otimes u_0)\) belongs to
\((X_2)_{\varepsilon g^2h}\).  If this component is zero, then
\(\varphi_2^\Phi(v_2\otimes u_0)=0\).  By
\eqref{eq:Gamma3-32-noncentral-vector}, its two terms lie in distinct
direct summands, and the coefficient of \(v_2\otimes u_0\) is
\(1+\rho(g)\).  Thus \(1+\rho(g)=0\), so \(\rho(g)=-1\).

(3) If \(\rho(g)=-1\), then
Lemma~\ref{lem:Gamma3-32-bar-reduction} applies and gives precisely
the three identities in \eqref{eq:Gamma3-32-summary-identities}.
\end{proof}

\subsubsection{The \texorpdfstring{$(3,1)$}{(3,1)} case with
	\texorpdfstring{$\dim\tau=2$}{dim tau=2}}

Let \(V=M(g,\rho)\) and \(W=M(h,\tau)\), with
\(\dim\rho=1\) and \(\dim\tau=2\).  Define
\begin{equation}
	\label{eq:Gamma3-Delta3}
	\Delta _3=
	\rho(h)^2\Phi_h(g,g)\tau(g^2)
	\in\operatorname{End}(W_h).
\end{equation}

\begin{lemma}
	\label{lem:Gamma3-dim2-normal-form}
	Define
	\begin{equation}
		\label{eq:Gamma3-a-epsilon}
		a_\varepsilon=
		\frac{\Phi_h(\varepsilon,g)}
		{\Phi_h(g,\varepsilon^2)\Phi_h(\varepsilon,\varepsilon)}.
	\end{equation}
	There exist \(b\in\mathbb C^\times\), a primitive third root of unity
	\(\zeta_3\), and a basis \(w,\overline w\) of \(W_h\), with
	\(\Delta_3=\rho(h)^2bI\), such that
	\begin{equation}
		\label{eq:Gamma3-dim2-standard-basis}
		\begin{array}{c|cc}
			&w&\overline w\\ \hline
			\varepsilon&\zeta_3 a_\varepsilon^{-1}w
			&\zeta_3^2a_\varepsilon^{-1}\overline w\\
			g&\overline w&bw
		\end{array}.
	\end{equation}
\end{lemma}

\begin{proof}
	Equation~\eqref{eq:Gamma3-local-commutator} and
	Lemma~\ref{lem:Gamma3-projective-dimensions} give
	\(\Phi_h(h,a)=\Phi_h(a,h)\) and
	\(\Phi_h(g^2,a)=\Phi_h(a,g^2)\) for
	\(a\in\{\varepsilon,g,h\}\).  Hence \(\tau(h)\) and
	\(\tau(g^2)\) commute with \(\tau(G)\), so both are scalar by
	Schur's lemma.  In particular, \(\Delta_3\) is scalar.
	
	Let \(E_\varepsilon=a_\varepsilon\tau(\varepsilon)\).
	Equation~\eqref{eq:YD-projective-action},  {applied to the pairs
	\((\varepsilon,g)\), \((\varepsilon,\varepsilon)\), and
	\((g,\varepsilon^2)\)}, gives
	\(E_\varepsilon\tau(g)=\tau(g)E_\varepsilon^2\).
	Since \(\tau(g)^2=\Phi_h(g,g)\tau(g^2)\) is scalar, conjugating
	twice by \(\tau(g)\) gives \(E_\varepsilon=E_\varepsilon^4\).
	As \(E_\varepsilon\) is invertible, it follows that
	\begin{equation}
		\label{eq:Gamma3-dim2-E-cubic}
		E_\varepsilon^3=I.
	\end{equation}
	Choose an eigenvector \(0\ne w\in W_h\) of \(E_\varepsilon\), write
	\(E_\varepsilon w=\mu w\), and let
	\(\overline w=g\rhd w\).  If \(\overline w\in\mathbb Cw\), then
	\(\mathbb Cw\) is \(G\)-stable, contrary to the irreducibility of
	\(\tau\). Hence \(w,\overline w\) is a basis.
	Equation~\eqref{eq:YD-projective-action} gives
	\(g\rhd\overline w=\Phi_h(g,g)\tau(g^2)w=bw\) for a nonzero
	scalar \(b\). Hence \(\Delta_3=\rho(h)^2bI\).  The relation above
	also gives \(E_\varepsilon\overline w=\mu^2\overline w\), while
	\eqref{eq:Gamma3-dim2-E-cubic} gives \(\mu^3=1\).
	If \(\mu=1\), then \(\langle\varepsilon,g^2,h\rangle\) acts by
	scalars.  An eigenline of \(\tau(g)\) would then be \(G\)-stable,
	contrary to irreducibility.  Hence \(\mu\ne1\), so \(\mu\) is a
	primitive third root of unity.  Let \(\zeta_3=\mu\).  Then the action
	is given by \eqref{eq:Gamma3-dim2-standard-basis}.
\end{proof}

\begin{lemma}
	\label{lem:Gamma3-dim2-X1-X2-support}
	The object $X_1$ is simple if and only if \(\Delta_3=I\).
	Under this condition, for \(0\ne v\in V_g\), one has
	\[
	(X_1)_{gh}=\mathbb Cx_1,
	\qquad x_1=v\otimes(w-\rho(h)\overline w).
	\]
	One also has \(X_2\ne0\).  If, in addition, \(X_2\) is simple,
	then \(\rho(g)=-1\).
\end{lemma}

\begin{proof}
	The braiding formula and \eqref{eq:Gamma3-dim2-standard-basis} give
	\[
	\begin{aligned}
		\varphi_1^\Phi(v\otimes w)
		&=v\otimes(w-\rho(h)\overline w),\\
		\varphi_1^\Phi(v\otimes\overline w)
		&=v\otimes(\overline w-\rho(h)b w).
	\end{aligned}
	\]
	These vectors span the same line if and only if
	\(\rho(h)^2b=1\), which is the equation \(\Delta_3=I\).  Since
	the first vector is nonzero, \(X_1\) is \(G\)-stable, and all
	homogeneous degrees of \(V\otimes W\) belong to \((gh)^G\), the
	support of \(X_1\) is \((gh)^G\).  By
	Lemma~\ref{lem:Gamma3-kernel-centralizers},
	\(C_G(gh)=\langle g,h\rangle\).
	Lemma~\ref{lem:Gamma3-projective-dimensions} and
	\eqref{eq:Gamma3-local-commutator} show that the
	\(\Phi_{gh}\)-twisted group algebra of
	\(C_G(gh)\) is commutative.  Proposition~\ref{prop:simple-objects}
	therefore shows that \(X_1\) is simple exactly when its degree-\(gh\)
	component is one-dimensional.  This proves the first assertion and
	the formula for \(x_1\).
	
	Assume \(\Delta_3=I\).  Equation~\eqref{eq:tensor-product} gives
	\[
	g\rhd x_1
	=-\Phi^g(g,h)\rho(g)\rho(h)^{-1}x_1.
	\]
Because \(g\) and \(h\) commute,
	\(\Phi^g(g,h)=\Phi_g(g,h)\), and
	\(\rho(gh)=\rho(g)\rho(h)/\Phi_g(g,h)\).  Hence the
summand \(-c_{X_1,V}c_{V,X_1}(P_0)\), where
	\(P_0=v\otimes x_1\), is \(\rho(g)^2P_0\).  Moreover,
	\(c_{1,2}^\Phi\) contributes \(\rho(g)\) when it interchanges the
first two copies of \(v\), and
	\(\varphi_1^\Phi(v\otimes(w-\rho(h)\overline w))=2x_1\). Thus the
 {last term in the recursive formula for \(\varphi_2^\Phi\)} is \(2\rho(g)P_0\).  Together with the
 {term contributed by the identity map}, this gives
	\begin{equation}
		\label{eq:Gamma3-dim2-central-X2}
		\varphi_2^\Phi(P_0)=(1+\rho(g))^2P_0.
	\end{equation}
For \((\varepsilon^2\rhd v)\otimes x_1\), the same three summands
lie respectively in
	\[
	V_{g\varepsilon^2}\otimes(X_1)_{gh},\qquad
	V_g\otimes(X_1)_{\varepsilon^2gh},\qquad
	V_{g\varepsilon}\otimes(X_1)_{\varepsilon gh}.
	\]
	The  {term contributed by the identity map} is nonzero, so the image is nonzero and has
	degree \(\varepsilon g^2h\).  Thus \(X_2\ne0\).
	Since \(g^2h\) is central whereas
	\(\varepsilon g^2h\) is not, these degrees belong to different
	conjugacy classes.  If \(X_2\) is simple, its degree-\(g^2h\)
	component must therefore vanish.  Since \(P_0\ne0\),
	\eqref{eq:Gamma3-dim2-central-X2} gives \(\rho(g)=-1\).
\end{proof}

\begin{lemma}
	\label{lem:Gamma3-dim2-X2-summary}
	Retain \(v,x_1\) from
	Lemma~\ref{lem:Gamma3-dim2-X1-X2-support}.  Assume
	\(\Delta_3=I\), \(\rho(g)=-1\), and let
	$
	v_i=\varepsilon^i\rhd v, \ x_{1,i}=\varepsilon^i\rhd x_1$ for 
	$ i\in\mathbb Z/3\mathbb Z,
	$
	and
	\[
	e_0=v_2\otimes x_{1,0},
	\qquad e_1=v_0\otimes x_{1,1},
	\qquad e_2=v_1\otimes x_{1,2}.
	\]
	These vectors are linearly independent.  There are nonzero scalars
	\(\mathfrak a,\mathfrak b,m_0,m_1,m_2\) such that
	\begin{equation}
		\label{eq:Gamma3-dim2-X2-summary-vector}
		x_2:=\varphi_2^\Phi(e_0)
		=e_0+\mathfrak a e_1+\mathfrak b e_2
	\end{equation}
	and
	\begin{equation}
		\label{eq:Gamma3-dim2-X2-summary-action}
		\varepsilon\rhd e_0=m_0(\mathfrak a e_1),
		\qquad
		\varepsilon\rhd(\mathfrak a e_1)=m_1(\mathfrak b e_2),
		\qquad
		\varepsilon\rhd(\mathfrak b e_2)=m_2e_0.
	\end{equation}
They satisfy
	\begin{equation}
		\label{eq:Gamma3-dim2-X2-summary-reduction}
		m_2=m_0\kappa_\Phi(\varepsilon)^{-1}.
	\end{equation}
\end{lemma}

\begin{proof}
{
Lemma~\ref{lem:Gamma3-dim2-second-adjoint-expansion} gives the
linear independence, the expansion
\eqref{eq:Gamma3-dim2-X2-summary-vector}, and the  action
\eqref{eq:Gamma3-dim2-X2-summary-action}, with all the stated
coefficients nonzero. Lemma~\ref{lem:Gamma3-dim2-bar-reduction}
gives \eqref{eq:Gamma3-dim2-X2-summary-reduction}.}
\end{proof}

\subsubsection{The \texorpdfstring{$(3,1)$}{(3,1)} case with
\texorpdfstring{$\dim\tau=1$}{dim tau=1}}

Let \(V=M(g,\rho)\) and \(W=M(h,\tau)\), with
\(\dim\rho=\dim\tau=1\).
Choose nonzero vectors
\(v\in V_g\) and \(w\in W_h\), and let
\(\widetilde x_1=v\otimes w\).

\begin{lemma}
\label{lem:Gamma3-dim1-X1-X2}
One has
\[
 X_1=0
 \quad\Longleftrightarrow\quad
 \rho(h)\tau(g)=1.
\]
If \(\rho(h)\tau(g)\ne1\), then
\(X_1\simeq M(gh,\rho_1)\) is simple, where
$
 \rho_1(A)=\Phi^A(g,h)\rho(A)\tau(A)$ for $A\in C_G(gh)$.
Under the same assumption, one has \(X_2\ne0\).  If \(X_2\) is
simple, then
\begin{equation}
\label{eq:Gamma3-dim1-support-equation}
 \bigl(1+\rho(g)\bigr)
 \bigl(1-\rho(g)\rho(h)\tau(g)\bigr)=0.
\end{equation}
\end{lemma}

\begin{proof}
The braiding formula gives
\(\varphi_1^\Phi(v\otimes w)
=\bigl(1-\rho(h)\tau(g)\bigr)v\otimes w\), which proves the
vanishing criterion.  Assume \(\rho(h)\tau(g)\ne1\).  The \(G\)-translates of
\(v\otimes w\) span \(V\otimes W\), and \(\varphi_1^\Phi\) 
 {commutes with the $G$-action}. Hence the translates of \(\widetilde x_1\) span
\(X_1\), whose support is \((gh)^G\).  Since \(h\) is central,
Lemma~\ref{lem:Gamma3-kernel-centralizers} gives
\(C_G(gh)=C_G(g)=\langle g,h\rangle\), and
\eqref{eq:tensor-product} gives
\[
 A\rhd\widetilde x_1
 =\Phi^A(g,h)\rho(A)\tau(A)\widetilde x_1
 \qquad(A\in C_G(gh)).
\]
Thus \((X_1)_{gh}=\mathbb C\widetilde x_1\) affords the
one-dimensional projective representation \(\rho_1\).
Proposition~\ref{prop:simple-objects} now gives
\(X_1\simeq M(gh,\rho_1)\), and hence \(X_1\) is simple.

For the second adjoint, the double braiding and the  {last term in the recursive formula for \(\varphi_2^\Phi\)} give, respectively,
\[
 \begin{aligned}
 c_{X_1,V}c_{V,X_1}(v\otimes\widetilde x_1)
 &=\rho(g)^2\rho(h)\tau(g)\,v\otimes\widetilde x_1,\\
 (\operatorname{id}_V\otimes\varphi_1^\Phi)c_{1,2}^\Phi
 (v\otimes\widetilde x_1)
 &=\rho(g)\bigl(1-\rho(h)\tau(g)\bigr)
  v\otimes\widetilde x_1.
 \end{aligned}
\]
Together with the  {term contributed by the identity map}, they give
\begin{equation}
\label{eq:Gamma3-dim1-central-X2}
 \varphi_2^\Phi(v\otimes\widetilde x_1)
 =\bigl(1+\rho(g)\bigr)
  \bigl(1-\rho(g)\rho(h)\tau(g)\bigr)
  v\otimes\widetilde x_1.
\end{equation}
 {The three terms in the recursive formula for \(\varphi_2^\Phi\), applied to}
\((\varepsilon^2\rhd v)\otimes\widetilde x_1\) {,} lie, respectively, in
\[
 V_{g\varepsilon^2}\otimes(X_1)_{gh},\qquad
 V_g\otimes(X_1)_{\varepsilon^2gh},\qquad
 V_{g\varepsilon}\otimes(X_1)_{\varepsilon gh}.
\]
These direct summands are distinct, and the  {term contributed by the identity map} is
nonzero.  Hence the image is nonzero and has degree
\(\varepsilon g^2h\), so \(X_2\ne0\).  Since \(g^2h\) is central whereas
\(\varepsilon g^2h\) is not, simplicity of \(X_2\) forces the
degree-\(g^2h\) vector in
\eqref{eq:Gamma3-dim1-central-X2} to vanish.  This gives
\eqref{eq:Gamma3-dim1-support-equation}.
\end{proof}

\begin{lemma}
\label{lem:Gamma3-dim1-X2-summary}
Assume \(\rho(h)\tau(g)\ne1\), and let
\(x_{1,i}=\varepsilon^i\rhd\widetilde x_1\) and
\(v_i=\varepsilon^i\rhd v\) for \(i=0,1,2\).  Let
\begin{equation*}
e_0=v_2\otimes x_{1,0},
\qquad e_1=v_0\otimes x_{1,1},
\qquad e_2=v_1\otimes x_{1,2}.
\end{equation*}
These vectors lie in the component of degree
\(\varepsilon g^2h\) and are linearly independent.
There are \(A,B\in\mathbb C^\times\) such that
\begin{equation}
\label{eq:Gamma3-dim1-X2-summary-vector}
 x_2:=\varphi_2^\Phi(e_0)
 =e_0+A e_1+B e_2.
\end{equation}
Define \(m_0,m_1,m_2\in\mathbb C^\times\) by
\begin{equation}
\label{eq:Gamma3-dim1-m-definition}
 \varepsilon\rhd e_0=m_0(Ae_1),
 \qquad
 \varepsilon\rhd(Ae_1)=m_1(Be_2),
 \qquad
 \varepsilon\rhd(Be_2)=m_2e_0.
\end{equation}
These scalars satisfy
\begin{equation}
\label{eq:Gamma3-dim1-m-ratios}
 \begin{aligned}
 \frac{m_0}{m_1}
 &=\kappa_\Phi(\varepsilon)
   \frac{1-\rho(h)\tau(g)}
        {\rho(g)^3\rho(h)^2\tau(g)^2},\\
 \frac{m_0}{m_2}
 &=-\frac{\kappa_\Phi(\varepsilon)}
 {\rho(g)^3\rho(h)\tau(g)
  \bigl(1-\rho(h)\tau(g)\bigr)},\\
 \frac{m_1}{m_2}
 &=-\frac{\rho(h)\tau(g)}
 {\bigl(1-\rho(h)\tau(g)\bigr)^2}.
 \end{aligned}
\end{equation}
\end{lemma}

\begin{proof}
Lemma~\ref{lem:Gamma3-dim1-three-elementary-tensors} shows that
\(e_0,e_1,e_2\) have degree \(\varepsilon g^2h\) and are linearly
independent.  It also gives nonzero scalars
\(\mathfrak a,\mathfrak b\) such that
\[
 \begin{aligned}
 A&=-\rho(g)^2\rho(h)\tau(g)\mathfrak a,\\
 B&=\rho(g)\bigl(1-\rho(h)\tau(g)\bigr)\mathfrak b.
 \end{aligned}
\]
With these choices of \(A,B\),
\eqref{eq:Gamma3-dim1-three-elementary-vector} gives
\eqref{eq:Gamma3-dim1-X2-summary-vector}.
The assumption \(\rho(h)\tau(g)\ne1\) ensures that
\(A,B\ne0\).
{
The  action in
Lemma~\ref{lem:Gamma3-dim1-three-elementary-tensors} therefore
gives nonzero scalars \(m_0,m_1,m_2\) satisfying
\eqref{eq:Gamma3-dim1-m-definition}.}

Let \(\mu_0,\mu_1,\mu_2\) be the nonzero scalars in
Lemma~\ref{lem:Gamma3-dim1-chain-identification}.
Multiplying the identities in
\eqref{eq:Gamma3-dim1-m-mu} gives
\(\mu_0\mu_1\mu_2=m_0m_1m_2\).
Since \(\varepsilon^3=1\), the  action in
\eqref{eq:Gamma3-dim1-m-definition} and
\eqref{eq:YD-projective-action} give
\[
\mu_0\mu_1\mu_2
=
\Phi_{\varepsilon g^2h}(\varepsilon,\varepsilon)
\Phi_{\varepsilon g^2h}(\varepsilon^2,\varepsilon).
\]
By Lemma~\ref{lem:Gamma3-dim2-bar-reduction},
the two chains in
\eqref{eq:Gamma3-dim2-comparison-cycles} are boundaries.
 {Evaluating \(\pi^*\Phi\) on these chains} and using the preceding
identity gives
\[
\frac{\mu_0}{\mu_2}=\kappa_\Phi(\varepsilon),
\qquad
\frac{\mu_0^2}{\mu_1\mu_2}
=\kappa_\Phi(\varepsilon)^{-1}.
\]
Consequently,
\[
\mu_2=\mu_0\kappa_\Phi(\varepsilon)^{-1},
\qquad
\mu_1=\mu_0\kappa_\Phi(\varepsilon)^2
     =\mu_0\kappa_\Phi(\varepsilon)^{-1},
\]
where the last equality uses
\(\kappa_\Phi(\varepsilon)^3=1\).
Substituting these values into
\eqref{eq:Gamma3-dim1-m-mu} proves
\eqref{eq:Gamma3-dim1-m-ratios}.
\end{proof}

\subsection{Vanishing of the cocycle obstruction}
\label{subsec:Gamma3-branch-obstructions}

The preceding adjoint calculations give the following three results.
Throughout this subsection, write
\(X_n=(\operatorname{ad}V)^n(W)\).

\begin{prop}
\label{prop:Gamma3-32-X2-obstruction}
Let \(V=M(g,\rho)\) and \(W=M(\varepsilon h,\sigma)\) have supports
as in the \((3,2)\) case, with \(\dim\rho=\dim\sigma=1\).  If \(X_1\)
and \(X_2\) are nonzero and simple, then
\(\rho(g)=-1\) and \(\kappa_\Phi(\varepsilon)=1\).
\end{prop}

\begin{proof}
Simplicity of \(X_1\) gives \(\Delta_1=1\) by
Lemma~\ref{lem:Gamma3-32-X1}, so
Lemma~\ref{lem:Gamma3-32-second-adjoint-coefficients} applies and gives
\((X_2)_{g^2h}\ne0\).  Since \(g^2h\) is central and \(X_2\) is
simple, Proposition~\ref{prop:simple-objects} gives
\(\operatorname{supp}X_2=\{g^2h\}\).  Thus
\((X_2)_{\varepsilon g^2h}=0\), and
Lemma~\ref{lem:Gamma3-32-second-adjoint-coefficients} yields
\(\rho(g)=-1\).

Let \(z_0=\varphi_2^\Phi(P_0)\), and retain the scalars from
Lemma~\ref{lem:Gamma3-32-second-adjoint-coefficients}.  On the span of
\(P_0,P_1,P_2\), define
\[
 T=\lambda^{-1}(\varepsilon\rhd-),
 \qquad
 \widetilde P_0=P_0,
 \quad \widetilde P_1=\lambda^{-1}t_0P_1,
 \quad \widetilde P_2=\lambda^{-2}t_0t_1P_2.
\]
Equations~\eqref{eq:Gamma3-32-X2-summary-central} and
\eqref{eq:Gamma3-32-summary-identities} give
\[
 T\widetilde P_0=\widetilde P_1,
 \qquad T\widetilde P_1=\widetilde P_2,
 \qquad T\widetilde P_2=\widetilde P_0.
\]
Furthermore,
\begin{align*}
 z_0
 =(1+s)\widetilde P_0
   +\frac{c_1\lambda}{t_0}\widetilde P_1
   +\frac{c_2\lambda^2}{t_0t_1}\widetilde P_2
 =(1+s)\widetilde P_0+r\widetilde P_1+r\widetilde P_2.
\end{align*}
Equation~\eqref{eq:Gamma3-32-summary-lambda-r} gives the last
equality.  The \(P_i\) lie in distinct direct summands, so the
\(\widetilde P_i\) are linearly independent.  Since \(r^2=s\), the
coordinate matrix of \(z_0,Tz_0,T^2z_0\) with respect to the ordered
basis \(\widetilde P_0,\widetilde P_1,\widetilde P_2\) has determinant
\[
 \det\begin{pmatrix}
  1+s&r&r\\
  r&1+s&r\\
  r&r&1+s
 \end{pmatrix}
 =((1+s)+2r)((1+s)-r)^2
 =(1+r^3)^2.
\]
Since \(g^2h\) is central and \(X_2\) is \(G\)-stable, these three
vectors belong to \((X_2)_{g^2h}\).  If
\(\kappa_\Phi(\varepsilon)\ne1\), the determinant is nonzero: if it
vanished, then \(r^3=-1\), whereas
\(\kappa_\Phi(\varepsilon)=s^3=r^6=1\).  Hence
\((X_2)_{g^2h}\) contains three linearly independent vectors.

Since \(g^2h\) is central and
\(G=\langle\varepsilon,g,g^2h\rangle\), the proof of
Lemma~\ref{lem:Gamma3-projective-dimensions}, with \(h\) replaced by
\(g^2h\), shows that every irreducible
\(\Phi_{g^2h}\)-projective \(G\)-module has dimension at most two.
 {By Proposition~\ref{prop:simple-objects}, the simple object \(X_2\)
with support \(\{g^2h\}\) therefore has dimension at most two,
contradicting the existence of the three linearly independent vectors above.}  Hence \(\kappa_\Phi(\varepsilon)=1\).
\end{proof}

\begin{prop}
\label{prop:Gamma3-dim2-X2-obstruction}
Let \(V=M(g,\rho)\) and \(W=M(h,\tau)\) have supports as in the
\((3,1)\) case, with \(\dim\rho=1\) and \(\dim\tau=2\).  If \(X_1\)
and \(X_2\) are nonzero and simple, then
\(\Delta_3=I\), \(\rho(g)=-1\), and
\(\kappa_\Phi(\varepsilon)=1\).
\end{prop}

\begin{proof}
Lemma~\ref{lem:Gamma3-dim2-X1-X2-support} gives
\(\Delta_3=I\) and \(\rho(g)=-1\), so
Lemma~\ref{lem:Gamma3-dim2-X2-summary} applies.  The definitions of
\(e_0,e_1,e_2\), together with
\(g\varepsilon g^{-1}=\varepsilon^2\), show that all three have
degree \(\varepsilon g^2h\).  They are linearly independent, and
\eqref{eq:Gamma3-dim2-X2-summary-vector} therefore gives
\(0\ne x_2\in(X_2)_{\varepsilon g^2h}\).

Since \(g^2\) is central,
Lemma~\ref{lem:Gamma3-kernel-centralizers} gives
\(C_G(\varepsilon g^2h)=C_G(\varepsilon h)
=\langle\varepsilon,g^2,h\rangle\).  This group is abelian, and
Lemma~\ref{lem:Gamma3-projective-dimensions} together with
\eqref{eq:Gamma3-local-commutator} shows that its
\(\Phi_{\varepsilon g^2h}\)-twisted group algebra is commutative.
By Proposition~\ref{prop:simple-objects},
\((X_2)_{\varepsilon g^2h}\) is an irreducible module over this
algebra and is therefore one-dimensional.  Consequently,
\((X_2)_{\varepsilon g^2h}=\mathbb Cx_2\) and
\(\varepsilon\rhd x_2\in\mathbb Cx_2\).  Equations
\eqref{eq:Gamma3-dim2-X2-summary-vector} and
\eqref{eq:Gamma3-dim2-X2-summary-action} give
\(\varepsilon\rhd x_2
=m_2e_0+m_0\mathfrak a e_1+m_1\mathfrak b e_2\).
Since \(e_0,\mathfrak a e_1,\mathfrak b e_2\) are linearly
independent, comparison of coefficients with \(x_2\) gives
\(m_0=m_1=m_2\).  As these scalars are nonzero,
\eqref{eq:Gamma3-dim2-X2-summary-reduction} yields
\(\kappa_\Phi(\varepsilon)=1\).
\end{proof}
\begin{prop}
\label{prop:Gamma3-dim1-X2-obstruction}
Let \(V=M(g,\rho)\) and \(W=M(h,\tau)\) have supports as in the
\((3,1)\) case, with \(\dim\rho=\dim\tau=1\).  If \(X_1\) and
\(X_2\) are nonzero and simple, then
\(\kappa_\Phi(\varepsilon)=1\).
\end{prop}

\begin{proof}
Since \(X_1\ne0\), Lemma~\ref{lem:Gamma3-dim1-X1-X2} gives
\(\rho(h)\tau(g)\ne1\), so
Lemma~\ref{lem:Gamma3-dim1-X2-summary} applies.  It
shows that \(e_0,e_1,e_2\) lie in degree \(\varepsilon g^2h\) and
are linearly independent.  Hence
\eqref{eq:Gamma3-dim1-X2-summary-vector} gives
\(0\ne x_2\in(X_2)_{\varepsilon g^2h}\).

 {As in the proof of
Proposition~\ref{prop:Gamma3-dim2-X2-obstruction}, simplicity gives
\((X_2)_{\varepsilon g^2h}=\mathbb Cx_2\).  Hence}
\eqref{eq:Gamma3-dim1-X2-summary-vector} and
\eqref{eq:Gamma3-dim1-m-definition} give
\(\varepsilon\rhd x_2
=m_2e_0+m_0Ae_1+m_1Be_2\in\mathbb Cx_2\).
Since \(A,B\ne0\), comparison of coefficients gives
\(m_0=m_1=m_2\).  The last equality in
\eqref{eq:Gamma3-dim1-m-ratios} now gives
\[
 1-\rho(h)\tau(g)+\rho(h)^2\tau(g)^2=0.
\]
Hence
\(\rho(h)\tau(g)\bigl(1-\rho(h)\tau(g)\bigr)=1\) and
\(\rho(h)^3\tau(g)^3=-1\).  The second equality in
\eqref{eq:Gamma3-dim1-m-ratios} gives
\(\kappa_\Phi(\varepsilon)=-\rho(g)^3\), while
\eqref{eq:Gamma3-dim1-support-equation} gives either
\(\rho(g)=-1\) or \(\rho(g)\rho(h)\tau(g)=1\).  In the latter case,
\(\rho(g)=\bigl(\rho(h)\tau(g)\bigr)^{-1}\), so again
\(\rho(g)^3=-1\).
Therefore \(\kappa_\Phi(\varepsilon)=1\).
\end{proof}

The preceding results yield the following necessary condition
for finite-dimensionality.

\begin{prop}
\label{prop:Gamma3-local-obstruction}
Let \(P=(V,W)\) be in the \((3,2)\) case with
\(\dim\rho=\dim\sigma=1\), or in the \((3,1)\) case with
\(\dim\rho=1\) and \(\dim\tau\in\{1,2\}\).
If \(\dim\mathcal B(P)<\infty\), then
\(\kappa_\Phi(\varepsilon)=1\).
\end{prop}

\begin{proof}By Theorem~\ref{thm:Nichols-finiteness-criterion},
\(P\) admits all reflections and its Cartan graph is finite.
By \eqref{eq:indecomposable-first-adjoint} and braided
indecomposability, \(X_1\ne0\). Hence \( {a_{12}^{[P]}\leq -1}\), and
Lemma~\ref{lem:root-string-adjoint-simplicity} shows that \(X_1\) is
simple.
In the \((3,2)\) case, Lemma~\ref{lem:Gamma3-32-X1} gives
\(\Delta_1=1\), and
Lemma~\ref{lem:Gamma3-32-second-adjoint-coefficients}(2) gives
\(X_2\ne0\).  In the two \((3,1)\) cases, the same conclusion
follows from Lemmas~\ref{lem:Gamma3-dim2-X1-X2-support} and
\ref{lem:Gamma3-dim1-X1-X2}, respectively.
Thus \(a_{12}^{[P]} \leq -2\), and
Lemma~\ref{lem:root-string-adjoint-simplicity} also shows that \(X_2\)
is simple.  Propositions~\ref{prop:Gamma3-32-X2-obstruction},
\ref{prop:Gamma3-dim2-X2-obstruction}, and
\ref{prop:Gamma3-dim1-X2-obstruction} now give
\(\kappa_\Phi(\varepsilon)=1\) in the respective cases.
\end{proof}

\subsection{Cohomological trivialization and braided equivalence}
\label{subsec:Gamma3-obstruction-untwist}

For \(x\in\Gamma _3\), write \(\overline x=\pi(x)\), and let
\(\widetilde\Phi=\pi^*\Phi\).  The value of
\(\kappa_\Phi(\overline\varepsilon)\) determines the cohomology class
of this pullback.

We use the coboundary convention
\begin{equation}
\label{eq:Gamma3-delta-J}
(\delta J)(a,b,c)=
\frac{J(b,c)J(a,bc)}{J(ab,c)J(a,b)}.
\end{equation}

\begin{prop}
\label{prop:Gamma3-cohomological-obstruction}
The class \([\widetilde\Phi]\) is trivial in
\(H^3(\Gamma _3,\mathbb C^\times)\) if and only if
\(\kappa_\Phi(\overline\varepsilon)=1\).  In this case there is a
normalized two-cochain
\(J:\Gamma _3^2\to\mathbb C^\times\) satisfying
\(\widetilde\Phi=\delta J\).
\end{prop}

\begin{proof}
Since \(\mathbb C^\times\) is divisible, the Ext term in the
universal coefficient theorem vanishes, and evaluation gives an
isomorphism
\[
 H^3(\Gamma _3,\mathbb C^\times)
 \simeq
 \operatorname{Hom}\bigl(
 H_3(\Gamma _3,\mathbb Z),\mathbb C^\times
 \bigr).
\]
By Lemma~\ref{lem:Gamma3-third-homology},
\(H_3(\Gamma _3,\mathbb Z)\simeq\mathbb Z/3\) and
\([\mathsf K]\) is a generator.  Moreover,
\eqref{eq:Gamma3-kappa} and \eqref{eq:Gamma3-kappa-cycle} give
\(\langle\widetilde\Phi,\mathsf K\rangle
=\kappa_\Phi(\overline\varepsilon)\).  Thus
\([\widetilde\Phi]=0\) if and only if
\(\kappa_\Phi(\overline\varepsilon)=1\).

If the class is trivial, choose a two-cochain \(J\) with
\(\widetilde\Phi=\delta J\).  Since \(\widetilde\Phi\) is normalized,
the cases in which the first or second argument in
\eqref{eq:Gamma3-delta-J} equals \(1\) give
\(J(1,a)=J(a,1)=J(1,1)\) for every \(a\in\Gamma _3\).
Dividing \(J\) by the constant \(J(1,1)\) leaves its coboundary
unchanged and makes \(J\) normalized.
\end{proof}

The cochain \(J\) trivializes \(\pi^*\Phi\) on \(\Gamma _3\). The
class of \(\Phi\) on \(G\) may remain nontrivial.
 { {We next regard these objects as \(\Gamma_3\)-graded objects by
lifting their supports, with \(\Gamma_3\) acting through \(\pi\).}}
\begin{prop}
\label{prop:Gamma3-universal-lift}
Let \(x\in\Gamma_3\), and let \(\rho\) be an irreducible
\(\Phi_{\overline x}\)-projective representation of
\(C_G(\overline x)\).
Then
\(\widetilde\rho=\rho\circ(\pi|_{C_{\Gamma_3}(x)})\)
is an irreducible \(\widetilde\Phi_x\)-projective
representation of \(C_{\Gamma_3}(x)\).

Via the bijection
\(\pi:x^{\Gamma_3}\longrightarrow\overline x^G\),
the objects \(M_{\Gamma_3}(x,\widetilde\rho)\) and
\(M_G(\overline x,\rho)\) can be realized on the same
vector space, with \(\Gamma_3\) acting through \(\pi\).
For any pair of such objects, the recursive adjoint maps
agree under these identifications.
 {The iterated adjoint objects have unchanged dimensions.
The Nichols algebras of the direct sum and of these adjoint objects
also have unchanged dimensions.}
\end{prop}

\begin{proof}
 {By \eqref{eq:Gamma3-centralizer-lift}, the map
\(\pi:C_{\Gamma_3}(x)\to C_G(\overline x)\) is surjective,
so pulling back \(\rho\) and its multiplier gives the irreducible
\(\widetilde\Phi_x\)-projective representation \(\widetilde\rho\).}
For \(y\in x^{\Gamma_3}\), assign degree \(y\) to the
component \(M_G(\overline x,\rho)_{\pi(y)}\).
This grading is well defined because
\(\pi:x^{\Gamma_3}\to\overline x^G\) is bijective.
Since
\[
\pi(aya^{-1})=\pi(a)\pi(y)\pi(a)^{-1},
\]
the action through \(\pi\) respects this grading, and the
Yetter--Drinfeld identities follow from those over \(G\).
The degree-\(x\) component affords \(\widetilde\rho\).
Hence Proposition~\ref{prop:simple-objects} identifies this
object with \(M_{\Gamma_3}(x,\widetilde\rho)\).

Since \(\widetilde\Phi=\pi^*\Phi\) and the action is
pulled back along \(\pi\), the associators and braidings
on corresponding tensor products agree.
Hence so do the recursive adjoint maps and quantum
symmetrizers, proving the remaining assertions.
\end{proof}

Assume now that \(\kappa_\Phi(\overline\varepsilon)=1\), and fix a
normalized \(J\) as in
Proposition~\ref{prop:Gamma3-cohomological-obstruction}.  For
\(x\in\Gamma _3\) and \(a\in C_{\Gamma _3}(x)\), let
\(L_x(a)=J(a,x)/J(x,a)\).
 {For \(a,b\in C_{\Gamma _3}(x)\), the identities
\(\widetilde\Phi=\delta J\), \(ax=xa\), and \(bx=xb\) give}
\begin{equation}
\label{eq:Gamma3-local-coboundary}
\begin{aligned}
 \widetilde\Phi_x(a,b)
 =\frac{(\delta J)(a,b,x)(\delta J)(x,a,b)}
         {(\delta J)(a,x,b)}
 =\frac{J(a,x)J(b,x)J(x,ab)}
         {J(x,a)J(x,b)J(ab,x)}
  =\frac{L_x(a)L_x(b)}{L_x(ab)}.
\end{aligned}
\end{equation}
For \(\rho,\widetilde\rho\) as in
Proposition~\ref{prop:Gamma3-universal-lift}, put
\begin{equation}
\label{eq:Gamma3-hat-character}
 \widehat\rho_x(a)=L_x(a)^{-1}\widetilde\rho(a)
 =\frac{J(x,a)}{J(a,x)}\rho(\pi(a)).
\end{equation}
The cochain \(J\) gives the following braided monoidal equivalence.
\begin{prop}
\label{prop:Gamma3-braided-untwisting}
There is a braided monoidal equivalence
\[
F_J:{}_{\Gamma_3}^{\Gamma_3}\mathcal{YD}^{\widetilde\Phi}
\longrightarrow{}_{\Gamma_3}^{\Gamma_3}\mathcal{YD}
\]
which is the identity on underlying graded vector spaces
and morphisms.
It preserves duals, iterated adjoint objects, all defined
reflections, and dimensions  of Nichols algebras.
\end{prop}

\begin{proof}
Define the action on \(F_J(V)\), for \(v_x\in V_x\), by
\begin{equation}
\label{eq:Gamma3-untwist-action}
a\rhd_Jv_x
=
\frac{J(axa^{-1},a)}{J(a,x)}(a\rhd v_x),
\end{equation}
and equip \(F_J\) with the tensor structure
\[
(F_J)_2(u_a\otimes v_b)
=
J(a,b)^{-1}u_a\otimes v_b.
\]
Since \(\widetilde\Phi=\delta J\) and \(J\) is normalized,
these formulas define a braided monoidal functor.
Replacing \(J\) by \(J^{-1}\) gives its inverse.
For \(a\in C_{\Gamma_3}(x)\), the factor in
\eqref{eq:Gamma3-untwist-action} is
\(J(x,a)/J(a,x)=L_x(a)^{-1}\).
Thus the inducing representation is
\(\widehat\rho_x\) from
\eqref{eq:Gamma3-hat-character}.

The assertions about duals, adjoint objects, and reflections
follow from \cite[Proposition~6.3]{reflection3}.
The assertion about Nichols algebras follows from
\eqref{eq:Nichols-under-equivalence}.
\end{proof}

For an inducing projective representation \(\xi\), write
\(\widehat\xi\) for the corresponding ordinary representation
defined by \eqref{eq:Gamma3-hat-character}.  For \(1\leq i\leq4\),
let \(\mathcal P_i\) denote, over \(\mathbb C\), the ordinary class
in row \(i\) of
\cite[Definition~7.1 and Table~2]{rank2-classification}, with the
central generator denoted there by \(z\) identified with \(h\).

\begin{prop}
\label{prop:Gamma3-parameter-comparison}
Assume \(\kappa_\Phi(\overline\varepsilon)=1\).  For every \(J\) as in
Proposition~\ref{prop:Gamma3-cohomological-obstruction} and every
\(1\leq i\leq4\),
a pair belongs to \(\mathcal P_i^\Phi\) if and only if the ordinary
pair obtained  {by the construction in
Proposition~\ref{prop:Gamma3-universal-lift} followed by \(F_J\)}
belongs to \(\mathcal P_i\).
\end{prop}

\begin{proof}
 {The constructions in }Propositions~\ref{prop:Gamma3-universal-lift} and
\ref{prop:Gamma3-braided-untwisting} preserve simplicity
and the dimensions of the inducing representations,
and identify the supports via \(\pi\).
 {By definition, \(L_x(y)L_y(x)=1\) whenever \(xy=yx\),
and \(L_x(x)=1\).}
Hence \eqref{eq:Gamma3-hat-character} gives
\[
\widehat\rho(g)=\rho(g),\qquad
\widehat\sigma(\varepsilon h)=\sigma(\varepsilon h),
\qquad
\widehat\tau(h)=\tau(h).
\]

For \(i=1\), the conditions
\(\rho(g)=\sigma(\varepsilon h)=-1\) are therefore
preserved.
 {Both constructions preserve simplicity of the first adjoint object.}
Lemma~\ref{lem:Gamma3-32-X1} and the same calculation
with \(\Phi=1\) give
\begin{equation}
\label{eq:Gamma3-translation-32}
\Delta_1=1
\quad\Longleftrightarrow\quad
\widehat\rho(h^2)\widehat\sigma(\varepsilon g^2)=1.
\end{equation}
Indeed, the ordinary simplicity criterion is
\(\widehat\rho(h)^2\widehat\sigma(\varepsilon^2)^2
\widehat\sigma(g^2)=1\), which has the displayed form
because the hatted representations are characters and
\(\varepsilon^3=1\).
Assume \(\Delta_1=1\) and \(\rho(g)=-1\).
By Lemma~\ref{lem:Gamma3-32-X1} and
\eqref{eq:Gamma3-untwist-action},  {\(g\) acts on the degree-\(gh\) component of \(F_J(X_1)\)
by the scalar \(L_{gh}(g)^{-1}\chi\).}
Since \(L_{gh}(g)L_g(gh)=1\), the definitions in
\eqref{eq:Gamma3-a},
\eqref{eq:Gamma3-b0-b1}, and
\eqref{eq:Gamma3-Delta1-s}, applied to the ordinary pair,
give
\begin{equation}
\label{eq:Gamma3-translation-s}
\begin{aligned}
s
=-\bigl(L_{gh}(g)^{-1}\chi\bigr)\widehat\rho(gh)
=\widehat\rho(g)^2\widehat\sigma(\varepsilon)^{-2}
 =\widehat\sigma(\varepsilon).
\end{aligned}
\end{equation}
The last equality uses
\(\widehat\rho(g)=-1\) and
\(\widehat\sigma(\varepsilon)^3=1\).
Thus \(1+s+s^2=0\) is equivalent to
\(1+\widehat\sigma(\varepsilon)
+\widehat\sigma(\varepsilon)^2=0\).
This identifies the conditions defining
\(\mathcal P_1^\Phi\) with those defining \(\mathcal P_1\).

For \(i=2\), the calculation above again gives
\eqref{eq:Gamma3-translation-32} and
\eqref{eq:Gamma3-translation-s}.
Thus \(s=1\) is equivalent to
\(\widehat\sigma(\varepsilon)=1\).
Under this condition,
\(\widehat\sigma(\varepsilon h)=\widehat\sigma(h)\),
and the defining conditions become
\[
\widehat\rho(g)=\widehat\sigma(h)=-1,\qquad
\widehat\rho(h^2)\widehat\sigma(g^2)=1,\qquad
\widehat\sigma(\varepsilon)=1.
\]
These are precisely the conditions defining \(\mathcal P_2\).

For \(i=3\), the conditions
\(\dim\tau=2\), \(\rho(g)=-1\), and \(\tau(h)=-I\)
are preserved.
Since
\(\tau(g)^2=\Phi_h(g,g)\tau(g^2)\) and
\(L_g(h)L_h(g)=1\), one has
\[
\begin{aligned}
\Delta_3
=\rho(h)^2\tau(g)^2
=\widehat\rho(h)^2\widehat\tau(g)^2
 =\widehat\rho(h)^2\widehat\tau(g^2).
\end{aligned}
\]
Thus \(\Delta_3=I\) is exactly the remaining condition
defining \(\mathcal P_3\).

For \(i=4\), the conditions
\(\dim\tau=1\) and \(\rho(g)=-1\) are preserved.
The identity \(L_g(h)L_h(g)=1\), together with
 {the multiplicativity of \(\widehat\rho\) and \(\widehat\tau\)} and
\(\widehat\tau(h)=\tau(h)\), gives
\[
\begin{aligned}
\widehat\rho(h)\widehat\tau(g)
&=\rho(h)\tau(g),\\
\widehat\rho(h)\widehat\tau(gh)
&=\rho(h)\tau(g)\tau(h).
\end{aligned}
\]
Hence the quadratic and product conditions in
\eqref{eq:Gamma3-P4} become
\[
1-\widehat\rho(h)\widehat\tau(g)
+\widehat\rho(h)^2\widehat\tau(g)^2=0,
\qquad
\widehat\rho(h)\widehat\tau(gh)=1.
\]
These are precisely the remaining conditions
defining \(\mathcal P_4\).
\end{proof}

\subsection{ {Necessity and closure under reflections}}
\label{subsec:Gamma3-necessity-reflections}

\begin{prop}
\label{prop:Gamma3-necessity}
If \(\dim\mathcal B(V\oplus W)<\infty\), then, after possibly
interchanging \(V\) and \(W\), generators
\(\varepsilon,g,h\in G\) inducing an epimorphism
\(\Gamma _3\twoheadrightarrow G\) can be chosen so that the resulting
ordered pair belongs to one of
\(\mathcal P_1^\Phi,\ldots,\mathcal P_4^\Phi\).
\end{prop}

\begin{proof}
{
 {By Proposition~\ref{prop:Gamma3-support-and-dimensions}, we may
interchange \(V\) and \(W\) and choose generators as required.}
Since interchanging the entries preserves their direct sum,
Proposition~\ref{prop:Gamma3-local-obstruction} gives
\(\kappa_\Phi(\varepsilon)=1\) for these generators.}

Choose \(J\) as in
Proposition~\ref{prop:Gamma3-cohomological-obstruction}, and let
\((V^J,W^J)\) be the ordinary pair over \(\Gamma _3\) obtained by
 {the construction in
Proposition~\ref{prop:Gamma3-universal-lift} followed by \(F_J\)}.
 {The constructions in }Propositions
\ref{prop:Gamma3-universal-lift} and
\ref{prop:Gamma3-braided-untwisting} preserve  {simplicity,
iterated adjoint objects, and the dimensions of Nichols algebras.}    By
\eqref{eq:indecomposable-first-adjoint}, the original first adjoint
object is nonzero.  Hence $(V^J,W^J)$ is braided-indecomposable and
\[
 \begin{aligned}
 \dim\mathcal B(V^J\oplus W^J)
 =\dim\mathcal B(V\oplus W)<\infty {.}
 \end{aligned}
\]
The support bijection in
Proposition~\ref{prop:Gamma3-universal-lift} and the fact that \(F_J\)
does not change the grading give
\[
 \left\langle\operatorname{supp}(V^J\oplus W^J)\right\rangle
 =
 \begin{cases}
  \langle g,g\varepsilon,\varepsilon h\rangle,&(3,2),\\
  \langle g,g\varepsilon,h\rangle,&(3,1),
 \end{cases}
 =\Gamma _3.
\]
Thus all the hypotheses of
\cite[Theorem~2.1]{rank2-classification} are satisfied.

For these two support types over \(\mathbb C\), the alternatives in
\cite[Examples~1.9--1.11]{rank2-classification} are exactly the first
four classes in
\cite[Definition~7.1 and Table~2]{rank2-classification}. The support
sizes fix the ordering.  Hence
{
\((V^J,W^J)\in\mathcal P_i\) for some \(1\le i\le4\), and
Proposition~\ref{prop:Gamma3-parameter-comparison} gives
\((V,W)\in\mathcal P_i^\Phi\).}
\end{proof}

\begin{prop}
\label{prop:Gamma3-reflection-transport}
Let \(P=(V,W)\) belong to one of the four families
\(\mathcal P_i^\Phi\).  The reflections \(R_1(P)\) and \(R_2(P)\)
are defined, and  {the generators \(g',\varepsilon',h'\)
for the reflected pairs} may be chosen as
\begin{equation}
\label{eq:Gamma3-reflected-generators}
 \begin{array}{c|ccc}
 &g'&\varepsilon'&h'\\ \hline
 R_1&g^{-1}&\varepsilon&g^2h\\
 R_2&gh&\varepsilon^{-1}&h^{-1}.
 \end{array}
\end{equation}
\end{prop}

\begin{proof}
Every pair in the four families satisfies
\(\kappa_\Phi(\varepsilon)=1\), as observed after
\eqref{eq:Gamma3-P4}.
Let \((V^J,W^J)\) be the ordinary pair over \(\Gamma_3\)
obtained by lifting along \(\pi\) and applying \(F_J\).
By Proposition~\ref{prop:Gamma3-parameter-comparison},
this pair belongs to the corresponding ordinary family.

{
 {For these ordinary families,}
\cite[Corollaries~4.13--4.16, 5.10--5.11, and
6.7--6.8]{rank2-classification} give both reflections and the
generators in \eqref{eq:Gamma3-reflected-generators}, with the
central generator in that paper identified with \(h\).
 {The constructions in }Propositions~\ref{prop:Gamma3-universal-lift} and
\ref{prop:Gamma3-braided-untwisting} preserve duals, iterated
adjoint objects, and simplicity, so the reflections of \(P\)
are defined and have the generators obtained by applying
\(\pi\) to these triples.}

Finally, since \(g^2,h\in Z(\Gamma_3)\), both substitutions
in \eqref{eq:Gamma3-reflected-generators} preserve the
defining relations of \(\Gamma_3\).
Each substitution is its own inverse, so both define
automorphisms of \(\Gamma_3\).
Their images under \(\pi\) therefore give generating
triples of \(G\) satisfying the required relations.
\end{proof}
\begin{prop}
\label{prop:Gamma3-reflection-closure}
Every pair in one of \(\mathcal P_1^\Phi,\ldots,\mathcal P_4^\Phi\)
admits all reflections.  Every pair obtained by successive
reflections again belongs to one of these four families, and every
Cartan matrix in its Cartan graph is
$\begin{pmatrix}2&-2\\-1&2\end{pmatrix}$.
\end{prop}

\begin{proof}
{
Let \((V^J,W^J)\) be the  pair used in the proof of
Proposition~\ref{prop:Gamma3-reflection-transport}.
Its reflections and Cartan matrices are given by
\cite[Equations~(7.1)--(7.2) and Lemmas~7.3--7.4]
{rank2-classification}.
 {The constructions in }Propositions~\ref{prop:Gamma3-universal-lift} and
\ref{prop:Gamma3-braided-untwisting} preserve the vanishing of
iterated adjoint objects, hence the reflections and Cartan integers.
By Proposition~\ref{prop:Gamma3-reflection-transport}, each reflected
generator triple is induced by an automorphism of \(\Gamma_3\).
The pullback of \(\Phi\) for that triple is therefore again a
coboundary. Applying Proposition~\ref{prop:Gamma3-parameter-comparison}
to the reflected pairs gives}
\begin{equation*}
\mathcal P_1^\Phi
\xleftrightarrow{\ R_1\ }
\mathcal P_4^\Phi,
\qquad
\mathcal P_2^\Phi
\xleftrightarrow{\ R_1\ }
\mathcal P_3^\Phi,
\qquad
R_2(\mathcal P_j^\Phi)=\mathcal P_j^\Phi
\quad(1\le j\le4).
\end{equation*}
 {Thus every successive reflection remains in the four families.
Since the Cartan integers are preserved, the original pair also has
a standard Cartan graph of type \(B_2\).}
\end{proof}
\subsection{Sufficiency and dimension formulas}
\label{subsec:Gamma3-sufficiency}

\begin{cor}
\label{cor:Gamma3-Nichols-dimension}
Every pair in one of the four families has a finite-dimensional
Nichols algebra, and its dimension is given by
\eqref{eq:Gamma3-dimensions}.
\end{cor}

\begin{proof}
{
Every family satisfies \(\kappa_\Phi(\varepsilon)=1\), as observed
after \eqref{eq:Gamma3-P4}. Thus
Proposition~\ref{prop:Gamma3-parameter-comparison} associates to
\((V,W)\in\mathcal P_i^\Phi\) a pair in \(\mathcal P_i\).
Over \(\mathbb C\), \cite[Theorems~8.1--8.4]{rank2-classification}
give its Nichols algebra dimension as \(10368\) for \(i=1,4\) and
\(2304\) for \(i=2,3\).
 {The constructions in }Propositions~\ref{prop:Gamma3-universal-lift} and
\ref{prop:Gamma3-braided-untwisting} preserve this dimension,
proving the assertion.}
\end{proof}

\begin{proof}[Proof of Theorem~\ref{thm:Gamma3-classification}]
Proposition~\ref{prop:Gamma3-necessity} proves necessity.  Conversely,
interchanging the two entries does not change their direct sum, so
Corollary~\ref{cor:Gamma3-Nichols-dimension} proves sufficiency and
\eqref{eq:Gamma3-dimensions}.  Proposition
\ref{prop:Gamma3-reflection-closure} gives the asserted standard
Cartan graph of type \(B_2\).
\end{proof}

\section{The complete classification}
\label{sec:complete-classification}

 {The following theorem combines the classifications in the four preceding sections.}
For \(M=(M_1,M_2)\) and \(\omega\in S_2\), write
\(M^\omega=(M_{\omega^{-1}(1)},M_{\omega^{-1}(2)})\).

\begin{thm}
\label{thm:complete-classification}
Let \(G\) be a finite non-abelian group, let
\(\Phi\in Z^3(G,\mathbb C^\times)\) be normalized, and let
\(M=(V,W)\) be a tuple of finite-dimensional simple objects in
\({}_G^G\mathcal{YD}^{\Phi}\).  Assume that
\begin{equation}
\label{eq:complete-classification-assumptions}
G=\langle\operatorname{supp}V\cup\operatorname{supp}W\rangle,
\qquad
M\text{ is braided-indecomposable}.
\end{equation}
Then the following statements are equivalent.
\begin{enumerate}[label=\textnormal{(\arabic*)}]
\item
The Nichols algebra \(\mathcal B(V\oplus W)\) is finite-dimensional.

\item
The pair \(M\) admits all reflections, its Cartan graph
\(\mathcal G(M)\) is finite, and
\[
\dim\mathcal B(Q_i)<\infty
\qquad
\text{for every }Q=(Q_1,Q_2)\in\mathcal F_2(M)
\text{ and }i\in\{1,2\}.
\]

\item
There exists \(\omega\in S_2\) such that \(M^\omega\) satisfies one of
the following alternatives.  All parameters are computed for this
ordering, the indicated generators and support representatives, and
the fixed cocycle \(\Phi\).
\begin{enumerate}[label=\textnormal{(\roman*)}]
\item
There are \(g,h,\varepsilon\in G\) inducing an epimorphism
\(\Gamma _2\twoheadrightarrow G\) such that
\[
M^\omega\simeq\bigl(M(g,\rho),M(h,\sigma)\bigr),
\quad
\operatorname{supp}M(g,\rho)=\{g,\varepsilon g\},
\quad
\operatorname{supp}M(h,\sigma)=\{h,\varepsilon h\}.
\]
Either
\eqref{eq:Gamma2-classification-A2} or
\eqref{eq:Gamma2-classification-G2} holds.

\item
There are \(\varepsilon,g,h\in G\) inducing an epimorphism
\(\Gamma _3\twoheadrightarrow G\) from
\eqref{eq:Gamma3-presentation} such that \(M^\omega\) is isomorphic to
an ordered pair satisfying one of
\eqref{eq:Gamma3-P1}--\eqref{eq:Gamma3-P4}.

\item
There are \(g,h,\varepsilon\in G\) inducing an epimorphism
\(\Gamma _4\twoheadrightarrow G\), such
that
\begin{align*}
&M^\omega\simeq\bigl(M(h,\rho),M(g,\sigma)\bigr),\\
&\operatorname{supp}M(h,\rho)=\{h,\varepsilon^{-1}h\},\ \ \ \operatorname{supp}M(g,\sigma)
=\{g,\varepsilon g,\varepsilon^2g,\varepsilon^3g\},
\end{align*}

and \eqref{eq:Gamma4-classification-B2} holds.

\item
There are \(z,x_1,x_2\in G\) inducing an epimorphism
\(T\twoheadrightarrow G\).  With
\(x_3=x_2x_1x_2^{-1}\) and
\(x_4=x_1x_2x_1^{-1}\), the elements
\(x_1,x_2,x_3,x_4\) are distinct,
\(x_1^G=\{x_1,x_2,x_3,x_4\}\),
and
\[
M^\omega\simeq\bigl(M(z,\rho),M(x_1,\sigma)\bigr).
\]
The conditions in \eqref{eq:T-classification-G2} hold.
\end{enumerate}
\end{enumerate}
For an ordering \(M^\omega\) satisfying \textnormal{(3)},
the Cartan type and the dimension are given in
Table~\ref{tab:complete-classification}.
\end{thm}
\begin{rmk}\upshape
    In alternative \textnormal{(i)}, the one-dimensionality of \(\rho\)
and \(\sigma\) is automatic by
Lemma~\ref{lem:Gamma2-projective-representations-one-dimensional}.   {In alternative \textnormal{(ii)},
Lemma~\ref{lem:Gamma3-projective-dimensions} gives the dimension
restrictions stated before \eqref{eq:Gamma3-P1}.
The two \((3,1)\) cases are distinguished by \(\dim\tau=1\) or \(2\).}  {
In alternative \textnormal{(iii)}, both
\(\dim\rho=1\) and \(\dim\sigma=1\) are automatic by
Lemma~\ref{lem:Gamma4-projective-representations-one-dimensional}.
} In alternative \textnormal{(iv)},
\(\dim\sigma=1\) is automatic by
Proposition~\ref{prop:T-sigma-one-dimensional}, whereas
\(\dim\rho=1\) remains part of \eqref{eq:T-classification-G2}.
\end{rmk}

\begin{proof}[Proof of Theorem~\ref{thm:complete-classification}]
The equivalence of \textnormal{(1)} and \textnormal{(2)} is
Theorem~\ref{thm:Nichols-finiteness-criterion}.

 \textnormal{(1)} $\Rightarrow$ \textnormal{(3)}.  Under
\eqref{eq:complete-classification-assumptions}, statement
\textnormal{(2)} and
Proposition~\ref{prop:pair-with-finite-Cartan-matrix} give a tuple
\[
N=R_{i_r}\cdots R_{i_1}(M)\in\mathcal F_2(M)
\]
with \(1\leq a_{12}^{[N]}a_{21}^{[N]}\leq3\).  Hence its rank-two
Cartan matrix is indecomposable and of finite type.  Choose
\(\omega_0\in S_2\) so that
\[
-a_{12}^{[N^{\omega_0}]}=1,
\qquad 1\leq-a_{21}^{[N^{\omega_0}]}\leq3.
\]
Thus
\[
(\operatorname{ad}(N^{\omega_0})_1)((N^{\omega_0})_2)\ne0,
\qquad
(\operatorname{ad}(N^{\omega_0})_1)^2((N^{\omega_0})_2)=0,
\qquad
(\operatorname{ad}(N^{\omega_0})_2)^4((N^{\omega_0})_1)=0.
\]
Equation~\eqref{eq:indecomposable-first-adjoint} shows that
\(N^{\omega_0}\) is braided-indecomposable.  Interchanging the entries
does not change  {the union of their supports}, and
Lemma~\ref{lem:reflections-preserve-support-group} gives
\(G_{N^{\omega_0}}=G_N=G_M=G\).  Hence
Theorem~\ref{thm:support-classification} shows that the support
quandle of \(N^{\omega_0}\) is one of
\[
Z_T^{4,1},\qquad Z_2^{2,2},\qquad Z_3^{3,1},\qquad
Z_3^{3,2},\qquad Z_4^{4,2}.
\]
The epimorphism supplied by that theorem provides the generators and
support representatives used in  {the section treating the corresponding case}.

By Lemma~\ref{lem:reflection-dimension-invariance},
\(\mathcal B(N_1\oplus N_2)\) is finite-dimensional.
Applying the appropriate one of
Theorems~\ref{thm:Gamma2-classification},
\ref{thm:Gamma3-classification},
\ref{thm:Gamma4-classification}, and
\ref{thm:T-classification}, with the ordering required
in that case, we obtain \(\nu\in S_2\) such that
\(P=N^\nu\) satisfies the corresponding alternative
in \textnormal{(3)}.
It remains to obtain the same conclusion for the original tuple.
By definition,
\((R_iQ)^\nu\simeq R_{\nu(i)}(Q^\nu)\).
Since reflections are involutive by
Lemma~\ref{lem:reflection-basic-properties},
\[
M^\nu
\simeq
R_{\nu(i_1)}\cdots R_{\nu(i_r)}(P)
\in\mathcal F_2(P).
\]

In the \(\Gamma_3\), \(\Gamma_4\), and \(T\) cases,
Propositions~\ref{prop:Gamma3-reflection-closure},
\ref{prop:Gamma4-explicit-B2-stability-under-reflections},
and~\ref{prop:T-G2-stability-under-reflections},
respectively, show that every tuple in
\(\mathcal F_2(P)\) satisfies the corresponding
alternative in \textnormal{(3)}.
This applies in particular to \(M^\nu\).

In the \(\Gamma_2\) case,
Lemma~\ref{lem:reflections-remain-Gamma2} shows that
\(M^\nu\) satisfies the same support assumptions as \(P\).
It is braided-indecomposable by
\eqref{eq:complete-classification-assumptions}, and
its Nichols algebra is finite-dimensional because
interchanging the entries does not change their direct sum.
Theorem~\ref{thm:Gamma2-classification} therefore gives
alternative \textnormal{(i)}, after a further interchange
if necessary.
Thus \textnormal{(3)} holds in every case.

 \textnormal{(3)} $\Rightarrow$  \textnormal{(1)}.   {The theorem for the corresponding case}
gives
\(\dim\mathcal B((M^\omega)_1\oplus(M^\omega)_2)<\infty\).
Since permuting the two summands does not change their direct sum,
\textnormal{(1)} follows.  Finally, the Cartan-type assertions follow from
Propositions~\ref{prop:Gamma2-standard-A2} and
\ref{prop:Gamma2-standard-G2-stability-under-reflections} and
Theorems~\ref{thm:Gamma3-classification},
\ref{thm:Gamma4-classification}, and
\ref{thm:T-classification}.
The dimension formulas follow from
Corollaries~\ref{cor:Gamma2-A2-Nichols-dimension},
\ref{cor:standard-G2-Nichols-dimension}, and
\ref{cor:Gamma4-Nichols-dimension} and
Theorems~\ref{thm:Gamma3-classification} and
\ref{thm:T-classification}.
\end{proof}

\clearpage
\appendix

\section{Guide and common notation for the appendices}

The appendices contain the longer calculations used in the proofs for
the four cases.  The table below shows where these calculations can be
found.  Links to the corresponding GAP programs are given next to the
lemmas whose calculations they verify.

\begingroup
\small
\setlength{\tabcolsep}{5pt}
\renewcommand{\arraystretch}{1.15}
\begin{center}
\begin{tabular}{@{}L{0.14\textwidth}L{0.17\textwidth}L{0.61\textwidth}@{}}
\hline
Case & Appendix & Calculations \\
\hline
\(\Gamma_2\)
&
Appendix~\ref{app:Gamma2-calculations}
&
 {coefficients for the first adjoint object},
  higher adjoints, reflection calculations, and diagonal braiding
  coefficients.
\\

\(\Gamma_4\)
&
Appendix~\ref{app:Gamma4-calculations}
&
 {%
the cocycle reduction for \(X_2\) and the calculation of \(z_0\)
used in the simplicity criterion for \(Y_2\).
}
\\

\(T\)
&
Appendix~\ref{app:T-calculations}
&
{
 {the cocycle on \(Q_8\) and the four reflections used to exclude
the case \(h=4\).}
}
\\

\(\Gamma_3\)
&
Appendix~\ref{app:Gamma3-calculations}
&
bar-complex calculations proving
\(\kappa_\Phi(\varepsilon)=1\) in the \((3,2)\) case and in the two
\((3,1)\) cases.
\\
\hline
\end{tabular}
\end{center}
\endgroup

Let \(G\) be a group and let \(\Phi\) be a normalized \(3\)-cocycle
on \(G\).  For \(a,b,c,d\in G\), let
\begin{equation}
\label{eq:3-cocycle-identity}
\mathscr C_\Phi(a,b,c,d)
:=
\frac{
\Phi(b,c,d)\Phi(a,bc,d)\Phi(a,b,c)
}{
\Phi(ab,c,d)\Phi(a,b,cd)
}.
\end{equation}
Thus the normalized \(3\)-cocycle identity is
\(\mathscr C_\Phi(a,b,c,d)=1\).  We use this notation throughout the
appendices.

\section{Calculations for the \texorpdfstring{$\Gamma_2$}{Gamma2} case}
\label{app:Gamma2-calculations}

Retain the notation and assumptions of
Section~\ref{sec:classification-Gamma2}. In particular,
\(hg=\varepsilon gh\), \(\varepsilon^2=1\), and
\(\varepsilon\in Z(G)\). 
Each cocycle quotient below is reduced to a product of powers
\(\mathscr C_\Phi(a_i,b_i,c_i,d_i)^{m_i}\), with
\(m_i\in\mathbb Z\).  Long products are recorded in the linked GAP
programs.
 {In these reductions, we use the normal form
\(\varepsilon^e g^a h^b\) in \(\Gamma_2\), where
\(e\in\mathbb Z/2\mathbb Z\) and \(a,b\in\mathbb Z\).}  

\subsection{ {Coefficient reduction for the simplicity criterion for
\texorpdfstring{$X_1$}{X1}}}

Retain \(a=\zeta/\Phi_g(h,\varepsilon)\) and the scalars \(A_1,B_1\)
from \eqref{eq:A1-definition} and \eqref{eq:B1-definition}.

\begin{lemma}
\label{lem:first-coefficient-reduction}
One has
\begin{equation}
\label{eq:first-coefficient-Delta}
\frac{
a^2A_1B_1
}{
\lambda'\Phi^h(g,h)\Phi^g(g,h)
}
=
\Delta,
\end{equation}
where \(\Delta\) is defined in \eqref{eq:Delta-definition}.
\end{lemma}

\begin{proof}
After cancelling the scalars from the one-dimensional projective
actions appearing on
both sides of \eqref{eq:first-coefficient-Delta},  {the only
remaining factor involving the values of \(\rho\) or \(\sigma\) is
\(\zeta^2\). Using \(\zeta^2=\Phi_g(\varepsilon,\varepsilon)\)},
the quotient of the left-hand side of
\eqref{eq:first-coefficient-Delta} by \(\Delta\) becomes an expression
involving only values of the cocycle.

The GAP program \gaplink{Gamma_2/delta_reduction/lemma_B1_first_coefficient_reduction.g}
verifies the following factorization:
\[
\frac{
a^2A_1B_1
}{
\lambda'\Phi^h(g,h)\Phi^g(g,h)\Delta
}
=
\prod_{i=1}^{15}
\mathscr C_\Phi(a_i,b_i,c_i,d_i)^{m_i},
\]
where the tuples \((a_i,b_i,c_i,d_i)\) and the exponents \(m_i\) are
given by
\[
\begin{array}{c c r @{\qquad\qquad} c c r}
\hline
i & (a_i,b_i,c_i,d_i) & m_i
& i & (a_i,b_i,c_i,d_i) & m_i\\
\hline
1 & (\varepsilon g,g,h,\varepsilon) & 1
& 9 & (g,\varepsilon g,h,\varepsilon) & -1\\
2 & (\varepsilon,g,g,h) & 1
& 10 & (h,\varepsilon,g,\varepsilon h) & 1\\
3 & (\varepsilon g,\varepsilon,\varepsilon,g) & 1
& 11 & (g,h,\varepsilon,\varepsilon g) & -1\\
4 & (\varepsilon,h,g,g) & 1
& 12 & (h,\varepsilon,g,g) & -1\\
5 & (\varepsilon g,h,\varepsilon g,\varepsilon) & -1
& 13 & (\varepsilon,\varepsilon g,\varepsilon,\varepsilon) & -1\\
6 & (h,\varepsilon g,\varepsilon,\varepsilon g) & 1
& 14 & (h,\varepsilon,h,g) & -1\\
7 & (g,\varepsilon h,\varepsilon g,\varepsilon) & 1
& 15 & (\varepsilon,g,\varepsilon h,g) & -1\\
8 & (h,\varepsilon,\varepsilon g,\varepsilon) & 1
& & &\\
\hline
\end{array}
\]
Every factor on the right-hand side is equal to \(1\). Therefore
the quotient is \(1\), proving
\eqref{eq:first-coefficient-Delta}.
\end{proof}

\subsection{Reduction of \texorpdfstring{\(\Xi_V\)}{Xi V}}
\label{app:Xi-V-Omega-V}

\begin{lemma}
\label{lem:Xi-V-Omega-V}
For \(\Xi_V\) defined in \eqref{eq:Xi-V-definition}, one has
\(\Xi_V=\Phi_g(h,g)^{-1}\).
\end{lemma}

\begin{proof}
Put
\begin{equation*}
\Omega_V
=
\frac{
\zeta^2
\Phi(g,\varepsilon g,\varepsilon h)
\Phi_g(g,h)\Phi_g(\varepsilon,h)
}{
\Phi(\varepsilon g,g,\varepsilon h)
\Phi^g(g,h)
\Phi_g(\varepsilon,hg)
\Phi_g(h,\varepsilon)^2
}.
\end{equation*}
Apply Lemma~\ref{lem:local-cocycle-coherence} with
$x=g,y=h,z=\varepsilon,w=g$, and then with
$x=g,y=\varepsilon,z=h,w=g$.  This gives
\begin{align*}
\Phi_g(\varepsilon,g)\Phi_g(h,\varepsilon g)
&=
\Phi_g(h,\varepsilon)\Phi_g(\varepsilon h,g),\\
\Phi_g(h,g)\Phi_g(\varepsilon,hg)
&=
\Phi_g(\varepsilon,h)\Phi_g(\varepsilon h,g).
\end{align*}
Eliminating \(\Phi_g(\varepsilon h,g)\), we obtain
\[
\Phi_g(\varepsilon,g)\Phi_g(h,\varepsilon g)
=
\frac{
\Phi_g(h,\varepsilon)
\Phi_g(h,g)
\Phi_g(\varepsilon,hg)
}{
\Phi_g(\varepsilon,h)
}.
\]
Substitution in \eqref{eq:Xi-V-definition} gives
\(\Xi_V=\Omega_V/\Phi_g(h,g)\).
 {Substituting \(\zeta^2=\Phi_g(\varepsilon,\varepsilon)\) and
expanding \(\Phi_x\) and \(\Phi^x\) gives}
\[
\Omega_V
=
\prod_{i=1}^{11}
\mathscr C_\Phi(a_i,b_i,c_i,d_i)^{m_i},
\]
where
\[
\begin{array}{c c r @{\qquad\qquad} c c r}
\hline
i & (a_i,b_i,c_i,d_i) & m_i
& i & (a_i,b_i,c_i,d_i) & m_i\\
\hline
1 & (\varepsilon g,\varepsilon,\varepsilon g,\varepsilon h) & -1
& 7 & (\varepsilon,\varepsilon g,h,\varepsilon) & -1\\
2 & (\varepsilon g,\varepsilon,\varepsilon g,h) & 1
& 8 & (\varepsilon,h,g,\varepsilon) & 1\\
3 & (\varepsilon g,\varepsilon,h,\varepsilon g) & 1
& 9 & (h,g,\varepsilon,\varepsilon) & 1\\
4 & (\varepsilon g,h,\varepsilon,\varepsilon g) & -1
& 10 & (h,\varepsilon,g,\varepsilon) & -1\\
5 & (g,h,\varepsilon,g) & -1
& 11 & (\varepsilon gh,\varepsilon,\varepsilon,g) & 1\\
6 & (\varepsilon,\varepsilon g,h,g) & -1
& & &\\
\hline
\end{array}
\]
 {This factorization is also verified in
\gaplink{Gamma_2/omegaV/lemma_B4_Omega_V_equals_one.g}.}
Every factor on the right-hand side is \(1\) by
\eqref{eq:3-cocycle-identity}. Hence \(\Omega_V=1\), and the asserted
identity follows.
\end{proof}

\subsection{ {The coefficient \texorpdfstring{\(\Theta_\Phi\)}{Theta Phi}
in the simplicity criterion for \texorpdfstring{$X_2$}{X2}}}

 {The coefficient comparison in the proof of
Lemma~\ref{lem:Gamma2-X2-simple} gives the quotient in the following
lemma}, where
$p=\frac{1}{\Phi_g(h,g)}$
and
$
D_1
=
\Phi^h(\varepsilon g,h)
\Phi_g(h,h)\mu\lambda'.$

\begin{lemma}
\label{lem:Theta-Phi-reduction}
Assume that \(\Delta=1\). Then
\begin{equation}
\label{eq:Theta-Phi-coefficient-reduction}
\frac{
\Phi^h(g,gh)
 A_1
}{
 p^2
\Phi^g(g,h)
\Phi^h(\varepsilon g,\varepsilon gh)
 \Phi_g(h,h)\mu B_1D_1
}
=\Theta_\Phi.
\end{equation}
\end{lemma}

\begin{proof}
 {Substitution of the definitions gives}
\[
\frac{1}{\Theta_\Phi}
\frac{
\Phi^h(g,gh)A_1
}{
p^2
\Phi^g(g,h)
\Phi^h(\varepsilon g,\varepsilon gh)
\Phi_g(h,h)\mu B_1D_1
}
=
\Delta^{-1}
\frac{{\zeta'}^2}{\Phi_h(\varepsilon,\varepsilon)}.
\]
 {This identity is also verified in
\gaplink{Gamma_2/theta_coefficient_reduction/theta_coefficient_reductions.g}.}
Since \(\varepsilon\rhd w=\zeta'w\), relation
\eqref{eq:YD-projective-action} gives
\({\zeta'}^2=\Phi_h(\varepsilon,\varepsilon)\).
 {Together with \(\Delta=1\), this identity gives
\eqref{eq:Theta-Phi-coefficient-reduction}.}
\end{proof}

\subsection{Cocycle identities for \texorpdfstring{$X_3$}{X3} and the
 {projective representations} after reflections}
\label{app:finite-reflection-cocycle-reductions}

\begin{lemma}
\label{lem:x3-cocycle-reduction}
Let $d_3$ and $\kappa_3$ be the scalars in
\eqref{eq:x3-double-braiding} and
\eqref{eq:kappa-3-definition}.  Then
\(d_3/(\lambda^2\kappa_3)=1\).
\end{lemma}

\begin{proof}
Substituting \eqref{eq:x3-double-braiding} and
\eqref{eq:kappa-3-definition}, and cancelling the common factors
\(\zeta\), \(\lambda^2\), and \(\Phi_g(\varepsilon,g)\), gives
\[
\frac{d_3}{\lambda^2\kappa_3}
=
\frac{
\Phi_g(h,g)
\Phi(\varepsilon g,g,\varepsilon gh)
\Phi_g(h,\varepsilon g)
\Phi^g(g,gh)
}{
\Phi_g(\varepsilon g,g)
\Phi_g(h,\varepsilon g^2)
\Phi(g,\varepsilon g,\varepsilon gh)
\Phi_g(g,h)
}.
\]
Apply Lemma~\ref{lem:local-cocycle-coherence} first with
$
(x,y,z,w)=(g,h,\varepsilon g,g).
$
Since \(g\) commutes with \(\varepsilon g\) and
\(h\varepsilon g=gh\), this gives
\[
\Phi_g(\varepsilon g,g)\Phi_g(h,\varepsilon g^2)
=
\Phi_g(h,\varepsilon g)\Phi_g(gh,g).
\]
Applying the same lemma with
$
(x,y,z,w)=(g,g,h,g)
$
gives
\[
\Phi_g(h,g)\Phi_g(g,hg)
=
\Phi_g(g,h)\Phi_g(gh,g).
\]
Consequently,
\[
\frac{
\Phi_g(h,g)\Phi_g(h,\varepsilon g)
}{
\Phi_g(\varepsilon g,g)
\Phi_g(h,\varepsilon g^2)
\Phi_g(g,h)
}
=
\frac{1}{\Phi_g(g,hg)}.
\] Since
$
g(gh)g^{-1}=\varepsilon gh=hg,
$ 
$(hg)g(hg)^{-1}=\varepsilon g,
$
the definitions of \(\Phi^g\) and \(\Phi_g\) give
$\Phi^g(g,gh)
=
\Phi(g,hg,g)
$
and
$
\Phi_g(g,hg)
=
\frac{
\Phi(g,hg,g)\Phi(\varepsilon g,g,hg)
}{
\Phi(g,\varepsilon g,hg)
}.
$
Using again \(\varepsilon gh=hg\), we obtain
\[
\frac{
\Phi(\varepsilon g,g,\varepsilon gh)\Phi^g(g,gh)
}{
\Phi(g,\varepsilon g,\varepsilon gh)
}
=
\Phi_g(g,hg).
\]
Combining the preceding two reductions gives
$
\frac{d_3}{\lambda^2\kappa_3}
=
\frac{1}{\Phi_g(g,hg)}\Phi_g(g,hg)
=1.
$
\end{proof}

\begin{lemma}
\label{lem:reflected-character-cocycle-reduction}
The two cocycle quotients in
\eqref{eq:K1-reflected-character}
and \eqref{eq:K1-op-reflected-character} satisfy
\(K_1=K_1^{\mathrm{op}}=1\).
\end{lemma}

\begin{proof}
Apply Lemma~\ref{lem:local-cocycle-coherence} with
\((x,y,z,w)=(g,h,\varepsilon,g)\).  Since \(\varepsilon\) is
central, this gives
\[\Phi_g(\varepsilon,g)\Phi_g(h,\varepsilon g)
=\Phi_g(h,\varepsilon)\Phi_g(\varepsilon h,g).\]
Consequently,
\[
K_1
=
\frac{
\Phi^{gh}(g,h)
}{
\Phi_h(g,h)\Phi_g(\varepsilon h,g)
}.
\]
The relations \(hg=\varepsilon gh\) and \(\varepsilon^2=1\) imply
that conjugation by \(gh\) sends \(g\) to \(\varepsilon g\) and
\(h\) to \(\varepsilon h\).  Thus the definitions of
\(\Phi^x\) and \(\Phi_x\) give
\[
\Phi^{gh}(g,h)
=
\frac{
\Phi(gh,g,h)\Phi(\varepsilon g,\varepsilon h,gh)
}{
\Phi(\varepsilon g,gh,h)
},
\qquad
\Phi_h(g,h)=\Phi(\varepsilon h,g,h),
\qquad
\Phi_g(\varepsilon h,g)=\Phi(\varepsilon g,\varepsilon h,g).
\]
Substitution shows that
\[
K_1
=
\frac{
\Phi(gh,g,h)\Phi(\varepsilon g,\varepsilon h,gh)
}{
\Phi(\varepsilon g,gh,h)
\Phi(\varepsilon h,g,h)
\Phi(\varepsilon g,\varepsilon h,g)
}
=
\mathscr C_\Phi(\varepsilon g,\varepsilon h,g,h)^{-1}
=1.
\]
Finally, \(gh=\varepsilon hg\).  Interchanging \(g\) and \(h\) in
the preceding calculation transforms
\eqref{eq:K1-reflected-character} into
\eqref{eq:K1-op-reflected-character}.  Hence
\(K_1^{\mathrm{op}}=1\).
\end{proof}

\subsection{Cocycle identities for the third and fourth adjoint objects}
\label{app:Gamma2-G2-higher-adjoints-reductions}

 {The following identities determine \(\sigma_3(g^3h)\) and the
coefficient \(d_4\) in \eqref{eq:Gamma2-x4-double-braiding}.}

\begin{lemma}
\label{lem:Gamma2-G2-cocycle-reductions}
The cocycle factor $\mathcal K_3$ in
\eqref{eq:Gamma2-K3} is equal to one.  Moreover, with $d_4$ and
$\kappa_4$ defined by \eqref{eq:Gamma2-x4-double-braiding} and
\eqref{eq:Gamma2-kappa4}, respectively, one has
$d_4=\lambda^3\kappa_4$.
\end{lemma}

\begin{proof}
 {The two factorizations below are verified in
\gaplink{Gamma_2/g2_propagation/lemma_B12_K3_and_d4_cocycle_reductions.g}.
Expanding \(\mathcal K_3\) gives the product of the thirteen factors
\(\mathscr C_\Phi(a,b,c,d)^m\) specified in the following table.}

\begin{center}
\small
\setlength{\tabcolsep}{5pt}
\renewcommand{\arraystretch}{1.1}
\begin{tabular}{@{}c r@{\qquad\qquad}c r@{}}
\hline
\((a,b,c,d)\) & \(m\) & \((a,b,c,d)\) & \(m\)\\
\hline
\((h,g,g,g)\) & \(-1\)
& \((\varepsilon g,h,\varepsilon g,g)\) & \(1\)\\
\((h,g,\varepsilon,g)\) & \(-1\)
& \((\varepsilon g,h,\varepsilon g^2,g)\) & \(1\)\\
\((h,g,\varepsilon g,g)\) & \(-1\)
& \((\varepsilon g,g,h,g)\) & \(-1\)\\
\((h,g,\varepsilon g^2,g)\) & \(-1\)
& \((\varepsilon g,\varepsilon h,g,h)\) & \(-1\)\\
\((g,\varepsilon g,h,g)\) & \(1\)
& \((\varepsilon g,\varepsilon gh,g,gh)\) & \(-1\)\\
\((\varepsilon g,h,g,g)\) & \(1\)
& \((\varepsilon g,\varepsilon g^2h,g,g^2h)\) & \(-1\)\\
\((\varepsilon g,h,\varepsilon,g)\) & \(1\)
& &\\
\hline
\end{tabular}
\end{center}

 {Substituting \eqref{eq:Gamma2-x4-double-braiding},
\eqref{eq:kappa-3-definition}, and \eqref{eq:Gamma2-kappa4}
into \(d_4/(\lambda^3\kappa_4)\), and using
\(\zeta^2=\Phi_g(\varepsilon,\varepsilon)\), gives the product
of the eighteen factors \(\mathscr C_\Phi(a,b,c,d)^m\)
specified in the following table.}

\begin{center}
\small
\setlength{\tabcolsep}{5pt}
\renewcommand{\arraystretch}{1.1}
\begin{tabular}{@{}c r@{\qquad\qquad}c r@{}}
\hline
\((a,b,c,d)\) & \(m\) & \((a,b,c,d)\) & \(m\)\\
\hline
\((h,g,g^2,g)\) & \(-1\)
& \((\varepsilon g,h,g,\varepsilon g)\) & \(-2\)\\
\((g,\varepsilon,g,\varepsilon)\) & \(1\)
& \((\varepsilon g,h,\varepsilon g,g)\) & \(2\)\\
\((g,\varepsilon,\varepsilon,g)\) & \(-1\)
& \((\varepsilon g,g,h,g)\) & \(-2\)\\
\((g,\varepsilon g,h,g)\) & \(2\)
& \((\varepsilon g,g,\varepsilon,g)\) & \(1\)\\
\((g,\varepsilon g,\varepsilon,g)\) & \(-1\)
& \((\varepsilon g,g,\varepsilon gh,\varepsilon g)\) & \(-1\)\\
\((g,\varepsilon g,\varepsilon gh,\varepsilon g)\) & \(1\)
& \((\varepsilon g,\varepsilon gh,\varepsilon g,\varepsilon g)\) & \(1\)\\
\((g,\varepsilon gh,g,\varepsilon g)\) & \(-1\)
& \((\varepsilon gh,g,\varepsilon g,\varepsilon g)\) & \(-1\)\\
\((g,\varepsilon gh,\varepsilon g,g)\) & \(1\)
& \((\varepsilon gh,\varepsilon g,g,\varepsilon g)\) & \(1\)\\
\((\varepsilon g,h,g,g^2)\) & \(1\)
& \((\varepsilon gh,\varepsilon g,\varepsilon g,g)\) & \(-1\)\\
\hline
\end{tabular}
\end{center}
\end{proof}

\subsection{Reflection identities in type
\texorpdfstring{$G_2$}{G2}}
\label{app:Gamma2-standard-G2-cocycle-identities}

Let $x,y\in G$ satisfy $yx=\varepsilon xy$, and let $q=x^2y$.   {Define}
\begin{equation}
\label{eq:Gamma2-X3-cocycle-factor}
\begin{aligned}
\mathfrak s_\Phi(x,y)
:={}&
\left(
\frac{
 \Phi(x,\varepsilon x,\varepsilon xy)\Phi_x(x,y)
}{
 \Phi_x(y,x)\Phi(\varepsilon x,x,\varepsilon xy)
 \Phi_x(\varepsilon,x)\Phi_x(y,\varepsilon x)\Phi^x(x,xy)
}
\right)^2
\\
&\times
\frac{
 \Phi_x(\varepsilon,\varepsilon)^2
 \Phi_x(y,\varepsilon)\Phi_y(\varepsilon,x^2)
}{
 \Phi_y(\varepsilon x,x)\Phi_x(y,y)
}
\\
&\times
\frac{
 \Phi^x(\varepsilon x,\varepsilon q)
 \Phi^y(\varepsilon x,\varepsilon q)
 \Phi_x(x,y)\Phi_q(x,x)\Phi_x(y,y)\Phi_q(y,x)
}{
 \Phi^x(x,q)\Phi^y(x,q)
 \Phi_x(\varepsilon,x)\Phi_x(y,\varepsilon x)
 \Phi_q(\varepsilon x,y)\Phi_q(x,\varepsilon)\Phi_x(x,x)^2
}
\\
&\times
\Phi^\varepsilon(x,xy)\Phi^\varepsilon(x,y)
\Phi^{x^2}(x,xy)\Phi^{x^2}(x,y).
\end{aligned}
\end{equation}

\begin{lemma}
\label{lem:Gamma2-X3-cocycle-factor-identity}
Let $x,y\in G$ satisfy $yx=\varepsilon xy$, and assume that
$\varepsilon^2=1$ and $\varepsilon\in Z(G)$.
Then the scalars defined in
\eqref{eq:Gamma2-X3-cocycle-factor} and
\eqref{eq:Theta-Phi-definition} satisfy
\[
 \mathfrak s_\Phi(x,y)=\Theta_\Phi(x,y)^2.
\]
\end{lemma}

\begin{proof}
 {Substitute the definitions of \(\Phi_x\) and \(\Phi^x\) in}
\(\mathfrak s_\Phi(x,y)/\Theta_\Phi(x,y)^2\), and reduce
the group elements to the form \(\varepsilon^e x^a y^b\),
using \(yx=\varepsilon xy\), \(\varepsilon^2=1\), and
the centrality of \(\varepsilon\).
The resulting quotient is a product of \(63\) factors
\(\mathscr C_\Phi(a,b,c,d)^m\), with \(m\in\mathbb Z\),
as recorded and verified in
\gaplink{Gamma_2/g2_pure_parameters/s_phi_equals_theta_squared/lemma_B14_s_phi_equals_theta_squared.g}.
Each factor is one by \eqref{eq:3-cocycle-identity},
so \(\mathfrak s_\Phi(x,y)=\Theta_\Phi(x,y)^2\).
\end{proof}

A further cocycle calculation gives the value of \(\Theta_\Phi\)
after the second reflection.

\begin{lemma}
\label{lem:Gamma2-R2-Theta-cocycle}
Under the hypotheses of the preceding lemma,
\begin{equation}
\label{eq:Gamma2-R2-Theta-cocycle}
 \Theta_\Phi(yx,y^{-1})=\Theta_\Phi(x,y).
\end{equation}
\end{lemma}

\begin{proof}
Since \(y^{-1}(yx)=x=\varepsilon(yx)y^{-1}\), the pair
\((yx,y^{-1})\) satisfies the same relation as \((x,y)\).
 {Substituting the definitions of \(\Phi_x\) and \(\Phi^x\) in}
\(\Theta_\Phi(yx,y^{-1})/\Theta_\Phi(x,y)\) and using
the group relations gives the factorization into \(94\)
integer powers of \(\mathscr C_\Phi(a,b,c,d)\)
recorded and verified in
\gaplink{Gamma_2/g2_pure_parameters/theta_r2_invariance/lemma_B15_theta_after_second_reflection.g}.
The quotient is therefore one by
\eqref{eq:3-cocycle-identity}, proving
\eqref{eq:Gamma2-R2-Theta-cocycle}.
\end{proof}

\subsection{ {Cocycle identities for the simple objects
corresponding to positive roots of type \texorpdfstring{$G_2$}{G2}}}

The following identities determine the values of the one-dimensional
projective representations and the edge labels of the generalized
Dynkin diagrams for the  {simple objects corresponding to the
positive roots of type \(G_2\)}.

\begin{lemma}
\label{lem:Gamma2-K2-cocycle-reduction}
The cocycle quotient $K_2$ in
\eqref{eq:Gamma2-K2-support-value} is equal to one.
\end{lemma}

\begin{proof}
Let \(K_1\) be the quotient in
\eqref{eq:K1-reflected-character}.  By
Lemma~\ref{lem:reflected-character-cocycle-reduction}, \(K_1=1\).
 {Comparing the definitions of \(K_2\) and \(K_1\), we cancel the
common factors. Moreover, \(g^2h=hg^2\). Conjugation by either
\(hg\) or \(g^2h\) sends \(g\) to \(\varepsilon g\), while
conjugation by \(g^2h\) sends \(gh\) to \(\varepsilon gh\).}  Substituting the
definitions of \(\Phi_x\) and \(\Phi^x\) therefore gives
\begin{align*}
\Phi_g(h,g)
&=\Phi(\varepsilon g,h,g),
&
\Phi_g(g,g)
&=\Phi(g,g,g),
\\
\Phi_g(h,g^2)
&=
\frac{
 \Phi(h,g^2,g)\Phi(\varepsilon g,h,g^2)
}{
 \Phi(h,g,g^2)
},
&
\Phi_{gh}(g,gh)
&=\Phi(\varepsilon gh,g,gh),
\\
\Phi^{g^2h}(g,gh)
&=
\frac{
 \Phi(g^2h,g,gh)
 \Phi(\varepsilon g,\varepsilon gh,g^2h)
}{
 \Phi(\varepsilon g,g^2h,gh)
}.
\end{align*}
Consequently,
\begin{align*}
K_2
&=
K_1
\frac{
 \Phi_g(h,g)\Phi^{g^2h}(g,gh)
}{
 \Phi_g(h,g^2)\Phi_g(g,g)\Phi_{gh}(g,gh)
}
\\
&=
\frac{
 \Phi(\varepsilon g,h,g)\Phi(h,g,g^2)
}{
 \Phi(h,g^2,g)\Phi(\varepsilon g,h,g^2)\Phi(g,g,g)
}
\frac{
 \Phi(g^2h,g,gh)\Phi(\varepsilon g,\varepsilon gh,g^2h)
}{
 \Phi(\varepsilon g,g^2h,gh)\Phi(\varepsilon gh,g,gh)
}
\\
&=
\mathscr C_\Phi(h,g,g,g)^{-1}
\mathscr C_\Phi(\varepsilon g,h,g,g)
\mathscr C_\Phi(\varepsilon g,\varepsilon gh,g,gh)^{-1}.
\end{align*}
  Each factor
\(\mathscr C_\Phi(a,b,c,d)\) is equal to \(1\), and hence \(K_2=1\).
\end{proof}

\begin{lemma}
\label{lem:Theta-diagonal-braiding-cocycle}
For every pair $x,y\in G$ satisfying $yx=\varepsilon xy$, one has
\begin{equation}
\label{eq:Theta-diagonal-braiding-cocycle}
\Theta_\Phi(x,y)
\frac{
\Phi_x(\varepsilon,\varepsilon)\Phi_x(x,y)
}{
\Phi_x(\varepsilon,x)^2\Phi_x(y,\varepsilon x)
}
=1.
\end{equation}
\end{lemma}

\begin{proof}
We first take \((x,y)=(g,h)\).
 {Substituting the definitions of \(\Phi_x\) and \(\Phi^x\) expresses} the left-hand side of
\eqref{eq:Theta-diagonal-braiding-cocycle} as the product of
the twenty factors \(\mathscr C_\Phi(a,b,c,d)^m\)
specified by the following table.

\begin{center}
\small
\renewcommand{\arraystretch}{1.08}
\begin{tabular}{@{}c r@{\qquad}c r@{}}
\hline
\((a,b,c,d)\) & \(m\) & \((a,b,c,d)\) & \(m\) \\
\hline
\((\varepsilon,\varepsilon,\varepsilon,h)\) & \(1\) &
\((\varepsilon,\varepsilon,\varepsilon,\varepsilon g)\) & \(-1\) \\

\((\varepsilon,g,g,h)\) & \(-1\) &
\((\varepsilon,g,h,\varepsilon)\) & \(1\) \\

\((\varepsilon,g,\varepsilon h,g)\) & \(1\) &
\((\varepsilon,h,\varepsilon,\varepsilon g)\) & \(-1\) \\

\((\varepsilon,h,g,g)\) & \(-1\) &
\((\varepsilon,\varepsilon h,g,\varepsilon)\) & \(-1\) \\

\((g,\varepsilon,\varepsilon,\varepsilon g)\) & \(1\) &
\((g,h,\varepsilon,gh)\) & \(-1\) \\

\((g,\varepsilon h,g,h)\) & \(-1\) &
\((h,\varepsilon,\varepsilon,\varepsilon h)\) & \(1\) \\

\((h,\varepsilon,g,\varepsilon)\) & \(-1\) &
\((h,\varepsilon,g,g)\) & \(1\) \\

\((h,\varepsilon,g,h)\) & \(-1\) &
\((h,\varepsilon g,\varepsilon,gh)\) & \(1\) \\

\((\varepsilon g,\varepsilon,\varepsilon g,\varepsilon h)\) & \(1\) &
\((\varepsilon g,h,g,h)\) & \(1\) \\

\((\varepsilon h,\varepsilon,\varepsilon g,\varepsilon h)\) & \(-1\) &
\((\varepsilon h,\varepsilon,\varepsilon h,g)\) & \(1\) \\
\hline
\end{tabular}
\end{center}

This factorization is verified in
\gaplink{Gamma_2/diagonal_braiding_theta/theta_diagonal_braiding_cocycle.g}.
Each factor is one by \eqref{eq:3-cocycle-identity}.
The calculation uses only \(hg=\varepsilon gh\),
\(\varepsilon^2=1\), and the centrality of \(\varepsilon\).
Replacing \(g,h\) by \(x,y\) therefore proves
\eqref{eq:Theta-diagonal-braiding-cocycle}.
\end{proof}

\begin{lemma}
\label{lem:Theta-opposite-symmetry}
Let $a,b\in G$ satisfy $ba=\varepsilon ab$, where
$\varepsilon^2=1$ and $\varepsilon\in Z(G)$.
Let $\Theta_\Phi(a,b)$ be the scalar obtained from
\eqref{eq:Theta-Phi-definition} after replacing $(g,h)$ by $(a,b)$.
Then
\begin{equation}
\label{eq:Theta-opposite-symmetry}
\Theta_\Phi(a,b)=\Theta_\Phi(b,a).
\end{equation}
\end{lemma}

\begin{proof}
Since \(ba=\varepsilon ab\) and \(\varepsilon^2=1\), one also has
\(ab=\varepsilon ba\).  Hence
Lemma~\ref{lem:Theta-diagonal-braiding-cocycle} applies to both
\((a,b)\) and \((b,a)\), and gives
\begin{align}
\Theta_\Phi(a,b)
&=
\frac{
\Phi_a(\varepsilon,a)^2\Phi_a(b,\varepsilon a)
}{
\Phi_a(\varepsilon,\varepsilon)\Phi_a(a,b)
},
\label{eq:Theta-ab-reduced}
\\
\Theta_\Phi(b,a)
&=
\frac{
\Phi_b(\varepsilon,b)^2\Phi_b(a,\varepsilon b)
}{
\Phi_b(\varepsilon,\varepsilon)\Phi_b(b,a)
}.
\label{eq:Theta-ba-reduced}
\end{align}
 {Substituting the definitions of \(\Phi_x\) and \(\Phi^x\) in}
\eqref{eq:Theta-ab-reduced} and
\eqref{eq:Theta-ba-reduced}, and using
\(ba=\varepsilon ab\), gives
\[
\frac{\Theta_\Phi(a,b)}{\Theta_\Phi(b,a)}
=
\prod_{i=1}^{14}
\mathscr C_\Phi(a_i,b_i,c_i,d_i)^{m_i},
\]
where the tuples and exponents are given by
\[
\begin{array}{c c r @{\qquad} c c r}
\hline
i & (a_i,b_i,c_i,d_i) & m_i
& i & (a_i,b_i,c_i,d_i) & m_i\\
\hline
1 & (b,a,\varepsilon,a) & 1
& 8 & (a,\varepsilon b,\varepsilon,a) & -1\\
2 & (b,\varepsilon,b,a) & -1
& 9 & (\varepsilon,b,\varepsilon,\varepsilon a) & -1\\
3 & (b,\varepsilon,\varepsilon a,b) & 1
& 10 & (\varepsilon,\varepsilon,b,\varepsilon a) & 1\\
4 & (a,b,\varepsilon,b) & -1
& 11 & (\varepsilon,\varepsilon,a,b) & -1\\
5 & (a,b,\varepsilon,\varepsilon) & 1
& 12 & (\varepsilon b,a,\varepsilon,b) & 1\\
6 & (a,\varepsilon,b,\varepsilon a) & -1
& 13 & (\varepsilon b,a,\varepsilon,\varepsilon) & -1\\
7 & (a,\varepsilon,a,b) & 1
& 14 & (\varepsilon b,\varepsilon,a,\varepsilon) & 1\\
\hline
\end{array}
\]
This factorization is also verified by the GAP program
\gaplink{Gamma_2/theta_symmetry/lemma_B19_theta_symmetry.g}.
Every factor is equal to \(1\) by
\eqref{eq:3-cocycle-identity}, proving
\eqref{eq:Theta-opposite-symmetry}.
\end{proof}

\section{Calculations for the \texorpdfstring{$\Gamma_4$}{Gamma4} case}
\label{app:Gamma4-calculations}

{
Retain the notation of
Section~\ref{sec:classification-Gamma4}.
} 
 {We use the normal form \(\varepsilon^i h^j g^k\) in \(\Gamma_4\),
where \(i\in\mathbb Z/4\mathbb Z\) and \(j,k\in\mathbb Z\).}  These
reductions depend only on the defining relations of \(\Gamma_4\) and
therefore hold in every finite quotient considered in
Section~\ref{sec:classification-Gamma4}.

\subsection{Cocycle reduction for \texorpdfstring{$X_2$}{X2}}

\begin{lemma}
\label{lem:Gamma4-X2-cocycle-reduction}
Recall that
$\alpha=\frac{\rho(\varepsilon^{-1})}
             {\Phi_h(g,\varepsilon^{-1})}$
and $$\kappa_X=
\frac{\rho(\varepsilon^2)}
     {\Phi_h(\varepsilon^2,h)\Phi_h(g,\varepsilon^2h)}.
$$
For
\[
\Xi_X=
\frac{\alpha\Phi_h(h,g)\rho(\varepsilon^{-1}h)}
     {\rho(h)\Phi^h(h,g)\Phi_h(g,\varepsilon^{-1}h)}
\frac{\Phi(h,\varepsilon^{-1}h,\varepsilon g)}
     {\Phi(\varepsilon^{-1}h,h,\varepsilon g)},
\]
one has \(\Xi_X=\kappa_X\).
\end{lemma}

\begin{proof}
 {Equation~\eqref{eq:YD-projective-action} gives}
\begin{align*}
\rho(\varepsilon^{-1})^2
&=\Phi_h(\varepsilon^{-1},\varepsilon^{-1})\rho(\varepsilon^2),\\
\rho(\varepsilon^{-1})\rho(h)
&=\Phi_h(\varepsilon^{-1},h)\rho(\varepsilon^{-1}h).
\end{align*}
Substituting these identities into the definitions above gives
\[
\frac{\Xi_X}{\kappa_X}
=
\frac{
 \Phi_h(\varepsilon^{-1},\varepsilon^{-1})
 \Phi_h(\varepsilon^2,h)\Phi_h(g,\varepsilon^2h)
 \Phi(h,\varepsilon^{-1}h,\varepsilon g)\Phi_h(h,g)
}{
 \Phi_h(g,\varepsilon^{-1})
 \Phi_h(g,\varepsilon^{-1}h)
 \Phi_h(\varepsilon^{-1},h)
 \Phi(\varepsilon^{-1}h,h,\varepsilon g)\Phi^h(h,g)
}.
\]

Apply Lemma~\ref{lem:local-cocycle-coherence} with
\((x,y,z,w)=(h,g,\varepsilon^{-1},h)\), \((h,g,\varepsilon^2,h)\), and
\((h,g,\varepsilon^{-1},\varepsilon^{-1})\).  Since
\(g\varepsilon^{-1}=\varepsilon g\), this gives
\begin{align*}
\Phi_h(\varepsilon^{-1},h)\Phi_h(g,\varepsilon^{-1}h)
&=\Phi_h(g,\varepsilon^{-1})\Phi_h(\varepsilon g,h),\\
\Phi_h(\varepsilon^2,h)\Phi_h(g,\varepsilon^2h)
&=\Phi_h(g,\varepsilon^2)\Phi_h(\varepsilon^2g,h),\\
\Phi_h(\varepsilon^{-1},\varepsilon^{-1})\Phi_h(g,\varepsilon^2)
&=\Phi_h(g,\varepsilon^{-1})
  \Phi_h(\varepsilon g,\varepsilon^{-1}).
\end{align*}
 {Substituting these identities and the definitions of
\(\Phi_x\) and \(\Phi^x\) expresses} \(\Xi_X/\kappa_X\) as
\[
\frac{
 \mathscr C_\Phi(h,g,h,\varepsilon^{-1})
 \mathscr C_\Phi(\varepsilon^{-1}h,\varepsilon g,
                  \varepsilon^{-1},h)
 \mathscr C_\Phi(\varepsilon^{-1}h,h,g,\varepsilon^{-1})
}{
 \mathscr C_\Phi(h,g,\varepsilon^{-1},h)
 \mathscr C_\Phi(\varepsilon^{-1}h,\varepsilon g,
                  h,\varepsilon^{-1})
 \mathscr C_\Phi(h,\varepsilon^{-1}h,g,\varepsilon^{-1})
}.
\]
This factorization is verified in
\gaplink{Gamma_4/x2_cocycle_reduction/gamma4_x2_cocycle_reduction_certificate.g}.
Each factor is one by \eqref{eq:3-cocycle-identity},
so \(\Xi_X=\kappa_X\).
\end{proof}

\subsection{ {Calculation of \texorpdfstring{$z_0$}{z0}}}

\begin{lemma}
\label{lem:Gamma4-z0-factorization}
The coefficient of \(w\otimes(w\otimes v_1)\) in \(z_0\) is
\(-\beta(1+\sigma(g))\).  If \(\Delta_4=1\) and
\(\sigma(g)=-1\), then \(z_0=0\).
\end{lemma}

\begin{proof}
Take \(a=w\) in \eqref{eq:Gamma4-general-Y2-formula}.
 {By \eqref{eq:Gamma4-y1}, the first term contributes
\(-\beta\) to the coefficient of \(w\otimes(w\otimes v_1)\).}  The definition of
\(\varphi_1^\Phi\) and \eqref{eq:YD-projective-action} give
\(\varphi_1^\Phi(w\otimes v_1)
=w\otimes v_1-\Phi_h(g,g)\rho(g^2)w_1\otimes v\).
Moreover, \(g\rhd w=\sigma(g)w\), and the cocycle quotient in
the last term of \eqref{eq:Gamma4-general-Y2-formula} is one.
Thus the last term contributes \(-\beta\sigma(g)\).  The first
\(W\)-factor in each of the other two terms has degree
\(\varepsilon^{-1}g\),  {so neither contributes to this coefficient.
Hence the coefficient is \(-\beta(1+\sigma(g))\).}

Assume that \(\Delta_4=1\) and \(\sigma(g)=-1\).
 {Expanding \(z_0=\varphi_2^\Phi(w\otimes y_1)\) using
\eqref{eq:Gamma4-general-Y2-formula}, we find that two of the four
multi-homogeneous components vanish. The coefficients of the other two
contain the factor \(1-\Delta_4\) and hence also vanish.}
These reductions are verified in
\gaplink{Gamma_4/reflected_y2/gamma4_reflected_y2_certificate.g}
using 
\eqref{eq:YD-projective-action} and the normalized
\(3\)-cocycle identity.
Thus \(z_0=0\).
\end{proof}

\section{Calculations for the \texorpdfstring{$T$}{T} case}
\label{app:T-calculations}

Retain the notation and assumptions of
Section~\ref{sec:classification-T}. All tensor products are
right-associated. For the  calculations in this appendix,
whenever \(\dim\rho=1\), put \(p=\rho(x_1)\sigma(z)\).
{
 {We specify the cocycle on \(Q_8\) and the four-reflection
calculation used to exclude the case \(h=4\). For the reflection
calculation, we use the cocycle in
\eqref{eq:T-order-two-finite-cocycle}.}
}

\subsection{A cocycle on the quaternion group}
\label{subsec:T-quaternion-cocycle}

We give the explicit cocycle
\(F:Q_8^3\to\mathbb Z_8\) used to define \(\Omega\) in
\eqref{eq:T-cocycle-representative}.
Recall that \(Q_8=\{\pm1,\pm i,\pm j,\pm k\}\), with
\(k=ij\).
Identify each element \((-1)^s i^a j^b\) with its
coordinates \((s,a,b)\in\mathbb F_2^3\).
For the automorphism \(A\) defined by
\(A(i)=j\), \(A(j)=k\), and \(A(k)=i\), one has
\[
(s,a,b)(t,c,d)
=(s+t+ac+bd+bc,a+c,b+d),
\qquad
A(s,a,b)=(s,b,a+b).
\]
All operations in these coordinates are taken modulo two.

The \(A\)-orbits in \(Q_8\) are
\(\{1\}\), \(\{-1\}\), \(\{i,j,k\}\), and
\(\{-i,-j,-k\}\).
We set \(F(1,v,w)=0\) and give tables for the first
arguments \(i,-1,-i\).
In the table with first argument \(u\), the entry
in row \(v\) and column \(w\) is \(F(u,v,w)\).
Both rows and columns are ordered as
\((1,i,j,k,-1,-i,-j,-k)\), and all entries are read
modulo eight.
The remaining values are defined by
\[
F(A^r u,v,w)
=
F(u,A^{-r}v,A^{-r}w)
\qquad
u\in\{i,-i\},\ r\in\{0,1,2\}.
\]
The table with first argument \(-1\) satisfies
\(F(-1,Av,Aw)=F(-1,v,w)\).
Thus these rules determine \(F\) on all of \(Q_8^3\)
and give
\(F(Au,Av,Aw)=F(u,v,w)\).

\begingroup\small\setlength{\parindent}{0pt}\setlength{\arraycolsep}{4pt}\renewcommand{\arraystretch}{0.95}
\begin{minipage}[t]{0.48\textwidth}\centering
\[u=i:\quad\begin{pmatrix}
0 & 0 & 0 & 0 & 0 & 0 & 0 & 0\\
0 & 1 & 7 & 6 & 0 & 1 & 3 & 2\\
0 & 3 & 0 & 3 & 0 & 7 & 4 & 3\\
0 & 6 & 3 & 3 & 0 & 2 & 3 & 7\\
0 & 0 & 2 & 2 & 4 & 4 & 6 & 6\\
0 & 5 & 5 & 4 & 0 & 5 & 1 & 0\\
0 & 7 & 6 & 1 & 0 & 3 & 2 & 1\\
0 & 2 & 5 & 1 & 4 & 2 & 1 & 1
\end{pmatrix}\]
\end{minipage}\hfill
\begin{minipage}[t]{0.48\textwidth}\centering
\[u=-1:\quad\begin{pmatrix}
0 & 0 & 0 & 0 & 0 & 0 & 0 & 0\\
0 & 2 & 4 & 6 & 4 & 6 & 0 & 2\\
0 & 6 & 2 & 4 & 4 & 2 & 6 & 0\\
0 & 4 & 6 & 2 & 4 & 0 & 2 & 6\\
0 & 4 & 4 & 4 & 4 & 0 & 0 & 0\\
0 & 6 & 4 & 6 & 0 & 6 & 4 & 6\\
0 & 6 & 6 & 4 & 0 & 6 & 6 & 4\\
0 & 4 & 6 & 6 & 0 & 4 & 6 & 6
\end{pmatrix}\]
\end{minipage}\par\medskip
\begin{minipage}[t]{0.48\textwidth}\centering
\[u=-i:\quad\begin{pmatrix}
0 & 0 & 0 & 0 & 0 & 0 & 0 & 0\\
0 & 3 & 7 & 4 & 0 & 3 & 3 & 0\\
0 & 1 & 6 & 7 & 0 & 5 & 2 & 7\\
0 & 6 & 1 & 1 & 4 & 6 & 5 & 1\\
0 & 4 & 2 & 2 & 0 & 4 & 2 & 2\\
0 & 3 & 1 & 6 & 4 & 7 & 1 & 6\\
0 & 5 & 4 & 1 & 4 & 5 & 4 & 5\\
0 & 2 & 7 & 7 & 4 & 2 & 3 & 7
\end{pmatrix}\]
\end{minipage}\par\medskip
\endgroup

\begin{lemma}
\label{lem:T-fixed-quaternion-cocycle}
The function \(F\) is a normalized \(A\)-invariant additive
three-cocycle, and
\[
F(i,j,k)=F(j,k,i)=F(k,i,j)=3.
\]
If \(u=(-1)^s i^{a(u)}j^{b(u)}\), then
\[
F(u,v,w)\equiv
\sum_{r=0}^2 a(A^ru)a(A^rv)b(A^rw)\pmod2.
\]
\end{lemma}

\begin{proof}
By definition, \(F(1,v,w)=0\).
The zero first rows and columns of the tables give
\(F(u,1,w)=F(u,v,1)=0\), so \(F\) is normalized.
Its \(A\)-invariance follows from the extension rule.
Using the multiplication rule and the table values,
the GAP program
\gaplink{T/verify_T_cohomology_reduction.g}
verifies
\[
F(v,w,t)-F(uv,w,t)+F(u,vw,t)
-F(u,v,wt)+F(u,v,w)=0\pmod8
\]
for all \(8^4\) quadruples \((u,v,w,t)\in Q_8^4\).
Thus \(F\) is an additive three-cocycle.
The table with first argument \(i\) gives
\(F(i,j,k)=3\).
Since \(A(i)=j\), \(A(j)=k\), and \(A(k)=i\),
\(A\)-invariance gives
\(F(j,k,i)=F(k,i,j)=3\).

For the congruence modulo two, both sides are invariant
under simultaneous application of \(A\).
It therefore suffices to consider
\(u\in\{1,-1,i,-i\}\).
For \(u=1,-1\), both sides vanish modulo two.
For \(u=i,-i\), reduction of the corresponding tables gives
\[
F(u,v,w)
\equiv a(v)b(w)+(a(v)+b(v))a(w)\pmod2.
\]
For these two choices of \(u\), the formula for \(A\)
gives
\[
\sum_{r=0}^2 a(A^ru)a(A^rv)b(A^rw)
=
a(v)b(w)+(a(v)+b(v))a(w)
\]
in \(\mathbb F_2\).
This proves the asserted congruence.
\end{proof}

\subsection{Four reflections for the nontrivial cocycle}
\label{subsec:T-order-two-reflections}

We compute the four  reflections used in
Lemma~\ref{lem:T-order-two-Cartan-periodicity}.
Retain \(G_*,\Phi_*\), and \((V_0,W_0)\) from
\eqref{eq:T-order-two-finite-cocycle} and
\eqref{eq:T-order-two-finite-pair}.
Recall that
\(H=Q_8\rtimes\langle t\mid t^3=1\rangle\leq G_*\).
The support elements are
\(x_1=-t\), \(x_2=kt\), \(x_3=jt\), and \(x_4=it\).
In particular, \(x_1x_2x_3=1\).
Assume that the reflections
\(R_2,R_1,R_2,R_1\) are defined, and denote the resulting
pairs by \((V_r,W_r)\), \(0\leq r\leq4\).

\begin{lemma}
\label{lem:T-order-two-reflection-calculation}
The dimensions, supports, and actions of \(z\) are given
by the following table, where \(1\leq i\leq4\) in each
support set and  {the columns labeled \(z|_{V_r}\) and
\(z|_{W_r}\) give the scalars by which \(z\) acts.}
\begin{equation}
\label{eq:T-order-two-reflected-pairs}
\begin{array}{c|c|c|c|c|c}
r&\dim V_r&\operatorname{supp}V_r&z|_{V_r}
&\operatorname{supp}W_r&z|_{W_r}\\ \hline
0&1&\{z\}&p^{-1}&\{x_i\}&p\\
1&2&\{z\}&p^2&\{x_i^{-1}\}&p^{-1}\\
2&2&\{z^{-1}\}&p^{-2}&\{z^2x_i^{-1}\}&-1\\
3&1&\{z^{-1}\}&p&\{z^{-2}x_i\}&-1\\
4&1&\{z\}&p^{-1}&\{z^3x_i\}&-p
\end{array}
\end{equation}
Each \(W_r\) has dimension four, and
\(c_{W_r,W_r}(w\otimes w)=-w\otimes w\)
for every homogeneous \(w\in W_r\).
 {For \(0\leq r\leq3\), the following table lists the dimensions of
the successive adjoint objects, ending with the first zero object.}
\begin{equation}
\label{eq:T-order-two-adjoint-dimensions}
\begin{array}{c|c|c}
r&
\dim(\operatorname{ad}V_r)^m(W_r),\ m=1,2,\ldots&
\dim(\operatorname{ad}W_r)^m(V_r),\ m=1,2,\ldots\\ \hline
0&4,0&4,4,2,0\\
1&4,4,0&4,4,1,0\\
2&4,4,0&4,4,1,0\\
3&4,0&4,4,2,0
\end{array}.
\end{equation}
There is a linear isomorphism \(f:W_4\to W_0\) such that
\[
f\bigl((W_4)_{z^3x_i}\bigr)=(W_0)_{x_i}
\qquad(1\leq i\leq4),
\]
and
$
f(a\rhd w)=a\rhd f(w)$ for 
\(a\in H,\ w\in W_4 \).
The action of \(H\) on \(V_4\) is trivial.
\end{lemma}

\begin{proof}
Choose \(0\neq v\in V_0\), and let \(w_1,\ldots,w_4\)
be the induced basis of \(W_0\) corresponding to
\(1,i,k,j\).
Formula~\eqref{eq:induced-action} gives
\[
c_{W_0,W_0}(w_i\otimes w_j)
=\beta_{ij}w_{i\brhd j}\otimes w_i,
\qquad
(\beta_{ij})=
\begin{pmatrix}
-1&-1&-1&-1\\
1&-1&-1&1\\
1&1&-1&-1\\
1&-1&1&-1
\end{pmatrix}.
\]
The mixed braidings are
\[
c_{V_0,W_0}(v\otimes w_i)=p\,w_i\otimes v,
\qquad
c_{W_0,V_0}(w_i\otimes v)=v\otimes w_i.
\]

All tensor products in the calculation are
right-associated.
For homogeneous vectors of degrees \(a,b,c\), the braiding
 {of the first two tensor factors} acts by
\[
u_a\otimes(v_b\otimes\xi_c)
\longmapsto
\frac{\Phi_*(a,b,c)}
     {\Phi_*(aba^{-1},a,c)}
(a\rhd v_b)\otimes(u_a\otimes\xi_c),
\]
where \(\xi_c\) denotes the remaining tensor, of total
degree \(c\).

Starting from these formulas, the GAP program
\gaplink{T/verify_T_order2_reflections.g}
 {computes the images of the recursively defined maps in
Lemma~\ref{lem:adjoint-object-realization}.}
At each stage, the recursive map is restricted to the
tensor product of the acting object with the preceding
image.
 {Elimination using exact arithmetic gives the dimensions in
\eqref{eq:T-order-two-adjoint-dimensions}, including
the zero images.}
The degrees and induced actions on these images and their duals give
\eqref{eq:T-order-two-reflected-pairs} and
\(c_{W_r,W_r}(w\otimes w)=-w\otimes w\)
for every homogeneous \(w\in W_r\).
The calculations take place over
\(\mathbb Q[p]/(p^2-p+1)\), so they apply to either
root of \(p^2-p+1\).

For the final pair, the same calculation gives
nonzero vectors \(w_i^{(4)}\in(W_4)_{z^3x_i}\)
such that the map \(f:W_4\to W_0\),
\(f(w_i^{(4)})=w_i\), satisfies
\[
f(g\rhd w_i^{(4)})=g\rhd w_i
\qquad g\in H,\ 1\leq i\leq4.
\]
It also gives \(g\rhd v_4=v_4\) for every
\(g\in H\) and \(v_4\in V_4\).
This proves the remaining assertions.
\end{proof}

\section{Calculations for the \texorpdfstring{$\Gamma _3$}{Gamma3} case}
\label{app:Gamma3-calculations}
Section~\ref{sec:classification-Gamma3} states the structural
conclusions.  This appendix gives the coefficient calculations,
bar-chain reductions, and GAP verifications used in their proofs.

All bar-chain identities below are computed in the normalized integral
bar complex of \(\Gamma _3\).  Because they use only the normal form
\(\varepsilon^i g^m h^n\), where
\(i\in\mathbb Z/3\mathbb Z\) and \(m,n\in\mathbb Z\), they descend to
every finite non-abelian quotient of \(\Gamma _3\). Throughout this appendix, we continue to write \(\Phi\) for its
pullback \(\pi^*\Phi\) to \(\Gamma_3\).

\subsection{ {The normalized bar complex}}
\label{app:Gamma3-bar-detector}

Define the normalized bar three-chains
\(L_x(a,b)=[a|b|x]+[{}^{ab}x|a|b]-[a|{}^bx|b]\),
\(U_a(x,y)=[a|x|y]+[{}^ax|{}^ay|a]-[{}^ax|a|y]\), and
\(\mathcal R(a;x,y)=[a|x|y]-[{}^ax|a|y]\).
 {Evaluating \(\Phi\) on these chains gives, respectively,}
\(\Phi_x(a,b)\), \(\Phi^a(x,y)\), and
\[
 R(a;x,y)=\frac{\Phi(a,x,y)}{\Phi({}^ax,a,y)}.
\]

Suppose that a normalized three-chain \(C\) satisfies
\(\partial C=0\) and \(\varpi_*C=\partial D\) for a normalized
four-chain \(D\) over \(S_3\).  The injectivity in
Lemma~\ref{lem:Gamma3-third-homology} implies that \(C\) is a
boundary over \(\Gamma _3\), so \(\langle\Phi,C\rangle=1\).
This also applies when \(\varpi_*C=0\), by taking \(D=0\).

\subsection{The \texorpdfstring{$(3,2)$}{(3,2)} obstruction}
\label{app:Gamma3-32-obstruction}

Retain \(V,W,v_0,w_0,v_i,w_1,E_0,E_1,u_0\) from
Lemma~\ref{lem:Gamma3-32-X1}.  With indices in
\(\mathbb Z/3\mathbb Z\), set
\[
 p_i=g\varepsilon^i,\qquad q_i=p_ih,
 \qquad y_0=\varepsilon h,\qquad y_1=\varepsilon^2h.
\]
Thus \(v_i\in V_{p_i}\), \(w_j\in W_{y_j}\). If \(\xi_x,\eta_y,\zeta_z\) are homogeneous, the definition of
\(c_{1,2}^\Phi\) and the associativity constraint give
\begin{equation}
\label{eq:Gamma3-32-c12}
 c_{1,2}^\Phi
 \bigl(\xi_x\otimes(\eta_y\otimes\zeta_z)\bigr)
 =R(x;y,z)(x\rhd\eta_y)
   \otimes(\xi_x\otimes\zeta_z).
\end{equation}
Retain \(a,b_0,\Delta_1,\chi,s\) from
\eqref{eq:Gamma3-a}--\eqref{eq:Gamma3-Delta1-s}.
The calculation in Lemma~\ref{lem:Gamma3-32-X1} gives
\(u_0:=\varphi_1^\Phi(E_0)=E_0-aE_1\), together with
\(g\rhd E_0=b_0E_1\).
Assume henceforth that \(\Delta_1=1\), and let
\(u_i=\varepsilon^i\rhd u_0\).
Then \((X_1)_{q_i}=\mathbb C u_i\).  Set
\begin{align*}
\ell_0={}&\Phi^\varepsilon(p_2,y_0)
 \Phi_g(\varepsilon,\varepsilon^2)\sigma(\varepsilon),
&k_2={}&
 \frac{\Phi_g(g,\varepsilon^2)}
      {\Phi_g(\varepsilon,g)}\rho(g),\\
\ell_1={}&\Phi^{\varepsilon^2}(p_1,y_1)
 \Phi_g(\varepsilon^2,\varepsilon)
 \frac{\Phi_{\varepsilon h}(\varepsilon^2,g)}
      {\Phi_{\varepsilon h}(g,\varepsilon)}\sigma(\varepsilon),
&k_1={}&
 \frac{\Phi_g(g,\varepsilon)}
      {\Phi_g(\varepsilon^2,g)}\rho(g),
\end{align*}
and define
\begin{align*}
  {c_1}& {=R(g;p_2,y_0)\frac{k_2}{\ell_0}},\\
  {c_2}& {=R(g;p_1,y_1)\frac{k_1}{\ell_1}}.
\end{align*}

\begin{lemma}
\label{lem:Gamma3-32-second-adjoint-expansion}
Let \(P_i=v_i\otimes u_i\).  These tensors have total degree
\(g^2h\), and
\begin{equation}
\label{eq:Gamma3-32-central-vector}
 \varphi_2^\Phi(P_0)=(1+s)P_0+c_1P_1+c_2P_2,
\end{equation}
where \(c_1c_2\ne0\).  Moreover, \(v_2\otimes u_0\) and
\(v_0\otimes u_1\) have total degree \(\varepsilon g^2h\), and
\begin{equation}
\label{eq:Gamma3-32-noncentral-vector}
 \varphi_2^\Phi(v_2\otimes u_0)
 -(1+\rho(g))(v_2\otimes u_0)
 \in\mathbb C(v_0\otimes u_1).
\end{equation}
\end{lemma}

\begin{proof}
Since \(p_i^2=g^2\), one has \(p_iq_i=g^2h\).
Also, \(p_2q_0=p_0q_1=\varepsilon g^2h\).
This proves the assertions about degrees.

The definitions of \(\ell_0,\ell_1\), together with
\eqref{eq:YD-projective-action} and
\eqref{eq:tensor-product}, give
\(\varepsilon\rhd E_0=\ell_0(v_0\otimes w_0)\) and
\(\varepsilon^2\rhd E_1=\ell_1(v_0\otimes w_1)\).
The map \(\varphi_1^\Phi\) commutes with the \(G\)-action.
Since \(g\rhd E_0=b_0E_1\) and
\(g\rhd u_0=\chi u_0\), it follows that
\[
\begin{aligned}
&\varphi_1^\Phi(E_1)
=b_0^{-1}(g\rhd u_0)=-a^{-1}u_0,\\
&\varphi_1^\Phi(v_0\otimes w_0)
=\ell_0^{-1}u_1,\\
&\varphi_1^\Phi(v_0\otimes w_1)
=-(a\ell_1)^{-1}u_2.
\end{aligned}
\]
Moreover, \eqref{eq:YD-projective-action} gives
\(g\rhd v_2=k_2v_1\) and \(g\rhd v_1=k_1v_2\).
Using \(u_0=E_0-aE_1\) and
\eqref{eq:Gamma3-32-c12}, we therefore obtain
\[
\begin{aligned}
(\operatorname{id}_V\otimes\varphi_1^\Phi)
c_{1,2}^\Phi(P_0)
={}&R(g;p_2,y_0)k_2v_1\otimes
       \varphi_1^\Phi(v_0\otimes w_0)-aR(g;p_1,y_1)k_1v_2\otimes
       \varphi_1^\Phi(v_0\otimes w_1)\\
={}&c_1P_1+c_2P_2.
\end{aligned}
\]
On the other hand, \(u_0\) has degree \(gh\), so
$
c_{X_1,V}c_{V,X_1}(P_0)
=\chi\rho(gh)P_0.
$
The defining recursion for \(\varphi_2^\Phi\) and
\(s=-\chi\rho(gh)\) now give
\eqref{eq:Gamma3-32-central-vector}.
All factors in the definitions of \(c_1,c_2\) are
nonzero, hence \(c_1c_2\neq0\).

It remains to compute the coefficient of
\(v_2\otimes u_0\).
We have
\(
\Phi_g(p_2,\varepsilon^2)
=\Phi(p_2,\varepsilon^2,g)
=\Phi_g(\varepsilon^2,g).
\)
Since \(p_2\varepsilon^2=\varepsilon^2g\),
\eqref{eq:YD-projective-action} implies
\(p_2\rhd v_2=\rho(g)v_2\).
Together with \(R(p_2;p_2,y_0)=1\), this gives
\[
(\operatorname{id}_V\otimes\varphi_1^\Phi)
c_{1,2}^\Phi(v_2\otimes E_0)
=\rho(g)(v_2\otimes u_0).
\]
The group relations give
\({}^{p_2}q_0=q_1\), \({}^{q_1}p_2=p_0\), and
\({}^{p_2}p_1=p_0\).
Also, \(p_2y_1=q_1\), so
\(\varphi_1^\Phi(v_2\otimes w_1)\in\mathbb Cu_1\).
Consequently,
\[
\begin{aligned}
c_{X_1,V}c_{V,X_1}(v_2\otimes u_0)
&\in\mathbb C(v_0\otimes u_1),\\
(\operatorname{id}_V\otimes\varphi_1^\Phi)
c_{1,2}^\Phi(v_2\otimes E_1)
&\in\mathbb C(v_0\otimes u_1).
\end{aligned}
\]
Substituting these three relations and
\(u_0=E_0-aE_1\) into the recursion proves
\eqref{eq:Gamma3-32-noncentral-vector}.
\end{proof}

Equation~\eqref{eq:tensor-product} gives
\(\varepsilon\rhd P_0=t_0P_1\),
\(\varepsilon\rhd P_1=t_1P_2\), and
\(\varepsilon\rhd P_2=t_2P_0\), where
\begin{align}
 t_0&=\Phi^\varepsilon(p_0,q_0),
 \label{eq:Gamma3-32-t0}\\
 t_1&=\Phi^\varepsilon(p_1,q_1)
      \Phi_g(\varepsilon,\varepsilon)
      \Phi_{gh}(\varepsilon,\varepsilon),
 \label{eq:Gamma3-32-t1}\\
 t_2&=\Phi^\varepsilon(p_2,q_2)
      \Phi_g(\varepsilon,\varepsilon^2)
      \Phi_{gh}(\varepsilon,\varepsilon^2).
 \label{eq:Gamma3-32-t2}
\end{align}
Indeed, the formula for \(t_0\) follows from
\(v_1=\varepsilon\rhd v_0\) and
\(u_1=\varepsilon\rhd u_0\).  For the other two,
\eqref{eq:YD-projective-action} gives
\(\varepsilon\rhd v_1
=\Phi_g(\varepsilon,\varepsilon)v_2\),
\(\varepsilon\rhd u_1
=\Phi_{gh}(\varepsilon,\varepsilon)u_2\),
\(\varepsilon\rhd v_2
=\Phi_g(\varepsilon,\varepsilon^2)v_0\), and
\(\varepsilon\rhd u_2
=\Phi_{gh}(\varepsilon,\varepsilon^2)u_0\).
 {Substitution in \eqref{eq:tensor-product} gives \(t_1\) and \(t_2\).
All three scalars are nonzero.}

\begin{lemma}
\label{lem:Gamma3-32-bar-reduction}
Assume \(\Delta_1=1\) and \(\rho(g)=-1\), and retain
\(s,\lambda,r\) from
\eqref{eq:Gamma3-Delta1-s} and
\eqref{eq:Gamma3-32-summary-lambda-r}.
Then the identities in \eqref{eq:Gamma3-32-summary-identities} hold.
\end{lemma}

\begin{proof}
Define the following normalized bar three-chains:
\begin{align*}
 \mathsf A={}&
 L_g(\varepsilon^2h,\varepsilon^2)
 -L_g(\varepsilon,h)-L_{y_0}(g,\varepsilon^2),\\
 \mathsf B_0={}&
 U_g(p_2,y_0)+L_g(g,\varepsilon^2)-L_g(\varepsilon,g),\\
 \mathsf S={}&
 \mathsf B_0-\mathsf A-L_g(g,h)+L_{y_0}(\varepsilon,\varepsilon),\\
 \mathsf E={}&
 L_{y_0}(\varepsilon,\varepsilon)
 +L_{y_0}(\varepsilon^2,\varepsilon),\\
 \mathsf C_1={}&
 \mathcal R(g;p_2,y_0)+L_g(g,\varepsilon^2)-L_g(\varepsilon,g)
 -U_\varepsilon(p_2,y_0)-L_g(\varepsilon,\varepsilon^2),\\
 \mathsf C_2={}&
 \mathcal R(g;p_1,y_1)+L_g(g,\varepsilon)-L_g(\varepsilon^2,g)
 -U_{\varepsilon^2}(p_1,y_1)-L_g(\varepsilon^2,\varepsilon)\\
 &\quad-L_{y_0}(\varepsilon^2,g)+L_{y_0}(g,\varepsilon),\\
 \mathsf T_0={}&U_\varepsilon(p_0,q_0),\\
 \mathsf T_1={}&U_\varepsilon(p_1,q_1)
 +L_g(\varepsilon,\varepsilon)+L_{gh}(\varepsilon,\varepsilon),\\
 \mathsf T_2={}&U_\varepsilon(p_2,q_2)
 +L_g(\varepsilon,\varepsilon^2)
 +L_{gh}(\varepsilon,\varepsilon^2).
\end{align*}
{
Substituting the evaluations of \(L,U,\mathcal R\) into the
formulas for \(a,b_0,\ell_i,k_i,c_i,t_i\), and using
\eqref{eq:YD-projective-action} and \(\rho(g)=-1\), gives
}
\begin{equation}
\label{eq:Gamma3-32-scalar-chain-dictionary}
\begin{aligned}
s&=\sigma(\varepsilon)^{-2}\langle\Phi,\mathsf S\rangle,
&\sigma(\varepsilon)^3&=\langle\Phi,\mathsf E\rangle,\\
c_i&=-\sigma(\varepsilon)^{-1}
\langle\Phi,\mathsf C_i\rangle\ (i=1,2),
&t_i&=\langle\Phi,\mathsf T_i\rangle\ (i=0,1,2).
\end{aligned}
\end{equation}
Let \(\mathsf K\) be the cycle in
\eqref{eq:Gamma3-kappa-cycle}, and set
\begin{align*}
\mathsf C_s&=3\mathsf S-2\mathsf E-\mathsf K,\\
\mathsf C_\lambda&=
3(\mathsf C_1+\mathsf T_1-\mathsf C_2)
-(\mathsf T_0+\mathsf T_1+\mathsf T_2),\\
\mathsf C_r&=
4\mathsf C_1+2\mathsf T_1-2\mathsf C_2
-2\mathsf T_0-\mathsf S.
\end{align*}
{
The identities in \eqref{eq:Gamma3-32-scalar-chain-dictionary}
and the following two properties are verified coefficientwise in
\gaplink{Gamma_3/verify_gamma3_32.g}: the three chains have zero
normalized bar differential, and their images under
\(\varpi:\Gamma_3\to S_3\) are boundaries. The program records
explicit integral normalized four-chains with these boundaries.
}

By the argument in Subsection~\ref{app:Gamma3-bar-detector}, all
three chains are boundaries over \(\Gamma_3\), and their evaluations
by \(\Phi\) are \(1\).
By \eqref{eq:Gamma3-32-scalar-chain-dictionary},
\eqref{eq:Gamma3-32-summary-lambda-r}, and
\(\langle\Phi,\mathsf K\rangle=\kappa_\Phi(\varepsilon)\), these
evaluations are respectively
\(s^3/\kappa_\Phi(\varepsilon)\),
\(\lambda^3/(t_0t_1t_2)\), and \(r^2/s\).
This proves \eqref{eq:Gamma3-32-summary-identities}.
\end{proof}

\subsection{The \texorpdfstring{$(3,1)$}{(3,1)} obstruction with
	\texorpdfstring{$\dim\tau=2$}{dim tau=2}}
\label{app:Gamma3-31-dim2-obstruction}

Retain \(a_\varepsilon\) from \eqref{eq:Gamma3-a-epsilon},
\(b,\zeta_3,w,\overline w\) from
Lemma~\ref{lem:Gamma3-dim2-normal-form}, and \(v,x_1\) from
Lemma~\ref{lem:Gamma3-dim2-X1-X2-support}.  Thus
\(x_1=v\otimes(w-\rho(h)\overline w)\).  For
\(i\in\mathbb Z/3\mathbb Z\), put
\(v_i=\varepsilon^i\rhd v\) and
\(x_{1,i}=\varepsilon^i\rhd x_1\), and let
\(e_0=v_2\otimes x_{1,0}\),
\(e_1=v_0\otimes x_{1,1}\), and
\(e_2=v_1\otimes x_{1,2}\).

\begin{lemma}
	\label{lem:Gamma3-dim2-second-adjoint-expansion}
	Assume that \(\Delta_3=I\) and \(\rho(g)=-1\).  The vectors
	\(e_0,e_1,e_2\) lie in
	\((V\otimes X_1)_{\varepsilon g^2h}\) and are linearly independent.
	There are nonzero scalars
	\(\mathfrak a,\mathfrak b,m_0,m_1,m_2\) such that
	\begin{equation}
		\label{eq:Gamma3-dim2-three-path-vector}
		x_2:=\varphi_2^\Phi(e_0)
		=e_0+\mathfrak a e_1+\mathfrak b e_2\ne0
	\end{equation}
	and
	\begin{equation}
		\label{eq:Gamma3-dim2-m-action}
		\varepsilon\rhd e_0=m_0(\mathfrak a e_1),\qquad
		\varepsilon\rhd(\mathfrak a e_1)=m_1(\mathfrak b e_2),\qquad
		\varepsilon\rhd(\mathfrak b e_2)=m_2e_0.
	\end{equation}
\end{lemma}

\begin{proof}
Set
\(C_0=\Phi_h(\varepsilon,\varepsilon)a_\varepsilon^2/
\Phi^{\varepsilon^2}(g,h)\).
Equations~\eqref{eq:YD-projective-action},
\eqref{eq:tensor-product}, and
\eqref{eq:Gamma3-dim2-standard-basis} give
\[
\begin{aligned}
\varepsilon^2\rhd(v\otimes w)
&=\zeta_3^2C_0^{-1}(v_2\otimes w),\\
\varepsilon^2\rhd(v\otimes\overline w)
&=\zeta_3C_0^{-1}(v_2\otimes\overline w).
\end{aligned}
\]
The calculation in the proof of
Lemma~\ref{lem:Gamma3-dim2-X1-X2-support} gives
\(\varphi_1^\Phi(v\otimes w)=x_1\) and
\(\varphi_1^\Phi(v\otimes\overline w)
=-\rho(h)^{-1}x_1\).
Since \(\varphi_1^\Phi\) commutes with the
\(\varepsilon^2\)-action and
\(x_{1,2}=\varepsilon^2\rhd x_1\), we obtain
\begin{equation}
\label{eq:Gamma3-dim2-shifted-X1}
\begin{aligned}
\varphi_1^\Phi(v_2\otimes w)
&=\zeta_3 C_0x_{1,2},\\
\varphi_1^\Phi(v_2\otimes\overline w)
&=-\zeta_3^2\rho(h)^{-1}C_0x_{1,2},\\
\varphi_1^\Phi
\bigl(v_2\otimes(w-\rho(h)\overline w)\bigr)
&=-C_0x_{1,2}.
\end{aligned}
\end{equation}
The last equality follows from the first two and
\(\zeta_3+\zeta_3^2=-1\).
    
  Define
\begin{align}
\mathfrak a
&=\frac{\Phi_g({}^{\varepsilon}(gh),\varepsilon^2)}
        {\Phi_{gh}(\varepsilon,g)},
\label{eq:Gamma3-dim2-mathfrak-a}\\
\mathfrak b
&= {\frac{
\Phi(\varepsilon g,g,h)
\Phi_h(\varepsilon,\varepsilon)a_\varepsilon^2
}{
\Phi(\varepsilon^2g,\varepsilon g,h)
\Phi_g(\varepsilon,g)\Phi^{\varepsilon^2}(g,h)
}}.
\label{eq:Gamma3-dim2-mathfrak-b}
\end{align}
Both scalars are nonzero.
 {The tensors \(e_0,e_1,e_2\) have total degree
\(\varepsilon g^2h\). Their first tensor factors have degrees
\(g\varepsilon^2,g,g\varepsilon\), respectively.}
These degrees are distinct, so the three tensors are
linearly independent.
Equations~\eqref{eq:YD-projective-action} and
\eqref{eq:tensor-product} give
\[
\begin{aligned}
(\varepsilon g)\rhd x_1
&=-\frac{\Phi^g(g,h)\rho(g)}
        {\rho(h)\Phi_{gh}(\varepsilon,g)}x_{1,1},\\
{}^\varepsilon(gh)\rhd v_2
&=\Phi_g({}^\varepsilon(gh),\varepsilon^2)
  \rho(gh)v.
\end{aligned}
\]
Since \(g\) and \(h\) commute,
\(\Phi^g(g,h)=\Phi_g(g,h)\) and
\(\rho(gh)=\rho(g)\rho(h)/\Phi_g(g,h)\).
Using \(\rho(g)=-1\), we obtain
\[
-c_{X_1,V}c_{V,X_1}(e_0)
=\rho(g)^2
 \frac{\Phi_g({}^{\varepsilon}(gh),\varepsilon^2)}
      {\Phi_{gh}(\varepsilon,g)}e_1
=\mathfrak a e_1.
\]
Similarly, \eqref{eq:YD-projective-action} gives
\((\varepsilon g)\rhd v
=\rho(g)\Phi_g(\varepsilon,g)^{-1}v_1\).
Equations~\eqref{eq:Gamma3-32-c12} and
\eqref{eq:Gamma3-dim2-shifted-X1} therefore yield
\[
\begin{aligned}
(\operatorname{id}_V\otimes\varphi_1^\Phi)
c_{1,2}^\Phi(e_0)
=-\rho(g)
  \frac{\Phi(\varepsilon g,g,h)C_0}
       {\Phi(\varepsilon^2g,\varepsilon g,h)
        \Phi_g(\varepsilon,g)}e_2
=\mathfrak b e_2.
\end{aligned}
\]
 {Together with the term \(e_0\) contributed by the identity map,}
the recursion \eqref{eq:recursive-varphi-2} gives
\eqref{eq:Gamma3-dim2-three-path-vector}.
This vector is nonzero by the linear independence of
\(e_0,e_1,e_2\).
Finally, \eqref{eq:tensor-product} and
\eqref{eq:YD-projective-action} give
\begin{equation}
\label{eq:Gamma3-dim2-eta-values}
\begin{aligned}
\varepsilon\rhd e_0&=\eta_0e_1,
&
\eta_0&=\Phi^\varepsilon(\varepsilon g,gh)
        \Phi_g(\varepsilon,\varepsilon^2),\\
\varepsilon\rhd e_1&=\eta_1e_2,
&
\eta_1&=\Phi^\varepsilon(g,{}^\varepsilon(gh))
        \Phi_{gh}(\varepsilon,\varepsilon),\\
\varepsilon\rhd e_2&=\eta_2e_0,
&
\eta_2&=\Phi^\varepsilon
        (\varepsilon^2g,{}^{\varepsilon^2}(gh))
        \Phi_g(\varepsilon,\varepsilon)
        \Phi_{gh}(\varepsilon,\varepsilon^2).
\end{aligned}
\end{equation}
Thus, with
\[
m_0=\frac{\eta_0}{\mathfrak a},
\qquad
m_1=\frac{\mathfrak a\eta_1}{\mathfrak b},
\qquad
m_2=\mathfrak b\eta_2,
\]
equation~\eqref{eq:Gamma3-dim2-m-action} follows.
All three scalars are nonzero.
\end{proof}

 {To relate \(m_0\) and \(m_2\), we express them
as evaluations of normalized bar three-chains.}
The same chains will be used in
Subsection~\ref{app:Gamma3-31-dim1-obstruction}.
Define
\begin{align*}
\mathsf E={}&
L_h(\varepsilon,g)-L_h(g,\varepsilon^2)
-L_h(\varepsilon,\varepsilon),\\
\mathsf A={}&
L_g({}^\varepsilon(gh),\varepsilon^2)
-L_{gh}(\varepsilon,g),\\
\mathsf B={}&
[\varepsilon g|g|h]
-[\varepsilon^2g|\varepsilon g|h]+2\mathsf E
+L_h(\varepsilon,\varepsilon)-L_g(\varepsilon,g)
-U_{\varepsilon^2}(g,h),\\
\mathsf T_0={}&
U_\varepsilon(\varepsilon g,gh)
+L_g(\varepsilon,\varepsilon^2),\\
\mathsf T_1={}&
U_\varepsilon(g,{}^\varepsilon(gh))
+L_{gh}(\varepsilon,\varepsilon),\\
\mathsf T_2={}&
U_\varepsilon(\varepsilon^2g,{}^{\varepsilon^2}(gh))
+L_g(\varepsilon,\varepsilon)
+L_{gh}(\varepsilon,\varepsilon^2).
\end{align*}
The definitions of the chains and scalar coefficients give
\begin{equation}
\label{eq:Gamma3-dim1-chain-evaluations}
\begin{aligned}
\langle\Phi,\mathsf E\rangle&=a_\varepsilon,
&
\langle\Phi,\mathsf A\rangle&=\mathfrak a,\\
\langle\Phi,\mathsf B\rangle&=\mathfrak b,
&
\langle\Phi,\mathsf T_i\rangle&=\eta_i
\qquad(0\leq i\leq2).
\end{aligned}
\end{equation}
Set
\begin{equation}
\label{eq:Gamma3-dim2-M-chains}
\mathsf M_0=\mathsf T_0-\mathsf A,\qquad
\mathsf M_1=\mathsf A+\mathsf T_1-\mathsf B,\qquad
\mathsf M_2=\mathsf B+\mathsf T_2.
\end{equation}
Retain \(\mathsf K\) from
\eqref{eq:Gamma3-kappa-cycle}, and put
$
\mathsf\Theta
=L_{\varepsilon g^2h}(\varepsilon,\varepsilon)
+L_{\varepsilon g^2h}(\varepsilon^2,\varepsilon).
$
Thus
\(\langle\Phi,\mathsf\Theta\rangle
=\Phi_{\varepsilon g^2h}(\varepsilon,\varepsilon)
 \Phi_{\varepsilon g^2h}(\varepsilon^2,\varepsilon)\)
and
\(\langle\Phi,\mathsf K\rangle
=\kappa_\Phi(\varepsilon)\).
Define
\begin{equation}
\label{eq:Gamma3-dim2-comparison-cycles}
\mathsf C_{02}
=\mathsf M_0-\mathsf M_2-\mathsf K,
\qquad
\mathsf C_{\mathrm{pow}}
=3\mathsf M_0-\mathsf\Theta+\mathsf K.
\end{equation}

\begin{lemma}
\label{lem:Gamma3-dim2-bar-reduction}
The chains \(\mathsf C_{02}\) and
\(\mathsf C_{\mathrm{pow}}\) are boundaries in the
normalized bar complex of \(\Gamma_3\).
Under the hypotheses of
Lemma~\ref{lem:Gamma3-dim2-second-adjoint-expansion},
the scalars \(m_0,m_2\) satisfy
\begin{equation}
\label{eq:Gamma3-common-three-path-reduction}
m_2=m_0\kappa_\Phi(\varepsilon)^{-1}.
\end{equation}
\end{lemma}

\begin{proof}
The normalized bar differential gives
\[
\begin{aligned}
\partial\mathsf M_i
=[\varepsilon|\varepsilon g^2h]
 -[\varepsilon g^2h|\varepsilon],
 \ 0\leq i\leq2,\ \ \
\partial\mathsf\Theta
=3\bigl(
 [\varepsilon|\varepsilon g^2h]
 -[\varepsilon g^2h|\varepsilon]
 \bigr).
\end{aligned}
\]
Since \(\mathsf K\) is a cycle,
both \(\mathsf C_{02}\) and
\(\mathsf C_{\mathrm{pow}}\) are cycles.
The GAP program
\gaplink{Gamma_3/verify_gamma3_31_dim2.g}
verifies these differential identities and gives integral
four-chains over \(S_3\) whose boundaries are
\(\varpi_*\mathsf C_{02}\) and
\(\varpi_*\mathsf C_{\mathrm{pow}}\), respectively.
The injectivity of \(\varpi_*\) in
Lemma~\ref{lem:Gamma3-third-homology}
therefore shows that both cycles are boundaries
over \(\Gamma_3\).

Under the stated hypotheses,
\eqref{eq:Gamma3-dim1-chain-evaluations},
\eqref{eq:Gamma3-dim2-M-chains}, and the definitions of
\(m_0,m_2\) give
\[
\langle\Phi,\mathsf M_0\rangle
=\frac{\eta_0}{\mathfrak a}=m_0,
\qquad
\langle\Phi,\mathsf M_2\rangle
=\mathfrak b\eta_2=m_2.
\]
Evaluating the boundary \(\mathsf C_{02}\) yields
\[
1=\langle\Phi,\mathsf C_{02}\rangle
=\frac{m_0}{m_2\kappa_\Phi(\varepsilon)},
\]
which proves
\eqref{eq:Gamma3-common-three-path-reduction}.
\end{proof}

\subsection{The \texorpdfstring{$(3,1)$}{(3,1)} obstruction with
\texorpdfstring{$\dim\tau=1$}{dim tau=1}}
\label{app:Gamma3-31-dim1-obstruction}

For the remainder of this subsection, retain \(V,W,v,w\) and
\(\rho_1\) from Lemma~\ref{lem:Gamma3-dim1-X1-X2}, assume
\(\rho(h)\tau(g)\ne1\), and retain \(a_\varepsilon\) from
\eqref{eq:Gamma3-a-epsilon}.  Applying
\eqref{eq:YD-projective-action} to
\((\varepsilon,g)\), \((\varepsilon,\varepsilon)\), and
\((g,\varepsilon^2)\) gives
\(
 \bigl(a_\varepsilon\tau(\varepsilon)\bigr)\tau(g)
 =\tau(g)\bigl(a_\varepsilon\tau(\varepsilon)\bigr)^2.
\)
Since these scalars are nonzero, cancellation yields
\begin{equation}
\label{eq:Gamma3-dim1-epsilon-normalization}
 a_\varepsilon\tau(\varepsilon)=1.
\end{equation}
Equation~\eqref{eq:YD-projective-action} also gives
\begin{equation}
\label{eq:Gamma3-dim1-epsilon-square}
 \tau(\varepsilon)^2
 =\Phi_h(\varepsilon,\varepsilon)\tau(\varepsilon^2),
 \qquad
 \frac{1}{\tau(\varepsilon^2)}
 =\Phi_h(\varepsilon,\varepsilon)a_\varepsilon^2.
\end{equation}
For \(i\in\mathbb Z/3\mathbb Z\), put
\(v_i=\varepsilon^i\rhd v\) and
\(x_{1,i}=\varepsilon^i\rhd(v\otimes w)\), and let
\(e_0=v_2\otimes x_{1,0}\),
\(e_1=v_0\otimes x_{1,1}\), and
\(e_2=v_1\otimes x_{1,2}\).

\begin{lemma}
\label{lem:Gamma3-dim1-three-elementary-tensors}
The vectors \(e_0,e_1,e_2\) lie in the component of degree
\(\varepsilon g^2h\) and are linearly independent.  There are
nonzero scalars
\(\mathfrak a,\mathfrak b,\eta_0,\eta_1,\eta_2\) such that
\begin{equation}
\label{eq:Gamma3-dim1-three-elementary-vector}
 x_2:=\varphi_2^\Phi(e_0)
 =e_0-\rho(g)^2\rho(h)\tau(g)\mathfrak a e_1
 +\rho(g)\bigl(1-\rho(h)\tau(g)\bigr)\mathfrak b e_2.
\end{equation}
Moreover, \(\varepsilon\rhd e_0=\eta_0e_1\),
\(\varepsilon\rhd e_1=\eta_1e_2\), and
\(\varepsilon\rhd e_2=\eta_2e_0\).
\end{lemma}
\begin{proof}
 {The nonzero tensors \(e_0,e_1,e_2\) have total degree
\(\varepsilon g^2h\). Their first tensor factors have degrees
\(g\varepsilon^2,g,g\varepsilon\), respectively.}
These degrees are distinct, so the tensors are
linearly independent.
Take \(\mathfrak a,\mathfrak b\) from
\eqref{eq:Gamma3-dim2-mathfrak-a} and
\eqref{eq:Gamma3-dim2-mathfrak-b}.
The definition of \(\rho_1\) and
\eqref{eq:YD-projective-action} give
\[
\begin{aligned}
(\varepsilon g)\rhd x_{1,0}
&=\frac{\rho_1(g)}
        {\Phi_{gh}(\varepsilon,g)}x_{1,1},\\
{}^\varepsilon(gh)\rhd v_2
&=\Phi_g({}^\varepsilon(gh),\varepsilon^2)
  \rho(gh)v.
\end{aligned}
\]
Since
\(\rho_1(g)=\Phi^g(g,h)\rho(g)\tau(g)\),
\(\Phi^g(g,h)=\Phi_g(g,h)\), and
\(\rho(gh)=\rho(g)\rho(h)/\Phi_g(g,h)\), we obtain
\begin{equation}
\label{eq:Gamma3-dim1-noncentral-double-braiding}
\begin{aligned}
c_{X_1,V}c_{V,X_1}(e_0)
=\rho_1(g)\rho(gh)\mathfrak a e_1
=\rho(g)^2\rho(h)\tau(g)\mathfrak a e_1.
\end{aligned}
\end{equation}
 {For the last term in the recursive formula for \(\varphi_2^\Phi\),}
\eqref{eq:tensor-product} gives
\(
x_{1,2}
=\Phi^{\varepsilon^2}(g,h)\tau(\varepsilon^2)
  (v_2\otimes w).
\)
Since \(\varphi_1^\Phi\) commutes with the
\(\varepsilon^2\)-action and
\(\varphi_1^\Phi(v\otimes w)
=(1-\rho(h)\tau(g))v\otimes w\), it follows that
\[
\varphi_1^\Phi(v_2\otimes w)
=\frac{1-\rho(h)\tau(g)}
       {\Phi^{\varepsilon^2}(g,h)\tau(\varepsilon^2)}
  x_{1,2}.
\]
Also, \eqref{eq:YD-projective-action} gives
\((\varepsilon g)\rhd v
=\rho(g)\Phi_g(\varepsilon,g)^{-1}v_1\).
Substituting these formulas into
\eqref{eq:Gamma3-32-c12}, and using
\eqref{eq:Gamma3-dim1-epsilon-square}, yields
\begin{equation}
\label{eq:Gamma3-dim1-noncentral-c12}
(\operatorname{id}_V\otimes\varphi_1^\Phi)
c_{1,2}^\Phi(e_0)
=\rho(g)\bigl(1-\rho(h)\tau(g)\bigr)
 \mathfrak b e_2.
\end{equation}
 {Together with the term \(e_0\) contributed by the identity map,}
\eqref{eq:Gamma3-dim1-noncentral-double-braiding} and
\eqref{eq:Gamma3-dim1-noncentral-c12}
give \eqref{eq:Gamma3-dim1-three-elementary-vector}
by \eqref{eq:recursive-varphi-2}.
Finally, applying \eqref{eq:tensor-product} and
\eqref{eq:YD-projective-action} to
\(v_i=\varepsilon^i\rhd v\) and
\(x_{1,i}=\varepsilon^i\rhd(v\otimes w)\) gives
\[
\varepsilon\rhd e_i=\eta_i e_{i+1}
\qquad i\in\mathbb Z/3\mathbb Z,
\]
with \(\eta_i\) as in
\eqref{eq:Gamma3-dim2-eta-values}.
The definitions show that
\(\mathfrak a,\mathfrak b,\eta_0,\eta_1,\eta_2\)
are nonzero.
\end{proof}

\begin{lemma}
\label{lem:Gamma3-dim1-chain-identification}
For the chains \(\mathsf M_i\) in
\eqref{eq:Gamma3-dim2-M-chains}, let
\(\mu_i=\langle\Phi,\mathsf M_i\rangle\).  Then
\(\mu_0,\mu_1,\mu_2\) are nonzero, and the scalars \(m_i\) in
\eqref{eq:Gamma3-dim1-m-definition} satisfy
\begin{equation}
\label{eq:Gamma3-dim1-m-mu}
 \begin{aligned}
 m_0&=-\frac{\mu_0}{\rho(g)^2\rho(h)\tau(g)},\\
 m_1&=-\frac{\rho(g)\rho(h)\tau(g)}
 {1-\rho(h)\tau(g)}\mu_1,\\
 m_2&=\rho(g)\bigl(1-\rho(h)\tau(g)\bigr)\mu_2.
 \end{aligned}
\end{equation}
\end{lemma}

\begin{proof}
Equations~\eqref{eq:Gamma3-dim1-chain-evaluations} and
\eqref{eq:Gamma3-dim2-M-chains} give
\[
\mu_0=\frac{\eta_0}{\mathfrak a},
\qquad
\mu_1=\frac{\mathfrak a\eta_1}{\mathfrak b},
\qquad
\mu_2=\mathfrak b\eta_2.
\]
These scalars are nonzero.
 {Comparing the coefficients in
\eqref{eq:Gamma3-dim1-three-elementary-vector} with
\eqref{eq:Gamma3-dim1-m-definition} gives
\eqref{eq:Gamma3-dim1-m-mu}.}
\end{proof}

\section*{Data and code availability}
The GAP programs supporting the symbolic and finite-group calculations
are linked at their first use. The complete supplementary repository is
available at
\href{https://github.com/lbwheizi/rank2/tree/835a6ad3388e06c3e67cc71832f20c8c3cc4613a}
{commit \texttt{835a6ad3388e}}.
Each certificate directory contains a README file, archived output, and
a checksum manifest. These computations were performed with GAP 4.16.0.
\section*{Acknowledgements}

The author is grateful to Professor Gongxiang Liu for his guidance
and support throughout this project. The author also acknowledges
the assistance of ChatGPT in writing GAP programs to verify
 {complicated cocycle identities in the appendices}, finding the example in
Subsection~\ref{subsec:order16-G2-example}, and checking  the manuscript. The author assumes full responsibility for all the content in this article.
\bibliographystyle{plain}\small
	\bibliography{ref}

\end{document}